\documentclass[10pt]{article}

\usepackage{graphicx} % Required for inserting images
\usepackage{scrextend}
\usepackage[utf8]{inputenc}
\usepackage{mathtools}
\usepackage{placeins}
\usepackage{lipsum}
\usepackage{float}
\usepackage{verbatim}
\usepackage{tikz-cd, tikz}
\usepackage{amsmath}
\usepackage{amsfonts}
\usepackage{nccmath}
\usepackage{amsthm}
\usepackage{amssymb}
\usepackage{setspace}
\usepackage{pifont}
\usepackage{array}
\usepackage{enumitem}
\usepackage[english]{babel}
\usepackage[all]{xy}
\usepackage[autostyle]{csquotes}
\usepackage{aliascnt} 
\usepackage[
  colorlinks=true,
  urlcolor=black,
  linkcolor=blue,
  citecolor=blue, 
  linktocpage=true,
]{hyperref}
\usepackage{cleveref}
\usepackage{csquotes} % highly recommended with biblatex

\usepackage[
  backend=biber,
  style=alphabetic,
  sorting=nyt,
  maxnames=9,
  maxalphanames=9,
  backref=page
]{biblatex}
\renewbibmacro{in:}{}

\usepackage{geometry}
\usetikzlibrary{plotmarks, positioning, fit}
\usetikzlibrary{calc}

\numberwithin{thmcounter}{section}     %include section number
\newaliascnt{thmauto}{thmcounter}

\newaliascnt{Defauto}{thmcounter}

\newaliascnt{lemauto}{thmcounter}

\newaliascnt{propauto}{thmcounter}

\newaliascnt{corauto}{thmcounter}

\newaliascnt{remauto}{thmcounter}

\newaliascnt{notauto}{thmcounter}

\newaliascnt{conauto}{thmcounter}

\newaliascnt{obsauto}{thmcounter}

\newaliascnt{exauto}{thmcounter}

\addto\extrasenglish{}
\addto\extrasenglish{}
\addto\extrasenglish{}

\newcommand{\suchthat}{\ifnum\currentgrouptype=16\middle\fi|}

\makeatletter
\newcommand{\oset}[2]{%
  {\mathop{#2}\limits^{\vbox to 15\ex@{\kern-\tw@\ex@
   \hbox{\scriptsize #1}\vss}}}}
\makeatother

\newtheorem{atheorem}{Theorem}

\newtheorem{theorem}[thmauto]{Theorem}
\newtheorem{lemma}[lemauto]{Lemma}
\newtheorem{proposition}[propauto]{Proposition}
\newtheorem{corollary}[corauto]{Corollary}

\theoremstyle{definition}
\newtheorem{definition}[Defauto]{Definition}
\newtheorem{notation}[notauto]{Notation}

\theoremstyle{remark}
\newtheorem{remark}[remauto]{Remark}

\newtheorem*{claim}{\bf Claim}
\newtheorem*{cong sub}{\bf Congruence Subgroup}
\newtheorem*{st mod}{\bf Steinberg module}

\newtheorem*{outline}{\bf Outline}

\newtheorem*{acknowledgments}{\bf Acknowledgments}

\DeclareMathOperator{\SL}{SL}
\DeclareMathOperator{\GL}{GL}
\DeclareMathOperator{\Pb}{P}
\DeclareMathOperator{\B}{B}

\DeclareMathOperator{\BA}{BA}
\DeclareMathOperator{\BDA}{BDA}
\DeclareMathOperator{\BD}{BD}
\DeclareMathOperator{\C}{C}
\DeclareMathOperator{\coker}{coker}
\DeclareMathOperator{\Image}{im}
\DeclareMathOperator{\St}{St}
\DeclareMathOperator{\redhom}{\widetilde{H}}
\DeclareMathOperator{\redchain}{\widetilde{C}}
\DeclareMathOperator{\MDA}{MDA}
\DeclareMathOperator{\MD}{MD}
\DeclareMathOperator{\SB}{SB}
\DeclareMathOperator{\SBA}{SBA}
\DeclareMathOperator{\SBDA}{SBDA}
\DeclareMathOperator{\SBD}{SBD}
\DeclareMathOperator{\E}{E}
\DeclareMathOperator{\Face}{F}
\DeclareMathOperator{\D}{D}
\DeclareMathOperator{\DA}{DA}
\DeclareMathOperator{\TA}{TA}

\DeclareMathOperator{\homology}{H}

\DeclareMathOperator{\sign}{sign}
\DeclareMathOperator{\Id}{Id}
\DeclareMathOperator{\stab}{Stab}
\DeclareMathOperator{\Aut}{Aut}
\DeclareMathOperator{\rel}{rel}
\DeclareMathOperator{\Link}{Link}
\DeclareMathOperator{\oc}{\mathcal{O}}

\providecommand{\N}{\mathbb N}
\providecommand{\F}{\mathbb{F}}
\providecommand{\Z}{\mathbb Z}
\providecommand{\Q}{\mathbb Q}
\providecommand{\T}{\mathbb T}

\newcounter{cases}
\newcounter{subcases}[cases]
\newenvironment{mycases}
  {%
    \setcounter{cases}{0}%
    \setcounter{subcases}{0}%
    \def\case
      {%
        \par\noindent
        \refstepcounter{cases}%
        \textbf{Case \thecases.}
      }%
    \def\subcase
      {%
        \par\noindent
        \refstepcounter{subcases}%
        \textit{Subcase (\thesubcases):}
      }%
  }
  {%
  }
\renewcommand*\thecases{\arabic{cases}}
\renewcommand*\thesubcases{\roman{subcases}}
\numberwithin{equation}{section}

\usepackage{titlesec}
\titlespacing*{\section}
{0pt}{2ex plus 1ex minus .2ex}{1ex plus .2ex}
\titlespacing*{\subsection}
{0pt}{2ex plus 1ex minus .2ex}{1ex plus .2ex}
\titlespacing*{\subsubsection}
{0pt}{2ex plus 1ex minus .2ex}{1ex plus .2ex}
\titleformat{\paragraph}[hang]{\normalfont\normalsize\bfseries}{\theparagraph}{1em}{}[]
\titlespacing*{\paragraph}{0pt}{1em}{1em}

\usepackage[textwidth=25mm, textsize=tiny]{todonotes}

\title{Top-dimensional rational cohomology of the congruence subgroup $\Gamma_{0,n}^+(p)$}
\author{Tatiana Abdelnaim}
\date{}

\begin{document}
\maketitle

\begin{abstract}
    Let $\Gamma_{0,n}^+(p)\subset \SL_n(\Z)$ be the congruence subgroup of level-$p$ whose first column is of the form $(*,0,\dots,0)^t\bmod p$. We prove that the top-dimensional cohomology group $\homology^{\binom{n}{2}}(\Gamma_{0,n}^+(p);\Q)$ vanishes for $p\in\{2,3,5,7,13\}$ if $n \geq 3$, as well as for $p \leq 6n-14$.
    Additionally, we prove a non-vanishing result, showing that this cohomology group is nonzero for $n = 2$ for every prime $p$, and for $n=3$ for all primes $p \notin \{2,3,5,7,13\}$.
\end{abstract}
\tableofcontents

\section{Introduction}

Many problems in algebraic $K$-theory and algebraic number theory could be better understood with improved knowledge of the cohomology of arithmetic groups, such as congruence subgroups of the general linear group $\GL_n(\Z)$.

Borel--Serre \cite{BS73} proved that the virtual cohomological dimension of 
$\GL_n(\Z)$ and $\SL_n(\Z)$, and of all their finite-index subgroups, is $\binom{n}{2}$.
In particular, for any finite-index subgroup $\Gamma$, the rational cohomology
$\homology^k(\Gamma;\Q)$ vanishes for all $k > \binom{n}{2}$. On the other hand, the top degree cohomology is only known in specific cases such as $\SL_n(\Z)$ (\cite{LS}) and the principal congruence subgroup $\Gamma_n(p)$ (\cite{MPP}). 

The goal of this paper is to study the cohomology in the top possible degree for the natural family 
\begin{equation}\label{congsub}
\Gamma_{0,n}^+(p)=\left\{A\in\SL_n(\Z)\middle|A\equiv
\begin{pmatrix*}
* & * & \cdots& * \\
0 & * &\cdots&*\\
\vdots &\vdots & \ddots&\vdots \\
0 & * & \cdots & *
\end{pmatrix*}
\bmod p\right\},\end{equation} for a prime $p$. For $n=2$, the cohomology of these subgroups contains classes coming from modular forms of weight $2$, while in larger $n$ it contains classes coming from automorphic forms. In the case $n = 2$, $\Gamma_{0}^+(p)$ is commonly denoted $\Gamma_{0}(p)$, and sometimes called the \emph{Hecke congruence subgroup} of level $p$. 

Our first theorem establishes a wide-ranging vanishing result.
\begin{atheorem}\label{thA}
Let $p$ be a prime and $n\geq 3$. If $p \in\{2,3,5,7,13\}$ or $ p \leq 6n-14$, the cohomology group
    $\homology^{\binom{n}{2}}\left(\Gamma_{0,n}^+(p);\Q\right)$ vanishes. 
\end{atheorem}
Moreover, we show that this group does not vanish in certain cases for small values of $n$. This in fact leads to the following natural question: given a prime $p$, what is the largest $n$ for which this group does not vanish? We expect this question to be challenging and to involve arithmetic features of the prime $p$.
\begin{atheorem}\label{thB}
    Let $p$ be a prime. The cohomology group $\homology^{\binom{n}{2}}\left(\Gamma_{0,n}^+(p);\Q\right)$ does not vanish for $n=2$ and all primes $p$, and for $n=3$ whenever $p \notin \{2,3,5,7,13\}$.
\end{atheorem}

\begin{st mod}
A key tool in our approach is the \emph{Steinberg module} which arises from the topology of the Tits building. Let $\mathcal{T}_n(\Q)$ denote the \emph{Tits building} associated with $\GL_n(\Q)$; this is a simplicial complex whose simplices correspond to flags of proper nontrivial subspaces of $\Q^n$. By the Solomon--Tits theorem \cite{solomon-tits}, $\mathcal{T}_n(\Q)$ is homotopy equivalent to a wedge of $(n-2)$-spheres. In particular, its reduced homology is concentrated in degree $n-2$. 

The \emph{Steinberg module} $\St_n(\Q)$ is defined as \[\St_n(\Q):=\redhom_{n-2}(\mathcal{T}_n(\Q);\Z).\] 
The group $\GL_n(\Z)$ acts on $\mathcal{T}_n(\Q)$ by simplicial automorphisms, and hence the Steinberg module is naturally a $\GL_n(\Z)$-module. Subgroups of $\GL_n(\Z)$ then act on $\St_n(\Q)$ by restriction. 
\end{st mod}

The Borel–Serre duality theorem \cite{BS73} identifies high degree rational cohomology of $\Gamma_{0,n}^+(p)$ with low degree homology with coefficients in the Steinberg module:
\[\homology^{\binom{n}{2}-k}\left(\Gamma_{0,n}^+(p);\Q\right)\cong \homology_k\left(\Gamma_{0,n}^+(p);\St_n(\Q)\otimes \Q\right).\]
Let $\Gamma_{0,n}^\pm(p)$ be defined as \begin{equation}\label{congsub1}\Gamma_{0,n}^\pm(p)=\left\{A\in\GL_n(\Z)~\middle\vert\ A\equiv \begin{pmatrix*}
* & * & \cdots& * \\
0 & * &\cdots&*\\
\vdots &\vdots & \ddots&\vdots \\
0 & * & \cdots & *
\end{pmatrix*}\bmod p \right\}.\end{equation} By Borel--Serre duality, we obtain
\begin{equation}\label{rel:SL,GL}\homology^{\binom{n}{2}-k}\left(\Gamma_{0,n}^+(p);\Q\right)\cong \homology_k\left(\Gamma_{0,n}^\pm(p);\St_n(\Q)\otimes\Q\right)\oplus \homology_k\left(\Gamma_{0,n}^\pm(p);\St_n(\Q)\otimes \Q^{\det}\right),\end{equation}where $\Q^{\det}$ is $\Q$ endowed with the determinant representation.
In particular, the top cohomological dimension of the subgroup $\Gamma_{0,n}^+(p)$ is determined by the coinvariants \[\left(\St_n(\Q)\otimes\Q\right)_{\Gamma_{0,n}^\pm(p)}\quad\text{and}\quad \left(\St_n(\Q)\otimes\Q^{\det}\right)_{\Gamma_{0,n}^\pm(p)}.\] 
Understanding the vanishing or non-vanishing of these coinvariants becomes a main goal in proving \autoref{thA} and \autoref{thB}, and motivates our next set of results.

\begin{atheorem}\label{thC}
Let $p$ be a prime and $n\geq 3$. If $p\in\{2,3,5,7,13\}$ or $p \leq 6n-14$, then
    \[\left(\St_n(\Q)\otimes\Q\right)_{\Gamma_{0,n}^\pm(p)}\cong 0.\]
\end{atheorem}

\begin{atheorem}\label{thD}
   Let $p$ be a prime and $n\geq 2$. Suppose that one of the following three conditions holds\begin{itemize}
       \item $n=2$ and $p\in\{2,3,5,7,13\}$, 
       \item $n$ is odd,
       \item $p\leq 6n-8$.
   \end{itemize} Then \[\left(\St_n(\Q)\otimes\Q^{\det}\right)_{\Gamma_{0,n}^\pm(p)}\cong 0.\]
\end{atheorem}

\begin{outline}
Our approach is inspired by the work of Miller--Patzt--Putman \cite{MPP}, who studied the cohomology of principal congruence subgroups $\Gamma_n(p)$. A main tool in their method is a presentation of the Steinberg module in the sense of Church–Putman \cite{CP}, originally due to Bykovski\u i \cite{Bykovskii}, in terms of simplicial complexes built from partial bases. Because our results are rational and involve non-free group actions, we work throughout in the framework of \emph{symmetric $\Delta$-complexes}, which provide a convenient setting for forming quotients and computing rational homology.

In \autoref{subsec:symm}, we review symmetric $\Delta$-complexes and their homology. In \autoref{cpx of part}, we introduce complexes of partial frames and their determinant-1 variants. These tools are then applied in \autoref{sec:n=2}, where we treat the case $n=2$ in full detail and obtain a complete description of $\homology^1\left(\Gamma_{0}(p);\Q\right)$. We list these results for $p\leq 37$ in the table below.\\[1em]
\[
\renewcommand{\arraystretch}{1.5}
\begin{tabular}{ c|c|c|c } 

 & $p=2,3,5,7,13$ & $p=11,17,19$ & $p=23,29,31,37$ \\[1ex]
\hline
$\left(\St_2(\Q)\otimes\Q\right)_{\Gamma_{0}^\pm(p)}$ & $\Q$ & $\Q^2$ & $\Q^3$ \\[1ex]
\hline
$\left(\St_2(\Q)\otimes\Q^{\det}\right)_{\Gamma_{0}^\pm(p)}$ & $0$ & $\Q$ & $\Q^2$\\[1ex]
\hline
$\homology^1\left(\Gamma_{0}(p);\Q\right)$ & $\Q$ & $\Q^3$ & $\Q^5$\\[1ex]
\hline
\end{tabular}
\]\\[1em]
Note that $\homology^1\!\left(\Gamma_{0}^+(p);\Q\right)$ can also be calculated using methods from Section 3 of \cite{DiamondFred2016Afci}.

Let $\T_n(\Q)$ denote the poset of proper nonzero subspaces of $\Q^n$, ordered by inclusion; the Tits building $\mathcal{T}_n(\Q)$ is then its order complex. Our main strategy to prove \autoref{thA} for $n\geq 3$ is to consider the quotient map \[\T_n(\Q)\longrightarrow \Gamma_{0,n}^\pm(p)\backslash \T_n(\Q),\]which induces a homomorphism \[\St_n(\Q)\otimes \Q\cong \redhom_{n-2}(\T_n(\Q);\Q)\longrightarrow \redhom_{n-2}(\Gamma_{0,n}^\pm(p)\backslash \T_n(\Q);\Q).\]
 Since this map is $\Gamma_{0,n}^\pm(p)$-equivariant, it factors through the coinvariants, yielding \begin{equation}\label{main map} \left(\St_n(\Q)\otimes \Q\right)_{\Gamma_{0,n}^\pm(p)}\longrightarrow \redhom_{n-2}(\Gamma_{0,n}^\pm(p)\backslash \T_n(\Q);\Q).\end{equation} 

In \autoref{sec:quot}, we study the quotients obtained from the action of the congruence subgroup $\Gamma_{0,n}^\pm(p)$ on the poset $\T_n(\Q)$ and on related symmetric $\Delta$-complexes. In \autoref{sec6}, we prove high-acyclicity results for these complexes in a suitable range. In \autoref{sec7}, we prove \autoref{thD} using the structure of the relevant symmetric $\Delta$-complexes, together with the presentation of the Steinberg module described in \cite{Bykovskii,CP}. In \autoref{sec8}, we prove that \[\redhom_{n-2}(\Gamma_{0,n}^\pm(p)\backslash \T_n(\Q);\Q)\cong 0\quad\text{for $n\geq 3$},\]providing a key input for the proof of \autoref{thC}. 

 Finally, in \autoref{sec:spsq}, we discuss when the map \eqref{main map} is an isomorphism and close with the proofs of \autoref{thB} and \autoref{thC}.
\end{outline}

\begin{acknowledgments}
    I would like to thank my advisor, Peter Patzt, for suggesting this project, for his endless patience and guidance throughout this work, and for the significant time and care he devoted to reading and revising this paper. I would also like to thank Jeremy Miller and Matthew Scalamandre for helpful discussions. This work was supported by the National Science Foundation grant DMS-2405310.
\end{acknowledgments}

\section{Symmetric \texorpdfstring{$\Delta$}{Lg}-complexes} \label{subsec:symm}
In this section, we set up notation and review various definitions and properties of symmetric $\Delta$-complexes.  These complexes generalize the concept of $\Delta$-complexes by allowing symmetries of simplices to act, making them well-suited for studying spaces with group actions. Our review follows {\cite[Section 3]{CGP}}. For $k\geq 0$, we set $[k]=\{0,\dots,k\}$, and denote by $\Sigma_{k+1}$ the permutation group on $[k]$. We begin by recalling the definition of $\Delta$-complexes.

\subsection{Definitions}

\begin{definition} Let $\Delta_{{\mathrm{inj}}}$ be the category whose
objects are $[k]$ for each integer $k\geq 0$ and whose morphisms $[k] \rightarrow [l]$ are the order-preserving injective maps. 
    A \emph{$\Delta$-complex} is a functor $Y\colon \Delta_{{\mathrm{inj}}}^{\text{op}}\rightarrow \text{Sets}$. Set $Y_k=Y([k])$ to be the set of $k$-simplices of a $\Delta$-complex $Y$.
    \end{definition}
Recall that the standard-simplex $\Delta^k$ is a $k$-dimensional polytope in the space $\mathbb{R}^{k+1}$ whose vertices are the $k$-standard basis vectors $e_0,\dots,e_k$, i.e. \[\Delta^k=\left\{\sum_{i=0}^kt_ie_i~\middle\vert\ \sum_{i=0}^kt_i=1~\text{and}~ t_i\geq 0 ~\forall i \right\}\subset \mathbb{R}^{k+1}.\] 
For a finite set $S$, we define the standard-simplex \[\Delta^S=\left\{a\colon S\longrightarrow [0,1]~\middle\vert\ \sum_{s\in S}a(s)=1 \right\}.\] Given a map between finite sets $\theta\colon S\longrightarrow T$, it induces a map between simplices $\theta_*\colon\Delta^S\longrightarrow \Delta^T$ by \[\left(\theta_*a\right)(t)=\sum_{\theta(s)=t}a(s).\]
The \emph{geometric realization} $|Y|$ of a $\Delta$-complex $Y$ is the quotient space \begin{equation}\label{geomreal}\left(\coprod_{\substack{k=0}}^{\infty} Y_k\times \Delta^k\right)/ \sim\end{equation}with the equivalence relation $(\sigma,\phi_* t)\sim(\phi^*\sigma,t)$ where $\phi$ is a morphism $[k]\rightarrow [l]$ in $\Delta_{{\mathrm{inj}}}$, $\phi^*\colon Y_l\longrightarrow Y_k$, $\sigma\in Y_l$, and $t \in \Delta^k$. 

In this paper, we work with certain simplicial complexes from the literature, viewing them as $\Delta$-complexes in the following way.
\begin{definition}\label{delta}
    If $Y$ is a simplicial complex, we define a $\Delta$-complex $Y^\Delta$ by fixing a total order on the vertex set $Y_0$, such that \[Y^\Delta\colon \Delta_\mathrm{inj}^\mathrm{op}\longrightarrow \mathrm{Sets}\] sends $[k]$ to the set of order-preserving injective maps $[k]\longrightarrow Y_0$ whose image is a simplex in $Y$, with morphisms given by precomposition. The geometric realizations $|Y|$ and $|Y^\Delta|$ are homeomorphic. 
\end{definition}
\begin{definition}Let $\text{S}\Delta_{{\mathrm{inj}}}$ be the category whose objects are $[k]$ for each integer $k\geq 0$ and whose morphisms $[k] \rightarrow [l]$ are all injective maps. 
  A \emph{symmetric $\Delta$-complex} is a functor $Y\colon \text{S}\Delta_{{\mathrm{inj}}}^{\text{op}}\rightarrow \text{Sets}$. 
\end{definition}
\begin{remark}
    $\Delta_{{\mathrm{inj}}}$ is a subcategory of $\mathrm{S}\Delta_{{\mathrm{inj}}}$.
\end{remark}
The \emph{geometric realization} $|Y|$ of a symmetric $\Delta$-complex $Y$ is the quotient space \begin{equation}\label{geomreal'}\left(\coprod_{\substack{k=0}}^{\infty} Y([k])\times \Delta^k\right)/ \sim\end{equation}with the equivalence relation $(\sigma,\phi_* t)\sim(\phi^*\sigma,t)$ where $\phi$ is a morphism $[k]\rightarrow [l]$ in $\mathrm{S}\Delta_{{\mathrm{inj}}}$, $\phi^*\colon Y([l])\longrightarrow Y([k])$, $\sigma\in Y([l])$, and $t \in \Delta^k$.

We now describe how a $\Delta$-complex can be viewed as a symmetric $\Delta$-complex.
\begin{definition}\label{sdelta}
     If $Y$ is a $\Delta$-complex, we define a symmetric $\Delta$-complex  \[\mathrm{S}Y\colon \mathrm{S}\Delta_{{\mathrm{inj}}}^\mathrm{op}\longrightarrow\mathrm{Sets}\] by setting $\mathrm{S}Y([k]):=Y_k\times \Sigma_{k+1}$. For an arbitrary injective map $\phi\colon [k]\rightarrow [l]$ and $(\sigma,\pi)\in Y_l \times \Sigma_{l+1}$, define \[\mathrm{S}Y(\phi)\colon Y_l\times \Sigma_{l+1}\rightarrow Y_k\times \Sigma_{k+1}\]by $\mathrm{S}Y(\phi)(\sigma,\pi)=\left(\iota_{\phi,\pi}^*(\sigma),\rho_\phi^{-1}\circ\pi\circ\phi\right)$ where \begin{itemize}
        \item $\pi\circ\phi=\iota_{\phi,\pi}\circ\tau_{\phi,\pi}$ is a unique factorization with $\iota_{\phi,\pi}\colon[k]\rightarrow[l]$ is the order-preserving injection and $\tau_{\phi,\pi}\in\Sigma_{k+1}$.
        \item $\rho_{\phi,\pi}\colon [k]\overset{\cong}\longrightarrow\mathrm{Im}(\pi\circ\phi)\subset[l]$ is the unique order-preserving bijection.
    \end{itemize} The geometric realizations $|Y|$ and $|\mathrm{S}Y|$ are homeomorphic.
\end{definition}
   
\begin{remark}\label{deltasimp}
    It follows by \autoref{delta} and \autoref{sdelta} that for any simplicial complex $Y$, we can define a symmetric $\Delta$‑complex $\mathrm{S} Y^\Delta$ where elements in $\mathrm{S}Y^\Delta([k])$ are pairs $(\{v_0<\dots<v_k\},\pi)$. The geometric realizations $|Y|$ and $|\mathrm{S}Y^\Delta|$ are homeomorphic. 
\end{remark}
\begin{notation}\label{notation:sdelta}
    Given a simplicial complex $Y$, we write for simplicity\[\mathrm{S}Y:=\mathrm{S}Y^\Delta\]for the symmetric $\Delta$-complex associated to $Y$ via the above construction. 
\end{notation}

While the definition of symmetric $\Delta$-complexes given above is convenient for many purposes, it is often helpful to work with a hands-on description. For positive integers $k\leq l$, we let $\Sigma\left(\{k+1,\dots,l\}\right)$ denote the permutation group that fixes $\{1,\dots,k\}$ and permutes $\{k+1,\dots,l\}$.

\begin{lemma}\label{lem:sdelta}
A symmetric $\Delta$-complex $Y$ is equivalent to a sequence of sets $Y_k$ for each $k\geq 0$ with a right action of the symmetric group $\Sigma_{k+1}$, and face maps $d_i\colon Y_k \rightarrow Y_{k-1}$ for $k\geq 1$ 

\[
\begin{tikzcd}[column sep=large]
\cdots \arrow[r, shift left=3, "d_0"] \arrow[r, shift left=1] 
     \arrow[r, shift right=1] \arrow[r, shift right=3, "d_3"'] & 
Y_2\arrow[loop below,"\Sigma_3"] \arrow[r, shift left=2, "d_0"] \arrow[r] \arrow[r, shift right=2, "d_2"'] &
Y_1 \arrow[loop below,"\Sigma_2"]\arrow[r, shift left=1, "d_0"] \arrow[r, shift right=1, "d_1"'] &
Y_0\arrow[loop below,"\Sigma_1"]
\end{tikzcd}
\]
such that:\begin{itemize} 
\item (Simplicial identities)\ \[d_i\circ d_j=d_{j-1}\circ d_i\quad\text{for $i<j$},\] 
\item (Compatibility with the symmetric group action) For each $\pi \in \Sigma_{k+1}$, let $\pi^*\colon Y_k \to Y_k$ denote the map induced by the $\Sigma_{k+1}$-action on $Y_k$. Then for all $0 \leq i \leq k$,\[d_i\circ \pi^* = (\pi^{(i)})^*  \circ d_{\pi(i)},\]where $\pi^{(i)}\in \Sigma_k$ is the unique permutation satisfying $\pi\circ\delta^i=\delta^{\pi(i)}\circ\pi^{(i)}$ with $\delta^i\colon[k-1]\hookrightarrow [k]$ is the order-preserving injection whose image is $[k]\setminus\{i\}$.
\item (Triviality of permutations) Let $\delta^{0}\colon[k]\rightarrow[k+1]$. For every $\pi\in\Sigma_{k+1}$ such that \[\pi\circ\delta^0=\delta^0,\]we have \[d_0\circ\pi^*=d_0.\]\end{itemize}
\end{lemma}

For a symmetric $\Delta$-complex $Y$, we will write $Y_k=Y([k])$. By abuse of terminology, we refer to an element $\sigma\in Y_k$ as a \emph{$k$-simplex}.
\begin{definition}\label{chaincpx} Let $K$ be a field of characteristic $0$ and $Y$ be a symmetric $\Delta$-complex.
    For $k\geq -1$, the group of cellular $k$-chains $\redchain_k(Y;K)$ is \[\redchain_k(Y; K)=\begin{cases} K&\text{$k=-1$}\\ K[Y_k] \otimes_{K[\Sigma_{k+1}]} K^\mathrm{sgn}&\text{$k\geq 0$}\end{cases}\]
where $K^\mathrm{sgn}$ denotes the action of $\Sigma_{k+1}$ on $K$ via the sign.
\end{definition} 
The differential map $\partial_k \colon \redchain_k(Y;K)\rightarrow \redchain_{k-1}(Y;K)$ is given by $\sum\limits_{i=0}^k(-1)^id_i,$ where $d_i:Y_k\rightarrow Y_{k-1}$ are the face maps. Additionally, there is an augmentation map $\varepsilon\colon \redchain_0(Y;K)\longrightarrow K$ sending a $0$-simplex $\sigma$ to $1$. The homology $\redhom_k(Y;K)$ is defined to be the homology $\homology_k(\redchain_*(Y;K))$.

Chan--Galatius--Payne \cite{CGP} defined a natural transformation between $\redchain_k(Y;K)$ and $\redchain_k^{\text{sing}}(|Y|;K)$ and proved the following. 

\begin{proposition}[{\cite[Lemma 3.8]{CGP}}]\label{h}
Let $K$ be a field of characteristic $0$ and let $Y$ be a symmetric $\Delta$-complex. For $k\geq -1$,
   \[\redhom_k(Y; K) \cong \redhom_k^{\mathrm{sing}}(|Y|; K).\]
\end{proposition} 
Analogous definitions and results hold for a pair $(Y,Z)$ of symmetric $\Delta$-complexes. 

\begin{definition}
    Let $K$ be a field of characteristic $0$ and let $Y,Z$ be symmetric $\Delta$-complexes with $Z\subseteq Y$. For $k\geq 0$, the group of relative cellular $k$-chains $\C_k(Y,Z;K)$ is given by \[\redchain_k(Y;K)/\redchain_k(Z;K).\]
\end{definition}
The differential map $\partial_k\colon \C_k(Y,Z;K)\longrightarrow \C_{k-1}(Y,Z;K)$ is induced by the differential map $\partial_k \colon \redchain_k(Y;K)\rightarrow \redchain_{k-1}(Y;K)$ that takes $\redchain_k(Z;K)$ to $\redchain_{k-1}(Z;K)$. The relative homology group $\homology_k(Y,Z;K)$ is defined to be the homology of $\homology_k(\C_*(Y,Z;K))$. The following result follows from \autoref{h}.
\begin{proposition}
   Let $K$ be a field of characteristic $0$ and let $Y,Z$ be symmetric $\Delta$-complexes with $Z\subseteq Y$. For $k\geq 0$, \[\homology_k(Y,Z;K)\cong \homology_k^\mathrm{sing}(|Y|,|Z|;K).\]
\end{proposition}

\subsection{Group actions}\label{subsec:gpact}
In this section, we define how a discrete group $G$ acts on a symmetric $\Delta$-complex $Y$, ensuring that the simplicial structure of $Y$ is preserved.
\begin{definition}
    Let $Y$ be a symmetric $\Delta$-complex. An \emph{automorphism} of $Y$ is a natural isomorphism $\alpha\colon Y \longrightarrow Y$; that is, a 
    collection of bijections  $\alpha_k\colon Y_k\longrightarrow Y_k$ for all $k\geq 0$, such that for every $\phi\colon [l]\longrightarrow [k]$ in $\mathrm{S}\Delta_\mathrm{inj}$,\[\phi^*\circ \alpha_k=\alpha_{\ell}\circ\phi^*.\]
    Denote by $\Aut(Y)$ the group of automorphisms of $Y$. 
\end{definition}
A left group action of $G$ on $Y$ is a group homomorphism $\rho\colon G\longrightarrow \Aut(Y)$, where $g\cdot y:=\rho(g)y$. Throughout, all group actions are left actions, except for the right action of the symmetric group $\Sigma_{k+1}$ on each $Y_k$.

\begin{definition}\label{qut}
    Let $G$ be a group acting on a symmetric $\Delta$-complex $Y$. The symmetric $\Delta$-complex $G\backslash Y$ is defined by \[\left(G\backslash Y\right)_k:=G\backslash Y_k,\] with face maps \[G\backslash Y_k\longrightarrow G\backslash Y_{k-1}\] and symmetric group actions induced from those on $Y$.
\end{definition}

In this paper, we consider quotients of simplicial complexes by group actions where the resulting spaces are not necessarily simplicial complexes. To study their homology groups, we reinterpret the original simplicial complex as a symmetric $\Delta$-complex, a setting in which quotient constructions are better behaved. 
\begin{proposition}\label{chaincoinv}
    Let $K$ be a field of characteristic $0$ and $G$ be a group acting on a symmetric $\Delta$-complex $Y$. The group of cellular $k$-chains $\redchain_k(G\backslash Y;K)$ is isomorphic to the coinvariants $\redchain_k(Y;K)_G$.
\end{proposition}
\begin{proof}
    By \autoref{qut}, the chain complex of $G\backslash Y$ is given by \[\redchain_k\left(G\backslash Y;K\right)= K[G \backslash Y_k]\otimes_{K[\Sigma_{k+1}]} K^\mathrm{sgn}.\] Moreover, \[\redchain_k\left(Y;K\right)_G= \left(K[Y_k]\otimes_{K[\Sigma_{k+1}]} K^\mathrm{sgn}\right)_G.\] Since the actions of $G$ and the permutation group commute, we have that \[\left(K[Y_k]\otimes_{K[\Sigma_{k+1}]} K^\mathrm{sgn}\right)_G\cong K[Y_k]_G\otimes_{K[\Sigma_{k+1}]} K^\mathrm{sgn}.\] Moreover, $K[Y_k]_G\cong K[G\backslash Y_k]$, giving the desired identification of chain groups.

It remains to check compatibility with the differentials. The face maps $ Y_k\rightarrow Y_{k-1}$, which are $G$-equivariant and $\Sigma_{k+1}$ by functoriality of the symmetric $\Delta$-complex structure, induce the boundary maps 
    \[K[Y_k]\otimes_{K[\Sigma_{k+1}]} K^\mathrm{sgn}\rightarrow K[Y_{k-1}]\otimes_{K[\Sigma_{k}]} K^\mathrm{sgn}.\]These boundary maps are then $G$-equivariant. Therefore, they descend to well defined maps on coinvariants and the following diagram commutes \[ \begin{tikzcd}[column sep=22pt,row sep=20pt]  \redchain_k(Y;K)_G\arrow[r,"\Phi_k "] \arrow[d,swap,"\partial"] & \redchain_k(G(G\backslash Y;K) \arrow[d,"\partial"] \\ \redchain_{k-1}(Y;K)_G \arrow[r,"\Phi_{k-1}"] & \redchain_{k-1}(G\backslash Y;K)\end{tikzcd}\qedhere\] \end{proof}
   We record the following result for later use in \autoref{cpx of part}.
    \begin{lemma}\label{lem:Sy}
    Let $G$ be a group acting on a simplicial complex $Y$, such that $G\backslash Y$ is again a simplicial complex. Then \[G\backslash SY\cong S\left(G\backslash Y\right).\]
\end{lemma}
\begin{proof}
   We have by \autoref{sdelta} and \autoref{qut} \[\left(\mathrm{S}Y\right)_k=Y_k\times \Sigma_{k+1}\quad\text{and}\quad \left(\mathrm{S}(G\backslash Y)\right)_k=G\backslash Y_k\times\Sigma_{k+1}.\] 
   Thus, we have a bijection \[G\backslash \left(\mathrm{S}Y\right)_k\rightarrow \mathrm{S}\left(G\backslash Y\right)_k.\] 
   The result then follows by \autoref{lem:sdelta}, since the bijection is compatible with the face maps and the symmetric group actions.
\end{proof}
\subsection{Orientation-preserving and reversing actions}
Fix a field $K$ of characteristic $0$. We will review the notions of orientation-preserving and orientation-reversing elements, and we will compute a basis of $\redchain_k(G\backslash Y;K)$.

\begin{definition} \label{def:or} Let $G$ be a group acting on a symmetric $\Delta$-complex $Y$. Define the following.
\begin{itemize}
\item Let $Y_k^\mathrm{rv}$ be the set of all $k$-simplices $\sigma$ of $Y$ on which the action of $G$ is \emph{orientation-reversing}; that is, there exist an element $g\in G$ and a permutation $\pi\in \Sigma_{k+1}$ with $\sign(\pi)=-1$ such that $g\cdot\sigma = \sigma\cdot \pi$. 
\item  Let $Y_k^\mathrm{pr}$ be the set of all $k$-simplices $\sigma$ of $Y$ on which the action of $G$ is \emph{orientation-preserving}; that is \[Y_k^\mathrm{pr}=Y_k\setminus Y_k^\mathrm{rv}.\]
\end{itemize}\end{definition} 

\begin{remark}\label{rmk:triv} If $\sigma\in Y_k^\mathrm{pr}$, the stabilizer of $G\sigma$ in $\Sigma_{k+1}$, \[\stab_{\Sigma_{k+1}}\left(G\sigma\right)=\left\{\pi\in \Sigma_{k+1}\mid g\cdot\sigma=\sigma\cdot\pi~\text{for some $g\in G$} \right\}\]is contained in the alternating group $A_{k+1}$. While if $\sigma\in Y_k^\mathrm{rv}$, \[\stab_{\Sigma_{k+1}}(G\sigma)\cap\left( \Sigma_{k+1}\setminus A_{k+1}\right)\neq \emptyset.\] 
\end{remark}
We now prove that the actions of $G$ and $\Sigma_{k+1}$ on $Y_k$ descend to actions on each of $Y_k^\mathrm{rv}$ and $Y_k^\mathrm{pr}$.
\begin{proposition}
   Let $G$ be a group acting on a symmetric $\Delta$-complex $Y$ and let $k\geq 0$. Then the actions of $\Sigma_{k+1}$ and $G$ descend to actions on both $Y_k^\mathrm{pr}$ and $Y_k^\mathrm{rv}$. 
\end{proposition}
\begin{proof}
Since $G$ and $\Sigma_{k+1}$ act on $Y_k$, it is enough to prove that $G$ and $\Sigma_{k+1}$ act on $Y_k^\mathrm{rv}$.   Let $\sigma\in Y_k^\mathrm{rv}$. Then there exist $g\in G$ and $\pi\in \Sigma_{k+1}$ such that $g\cdot \sigma=\sigma\cdot \pi$ with $\sign(\pi)=-1$. Let $h\in G$ and $\tau\in \Sigma_{k+1}$. We have \[(hgh^{-1}) (h\sigma)=(h\sigma)\pi \quad\text{with $hgh^{-1}\in G$},\] and\[g (\sigma\tau)=(\sigma\tau)(\tau^{-1}\pi\tau)\quad\text{with $\sign(\tau^{-1}\pi\tau)=\sign(\pi)=-1$}.\] Thus $h\sigma$ and $\sigma\tau$ are in $Y_k^\mathrm{rv}$. 
\end{proof}
\begin{proposition} \label{coinv} Let $G$ be a group acting on a symmetric $\Delta$-complex $Y$. Then for $k\geq 1$, \[\redchain_k(G\backslash Y ; K)\cong K\left[G\backslash Y_k\right]\otimes_{K[\Sigma_{k+1}]}K^\mathrm{sgn},\]with the following relations:\begin{itemize}
    \item $G\sigma\otimes 1=G\sigma\pi\otimes\sign(\pi)$ if $\sigma\in Y_k$ and $\pi\in\Sigma_{k+1},$
     \item $G\sigma\otimes 1=0$ if $\sigma\in Y_k^\mathrm{rv}$.
\end{itemize} In particular, \[\redchain_k(G\backslash Y ; K)\cong K\left[G\backslash Y_k^\mathrm{pr}\right]\otimes_{K[\Sigma_{k+1}]}K^\mathrm{sgn}.\]Moreover, $\redchain_k(G\backslash Y ; K)$ has a non-canonical basis bijective to $G\backslash Y_k^\mathrm{pr}/\Sigma_{k+1}$, given by a choice of orientation for every simplex. 
     \end{proposition}
     See {\cite[Lemma 3.10]{CGP}}.
     \begin{proof} We have the isomorphism \[ \redchain_k(G\backslash Y;K) = \left(K[Y_k]\otimes K^\mathrm{sgn}\right)_G\cong K[G\backslash Y_k]\otimes K^\mathrm{sgn}.\]  The tensor product over $K[\Sigma_{k+1}]$ imposes the relations\[
G\sigma\pi\otimes 1=
G\sigma\otimes \sign(\pi)\]
for all $\sigma\in Y_k$ and $\pi\in \Sigma_{k+1}$. We now show the second relation. Let $\sigma\in Y_k^\mathrm{rv}$. Then there exist $g \in G$ and $\pi\in\Sigma_{k+1}$ such that $g\cdot\sigma=\sigma\cdot \pi$ with $\sign(\pi)=-1$. This implies that \[G\sigma\otimes 1=Gg\cdot \sigma\otimes 1=G\sigma \cdot \pi\otimes 1=G\sigma\otimes -1=-\left(G\sigma\otimes 1\right),\]which gives that \[G\sigma\otimes 1=0.\]Thus\begin{align*} \redchain_k(G\backslash Y;K) &\cong
     \left(K[G\backslash Y_k^\mathrm{pr}]\otimes_{K[\Sigma_{k+1}]}K^{\mathrm{sgn}}\right)\quad\oplus\quad \left(K[G\backslash Y_k^\mathrm{rv}]\otimes_{K[\Sigma_{k+1}]}K^{\mathrm{sgn}}\right)\\&\cong K[G\backslash Y_k^\mathrm{pr}]\otimes_{K[\Sigma_{k+1}]}K^{\mathrm{sgn}}.\end{align*}
 Additionally, the $\Sigma_{k+1}$-stabilizers act trivially on $K^\mathrm{sgn}$. So \[ K[G\backslash Y_k^\mathrm{pr}]\otimes_{K[\Sigma_{k+1}]}K^{\mathrm{sgn}}\cong K\left[G\backslash Y_k^{\mathrm{pr}}/\Sigma_{k+1}\right].\]Note that this is a non-canonical isomorphism, where a choice of basis of $K[G\backslash Y_k^\mathrm{pr}]\otimes_{K[\Sigma_{k+1}]}K^{\mathrm{sgn}}$ is equivalent to a choice of orientation for every simplex in $G\backslash Y_k^{\mathrm{pr}}/\Sigma_{k+1}.$ 
    \end{proof}

Thus $\redhom_k\left(G\backslash Y;K\right)$ may be calculated from a chain complex with one generator for each element in a set of representatives for orbits of elements whose orientations are preserved by the group $G$.

\subsection{Twisted and untwisted actions}
Fix a field $K$ of characteristic $0$. We will introduce the twisted and the untwisted actions related to a character (i.e. a group homomorphism) $\chi\colon G\longrightarrow K^\times$ of a group $G$. Our main examples will be subgroups of $\GL_n(\Z)$ with $\chi=\det$.

\begin{definition}\label{def:tw}
Let $G$ be a group acting on a symmetric $\Delta$-complex $Y$. Define the following.\begin{itemize}
    \item We say a simplex $\sigma$ is \emph{twisted} if there exist an element $g\in G$ and a permutation $\pi\in\Sigma_{k+1}$ with 
$\chi(g)\sign(\pi)\neq 1$ such that $g\cdot\sigma=\sigma\cdot \pi$. Let $Y_k^\mathrm{tw}$ denote the set of all twisted simplices. 
    \item We say a simplex $\sigma$ is \emph{untwisted} if $\sigma\in Y_k\setminus Y_k^\mathrm{tw}$. Let $Y_k^\mathrm{utw}=Y_k\setminus Y_k^\mathrm{tw}$ denote the set of all untwisted simplices. 
\end{itemize}
\end{definition}
\begin{comment}\begin{remark} If $\sigma\in Y_k^\mathrm{utw}$, the stabilizer of $\sigma$ in $\Sigma_{k+1}$, \[\stab_{\Sigma_{k+1}}\left(\sigma\right)=\left\{\pi\in \Sigma_{k+1}\mid \sigma=\sigma\cdot \pi~\text{for some $\pi\in \Sigma_{k+1}$}\right\}=\left\{\pi\in \Sigma_{k+1}\mid \Id\cdot\sigma=\sigma\cdot \pi~\text{for some $\pi\in \Sigma_{k+1}$}\right\}\]is contained in the alternating group $A_{k+1}$ since $\chi(\Id)=1$ While if $\sigma\in Y_k^\mathrm{tw}$, \[\stab_{\Sigma_{k+1}}(\sigma)\cap\left(\Sigma_{k+1}\backslash A_{k+1}\right)\neq \emptyset.\]
\end{remark}\end{comment}
We will prove that the actions of $G$ and $\Sigma_{k+1}$ on $Y_k$ induce well-defined actions on $Y_k^\mathrm{tw}$ and $Y_k^\mathrm{utw}$.
\begin{proposition}
   Let $G$ be a group acting on a symmetric $\Delta$-complex $Y$ and let $k\geq 0$. Then the actions of $\Sigma_{k+1}$ and $G$ descend to actions on both $Y_k^\mathrm{tw}$ and $Y_k^\mathrm{utw}$. 
\end{proposition}
\begin{proof}
Since $G$ and $\Sigma_{k+1}$ act on $Y_k$, it is enough to prove that $G$ and $\Sigma_{k+1}$ act on $Y_k^\mathrm{tw}$.  
    Let $\sigma\in Y_k^\mathrm{tw}$. Then there exist $g\in G$ and $\pi\in \Sigma_{k+1}$ such that $g\cdot \sigma=\sigma\cdot \pi$ with $\chi(g)\sign(\pi)\neq 1$. Let $h\in G$ and $\tau\in \Sigma_{k+1}$. We have \[(hgh^{-1}) (h\sigma)=(h\sigma)\pi, \]with $hgh^{-1}\in G$ and \[\chi(hgh^{-1})\sign(\pi)=\chi(g)\sign(\pi)\neq 1.\]Moreover, \[g(\sigma\tau)=(\sigma\tau)(\tau^{-1}\pi\tau),\]with \[\chi(g)\sign(\tau^{-1}\pi\tau)=\chi(g)\sign(\pi)\neq 1.\] Thus, $h\sigma$ and $\sigma\tau\in Y_k^\mathrm{tw}$. 
\end{proof}
\begin{definition}
 For a $K$-vector space $V$, we denote by $V^\chi$, $V$ endowed with the $K[G]$-module structure arising from the action \[g\cdot v=\chi(g)\cdot v\quad\text{for $g\in G$ and $v\in V$}.\]
\end{definition}
\begin{proposition}\label{coinv,tw}
      Let $G$ be a group acting on a symmetric $\Delta$-complex $Y$. Then for $k\geq 0$, \[\left(\redchain_k(Y;K)\otimes K^{\chi}\right)_G=\left(K[Y_k]\otimes_{K[\Sigma_{k+1}]}K^{\mathrm{sgn}}\otimes K^{\chi}\right)_G,\]with the following relations:\begin{itemize}
    \item $g\cdot \sigma\cdot \pi\otimes 1\otimes1=\sigma\otimes\sign(\pi)\otimes\chi(g)$ for $\sigma\in Y_k$, $\pi\in\Sigma_{k+1},$ and $g\in G$,
     \item $ \sigma\otimes 1\otimes 1=0$ if $\sigma\in Y_k^\mathrm{rv}$.
\end{itemize} In particular, \[\left(\redchain_k( Y ; K)\otimes K^\chi\right)_G\cong \left(K\left[ Y_k^\mathrm{utw}\right]\otimes_{K[\Sigma_{k+1}]}K^{\mathrm{sgn}}\otimes K^{\chi}\right)_G.\]Moreover, $\left(\redchain_k( Y ; K)\otimes K^\chi\right)_G$ has a non-canonical basis bijective to $G\backslash Y_k^\mathrm{utw}/\Sigma_{k+1}$, given by a choice of twisting for every simplex.
\end{proposition}
 \begin{proof} We have by definition\[\left(\redchain_k(Y;K)\otimes K^{\chi}\right)_G =\left(K[Y_k]\otimes_{K[\Sigma_{k+1}]}K^{\mathrm{sgn}}\otimes K^{\chi}\right)_G.\] The tensor product over $K[\Sigma_{k+1}]$ imposes the relations\[
g\cdot \sigma\cdot \pi\otimes 1\otimes1=\sigma\otimes\sign(\pi)\otimes\chi(g)\]
for all $\sigma\in Y_k$, $\pi\in \Sigma_{k+1}$ and $g\in G$.
    Let $\sigma\in Y_k^\mathrm{tw}$. Then, there exist $g\in G$ and $\pi\in \Sigma_{k+1}$ such that $g\cdot \sigma=\sigma\cdot \pi$ with $\chi(g)\sign(\pi)\neq1$. Thus, we have the following in $\left(K[Y_k^{\text{tw}}]\otimes_{K[\Sigma_{k+1}]}K^{\mathrm{sgn}}\otimes K^{\chi}\right)_G$:
\[\sigma\otimes 1\otimes 1=g\cdot\left(\sigma\otimes1\otimes 1\right)=g\cdot \sigma \otimes 1\otimes \chi(g)=\sigma\cdot \pi\otimes 1\otimes \chi(g)=\sigma\otimes\sign(\pi)\otimes \chi(g)=\chi(g)\sign(\pi)\left(\sigma\otimes1\otimes 1\right).\]
This leads to \[ \left(1-\chi(g)\sign(\pi)\right)\sigma\otimes 1\otimes 1=0.\] As $\chi(g)\sign(\pi)\neq 1$, we conclude that \[\sigma\otimes 1\otimes 1=0.\] Thus $\left(K[Y_k^{\text{tw}}]\otimes_{K[\Sigma_{k+1}]}K^{\mathrm{sgn}}\otimes K^{\chi}\right)_G=0$. Moreover, \begin{align*}\left(\redchain_k(Y;K)\otimes K^{\chi}\right)_G &\cong \left(K[Y_k^{\text{utw}}]\otimes_{K[\Sigma_{k+1}]}K^{\mathrm{sgn}}\otimes K^{\chi}\right)_G\quad\oplus\quad \left(K[Y_k^{\text{tw}}]\otimes_{K[\Sigma_{k+1}]}K^{\mathrm{sgn}}\otimes K^{\chi}\right)_G\\&\cong\left(K[Y_k^{\text{utw}}]\otimes_{K[\Sigma_{k+1}]}K^{\mathrm{sgn}}\otimes K^{\chi}\right)_G.\end{align*}
 Additionally, $G\times\Sigma_{k+1}$ acts on $Y_k^\mathrm{utw}$ via the action $\sigma\cdot(g,\pi)=g^{-1}\cdot\sigma\cdot\pi $, and acts on $K^\chi\otimes K^\mathrm{sgn}$ via $(g,\pi)\cdot 1\otimes1=\chi(g)\otimes\sign(\pi)$. Since the $G\times\Sigma_{k+1}$-stabilizers act trivially on $K^\chi\otimes K^\mathrm{sgn}$, we obtain \[\left(K[Y_k^{\text{utw}}]\otimes_{K[\Sigma_{k+1}]}K^{\mathrm{sgn}}\otimes K^{\chi}\right)_G\cong K[Y_k^\mathrm{utw}]\otimes_{K[G\times\Sigma_{k+1}]}(K^\chi\otimes K^\mathrm{sgn})\overset{\star}\cong K\left[G\backslash Y_k^{\mathrm{utw}}/\Sigma_{k+1}\right].\]Note that $\star$ is a non-canonical isomorphism, where a choice of basis of $\left(K[Y_k^{\text{utw}}]\otimes_{K[\Sigma_{k+1}]}K^{\mathrm{sgn}}\otimes K^{\chi}\right)_G$ is equivalent to a choice of twisting for each simplex in $G\backslash Y_k^{\mathrm{utw}}/\Sigma_{k+1}.$\qedhere

\end{proof}
\section{Complexes of partial frames}\label{cpx of part}
Fix a Euclidean domain $R$. In this section, we consider simplicial complexes associated to a free $R$-module and their quotients. Via \autoref{deltasimp}, we treat these complexes as symmetric $\Delta$-complexes, as this approach is well-suited for studying quotients that are no longer simplicial complexes and for computing their rational homology.

 \begin{definition} \
\begin{itemize}
\item A \emph{partial basis} for $R^n$ is a set $\left\{v_1,\dots,v_{k}\right\}$ that can be completed to a basis for $R^n$. 
    \item An \emph{augmented basis} for $R^n$ is a set $\{v_0, v_1, \hdots, v_n\}$ of vectors of $R$ such that $\{v_1, \hdots, v_n\}$ is a basis for $R$ and $\pm v_0\pm v_1 \pm v_2 = 0$ for some choice of signs.
    \item A \emph{partial augmented basis} for $R^n$ is a set of vectors that is either a basis or an augmented basis for a direct summand of $R$.\end{itemize}
\end{definition}

\begin{definition}
    A partial basis $\{v_1, \hdots, v_k\}$ for $R^n$ is a \emph{determinant-$1$ partial basis} if it satisfies the following conditions:
     \begin{itemize}
        \item If $k=n$,  the matrix $\left(v_1| v_2| \dots| v_n\right)$ whose columns are the vectors $v_i$ has determinant equal to either $1$ or $-1$. Note that this is independent of the order of the vectors $v_i$.
        \item If $k<n$, no additional condition is needed.
    \end{itemize}
\end{definition}

\begin{definition}
    A partial augmented basis $\{v_0, v_1, \hdots, v_k\}$ for $R^n$ with $\pm v_0\pm v_1\pm v_2=0$ is a \emph{determinant-$1$ partial augmented basis} if $\{v_1, \hdots, v_k\}$ is a determinant-$1$ partial basis.
\end{definition}
We now introduce the $\pm$-vectors and we make analogous definitions.
\begin{definition} \ \begin{itemize}
 \item A \emph{$\pm$-vector} in $R^n$ is a $2$-element set $\{v, -v\}$ of primitive vectors in $R^n$ denoted by $v^{\pm}$. 
   \item A \emph{frame} for $R^n$ is a set $\{v_1^{\pm}, \hdots, v_n^{\pm}\}$ such that $\{v_1, \hdots, v_n\}$ is a basis for $R^n$.
    \item A \emph{partial frame} for $R^n$ is a set $\{v_1^{\pm}, \hdots, v_k^{\pm}\}$ such that $\{v_1, \hdots, v_k\}$ is a partial basis for $R^n$. \end{itemize}
Define the \emph{complex of partial frames for $R^n$}, denoted $\B(R^n)$, to be the $(n-1)$-dimensional simplicial complex whose $k$-simplices are partial frames for $R^n$ of cardinality $(k+1)$. 
\end{definition}
\begin{definition} \
 \begin{itemize}
    \item An \emph{augmented frame} for $R^n$ is a set $\{v_0^{\pm}, v_1^{\pm}, \hdots, v_n^{\pm}\}$ of $\pm$-vectors of $R$ such that $\{v_0, \hdots, v_n\}$ is an augmented basis for $R^n$.
    \item A \emph{partial augmented frame} for $R^n$ is a set $\{v_0^{\pm}, v_1^{\pm}, \hdots, v_k^{\pm}\}$ of $\pm$-vectors of $R$ such that $\{v_0, v_1 \hdots, v_k\}$ is a partial augmented basis for $R^n$.
\end{itemize}
Define the \emph{complex of partial augmented frames for $R^n$}, denoted $\BA_n(R)$, to be the $n$-dimensional simplicial complex whose $k$-simplices are partial augmented frames for $R^n$ of cardinality ($k+1$).
\end{definition}

\begin{definition}
  A \emph{determinant-$1$ partial frame} is a set $\{v_1^\pm,\dots,v_k^\pm\}$ such that $\{v_1,\dots,v_k\}$ is a determinant-$1$ partial basis for $R^n$.
  
Define \emph{the complex of determinant-$1$ partial frames} for $R^n$, denoted $\BD(R^n)$, to be the $(n-1)$-dimensional simplicial complex whose $k$-simplices are determinant-$1$ partial frames for $R^n$. This definition is independent of the choice of representatives $v_i$.
\end{definition}
\begin{definition} A \emph{determinant-$1$ partial augmented frame} is a set $\{v_0^\pm,v_1^\pm,\dots,v_k^\pm\}$ such that $\{v_0,v_1,\dots,v_k\}$ is a determinant-$1$ partial augmented basis for $R^n$.

Define \emph{the complex of determinant-$1$ partial augmented frames} for $R^n$, denoted $\BDA(R^n)$, to be the $n$-dimensional simplicial complex whose $k$-simplices are determinant-$1$ partial augmented frames for $R^n$. 
\end{definition}
\begin{definition}
    We define a left $\GL_n(R)$-action on $\BA_n(R)$, as follows. Let $A\in \GL_n(R),$ \[A\cdot \{v_0^\pm,v_1^\pm,\dots,v_k^\pm\} = \{(A\cdot v_0)^\pm,(A\cdot v_1)^\pm,\dots,(A \cdot v_k)^\pm\}.\]
\end{definition} 
It is important to note that $\B(R^n), \BD(R^n)$, and $ \BDA(R^n)$ are subcomplexes of $\BA_n(R)$ and they remain invariant under the $\GL_n(R)$-action. 

In \cite{CP}, Church and Putman use the complexes of partial (augmented) frames over $\Z$ to find a presentation of $\St_n(\Q)$. This presentation, originally proven by Bykovski\u i \cite{Bykovskii}, phrased in terms of these complexes is the following.

\begin{proposition} \label{resol}
For all $n\geq 1$,
\begin{equation}\label{eq:resol}\redchain_n(\BA(\Z^n);\Q) \rightarrow \redchain_{n-1}(\B(\Z^n);\Q)\rightarrow  \St_n(\Q)\otimes \Q \rightarrow 0\end{equation} is a flat $\SL_n(\Z)$-resolution of $\St_n(\Q)\otimes \Q$.
\end{proposition} 
We recall that our goal is to study the coinvariants of the $\St_n(\Q)\otimes\Q$ under the congruence subgroup  \[\Gamma_{0,n}^\pm(p)=\left\{A\in\GL_n(\Z)~\middle\vert\ A\equiv \begin{pmatrix*}
* & * & \cdots& * \\
0 & * &\cdots&*\\
\vdots &\vdots & \ddots&\vdots \\
0 & * & \cdots & *
\end{pmatrix*}\bmod p\right\}.\] To this end, we study the action of $\Gamma_{0,n}^\pm(p)$ on the complex of partial (augmented) frames, $\B(\Z^n)$ and $\BA(\Z^n)$. Since their quotients by $\Gamma_{0,n}^\pm(p)$ are no longer simplicial complexes, we instead view them as symmetric $\Delta$-complexes, to better understand their chain complexes.

We denote by $\SB(R^n), \SBA(R^n), \SBD(R^n)$, and $\SBDA(R6n)$ the symmetric $\Delta$-complexes associated to $\B(R^n), \BA(R^n), \BD(R^n)$, and $\BDA(R^n)$, respectively. We also recall that the geometric realizations of these symmetric $\Delta$-complexes are homeomorphic to those of the corresponding simplicial complexes. As noted in \autoref{deltasimp}, simplices in these complexes are given by pairs $(\sigma, \pi)$, where $\sigma$ is an ordered partial (augmented) frame $\{v_0^\pm<\dots <v_k^\pm\}$ and $\pi$ is a permutation. For simplicity, we will denote our pairs by tuples of the form $(v_0^\pm, \dots, v_k^\pm)$.

These complexes admit a $\GL_n(R)$-action, which respects the symmetric $\Delta$-complex structure. Reformulating \eqref{eq:resol} in terms of the corresponding symmetric $\Delta$-complexes gives the following presentation
\begin{equation}\label{STresol}\redchain_n(\SBA(\Z^n);\Q) \rightarrow \redchain_{n-1}(\SB(\Z^n);\Q)\rightarrow  \St_n(\Q)\otimes\Q \rightarrow 0\end{equation} of $\St_n(\Q)\otimes \Q$.
\begin{definition}
Let $p$ be a prime. Let $\Pb_n^\pm(\F_p)$ denote the subgroup of $\GL_n(\F_p)$ defined as
 \begin{align*}\Pb_n^\pm(\F_p)&=\left\{A=\begin{pmatrix*}
* & * & \cdots& * \\
0 & * &\cdots&*\\
\vdots &\vdots & \ddots&\vdots \\
0 & * & \cdots & *
\end{pmatrix*}\in\GL_n(\F_p)~\middle\vert\ \det(A)=\pm 1\right\}\\&=\left\{A\in \GL_n(\F_p)\mid \det(A)=\pm 1 ~\text{and}~A\cdot e_1=\lambda e_1~\text{for some $\lambda\in\F_p^\times$} \right\}\end{align*}where $e_1\in\F_p^n$ is the first standard basis vector. Note then that $\Pb_n^\pm(\F_p)$ stabilizes the line spanned by $e_1$.
\end{definition}
 With this definition, we have the following short exact sequence 
     \begin{equation} \label{ses} 1 \longrightarrow \Gamma_n(p) \longrightarrow \Gamma_{0,n}^\pm(p) \longrightarrow  \Pb_n^\pm(\F_p) \longrightarrow 1 
\end{equation} where $\Gamma_n(p)$ is the kernel of the surjection $\SL_n(\Z)\rightarrow \SL_n(\F_p)$. 

Miller--Patzt--Putman \cite{MPP} proved the following two results.

\begin{lemma}[{\cite[Lemma 2.35]{MPP}}]\label{lem1} For a prime $p$, we have $\Gamma_n(p)\backslash\B(\Z^n) \cong \BD(\F_p^n)$ for all $n \geq 1$.
\end{lemma}

\begin{lemma}[{\cite[Lemma 2.43]{MPP}}]\label{lem2} For a prime $p$, we have $\Gamma_n(p)\backslash\BA(\Z^n) \cong \BDA(\F_p^n)$ for all $n \geq 2$.
\end{lemma} 

It then follows by \eqref{ses}, that we can describe the quotients $\Gamma_{0,n}^\pm(p)\backslash\SB(\Z^n)$ and $\Gamma_{0,n}^\pm(p)\backslash\SBA(\Z^n)$ as in following corollary.

\begin{corollary}\label{good q} For all primes $p$ and $n\geq 1$, 
    \[\Gamma_{0,n}^\pm(p)\backslash \SB(\Z^n)\cong \Pb_n^\pm(\F_p)\backslash\SBD(\F_p^n)\]and\[\Gamma_{0,n}^\pm(p)\backslash \SBA(\Z^n)\cong \Pb_n^\pm(\F_p)\backslash\SBDA(\F_p^n).\]
\end{corollary}
\begin{proof}
  It follows from \autoref{lem:Sy} the isomorphisms  \[\Gamma_n(p)\backslash \SB(\Z^n)\cong  \SBD(\F_p^n),\]and  \[\Gamma_n(p)\backslash \SBA(\Z^n)\cong  \SBDA(\F_p^n).\]
    Therefore, using \eqref{ses}, we conclude that  \[\Gamma^\pm_{0,n}(p)\backslash \SB(\Z^n)\cong \Pb_n^\pm(\F_p)\backslash\left(\Gamma_n(p)\backslash \SB(\Z^n)\right)\cong  \Pb_n^\pm(\F_p)\backslash \SBD(\F_p^n),\]and similarly,\[\Gamma^\pm_{0,n}(p)\backslash \SBA(\Z^n)\cong \Pb_n^\pm(\F_p)\backslash\left(\Gamma_n(p)\backslash \SBA(\Z^n)\right)\cong  \Pb_n^\pm(\F_p)\backslash \SBDA(\F_p^n).\qedhere\]\end{proof}
We now introduce two sets of matrices, which will be used to describe combinatorial data associated with the complexes $\SBD(\F_p^n)$ and $\SBDA(\F_p^n)$. For vectors $v_1,\dots,v_k\in\F_p^n$, we write \[\left(v_1|\dots|v_k\right)\]for the matrix whose columns are the vectors $v_1,\dots,v_k$.
\begin{definition} \label{matrix}\
    \begin{itemize}
        \item Let $k\leq n-1$. Define \[\MD(\F_p^n)_k=\left\{(v_0|\dots|v_k)~\middle\vert\begin{array}{c}v_0,\dots,v_k~\text{are linearly independent vectors in $\F_p^n$}, \\\text{and}~\det\left(v_0|\dots|v_k\right)
=\pm 1 ~\text{if $k=n-1$}\end{array}\right\}.\]
        \item  Let $2\leq k\leq n$. Define \[\MDA(\F_p^n)_k=\left\{(v_0|\dots| v_k)~\middle\vert\begin{array}{c}v_1,\dots,v_k~\text{are linearly independent vectors in $\F_p^n$}, \\v_0+v_1+v_2=0,~\text{and}~\det\left(v_1|\dots| v_k\right)
=\pm 1 ~\text{if $k=n$}\end{array}\right\}.\]
    \end{itemize}
\end{definition}
We will define the symmetry groups of $\MD(\F_p^n)_k$ and $\MDA(\F_p^n)_k$.

\begin{definition}
    For $k \in\N$. Set \[T_k=\Sigma_{k+1}\ltimes \{-1,1\}^{k+1}.\] For $k\geq 2$, set \[G_k=\Sigma_{\{0,1,2\}}\times \{-1,1\}\times\left(\Sigma_{\{3,\dots,k\}}\ltimes \{-1,1\}^{k-2}\right).\]Note that $T_0=\{-1,1\}$ and $G_2=\Sigma_3\times \{-1,1\}$.
\end{definition}
We write elements in $T_k$ and $G_k$ in the form $X=(\pi,\varepsilon_0,\dots,\varepsilon_k)$ and $X=(\pi,\varepsilon,\tau,\varepsilon_3,\dots,\varepsilon_k)$, respectively. $T_k$ is a group with multiplication given by \[(\pi,\varepsilon_0,\dots,\varepsilon_k)\cdot(\pi',\varepsilon_0',\dots,\varepsilon_k')=(\pi\pi',\varepsilon_{\pi'(0)}\varepsilon'_0,\dots,\varepsilon_{\pi'(k)}\varepsilon_k')\]and $G_k$ is a group with multiplication defined by  \[(\pi,\varepsilon,\tau,\varepsilon_3,\dots,\varepsilon_k)\cdot(\pi',\varepsilon',\tau',\varepsilon_3,',\dots,\varepsilon_k')=(\pi\pi',\varepsilon\varepsilon',\tau\tau',\varepsilon_{\tau'(3)}\varepsilon_3',\dots,\varepsilon_{\tau'(k)}\varepsilon_k').\]

We next define an action of the symmetry groups $T_k$ and $G_k$ on the matrix sets $\MD(\F_p^n)_k$ and $\MDA(\F_p^n)_k$, respectively. This will allow us to describe the simplices of our symmetric $\Delta$-complexes as orbits under these group actions.
\begin{definition}
    Let $p$ be a prime and $n\geq 1$. For all $0\leq k\leq n-1$, the group $T_k$ acts on $\MD(\F_p^n)_k$ as follows: for $B=(v_0|\dots|v_k)\in \MD(\F_p^n)_k$ and $X=(\pi,
    \varepsilon_0|\dots | \varepsilon_k)\in T_k$, we define \[B\cdot X=\left(\varepsilon_0v_{\pi(0)}\mid\dots\mid\varepsilon_kv_{\pi(k)}\right).\]For all $2\leq k\leq n$, the group $G_k$ acts on $\MDA(\F_p^n)_k$ as follows: for $B=(v_0|\dots|v_k)\in \MDA(\F_p^n)_k$ and $X=(\pi,\varepsilon,\tau,
    \varepsilon_3,\dots,\varepsilon_k)\in G_k$, we define \[B\cdot X=\left(\varepsilon v_{\pi(0)},\varepsilon v_{\pi(1)},\varepsilon v_{\pi(2)},\varepsilon_3v_{\tau(3)},\dots,\varepsilon_kv_{\tau(k)}\right).\]\end{definition}
    
\begin{lemma} \label{matrixrep} Let $p$ be a prime and $n\geq 2$. For $k\geq 0$, \[\SBD(\F_p^n)_k/\Sigma_{k+1} \cong \MD(\F_p^n)_k / T_k,\]\[\Q\left[\Pb_n^\pm(\F_p)\backslash \SBD(\F_p^n)_k\right]\otimes_{\Q[\Sigma_{k+1}]}\Q^\mathrm{sgn} \cong\Q\left[\Pb_n^\pm(\F_p)\backslash \MD(\F_p^n)_k\right]\otimes_{\Q[T_k]}\Q^\mathrm{sgn}\]and for $k\geq 2$,\[  \left(\SBDA(\F_p^n)_k\setminus \SBD(\F_p^n)_k\right)/\Sigma_{k+1} \cong \MDA(\F_p^n)_k / G_k,\]\[\Q\left[\Pb_n^\pm(\F_p)\backslash\left( \SBDA(\F_p^n)_k-\SBD(\F_p^n)_k\right)\right]\otimes_{\Q[\Sigma_{k+1}]}\Q^\mathrm{sgn}\cong \Q\left[\Pb_n^\pm(\F_p)\backslash \MDA(\F_p^n)_k\right]\otimes_{\Q[G_k]}\Q^\mathrm{sgn}.\]
\end{lemma}

\begin{proof}
We recall that elements in $\SBD(\F_p^n)_k$ are ordered determinant-$1$ partial frames $(v_0^\pm,\dots,v_k^\pm)$ where $v_{i}^\pm=\{v_i,-v_i\}$. Since $T_k\cong \{\pm\}^{k+1}\ltimes \Sigma_{k+1}$, we conclude the isomorphism \[\SBD(\F_p^n)_k/\Sigma_{k+1} \cong \MD(\F_p^n)_k / T_k.\]
As for the second isomorphism, let $\sigma\in \SBD(\F_p^n)_k$. Choose an element $B\in\MD(\F_p^n)_k$ and $\pi\in \Sigma_{k+1}$ such that $\sigma\cdot \pi=B\cdot X$ for some $X\in T_k$ with $\sign(X)=\sign(\pi).$ Define the map $\Phi$ sending $\Pb_n^\pm(\F_p)\sigma\otimes 1$ to $\Pb_n^\pm(\F_p)B\otimes 1$. We also define the inverse map the following way: Let $B=(v_0,\dots,v_k)\in\MD(\F_p^n)_k$, \[\Phi^{-1}(\Pb_n^\pm(\F_p)B\otimes1)=\Pb_n^\pm(\F_p)(v_0^\pm,\dots,v_k^\pm)\otimes1.\] 

We now prove the isomorphism \[  \left(\SBDA(\F_p^n)_k-\SBD(\F_p^n)_k\right)/\Sigma_{k+1} \cong \MDA(\F_p^n)_k / G_k.\]we first observe that for every $\sigma\in  \SBDA(\F_p^n)_k\setminus \SBD(\F_p^n)_k$, we can write $\sigma=\sigma_1\pi$ where $\sigma_1=(v_0^\pm,\dots,v_k^\pm)$ satisfies $\pm v_0\pm v_1\pm v_2=0$ for the first three vectors, for some $\pi\in \Sigma_{k+1}$. Furthermore, since $v^\pm=\{v,-v\}$, we may choose $\sigma_1$ such that $v_0+v_1+v_2=0$. Thus, we may assume that elements in the quotient $ \left(\SBDA(\F_p^n)_k\setminus \SBD(\F_p^n)_k\right)/\Sigma_{k+1} $ are determinant-$1$ partial augmented frames of the form $\sigma=(v_0^\pm,\dots,v_k^\pm)$ with $v_0+v_1+v_2=0$. Consider the map \[\psi\colon \left(\SBDA(\F_p^n)_k\setminus \SBD(\F_p^n)_k\right)/\Sigma_{k+1}\longrightarrow \MDA(\F_p^n)_k/G_k\] defined by \[\sigma\Sigma_{k+1}\mapsto (v_0|\dots|v_k)G_k.\]Take $\sigma'=(v_0'^\pm,\dots,v_k'^\pm)$ such that $ v_0'+v_1'+v_2'=0$ and $\sigma'=\sigma\cdot\pi$ for some $\pi\in \Sigma_{k+1}$. Since $v_0'+v_1'+v_2'=0$ and $v_1,\dots,v_k$ are linearly independent in $\F_p^n$, it follows that $\pi=(\pi_1,\pi_2)\in\Sigma_3\times \Sigma_{k-2}$. Moreover, we have $v_i'=\varepsilon_iv_{\pi(i)}$ for some $\varepsilon_i\in\{-1,1\}$. It implies that\[\varepsilon_0v_{\pi_1(0)}+\varepsilon_1v_{\pi_1(1)}+\varepsilon_2v_{\pi_1(2)}=0.\]
We will show that if $p\neq 2$, then $\varepsilon_0=\varepsilon_1=\varepsilon_2$. Suppose for contradiction that $\varepsilon_i\neq \varepsilon_j$ for some $i,j$. And without loss of generality, we assume that $\varepsilon_0=\varepsilon_1=-\varepsilon_2$ and $\pi_1=\Id$. Then the relation $\varepsilon_0v_{\pi_1(0)}+\varepsilon_1v_{\pi_1(1)}+\varepsilon_2v_{\pi_1(2)}=0$ becomes \[v_0+v_1-v_2=0.\] On the other hand, since $v_0+v_1+v_2=0$, it follows that $2v_2=0$. As $p\neq 2$, this forces $v_2=0$, which contradicts the assumption that $v_2\in \F_p^n\setminus\{\underline{0}\}$. Thus, we must have\[\varepsilon_0=\varepsilon_1=\varepsilon_2=\varepsilon=\pm 1.\]When $p = 2$, we have $-v_i = v_i$ for all $i$, so in particular $-\varepsilon v_i=\varepsilon v_i$.
Therefore, \begin{align*} \psi(\sigma'\Sigma_{k+1})&=\left(\varepsilon v_{\pi_1(0)}\mid\varepsilon v_{\pi_1(1)}\mid \varepsilon v_{\pi_1(2)}\mid\varepsilon_3v_{\pi_2(3)}\dots\mid \varepsilon_kv_{\pi_2(k)}\right)G_k\\&=(v_0|\dots|v_k)\cdot (\pi_1,\varepsilon,\pi_2,\varepsilon_3,\dots,\varepsilon_k)G_k\\&=\psi(\sigma\Sigma_{k+1}).
\end{align*}This implies that $\psi$ is well defined.Bijectivity follows immediately. Lastly, the final isomorphism is proved similarly to the second isomorphism.
\end{proof}
\begin{remark}\label{rmk:ident}
Via the isomorphisms in \autoref{matrixrep}, if $\sigma\in\SBD(\F_p^n)_k$, then the coset $\sigma\Sigma_{k+1}$ is identified with the coset $B\cdot T_k$ for some $B\in\MD(\F_p^n)_k$. If $\sigma\in\SBDA(\F_p^n)_k\setminus \SBD(\F_p^n)_k$, then the coset $\sigma\Sigma_{k+1}$ is identified with the coset $B\cdot G_k$ for some $B\in\MDA(\F_p^n)_k$.
 
\end{remark}
We now define a sign homomorphism on the symmetry groups $T_k$ and $G_k$.
\begin{definition} Let $k \in\N$. Define the function $\sign$ on $T_k$ as follows
\begin{align*} \sign\colon T_k &\longrightarrow \{-1,1\}\\ X=(\pi,\varepsilon_0,\dots,\varepsilon_k) &\mapsto \sign(\pi)
\end{align*}and on $G_k$ in the following way
\begin{align*} \sign\colon G_k &\longrightarrow \{-1,1\}\\ X=(\pi,\varepsilon,\tau,\varepsilon_3,\dots,\varepsilon_k) &\mapsto \sign(\pi)\sign(\tau)
\end{align*}
\end{definition}
Using the isomorphisms in \autoref{matrixrep}, we see that $\Pb_n^\pm(\F_p)$ acts naturally on $\MD(\F_p^n)_k$ and $\MDA(\F_p^n)_k$ by left multiplication. We therefore make the following definitions.
\begin{definition}\label{def:action on B}
    Let $p$ be a prime, $n\geq 1$. Define the following. \begin{itemize}
        \item Let $\MD(\F_p^n)^\mathrm{rv}_k$ be the set of all matrices $B\in \MD(\F_p^n)_k$ for which there exist $A\in \Pb_n^\pm(\F_p)$ and $X\in T_k$ with $\sign(X)=-1$ such that $A\cdot B=B\cdot X$. Let $\MD(\F_p^n)_k^\mathrm{pr}=\MD(\F_p^n)_k -  \MD(\F_p^n)_k^\mathrm{rv}$.
        \item Let $\MDA(\F_p^n)^\mathrm{rv}_k$ be the set of all matrices $B\in \MDA(\F_p^n)_k$ for which there exist $A\in \Pb_n^\pm(\F_p)$ and $X\in G_k$ with $\sign(X)=-1$ such that $A\cdot B=B\cdot X$. Let $\MDA(\F_p^n)_k^\mathrm{pr}=\MDA(\F_p^n)_k -  \MDA(\F_p^n)_k^\mathrm{rv}$.
         \item Let $\MD(\F_p^n)^\mathrm{tw}_k$ be the set of all matrices $B\in \MD(\F_p^n)_k$ for which there exist $A\in \Pb_n^\pm(\F_p)$ and $X\in T_k$ with $\det(A)\sign(X)=-1$ such that $A\cdot B=B\cdot X$. Let $\MD(\F_p^n)_k^\mathrm{utw}=\MD(\F_p^n)_k -  \MD(\F_p^n)_k^\mathrm{tw}$.
        \item Let $\MDA(\F_p^n)^\mathrm{utw}_k$ be the set of all matrices $B\in \MDA(\F_p^n)_k$ for which there exist $A\in \Pb_n^\pm(\F_p)$ and $X\in G_k$ with $\det(A)\sign(X)=-1$ such that $A\cdot B=B\cdot X$. Let $\MDA(\F_p^n)_k^\mathrm{utw}=\MDA(\F_p^n)_k -  \MDA(\F_p^n)_k^\mathrm{tw}$.
    \end{itemize}
\end{definition}
\begin{lemma}\label{lem:goodmat}
   Let $p$ be a prime, $n\geq 1$.  Then \[\SBD(\F_p^n)_k^\mathrm{rv}/\Sigma_{k+1}\cong \MD(\F_p^n)_k^\mathrm{rv}/T_k,\quad\left(\SBDA(\F_p^n)_k ^\mathrm{rv} -\SBD(\F_p^n)_k^\mathrm{rv}\right)/\Sigma_{k+1}\cong \MDA(\F_p^n)_k^\mathrm{rv}/G_k,\]\[\SBD(\F_p^n)_k^\mathrm{pr}/\Sigma_{k+1}\cong \MD(\F_p^n)_k^\mathrm{pr}/T_k,\quad\left(\SBDA(\F_p^n)_k ^\mathrm{pr} -\SBD(\F_p^n)_k^\mathrm{pr}\right)/\Sigma_{k+1}\cong \MDA(\F_p^n)_k^\mathrm{pr}/G_k.\]Moreover, \[\SBD(\F_p^n)_k^\mathrm{tw}/\Sigma_{k+1}\cong \MD(\F_p^n)_k^\mathrm{tw}/T_k,\quad\left(\SBDA(\F_p^n)_k ^\mathrm{tw}- \SBD(\F_p^n)_k^\mathrm{tw}\right)/\Sigma_{k+1}\cong \MDA(\F_p^n)_k^\mathrm{tw}/G_k,\] \[\SBD(\F_p^n)_k^\mathrm{utw}/\Sigma_{k+1}\cong \MD(\F_p^n)_k^\mathrm{utw}/T_k,\quad\left(\SBDA(\F_p^n)_k^\mathrm{utw} - \SBD(\F_p^n)_k^\mathrm{utw}\right)/\Sigma_{k+1}\cong \MDA(\F_p^n)_k^\mathrm{utw}/G_k.\]
\end{lemma}
\begin{proof}
   We will only prove the isomorphisms \[\SBD(\F_p^n)_k^\mathrm{rv}/\Sigma_{k+1}\cong \MD(\F_p^n)_k^\mathrm{rv}/T_k\quad\text{and}\quad \SBD(\F_p^n)_k^\mathrm{pr}/\Sigma_{k+1}\cong \MD(\F_p^n)_k^\mathrm{pr}/T_k,\] as the remaining isomorphisms follow by similar reasoning.
    
    We recall from \autoref{def:or} that $\SBD(\F_p^n)_k^\mathrm{rv}$ is the set of all $k$-simplices $\sigma$ for which there exist $A\in \Pb_n^\pm(\F_p)$ and $\pi\in \Sigma_{k+1}$ with $\sign(\pi)=-1$ such that $A\cdot \sigma=\sigma\cdot \pi$. Let $\sigma=(v_0^\pm,\dots,v_k^\pm)\in\SBD(\F_p^n)_k^\mathrm{rv}$ and let $B=(v_0|\dots|v_k)$. Then, there exists $A\in \Pb_n^\pm(\F_p)$ such that \[
    \left((A\cdot v_0)^\pm,\dots,(A\cdot v_k)^\pm \right)=A\cdot\sigma=\sigma\cdot\pi=(v_{\pi(0)}^\pm,\dots,v_{\pi(k)}^\pm).\]This is equivalent to \[A\cdot B=B\cdot (\pi,\varepsilon_0,\dots,\varepsilon_k),\]for some $\varepsilon_i\in\{-1,1\}$ with $\sign\left((\pi,\varepsilon_0,\dots,\varepsilon_k)\right)=\sign(\pi)=-1$. 
    So the isomorphism in \autoref{matrixrep} induces an isomorphism \begin{align*}
        \SBD(\F_p^n)_k^\mathrm{rv}/\Sigma_{k+1}&\rightarrow \MD(\F_p^n)_k^\mathrm{rv}/T_k\\
        \sigma\Sigma_{k+1}&\mapsto B\cdot T_k. 
    \end{align*}
Furthermore, we have by \autoref{matrixrep}, \[\SBD(\F_p^n)_k/\Sigma_{k+1}\cong \MD(\F_p^n)_k/T_k.\]
Therefore,
    \begin{align*}\SBD(\F_p^n)_k^\mathrm{pr}/\Sigma_{k+1}&\cong \left(\SBD(\F_p^n)_k/\Sigma_{k+1}\right)\setminus\left(\SBD(\F_p^n)_k^\mathrm{rv}/\Sigma_{k+1}\right)\\&\cong \left(\MD(\F_p^n)_k/T_k\right)\setminus\left( \MD(\F_p^n)_k^\mathrm{rv}/T_k\right)\\&\cong \MD(\F_p^n)_k^\mathrm{pr}/T_k.\qedhere\end{align*}
\end{proof}
\section{Case \texorpdfstring{$n=2$}{Lg}}\label{sec:n=2}
Let $p$ be a prime. In the base case $n = 2$, we find that the symmetric $\Delta$-complex $\Gamma^\pm_{0}(p)\backslash\SBA(\Z^2)$ has an explicit description. Specifically, we compute the dimensions of the chain complexes, which are then used to prove our theorems in this case. For convenience, we adopt the notation $\Gamma_{0}^\pm(p)$ for $\Gamma_{0,2}^\pm(p)$.

\subsection{The quotient \texorpdfstring{$\Gamma_{0}^\pm(p)\backslash \SBA(\Z^2)$}{Lg}}
Recall from \autoref{good q} that \[\Gamma_{0}^\pm(p)\backslash\SBA(\Z^2) \cong \Pb_2^\pm(\F_p)\backslash \SBDA(\F_p^2).\]
Also, from \autoref{matrixrep}, we have the isomorphisms\[\SBDA(\F_p^2)_k/\Sigma_{k+1}\cong \MD_2(\F_p)_k/T_k\quad\text{for $k=0,1$}\]and\[ \SBDA(\F_p^2)_2/\Sigma_3\cong \MDA_2(\F_p)_2/G_2.\]
We will describe the complex $\Gamma_{0}^\pm(p)\backslash\SBA(\Z^2)_k/\Sigma_{k+1}$ by studying the orbits of the actions of $\Pb_2^\pm(\F_p)$ and the corresponding symmetry groups on $\MD_2(\F_p)_0$, $\MD_2(\F_p)_1$ and $\MDA_2(\F_p)_2$. We start with the set $\MD_2(\F_p)_0$.

\begin{proposition}\label{vert} 
The orbits of the action of $\Pb_2^\pm(\F_p)$ on the set $\MD_2(\F_p)_0$ are
\[U=\{\lambda e_1~\text{for some $\lambda\in \F_p^\times$}\},\]and\[V = \{v \in\F_p^2 \setminus\{\underline{0}\} \mid v \notin U\}.\]

\end{proposition}
\begin{proof}
We note that $\MD_2(\F_p)_0=\F_p^2\setminus\{\underline{0}\},$ so $U$ and $V$ comprise the whole set $\MD_2(\F_p)_0$.
We will show that $\Pb_2^\pm(\F_p)$ acts on $U$ and $V$ transitively.

By definition, every element of $\Pb_2^\pm(\F_p)$ fixes the one-dimensional subspace $\operatorname{span}\{e_1\}$. Thus, $\Pb_2^\pm(\F_p)$ acts on $U$, and so it acts on $V$. To prove transitivity, it is enough to show that there exist elements $A_1,A_2\in \Pb_2^\pm(\F_p)$ such that $A_1\cdot u=e_1$ and $A_2\cdot v=e_2$ for $u\in U$ and $v\in V$. Write $u=\lambda e_1$ for $\lambda\in\F_p^\times$. Define $A_1$ on the standard basis in the following way\begin{align*}
        e_1&\longmapsto \lambda^{-1}e_1\\
        e_2&\longmapsto \lambda e_2
    \end{align*}
Thus $A_1 \cdot u = e_1$. Since $A_1 \in \Pb_2^\pm(\F_p)$ and $u$ was arbitrary in $U$, the action is transitive on $U$.

Now let $v$ be given by $v=\lambda e_1+\mu e_2$, where $\lambda\in \F_p$ and $\mu\in \F_p^\times$. Define $A_2$ as follows\begin{align*}
        e_1&\longmapsto \mu e_1\\
        e_2&\longmapsto -\lambda e_1+\mu^{-1}e_2
    \end{align*}
Observe that $A_2\in \Pb_2^\pm(\F_p)$ and $A_2\cdot v=e_2$. This implies that $\Pb_2^\pm(\F_p)$ acts transitively on $V$.
\end{proof}
\begin{corollary}\label{cor:vert}
    The chain group $\redchain_0\left(\Pb_2^\pm(\F_p)\backslash \SBDA(\F_p^2);\Q\right)$ is two-dimensional with basis elements the two $\Pb_2^\pm(\F_p)$-orbits of $\MD_2(\F_p)_0$, $U$ and $V$.
\end{corollary}
\begin{proof}
    By \autoref{qut}, the quotient $\Pb_2^\pm(\F_p)\backslash \SBDA(\F_p^2)$ is a symmetric $\Delta$-complex. It then follows by \autoref{chaincpx} that \[\redchain_0(\Pb_2^\pm(\F_p)\backslash \SBDA(\F_p^2);\Q)\cong\Q\left[\Pb_2^\pm(\F_p)\backslash \SBDA(\F_p^2)_0\right]\otimes_{\Q[\Sigma_1]}\Q^\mathrm{sgn},\]has a non-canonical basis bijective to $\Pb_2^\pm(\F_p)\backslash \SBDA(\F_p^2)_0/T_0$ by \autoref{chaincoinv}, which is bijective to $\Pb_2^\pm(\F_p)\backslash \MDA_2(\F_p)_0/T_0$ by \autoref{matrixrep}.
    
    We recall from \autoref{vert} that $\Pb_2^\pm(\F_p)\backslash \MDA_2(\F_p)_0=\{U,V\}$. Since each $U$ and $V$ is invariant under the $T_0$-action, their $T_0$-orbits satisfy \[U\cdot T_0= U,\quad\text{and}\quad V\cdot T_0=V.\] We conclude that  \[\redchain_0(\Pb_2^\pm(\F_p)\backslash \SBDA(\F_p^2);\Q)\cong \Q[U,V].\qedhere\]
\end{proof}

\begin{proposition}\label{edgetype}
    For all $\underline{a}=(a_1,a_2)\in\F_p^2\setminus\{\underline{0}\}$, define the set
   \[\E(\underline{a})=\left\{ B=(v_1| v_2)\in\MD_2(\F_p)_1~\middle\vert\ \lambda e_1=a_1v_1+a_2v_2~\text{for some $\lambda\in \F_p^\times$}\right\}.\] These sets satisfy the following properties. \begin{enumerate}[label=(\alph*)]
   \item $\E(\underline{a})$ is nonempty.
       \item $
\E(\underline{a})=\E(\underline{b})\quad\Longleftrightarrow\quad \underline a= \lambda \underline b$ for some $\lambda\in \F_p^\times$.
  \item  The right action of $T_1$ on $\MD_2(\F_p)_1$ satisfies \[
       \E(\underline{a})\cdot X=\E(\underline{a}\cdot X)
  \]for all $X=(\pi,\varepsilon_1,\varepsilon_2)\in T_1$ with $\underline{a}\cdot X=\left(\varepsilon_1a_{\pi(1)},\varepsilon_2a_{\pi(2)}\right)$.
   \item  The set of the $\Pb_2^\pm(\F_p)$-orbits in $\MD_2(\F_p)_1$ is $\left\{\E(\underline{a})~\middle\vert~ \underline{a}\in\F_p^2\setminus  \{\underline{0}\}\right\}$.
   \end{enumerate}
\end{proposition}
\begin{proof}Recall that \[\MD_2(\F_p)_1=\left\{B=(v_1| v_2)~\middle\vert\begin{array}{c}v_1,v_2~\text{are linearly independent vectors in $\F_p^2$}, \\\det\left(v_1| v_2\right)
=\pm 1 \end{array}\right\}.\] 
\noindent\emph{Proof of (a).} Let $\underline{a}=(a_1,a_2)\in\F_p^2\setminus\{\underline{0}\}$. Then $\E(\underline{a})\neq \emptyset$: $\underline{a}=(a_1,a_2)\in\F_p^2\setminus\{\underline{0}\}$, so $a_i\neq 0$ for some $i\in\{1,2\}$. If $a_1\neq 0$, set 
\[
v_1=\begin{pmatrix}
-a_1^{-1} \\
-a_2
\end{pmatrix},\quad v_2=\begin{pmatrix}
    0\\ a_1
\end{pmatrix}.\] Then \[a_1v_1+a_2v_2=e_1,\]so $B=(v_1| v_2)\in \E(\underline{a})$. If $a_2\neq 0$, swap the roles of $(a_1,v_1)$ and $(a_2,v_2)$ in the same construction. 

\noindent\emph{Proof of (b).} Let $\underline{b}=(b_1,b_2)\in \F_p^2\setminus\{\underline{0}\}$, and assume that $\E(\underline{a})=\E(\underline{b})$. Since $\E(\underline{a})\neq \emptyset$, pick an arbitrary $B=(v_1| v_2)\in \E(\underline{a})$. Then there exist $\lambda_1,\lambda_2\in \F_p^\times$ such that \[\lambda_1 e_1=a_1v_1+a_2v_2\]and\[\lambda_2 e_1=b_1v_1+b_2v_2.\] 
Thus, \[\lambda_1^{-1}a_1v_1+\lambda_1^{-1}a_2v_2=\lambda_2^{-1}b_1v_1+\lambda_2^{-1}b_2v_2.\]
Since $v_1$ and $v_2$ are linearly independent in $\F_p^2$, it implies that \[(a_1,a_2)=\lambda_1\lambda_2^{-1}(b_1,b_2).\] Conversely, suppose $\underline{a}=\mu \underline{b}$ for some $\mu\in \F_p^\times$. Then \begin{align*}
    \E(\underline{a})&=\left\{ (v_1| v_2)\in\MD_2(\F_p)_1~\middle\vert\ \lambda e_1=a_1v_1+a_2v_2~\text{for some $\lambda\in \F_p^\times$}\right\}\\
    &=\left\{ (v_1| v_2)\in\MD_2(\F_p)_1~\middle\vert\ \lambda e_1=\mu b_1v_1+\mu b_2v_2~\text{for some $\lambda\in \F_p^\times$}\right\}\\
    &=\left\{ (v_1| v_2)\in\MD_2(\F_p)_1~\middle\vert\ \mu^{-1}\lambda e_1=b_1v_1+b_2v_2~\text{for some $\lambda\in \F_p^\times$}\right\}\\
    &=\E(\underline{b}).
\end{align*}
This proves property (b).

\noindent\emph{Proof of (c).} Let $X=(\pi,\varepsilon_1,\varepsilon_2)\in T_1$ and $B=\left(v_1| v_2\right)\in\MD_2(\F_p)_1$. If $B\in \E(\underline{a})$, then $\lambda_1 e_1=a_1v_1+a_2v_2$ for some $\lambda_1\in \F_p^\times$ implies that \[\lambda_1 e_1=\varepsilon_1 a_{\pi(1)}\left(\varepsilon_1v_{\pi(1)}\right)+\varepsilon_2 a_{\pi(2)}\left(\varepsilon_2v_{\pi(2)}\right).\]
So the matrix \[B\cdot X=\left(\varepsilon_1 v_{\pi(1)}\mid \varepsilon_2 v_{\pi(2)}\right)\in E\left(\varepsilon_1 a_{\pi(1)},\varepsilon_2a_{\pi(2)}\right)= E\left(\underline{a}\cdot X\right),\]implying \begin{equation}\label{in}\E(\underline{a})\cdot X\subseteq \E(\underline{a}\cdot X)\quad\text{for all $X\in T_1$}.\end{equation} 
 Applying \eqref{in} to $X^{-1}$, we obtain \[ \E(\underline{a}\cdot X)\cdot X^{-1}\subseteq \E(\underline{a}).\] Thus \[ \E(\underline{a}\cdot X)\subseteq \E(\underline{a})\cdot X.\]concluding $\E(\underline{a}\cdot X)=\E(\underline{a})\cdot X$ 
\noindent\emph{Proof of (d).} Let $B=(v_1| v_2)\in \E(\underline{a})$ such that \[\lambda_1  e_1=a_1 v_1+a_2v_2,\]for some $\lambda_1\in\F_p^\times$. 
For any $A\in \Pb_2^\pm(\F_p)$, we know $A\cdot e_1=\lambda e_1$ for some $\lambda \in \F_p^\times$. Hence,\[a_1(A\cdot v_1)+a_2(A\cdot v_2)=\lambda_1\lambda e_1.\]Thus,  \[A\cdot B\in \E(\underline{a}),\]and so $\E(\underline{a})$ is a $\Pb_2^\pm(\F_p)$-invariant subset of $\MD_2(\F_p)_1$.
Moreover, we show that it acts transitively. Let $C=\left(u_1| u_2\right)$ be another matrix in $\E(\underline{a})$, such that \[\lambda_2  e_1=a_1 u_1+a_2u_2,\]for some $\lambda_2 \in\F_p^\times$. We now consider the element $A'$ defined on the basis as follows. \[
            v_i\longmapsto u_i\quad\text{for all $i=1,2$}.
           \]Observe that \[A'\cdot  B= C.\]
Additionally,\begin{align*}
            A'\cdot  e_1 &=\lambda_1^{-1}\left(a_1(A'\cdot  v_1)+a_2(A'\cdot  v_2)\right)\\&=\lambda_1^{-1}\left(a_1u_1+a_2u_2\right)\\&=\lambda_1^{-1}\lambda_2 e_1.
        \end{align*}
It follows that $A'\in\Pb_2^\pm(\F_p)$, and therefore $\Pb_2^\pm(\F_p)$ acts transitively on $\E(\underline{a})$.      
As the union $\bigcup\limits_{\underline{a}\in\F_p^2\setminus\{\underline{0}\}} \E(\underline{a})$ comprises the set $\MD_2(\F_p)_1$, we are done.
\end{proof}
\begin{definition}
    We define the equivalence relation $\sim$ on the elements of $\F_p$ as follows.
\[x\sim y \quad\text{if and only if}\quad \E(1,x)\cdot T_1=E(1,y)\cdot T_1~\text{in $\Pb_2^\pm(\F_p)\backslash \MD_2(\F_p)_1/T_1$}.\]
\end{definition}

\begin{proposition}\label{edge:c}
For all $c\in \F_p$, every orbit in the double quotient $\Pb_2^\pm(\F_p)\backslash \MD_2(\F_p)_1/T_1$ has the form $\E(1,c)\cdot T_1$, with \[c\sim c^{-1}\sim -c\sim -c^{-1}.\]
\end{proposition}
The expressions $c^{-1}$ and $-c^{-1}$ are understood only for $c\in \F_p^\times$.

\begin{proof}
    Recall from \autoref{edgetype} that $\E(\underline{a})$ are the orbits of $\Pb_2^\pm(\F_p)\backslash\MD_2(\F_p)_1$ such that \begin{equation}\label{1}\E(\underline{a})=\E(\underline{b})\Longleftrightarrow \underline{a}=\lambda\underline{b}~\text{some $\lambda\in\F_p^\times$}\end{equation}and\begin{equation}\label{2}\E(\underline{a})\cdot X=\E(\underline{a}\cdot X)\quad\text{for all $X\in T_1$}.\end{equation}
    It follows that every orbit in the double quotient $\Pb_2^\pm(\F_p)\backslash \MD_2(\F_p)_1/T_1$ has the form $\E(\underline{a})\cdot T_1$, satisfying the following equalities \[\E(a_1,a_2)\cdot T_1\overset{\text{by \eqref{1}}}{\underset{\lambda=a_1^{-1}}{=}}E(1,a_2a_1^{-1})\cdot T_1\quad\text{if $a_1\neq 0$},\] and
\[\E(a_1,a_2)\cdot T_1\overset{\text{by \eqref{2}}}{\underset{X=((1\,2),1,1)}{=}}E(a_2,a_1)\cdot T_1\overset{\text{by \eqref{1}}}{\underset{\lambda=a_2^{-1}}{=}}E(1,a_1a_2^{-1})\cdot T_1\quad\text{if $a_2\neq 0$}.\]Since at least one of the $a_i$'s must be nonzero, we thus observe that in $\Pb_2^\pm(\F_p)\backslash \MD_2(\F_p)_1/T_1$, every orbit $\E(\underline{a})\cdot T_1$ may be written in the form \[E(1,c)\cdot T_1\quad\text{for $c\in\F_p$}.\]Furthermore, it follows by the above equalities that for $c\neq 0$, we have \[E(1,c)\cdot T_1=E(c^{-1},1)\cdot T_1=E(1,c^{-1})\cdot T_1.\]Moreover, \[\E(1,c)\cdot T_1\overset{\text{by \eqref{2}}}{\underset{X=(\Id,1,-1)}{=}}\E(1,-c)\cdot T_1.\] Therefore, every orbit in $\Pb_2^\pm(\F_p)\backslash \MD_2(\F_p)_1/T_1$ has the form $\E(1,c)\cdot T_1$ for all $c\in F_p$, with \[ \E(1,c)\cdot T_1=E(1,-c)\cdot T_1=E(1, c^{-1})\cdot T_1=E(1,-c^{-1})\cdot T_1;\]that is, \[c\sim c^{-1}\sim -c\sim -c^{-1}.\qedhere \] 
\end{proof}

\begin{corollary}\label{cor:edge}  $\Pb_2^\pm(\F_p)\backslash \MD_2(\F_p)_1/T_1\cong \F_p/\sim$ where the equivalence classes of $\sim$ are \[\{c, c^{-1}, -c, -c^{-1}\}\] for $c\in \F_p$. Moreover, the only situations in which two of these elements are equal occur precisely when $c\in\{0,-1,1,-i,i\}$, where $i$ denotes an element of $\F_p$ such that $i^2=-1$. Note that the equation $i^2=-1$ has a solution in $\F_p$ if and only if $p\equiv 1\bmod 4$.
\end{corollary}
Note that $\F_p$ has a square root of $-1$ if and only if $p\equiv 1\bmod 4$.
\begin{proof}
  We show that if $\E(1,c)\cdot T_1=\E(1,c')\cdot T_1$ for some $c,c'\in\F_p$, then $c'\in\{\pm c,\pm c^{-1}\}$. Together with \autoref{edge:c}, this implies the claimed bijection.

  If $\E(1,c)\cdot T_1=\E(1,c')\cdot T_1$, then by \autoref{edgetype}, there exists $\lambda\in\F_p^\times$ such that \begin{equation}\label{eqq}1=\pm \lambda\quad\text{and}\quad c=\pm \lambda c',\end{equation}
  or \begin{equation}\label{eqq'}1=\pm \lambda c'\quad\text{and}\quad c=\pm \lambda.\end{equation}
  \eqref{eqq} implies that $c'=\pm c$ and \eqref{eqq'} gives $c'=\pm c^{-1}$.
  
  Two of the elements $\{\pm c,\pm c^{-1}\}$ are equal in $\F_p$ if and only if $c=0$, $c=\pm 1$ or $c^2=-1$. The equation $c^2=-1$ has a solution in $\F_p$ if and only if $p\equiv 1\bmod 4$, in which case $c=\pm i$.
    \end{proof}
We proceed to describe the orbits of the action of $\Pb_2^\pm(\F_p)$ on the set $\MDA_2(\F_p)_2$. We will treat the cases $p\neq 3$ and $p=3$ separately, starting with $p\neq 3$. 

\begin{proposition}\label{facetype} 
  For all primes $p\neq 3$ and $\underline{a}=(a_0,a_1,a_2)\in\F_p^3\setminus\{\underline{0}\}$ such that $a_0+a_1+a_2=0$, define the set
   \[\Face(\underline{a})=\left\{ B=(v_0|v_1| v_2)\in\MDA_2(\F_p)_2~\middle\vert\ \lambda e_1=a_0v_0+a_1v_1+a_2v_2~\text{for some $\lambda\in \F_p^\times$}\right\}.\] These sets satisfy the following properties. \begin{enumerate}[label=(\alph*)]
   \item $\Face(\underline{a})$ is nonmepty.
       \item $
\Face(\underline{a})=\Face(\underline{b})\quad\Longleftrightarrow\quad \underline{a}= \lambda \underline b$ for some $\lambda\in \F_p^\times$.
  \item  The right action of $G_2$ on $\MDA_2(\F_p)_2$ satisfies \[
       \Face(\underline{a})\cdot X=\Face(\underline{a}\cdot X)
  \]for all $X=(\pi,\varepsilon)\in G_2$ with $\underline{a}\cdot X=\left(\varepsilon a_{\pi(0)},\varepsilon a_{\pi(1)},\varepsilon a_{\pi(2)}\right)$.
   \item  The set of the $\Pb_2^\pm(\F_p)$-orbits in $\MDA_2(\F_p)_2$ is $\left\{\Face(\underline{a})~\middle\vert~ \underline{a}=(a_0,a_1,a_2)\in\F_p^3\setminus\{\underline{0}\}, a_0+a_1+a_2=0\right\}$.
   \end{enumerate}
\end{proposition}
\begin{proof} Recall that \[\MDA_2(\F_p)_2=\left\{B=(v_0|v_1| v_2)~\middle\vert\begin{array}{c}v_1,v_2~\text{are linearly independent vectors in $\F_p^2$}, \\v_0+v_1+v_2=0,~\text{and}~\det\left(v_1| v_2\right)
=\pm 1 \end{array}\right\}.\] 
\noindent\emph{Proof of (a).} Let $\underline{a}=(a_0,a_1,a_2)\in\F_p^3\setminus\{\underline{0}\}$ such that $a_0+a_1+a_2=0$, then $\Face(\underline{a})\neq \emptyset$: since $p\neq 3$, then not all $a_i$ are equal. So there exist $i\neq j$ with $a_i\neq a_j$. For instance, if $a_1\neq a_2$ set
\[B=
\begin{pmatrix}
-(a_1-a_2)^{-1} & (a_1-a_2)^{-1} & 0\\
a_0-a_1 & -(a_0-a_2)& a_1-a_2\\
\end{pmatrix}.
\]A direct check shows that \[a_0v_0+a_1v_1+a_2v_2=(a_1-a_0)(a_1-a_2)^{-1}e_1,\] and that $B\in \MDA_2(\F_p)_2$. Similar explicit constructions can be done when $a_0\neq a_1$ or $a_0\neq a_2$, showing that $\Face(\underline{a})$ is nonempty.

\noindent\emph{Proof of (b).} Let $\underline{b}=(b_0,b_1,b_2)\in \F_p^3\setminus\{\underline{0}\}$ such that $b_0+b_1+b_2=0$, and assume that $\Face(\underline{a})=\Face(\underline{b})$. Since $\Face(\underline{a})\neq \emptyset$, pick an arbitrary $B=(v_0|v_1| v_2)\in \Face(\underline{a})$. Then there exist $\lambda_1,\lambda_2\in \F_p^\times$ such \[\lambda_1 e_1=a_0v_0+a_1v_1+a_2v_2\]and\[\lambda_2 e_1=b_0v_0+b_1v_1+b_2v_2.\] Thus, \[\lambda_1^{-1}a_0v_0+\lambda_1^{-1}a_1v_1+\lambda_1^{-1}a_2v_2=\lambda_2^{-1}b_0v_0+\lambda_2^{-1}b_1v_1+\lambda_2^{-1}b_2v_2.\] As $v_0=-v_1-v_2$, it follows that \[\left(\lambda_1^{-1}a_1-\lambda_1^{-1}a_0\right)v_1+\left( \lambda_1^{-1}a_2-\lambda_1^{-1}a_0\right)v_2=\left(\lambda_2^{-1}b_1-\lambda_2^{-1}b_0\right)v_1+\left(\lambda_2^{-1}b_2-\lambda_2^{-1}b_0\right)v_2.\]
Since $a_0+a_1+a_2=b_0+b_1+b_2=0$, and $v_1,v_2$ are linearly independent in $\F_p^2$, we obtain \begin{align*}
\lambda_1^{-1}(2a_1+a_2)&=\lambda_2^{-1}(2b_1+b_2)\\
\lambda_1^{-1}(2a_2+a_1)&=\lambda_2^{-1}(2b_2+b_1)
\end{align*}
Thus, \[3\lambda_1^{-1}a_1=3\lambda_2^{-1}b_1\quad\text{and}\quad 3\lambda_1^{-1}a_2=3\lambda_2^{-1}b_2.\]
As $p\neq 3$, we conclude that $(a_1,a_2)=\lambda_1\lambda_2^{-1}(b_1,b_2)$. Therefore \[\underline{a}=\lambda_1\lambda_2^{-1}\underline{b}.\]Conversely, suppose $\underline{a}=\mu \underline{b}$ for some $\mu\in \F_p^\times$. Then \begin{align*}
    \Face(\underline{a})&=\left\{ (v_0|v_1| v_2)\in\MDA_2(\F_p)_2~\middle\vert\ \lambda e_1=a_0v_0+a_1v_1+a_2v_2~\text{for some $\lambda\in \F_p^\times$}\right\}\\
    &=\left\{ (v_0|v_1| v_2)\in\MDA_2(\F_p)_2~\middle\vert\ \lambda e_1=\mu b_0v_0+\mu b_1v_1+\mu b_2v_2~\text{for some $\lambda,\mu\in \F_p^\times$}\right\}\\
    &=\left\{ (v_0|v_1| v_2)\in\MDA_2(\F_p)_2~\middle\vert\ \mu^{-1}\lambda e_1=b_0v_0+b_1v_1+b_2v_2~\text{for some $\lambda,\mu\in \F_p^\times$}\right\}\\
    &=\Face(\underline{b}).
\end{align*}

\noindent\emph{Proof of (c).} Let $X=\left(\pi,\varepsilon\right)\in G_2$ and $B\in\MDA_2(\F_p)_2$. If $B=(v_0|v_1|v_2)\in \Face(\underline{a})$, then there exists $\lambda_1 \in\F_p^\times$ such that \[\lambda_1 e_1=a_0v_0+a_1v_1+a_2v_2=\varepsilon a_{\pi(0)}(\varepsilon v_{\pi(0)})+\varepsilon a_{\pi(1)}(\varepsilon v_{\pi(1)})+\varepsilon a_{\pi(2)}(\varepsilon v_{\pi(2)}).\]
It follows that the matrix\[B\cdot X=\left(\varepsilon v_{\pi(0)}\mid \varepsilon v_{\pi(1)}\mid \varepsilon v_{\pi(2)}\right)\in \Face(\underline{a}\cdot X),\]implying \begin{equation}\label{ine}\Face(\underline{a})\cdot X\subseteq \Face(\underline{a}\cdot X)\quad\text{for all $X\in G_2$}.\end{equation}
 Applying \eqref{ine} to $X^{-1}$, we obtain \[ \Face(\underline{a}\cdot X)\cdot X^{-1}\subseteq \Face(\underline{a}).\] Thus \[ \Face(\underline{a}\cdot X)\subseteq \Face(\underline{a})\cdot X.\]
Thus, \[B\in \Face(\underline{a})\cdot X.\]
\noindent\emph{Proof of (d).} We show that $\Pb_2^\pm(\F_p)$ acts transitively on $\Face(\underline{a})$. Suppose $B=(v_0|v_1| v_2)\in \Face(\underline{a})$. Then \[\lambda_1 e_1=a_0v_0+a_1v_1+a_2v_2~\text{for some $\lambda_1\in\F_p^\times$}.\]
For any $A\in \Pb_2^\pm(\F_p)$, we know $A\cdot e_1=\lambda e_1$ for some $\lambda \in \F_p^\times$. Hence,\[a_0(A\cdot v_0)+a_1(A\cdot v_1)+a_2(A\cdot v_2)=\lambda_1\lambda e_1.\]Thus,  \[A\cdot B\in \Face(\underline{a}).\] 
For the transitivity part, we let $C=\left(u_0| u_1| u_2\right)$ be another matrix in $\Face(\underline{a})$. Then \[\lambda_2  e_1=a_0u_0+a_1 u_1+a_2u_2\quad\text{for some $\lambda_2 \in\F_p^\times$}.\]We define $D$ on the basis $\{v_1,v_2\}$ in the following way \[
            v_i\longmapsto u_i\quad\text{for all $i=1,2$}.
           \]
Then, \[A'\cdot  v_0=-\left(A'\cdot  v_1+A'\cdot  v_2\right)=-\left(u_1+u_2\right)=u_0,\]which implies that \[A'\cdot  B= C.\]
Additionally, \[A'\cdot  e_1=\lambda_1^{-1}\left(a_0u_0+a_1u_1+a_2u_2\right)=\lambda_1^{-1}\lambda_2 e_1.\]
Therefore \[A'\in\Pb_2^\pm (\F_p).\] 

It remains to show that the union \[\bigcup\limits_{\substack{\underline{a}=(a_0,a_1,a_2)\in\F_p^3\setminus\{\underline{0}\}\\a_0+a_1+a_2=0}} \Face(\underline{a})\] comprises the whole set $\MDA_2(\F_p)_2$. We will show that for every matrix $\left(v_0|v_1| v_2\right)$ in $ \MDA_2(\F_p)_2$, there exist $\lambda\in\F_p^\times$ and $\underline{a}=(a_0,a_1,a_2)\in\F_p^3\setminus\{\underline{0}\}$, such that $a_0+a_1+a_2=0$ and \[\lambda e_1=a_0v_0+a_1v_1+a_2v_2.\]Let $\left(v_0|v_1| v_2\right)\in \MDA_2(\F_p)_2$, then $v_1$ and $v_2$ are linearly independent in $\F_p^2$ and $v_0+v_1+v_2=0$. Since $e_1\in \operatorname{span}\{v_1,v_2\},$ there exist some $\lambda\in\F_p^\times$ and $(b_1,b_2)\in\F_p^2\setminus\{\underline{0}\}$, such that\[\lambda e_1=b_1v_1+b_2v_2.\]
Let \[a_0=-3^{-1}(b_1+b_2),\]\[a_1=3^{-1}(2b_1-b_2)\]\[a_2=3^{-1}(2b_2-b_1).\]Then $\lambda e_1=a_0v_0+a_1v_1+a_2v_2$, and $a_0+a_1+a_2=0$.
\end{proof}
\begin{definition}
    For all primes $p\neq 3$, we define the equivalence relation $\approx$ on the elements of $\F_p$ as follows:
\[x\approx y \quad\text{if and only if}\quad \Face(1,x,-1-x)\cdot G_2=\Face(1,y,-1-y)\cdot G_2~\text{in $\Pb_2^\pm(\F_p)\backslash \MDA_2(\F_p)_2/G_2$}.\]
\end{definition}
\begin{proposition}\label{face:c}
For all primes $p\neq 3$ and for all $c\in\F_p$, every orbit in the double quotient $\Pb_2^\pm(\F_p)\backslash\MDA_2(\F_p)_1/G_2$ has the form $\Face(1,c,-1-c)\cdot G_2$, with \[c\approx c^{-1}\approx -(1+c)\approx -(1+c^{-1})\approx-(1+c)^{-1}\approx -1+(1+c)^{-1}.\]\end{proposition}
The expressions $c^{-1}$ and $(1+c)^{-1}$ are understood only for $c\in \F_p^\times$ and $c\neq -1$, respectively.

\begin{proof}
    We have by \autoref{facetype} that $\Face(-a_1-a_2,a_1,a_2)$ are the orbits of $\Pb_2^\pm(\F_p)\backslash\MDA_2(\F_p)_1$ such that \begin{equation}\label{3}\Face(\underline{a})=\Face(\underline{b})\Longleftrightarrow\underline{a}=\lambda \underline{b}~\text{for some $\lambda\in \F_p^\times$},\end{equation}and\begin{equation}\label{4}\Face(\underline{a}\cdot X)=\Face(\underline{a})\cdot X\quad\text{for all $X\in G_2$}.\end{equation}
This implies that every orbit in the double quotient $\Pb_2^\pm(\F_p)\backslash \MDA_2(\F_p)_2/G_2$ has the form $\Face(-a_1-a_2,a_1,a_2)\cdot G_2$, satisfying the following equalities:
\[\Face(a_0,a_1,a_2)\cdot G_2\overset{\text{by \eqref{3}}}{\underset{\lambda=a_0^{-1}}{=}}\Face(1,a_1a_0^{-1},a_2a_0^{-1})\cdot G_2\quad\text{if $a_0\neq 0$},\]and
\[\Face(a_0,a_1,a_2)\cdot G_2\overset{\text{by \eqref{4}}}{\underset{X=((1\,2),1)}{=}}\Face(a_1,a_0,a_2)\cdot G_2\overset{\text{by \eqref{3}}}{\underset{\lambda=a_1^{-1}}{=}}\Face(1,a_1^{-1}a_0,a_2a_1^{-1})\cdot G_2\quad\text{if $a_1\neq 0$}.\]
Since at least one of the $a_i$'s must be nonzero, it follows that every orbit in $\Pb_2^\pm(\F_p)\backslash \MDA_2(\F_p)_2/G_2$ has the form $\Face(1,c,-1-c)\cdot G_2$ for $c\in \F_p$.

Additionally, if $c\neq 0$, we have \[\Face(1,c,-1-c)\cdot G_2=\Face(c^{-1},1,-1-c^{-1})\cdot G_2=\Face(1,c^{-1},-1-c^{-1})\cdot G_2.\]
And if $c\neq -1$, we get \[\Face(1,c,-1-c)\cdot G_2=\Face(-(1+c)^{-1},-c(1+c)^{-1},1)\cdot G_2=\Face(1,-(1+c)^{-1},-c(1+c)^{-1})\cdot G_2.\] Since $-c(1+c)^{-1}=-1+(1+c)^{-1}$, we obtain  \[\Face(1,c,-1-c)\cdot G_2=\Face(1,-(1+c)^{-1},-1+(1+c)^{-1})\cdot G_2\quad (c\neq -1).\]
Thus, \begin{equation}\label{rel2}c\approx c^{-1}\approx -(1+c)\approx -(1+c^{-1})\approx -(1+c)^{-1}\approx -1+(1+c)^{-1}. \qedhere\end{equation}

\end{proof}
\begin{corollary}\label{cor:face} $\Pb_2^\pm(\F_p)\backslash \MDA_2(\F_p)_2/G_2\cong \F_p/\approx$ where the equivalence classes of $\approx$ are \[\{c,c^{-1}, -(1+c), -(1+c^{-1}), -(1+c)^{-1}, -1+(1+c)^{-1}\}\] for $c\in \F_p$. Moreover, the only situations in which two of these elements are equal occur precisely when $c\in\{0,-2,-2^{-1},-1,1,\gamma_1,\gamma_2\}$ where $\gamma_i^2+\gamma_i+1=0$ for $i=1,2$.
\end{corollary}
Note that $\gamma_i^2+\gamma_i+1=0$ occurs if and only if $p\equiv 1\bmod 3$.
\begin{proof}
  We show that if $\Face(1,c,-1-c)\cdot G_2=\Face
(1,c',-1-c')\cdot G_2$ for some $c,c'\in\F_p$, then $c'\in\{c, c^{-1},-(1+c), -(1+c^{-1}), -(1+c)^{-1}, -1+(1+c)^{-1}\}$. Together with \autoref{face:c}, this implies the claimed bijection.

  If $\Face(1,c,-1-c)\cdot G_2=\Face(1,c',-1-c')\cdot G_2$, then by \autoref{facetype}, there exists $\lambda\in\F_p^\times$ and $X=(\pi,\varepsilon)\in G_2$ such that \[(1,c,-1-c)=\lambda\cdot(1,c',-1-c')\cdot X.\]Thus one of the following equalities holds.\begin{equation*}
\begin{tabular}[b]{c | c | c}
$\pi$ & \text{Equalities} & \text{Solutions} \\
\hline 
$\pi=\Id$ & $\begin{aligned}
1 &= \lambda\\
c &=\ \lambda c'\\
-1 - c &= \lambda(-1 - c').
\end{aligned}$
& $c'=c$\\[10pt]\hline
$\pi=(0\,1)$ & $\begin{aligned}
1 &= \lambda c'\\
c &= \lambda\\
-1 - c &= \lambda(-1 - c')
\end{aligned}$
& $c' = c^{-1}\quad (c\neq 0)$ \\[10pt]\hline
$\pi=(1\,2)$ & $\begin{aligned}
1 &= \lambda,\\
c &= \lambda(-1 - c')\\
-1 - c &= \lambda c'
\end{aligned}$
& $c' = -(1+ c)$ \\[10pt]\hline
$\pi=(1,2,3)$ & $\begin{aligned}
1 &= \lambda(-1 - c')\\
c &= \lambda\\
-1 - c &= \lambda c'
\end{aligned}$
& $c' = -(1+c^{-1})\quad (c\neq 0)$\\[10pt]\hline
$\pi=(1,3,2)$ & $\begin{aligned}
1 &= \lambda c'\\
c &= \lambda(-1 - c')\\
-1 - c &= \lambda
\end{aligned}$
& $c' =-(1+ c)^{-1}\quad (c\neq -1)$ \\[10pt]\hline
 $\pi=(1,3)$ & $\begin{aligned}
1 &= \lambda(-1 - c')\\
c &= \lambda c'\\
-1 - c &= \lambda
\end{aligned}$
& $c' = -1+ (1+c)^{-1}\quad (c\neq -1)$ \\[10pt]
\end{tabular}
\end{equation*}Thus $c'$ lies in the desired list.

  We prove the second part of the corollary in the table below.
\begin{equation*}
\begin{tabular}[b]{c | c}
\text{Equalities} & \text{Solutions} \\
\hline
$c = c^{-1}$ & $c = \pm 1$ \\
$c = -(1 + c)$ & $c = -2^{-1}$ \\
$c = -(1 + c^{-1})$ & $c^2 + c + 1 = 0$ \\
$c = -(1 + c)^{-1}$ & $c^2 + c + 1 = 0$ \\
$c = -1 + (1 + c)^{-1}$ & $c = -2$ \\
$c^{-1} = -(1 + c)$ & $c^2 + c + 1 = 0$ \\
$c^{-1} = -(1 + c^{-1})$ & $c = -2$ \\
$c^{-1} = -(1 + c)^{-1}$ & $c = -2^{-1}$ \\
$c^{-1} = -1 + (1 + c)^{-1}$ & $c^2+c+1=0$ \\
$-(1 + c) = -(1 + c^{-1})$ & $c = \pm 1$ \\
$-(1 + c) = -(1 + c)^{-1}$ & $c = -2$ \\
$-(1 + c) = -1 + (1 + c)^{-1}$ & $c^2 + c + 1 = 0$ \\
$-(1 + c^{-1}) = -(1 + c)^{-1}$ & $c^2 + c + 1 = 0$ \\
$-(1 + c^{-1}) = -1 + (1 + c)^{-1}$ & $c = -2^{-1}$ \\
$-(1 + c)^{-1} = -1 + (1 + c)^{-1}$ & $c = 1$ \\
\end{tabular}\qedhere
\end{equation*}
    \end{proof}

Since \autoref{facetype} relies on dividing by $3$, it does not apply when $p=3$. We therefore classify orbits using the following alternative parameterization.
\begin{proposition}\label{F,p=3}
 For all $\underline{a}=(a_1,a_2)\in\F_3^2\setminus\{\underline{0}\}$, define the set
   \[\Face'(\underline{a})=\left\{ (v_0|v_1| v_2)\in\MDA(\F_3^2)_2~\middle\vert\ \lambda e_1=a_1v_1+a_2v_2~\text{for some $\lambda\in \F_3^\times$}\right\}.\] These sets satisfy the following properties. \begin{enumerate}[label=(\alph*)]
   \item $\Face'(\underline{a})$ is nonempty.
       \item \refstepcounter{equation}\label{prop,b}
       $
\Face'(\underline{a})=\Face'(\underline{b})\quad\Longleftrightarrow\quad \underline{a}= \pm \underline b$.
\hfill  \textup{(\theequation)}
  \item  The right action of $G_2$ on $\MDA(\F_3^2)_2$ satisfies \begin{equation}\label{Face',1}
       \Face'(a_1,a_2)\cdot X_1=\Face'(a_2,a_1)\quad\text{for some $X_1\in G_2$}
  \end{equation}\begin{equation}\label{Face',2}
       \Face'(a_1,a_2)\cdot X_2=\Face'(1,0)\quad\text{if $a_1\cdot a_2=1$} \quad\text{for some $X_2\in G_2$}.
  \end{equation}
   \item The set of the $\Pb_2^\pm(\F_3)$-orbits in $\MDA(\F_3^2)_2$ is $\left\{\Face'(\underline{a}) ~\middle\vert~\underline{a}\in\F_p^2\setminus\{\underline{0}\}\right\}$.
   \end{enumerate} 
\end{proposition}
\begin{proof}
 Property (a) uses the same construction as that of property (a) in \autoref{edgetype}, and property (b) follows the same reasoning as property (b) in \autoref{edgetype}. 

\textit{Proof of (c).} Let $B=\left(v_0|v_1| v_2\right)\in \Face'(a_1,a_2)$, such that \begin{equation}\label{eq:lam}\lambda e_1=a_1 v_1+a_2v_2\end{equation}for some $\lambda\in\F_3^\times$. We observe that for $X=\left((1\,2),1\right)\in G_2$, we have \[B\cdot X=\left(v_0| v_2| v_1\right).\] 
Moreover, we can write \[\lambda e_1=a_2v_2+a_1v_1.\] Thus, 
\[\Face(a_1,a_2)\cdot X\subseteq \Face'(a_2,a_1).\]
Since $g$ is invertible, we conclude that $\Face(a_1,a_2)\cdot X= \Face'(a_2,a_1)$.

Now we assume $a_1\cdot a_2=1$; that is, $a_1=a_2=\pm 1$. Take $\pi=(1\; 0)$ the transposition on $\{0,1,2\}$ sending $1\leftrightarrow 0$, and consider $X=\left(\pi,1\right)\in G_2$. Then \[B\cdot X=\left(v_1| v_0| v_2\right).\] Writing \eqref{eq:lam} as\[\lambda e_1=-a_1 v_0+(a_2-a_1)v_2,\]and using $a_1=a_2$, it follows that \[\lambda e_1=-a_1v_0.\]Thus, \[B\cdot X\in \Face'(-a_1,0),\]and consequently, \[\Face'(a_1,a_2)\cdot X\subseteq \Face'(-a_1,0).\] Similarly, we can show that \[\Face'(a_1,a_2)\cdot X= \Face'(-a_1,0).\]
Using property (b), we conclude that $\Face'(a_1,a_2)\cdot X= \Face'(1,0).$

\textit{Proof of (d).}Let $B=\left(v_0|v_1| v_2\right)\in \Face'(a_1,a_2)$, such that \[\lambda_1 e_1=a_1 v_1+a_2v_2\]for some $\lambda_1\in\F_3^\times$. Since for any $A\in \Pb_2^\pm(\F_3)$, we have by definition that $A\cdot e_1=\lambda e_1$ for some $\lambda\in \F_3^\times$, then\begin{align*}
        a_1\left(A\cdot v_1\right)+ a_2\left(A\cdot v_2\right)&=A\left(a_1 v_1+a_2 v_2\right)\\&=A\cdot\left(\lambda_1 e_1\right)\\&=\lambda_1 \lambda e_1.
    \end{align*}
    This implies that $A\cdot B\in \Face'(a_1,a_2)$, concluding that $\Pb_2^\pm(\F_3)$ acts on $\Face'(a_1,a_2)$. To show transitivity, let $C=\left(u_0| u_1| u_2\right)$ be another matrix in $\Face'(a_1,a_2)$. Then \[\lambda_2  e_1=a_1 u_1+a_2u_2,\]for some $\lambda_2 \in\F_3^\times$. We take the element $A'$ defined as follows\[v_i\longmapsto u_i\quad\text{for each $i\in\{1,2\}$}.\]Observe that \[A'\cdot  B= C\quad\text{and}\quad A'\cdot  e_1=\lambda_1^{-1}\lambda_2  e_1.\]
   It follows that $A'\in\Pb_2^\pm(\F_3)$, proving that $\Pb_2^\pm(\F_3)$ acts transitively on $\Face'(a_1,a_2)$.
   As the union $\bigcup\limits_{\substack{\underline{a}\in\F_3^2\setminus\{\underline{0}\}}} \Face'(\underline{a})$ comprises the whole set $\MDA(\F_3^2)_2$, we are done.\end{proof}

\subsection{The coinvariants \texorpdfstring{$\left(\St_2(\Q)\otimes\Q\right)_{\Gamma_{0}^\pm(p)}$}{Lg}}\label{n=2,coinv1}

Since $\BA_2(\Z)$ is contractible (\cite[Remark 1.4]{CP}), the presentation \eqref{STresol} gives the short exact sequence \begin{equation*}\label{St:ses}0\longrightarrow \redchain_2\left(\SBA(\Z^2);\Q\right)\longrightarrow \redchain_1\left(\SBA(\Z^2);\Q\right)\longrightarrow \St_2(\Q)\otimes \Q\longrightarrow 0,\end{equation*} which induces the following long exact sequence that we will use to compute the coinvariants $\left(\St_2(\Q)\otimes\Q\right)_{\Gamma_{0}^\pm(p)}$

\FloatBarrier
\begin{figure}[H]
    \begin{tikzpicture}
        % Define styles
        \tikzstyle{myarr}  = [-stealth] % Arrows

        % Nodes
        \node (C2) at (0.6,0)       {$\redchain_2(\SBA(\Z^2);\Q)_{\Gamma_{0}^\pm(p)}$};
        \node (C1) at (5.2,0)       {$\redchain_1(\SBA(\Z^2);\Q)_{\Gamma_{0}^\pm(p)}$};
        \node (C0) at (9.5,0)       {$\left(\St_2(\Q)\otimes\Q\right)_{\Gamma_{0}^\pm(p)}$};
        \node (Zero) at (12,0)     {$0$};
        \node (dots2) at (1,1.5)   {$\dots$};
        \node (C4) at (4.5,1.5)       {$\homology_1(\Gamma_{0}^\pm(p);\redchain_1(\SBA(\Z^2);\Q))$};
        \node (C3) at (9.5,1.5)       {$\homology_1(\Gamma_{0}^\pm(p);\St_2(\Q)\otimes\Q)$};

        % Horizontal arrows
        \draw [myarr] (C2) -- (C1);
        \draw [myarr] (C1) -- (C0);
        \draw [myarr] (C0) -- (Zero);
        \draw [myarr] (dots2) -- (C4);
        \draw [myarr] (C4) -- (C3);
        
        \draw [myarr] (C3)  .. controls  (15,1.2) and (-2,1).. (C2);\end{tikzpicture}\end{figure}
Moreover, $\redchain_1\left(\SBA(\Z^2);\Q\right)$ is a flat $\Gamma_{0}^\pm(p)$-module (e.g. see \cite[Lemma 3.2]{CP}), giving \[\homology_1\left(\Gamma_{0}^\pm(p);\redchain_1\left(\SBA(\Z^2);\Q\right)\right)\cong  0.\]
In addition, results of Putman--Studenmund in \cite[Theorem C]{Putman-Studenmund}, show that \[\homology_1(\Gamma_{0}^\pm(p);\St_2(\Q)\otimes \Q)\cong \homology^0(\Gamma_{0}^\pm(p);\Q^{\det})\cong 0.\] 
As a consequence, the dimension of $(\St_2(\Q)\otimes\Q)_{\Gamma_{0}^\pm(p)}$ is\begin{equation}\label{dim=St}\dim \left((\St_2(\Q)\otimes\Q)_{\Gamma_{0}^\pm(p)}\right)=\dim\left(\redchain_1\left(\SBA(\Z^2);\Q\right)_{\Gamma_{0}^\pm(p)}\right)-\dim\left(\redchain_2\left(\SBA(\Z^2);\Q\right)_{\Gamma_{0}^\pm(p)}\right).\end{equation} 
Moreover, \autoref{chaincoinv} gives \[\redchain_1\left(\SBA(\Z^2);\Q\right)_{\Gamma_{0}^\pm(p)}\cong \redchain_1\left(\Gamma_{0}^\pm(p)\backslash \SBA(\Z^2);\Q\right),\]which is isomorphic to $\redchain_1\left(\Pb_n^\pm(\F_p)\backslash \SBA_2(\F_p);\Q\right)$ by \autoref{good q}. We then deduce from \autoref{coinv} that $\redchain_1\left(\SBA(\Z^2);\Q\right)_{\Gamma_{0}^\pm(p)}$ has a non-canonical basis bijective to $\Pb_2^\pm(\F_p)\backslash \SBDA(\F_p^2)_1^\mathrm{pr}/\Sigma_{2}$, and that $\redchain_2\left(\SBA(\Z^2);\Q\right)_{\Gamma_{0}^\pm(p)}$ has a non-canonical basis bijective to $\Pb_2^\pm(\F_p)\backslash \SBDA(\F_p^2)_2^\mathrm{pr}/\Sigma_{3}$. 
Therefore, to determine $\left(\St_2(\Q)\otimes\Q\right)_{\Gamma_{0}^\pm(p)}$, we will find all the simplices in the symmetric $\Delta$-complex $\SBDA(\F_p^2)$ on which the action of $\Pb_2^\pm(\F_p)$ is orientation-preserving. We start with the $1$-simplices.

Recall from \autoref{matrixrep} that \[\SBDA(\F_p^2)_1/\Sigma_2\cong \MD_2(\F_p)_1/T_1.\] 
Thus, as mentioned in \autoref{rmk:ident}, we may identify the coset $\sigma \Sigma_2$ with the coset $B\cdot T_1$ for some $B\in \MD_2(\F_p)_1$. Moreover, we recall from \autoref{edge:c}, that every orbit of $\Pb_2^\pm(\F_p)\backslash \MD_2(\F_p)_1/T_1$ has the form $E(1,c)\cdot T_1$ for some $c\in\F_p$. 

\begin{proposition} \label{edgeinp=1mod4} Let $\sigma\in \SBDA(\F_p^2)_1$ and let $B\in \MD_2(\F_p)_1$ such that $\sigma\Sigma_2=B\cdot T_1$. Suppose $B\in \E(1,c)$ for some $c\in\F_p$. Then the action of $\Pb_2^\pm(\F_p)$ on $\sigma$ is orientation-reversing if and only if $c^4=1$.
\end{proposition}

\begin{proof}
Recall that for $(a_1,a_2)\in\F_p^2\setminus\{\underline{0}\}$, \[\E(a_1,a_2)=\left\{ B=(v_1| v_2)\in\MD_2(\F_p)_1~\middle\vert\ \lambda e_1=a_1v_1+a_2v_2~\text{for some $\lambda\in \F_p^\times$}\right\}.\]
Let $ \sigma \Sigma_2= B\cdot T_1$ with $B\in \E(1,c)$ for some $c\in\F_p$. First, We will find an equivalent condition for the action of $\Pb_2^\pm(\F_p)$ on $\sigma$ to be orientation-reversing. We have by \autoref{lem:goodmat} that \[\SBDA(\F_p^2)_1^\mathrm{rv}/\Sigma_{2}\cong \MDA_2(\F_p)_1^\mathrm{rv}/T_1.\]
Thus, the action of $\Pb_2^\pm(\F_p)$ on $\sigma$ is orientation-reversing if and only if the action of $\Pb_2^\pm(\F_p)$ on $B$ is orientation-reversing. By \autoref{def:action on B}, this means that there exist \[A\in \Pb_2^\pm(\F_p),\quad X\in T_1 \text{ with $\sign(X)=-1$},\] such that \[B\cdot X=A\cdot B.\]
Since $\E(1,c)=\Pb_2^\pm(\F_p)\cdot B$, property (c) of \autoref{edgetype} gives that this condition is equivalent to \[E((1,c)\cdot X)=E(1,c).\] 
Using property (b) and property (c) of \autoref{edgetype}: \[E((1,c)\cdot X)=E(1,c)\Longleftrightarrow (1,c)\cdot X=\lambda(1,c)~\text{for some $\lambda\in \F_p^\times$}.\] It follows that the action is orientation-reversing if and only if there exist $\lambda\in \F_p^\times$ and $X\in T_1$ with $\sign(X)=-1$ such that \begin{equation}\label{X}(1,c)\cdot X=(\lambda,\lambda c).\end{equation}Note that $\sign(X)=-1$ if and only if $X=((1\; 2),\varepsilon_1,\varepsilon_2)\in T_1$.
Writing $X=((1\; 2),\varepsilon_1,\varepsilon_2)\in T_1$ with $\varepsilon_1,\varepsilon_2\in\{-1,1\}$, the condition \eqref{X} is equivalent to the system
\begin{equation}\label{syst}\begin{cases}
\lambda = \varepsilon_1 c \\
\lambda c = \varepsilon_2.\end{cases}
\end{equation}
If \eqref{syst} holds, then
\[c^2=\varepsilon_1\varepsilon_2=\pm1,\]
and hence
\[c^4=1.\]
This proves the ``only if'' direction.

Conversely, if $c^2=\pm1$, choose $\varepsilon_1,\varepsilon_2\in\{\pm1\}$ such that
$\varepsilon_1\varepsilon_2=c^2$, and set $\lambda=\varepsilon_1 c$.
Then
\[\lambda c=\varepsilon_1 c^2=\varepsilon_2,\]
so \eqref{syst} is satisfied. This completes the proof. \end{proof} 
We are now ready to compute the dimension of $\Q\left[\Pb_2^\pm(\F_p)\backslash \SBDA(\F_p^2)_1^\mathrm{pr}/\Sigma_2\right]$.
\begin{corollary} \label{edgecount}
The dimension of $\Q\left[\Pb_2^\pm(\F_p)\backslash \SBDA(\F_p^2)_1^\mathrm{pr}/\Sigma_2\right]$ is 
    \[\dim\left( \Q\left[\Pb_2^\pm(\F_p)\backslash \SBDA(\F_p^2)_1^\mathrm{pr}/\Sigma_2\right]\right)=\begin{cases}
    1& ~\text{if}~ p =2,\\[1em]
    \frac{p+1}{4} & ~\text{if}~ p \equiv 3 \bmod 4,\\[1em]
    \frac{p-1}{4} & ~\text{if}~ p \equiv 1 \bmod 4.\\[1em]
    \end{cases}\]
\end{corollary}

\begin{proof} 
We recall from \autoref{lem:goodmat} that \[\Pb_2^\pm(\F_p)\backslash\SBDA(\F_p^2)^\mathrm{pr}/\Sigma_2\cong \Pb_2^\pm(\F_p)\backslash\MD_2(\F_p)_1^\mathrm{pr}/T_1.\]We also recall from \autoref{edge:c} that every orbit of $\Pb_2^\pm(\F_p)\backslash\MD_2(\F_p)_1/T_1$ may be written in the from $\E(1,c)\cdot T_1$ for some $c\in\F_p$.

It follows from \autoref{cor:edge} the bijection \begin{align*}\Pb_2^\pm(\F_p)\backslash\MD_2(\F_p)_1/T_1&\overset{\cong}\rightarrow\F_p/\sim\\\E(1,c)\cdot T_1&\mapsto c\end{align*} where the equivalence classes of $\sim$ are \begin{equation*}\{c, c^{-1}, -c, -c^{-1}\}\end{equation*} and the only values of $c$ for which two of these elements coincide are $c\in\{0,-1,1,-i,i\}$ where $i$ denotes an element of $\F_p$ such that $i^2=-1$. Note that the equation $i^2=-1$ has a solution in $\F_p$ if and only if $p\equiv 1\bmod 4$.
 
Additionally, by \autoref{edgeinp=1mod4}, the value $c$ corresponds to an orientation-reversing action if and only if $c^4=1$; that is $c\in\{\pm1, \pm i\}$. 
Therefore, \begin{itemize}
    \item If $p=2$, then $c\in\{0,1\}$. The value $c=1$ corresponds to an orientation-reversing action. So the \[
\dim \Q\left[\Pb_2^\pm(\F_2)\backslash
\SBDA(\F_2^2)_1^{\mathrm{pr}}/\Sigma_2\right]=1.\]
    \item If $p\equiv 3\bmod 4$, two of the elements $c,c^{-1},-c-c^{-1}$ coincide if and only if $c\in\{ 0,-1,1\}$. The remaining $p-3$ values of $c$ form $\frac{p-3}{4}$ orientation-preserving actions. The values $c=\pm 1$ correspond to orientation-reversing actions, so the dimension is \[1+\frac{p-3}{4}=\frac{p+1}{4}.\]
    \item If $p\equiv 1\bmod4$, two of the elements $c,c^{-1},-c-c^{-1}$ coincide if and only if $c\in\{0,-1,1,-i,i\}$.
The remaining $p-5$ values of $c$ form $\frac{p-5}{4}$ orientation-preserving actions. The values $c=\pm1,\pm i$ correspond to orientation-reversing simplices,  giving dimension \[1+\frac{p-5}{4}=\frac{p-1}{4}.\qedhere\]
\end{itemize}
\end{proof}
Next, we study the $2$-simplices of $\SBDA(\F_p^2)$. Recall from \autoref{matrixrep} that \[\SBDA(\F_p^2)_2/\Sigma_3\cong \MDA_2(\F_p)_2/G_2.\]
Let $\sigma\in \SBDA(\F_p^2)_2$. It follows that, under this isomorphism, we may identify the coset $\sigma \Sigma_3$ with the coset $B\cdot X_2$ for some $B\in \MDA_2(\F_p)_2$.

\begin{proposition}\label{faceorbit} 
Let $\sigma \in \SBDA(\F_p^2)_2$ and let $B\in \MDA_2(\F_p)_2$ such that $\sigma\Sigma_3 = B\cdot X_2$.
    \begin{itemize}
        \item If $p=3$, then the action of $\Pb_2^\pm(\F_p)$ on $\sigma$ is orientation-reversing.
        \item If $p\neq 3$ and $B\in \Face(\underline{a})$ for some $\underline{a}=(a_0,a_1,a_2) \in \F_p^3\setminus\{\underline{0}\}$, then the action of $\Pb_2^\pm(\F_p)$ on $\sigma$ is orientation-reversing if and only if $a_i=\pm a_j$ for some $i,j$.
    \end{itemize}\end{proposition}
    
    \begin{proof} Let $\sigma \Sigma_3= B\cdot X_2$. We begin with the case $p=3$. By property (c) of \autoref{F,p=3}, each $\Pb_2^\pm(\F_3)$-orbit of $\MDA(\F_3^2)_2$ is of the form $\Face'(a_1,a_2)$ for some $(a_1,a_2) \in \F_3^2 \backslash  {0}$. Thus, up to the action of $\Pb_2^\pm(\F_3)$, we may assume $B\in \Face'(a_1,a_2)$. Recall that for $(a_1,a_2)\in\F_3^2\setminus\{\underline{0}\},$\[\Face'(a_1,a_2)=\left\{ (v_0|v_1| v_2)\in\MDA(\F_3^2)_2~\middle\vert\ \lambda e_1=a_1v_1+a_2v_2~\text{for some $\lambda\in \F_3^\times$}\right\}.\] 
Since \[\SBDA(\F_3^2)_2^\mathrm{rv}/\Sigma_{3}\cong \MDA(\F_3^2)_2^\mathrm{rv}/G_2\] by \autoref{lem:goodmat}, we have the following equivalences:
   \begin{align*}\Pb_2^\pm(\F_3)~\text{reverses the orientation} &~\text{of $\sigma$}\\&\Longleftrightarrow \text{there exists $X\in G_2$ with $\sign(X)=-1$ and $B\cdot X=A\cdot B$ for some $A\in \Pb_2^\pm(\F_3)$}\\&\Longleftrightarrow \text{there exists $X\in G_2$ with $\sign(X)=-1$ and $B\cdot X\in \Face'(\underline{a})$}\\ &\Longleftrightarrow \text{there exists $X\in G_2$ with $\sign(X)=-1$ and $\Face'(\underline{a})\cdot X=\Face'(\underline{a})$},\end{align*}where the last equivalence follows from property (c) of \autoref{F,p=3}. 
   
   Note that $a_1,a_2\in\F_3$. So either \[a_1=\pm a_2\quad\text{or}\quad (a_1,a_2)\in\{(\pm 1,0), (0,\pm 1)\}.\]
  \begin{itemize}
    \item If $a_1=-a_2$, then the element $X=((1\,2),-1)\in G_2$ has
$\sign(X)=-1$ and \[\Face'(\underline{a})\cdot X= \left\{\left(-v_0| -v_2| -v_1\right)\in\MDA(\F_3^2)_2~\middle\vert~\begin{array}{c}\lambda e_1= a_1v_1+a_2v_2 \\ \text{for some $\lambda\in\F_3^\times$}\end{array}\right\}.\]Using $a_1=-a_2$, we obtain\begin{align*}\Face'(\underline{a})\cdot X&=\left\{\left(-v_0| -v_2| -v_1\right)\in\MDA(\F_3^2)_2~\middle\vert~\begin{array}{c}\lambda e_1= -a_2v_1-a_1v_2 \\ \text{for some $\lambda\in\F_3^\times$}\end{array}\right\}\\&=\left\{\left(-v_0| -v_2| -v_1\right)\in\MDA(\F_3^2)_2~\middle\vert~\begin{array}{c}\lambda e_1= a_2(-v_1)+a_1(-v_2) \\ \text{for some $\lambda\in\F_3^\times$}\end{array}\right\}\\&=\left\{\left(-v_0| -v_2| -v_1\right)\in\MDA(\F_3^2)_2~\middle\vert~\begin{array}{c}\lambda e_1= a_1(-v_2)+a_2(-v_1) \\ \text{for some $\lambda\in\F_3^\times$}\end{array}\right\}\\\\&=\Face'(\underline{a}).\end{align*}
\item If $a_1=a_2$, then the element $X=((1\,2),1)$ has $\sign(X)=-1$ and
again \[\Face'(\underline{a})\cdot X= \left\{\left(v_0| v_2| v_1\right)\in\MDA(\F_3^2)_2~\middle\vert~\begin{array}{c}\lambda e_1= a_1v_1+a_2v_2 \\ \text{for some $\lambda\in\F_3^\times$}\end{array}\right\}.\]Using $a_1=a_2$, we obtain\begin{align*}\Face'(\underline{a})\cdot X&=\left\{\left(v_0| v_2| v_1\right)\in\MDA(\F_3^2)_2~\middle\vert~\begin{array}{c}\lambda e_1= a_1(v_2)+a_2(v_1) \\ \text{for some $\lambda\in\F_3^\times$}\end{array}\right\} \\&=\Face'(\underline{a}).
\end{align*}
    \item If $(a_1,a_2)\in\{(\pm1,0),(0,\pm1)\}$, then by properties (a) and (c)
of \autoref{F,p=3}, there exists $Y\in G_2$ such that $\Face'(a_1,a_2)\cdot Y=\Face'(1,1)$. Let $X=((1\,2),1)$, be as in the previous case. Then, \[\Face'(\underline{a})\cdot YXY^{-1}=\Face'(1,1)\cdot XY^{-1}=\Face'(1,1)\cdot Y^{-1}=\Face'(\underline{a}),\]
    where $\sign(YXY^{-1})=\sign(X)=-1$.
    \end{itemize}Thus, in all cases, the action of $\Pb_2^\pm(\F_3)$ on $\sigma$ is orientation-reversing.

    Now suppose $p\neq 3$ and that $B\in \Face(\underline{a})$ for some $\underline{a}=(a_0,a_1,a_2)\in\F_p^3\setminus\{\underline{0}\}$ satisfying $a_0+a_1+a_2=0$. By \autoref{lem:goodmat}, the action of $\Pb_2^\pm(\F_p)$ on $\sigma$ is orientation-reversing if and only if there exist $A\in \Pb_2^\pm(\F_p)$ and $X\in G_2$ with $\sign(X)=-1$ such that \[B\cdot X=A\cdot B.\]
    By property (b) of \autoref{faceorbit}, this condition is equivalent to the existence of an element $X\in G_2$ with $\sign(X)=-1$ such that
\[\Face(\underline{a})\cdot X=\Face(\underline{a}).\]
Using property (a) of \autoref{faceorbit}, this holds if and only if there exist $\lambda\in \F_p^\times$ and $X\in G_2$ with $\sign(X)=-1$ such that
\[\lambda\,\underline{a}=\underline{a}\cdot X.\] 
This occurs precisely when there exists $\lambda\in\F_p^\times$ and $\varepsilon\in\{-1,1\}$ such that
\begin{equation}\label{face:rv}
\lambda a_{i}=\varepsilon a_{i},\quad
\lambda a_j= \varepsilon a_\ell,\quad
\lambda a_\ell= \varepsilon a_j,
\end{equation}for distinct indices $i,j,\ell\in\{0,1,2\}$. We now proceed with the proof, using that \eqref{face:rv} characterizes when the action of $\Pb_2^\pm(\F_3)$ on $\sigma$ is orientation-reversing.

\begin{itemize}
         \item If $a_j=0$, then \eqref{face:rv} implies that $a_{\ell}=0$ since $\lambda\neq 0$. It follows by the equality $a_0+a_1+a_2=0$ that $\underline{a}=\underline{0}$, which is a contradiction.
        \item If $a_j\neq 0$, then also $a_\ell\neq 0$, and \eqref{face:rv} implies $\lambda ^2=\varepsilon^2=1$.
       
    \end{itemize}
   We conclude that $\eqref{face:rv}$ implies $\lambda =\pm 1$ and therefore $a_j=\pm a_\ell$. This proves the ``only if'' direction.

To show the converse, assume that $a_{i_0}=\pm a_{i_1}$ for some distinct indices $i_0,i_1\in\{0,1,2\}$, and let $i_2\in\{0,1,2\}\setminus\{i_1,i_1\}$. We consider two cases.\begin{itemize}
    \item If $a_{i_0}=a_{i_1}$, take $\lambda =\varepsilon=1$ to satisfy
    \eqref{face:rv}.
    \item If $a_{i_0}=- a_{i_1}$. Then $a_{i_2}=0$ since $a_0+a_1+a_2=0$. In this case take $\lambda=-\varepsilon=-1$ to satisfy \eqref{face:rv}.
\end{itemize} 
 This completes the proof.\end{proof}

\begin{corollary}\label{facecount}The dimension of $\Q\left[\Pb_2^\pm(\F_p)\backslash \SBDA(\F_p^2)^\mathrm{pr}_2/\Sigma_3\right]$ is
$$\dim\left(\Q\left[\Pb_2^\pm(\F_p)\backslash \SBDA(\F_p^2)^\mathrm{pr}_2/\Sigma_3\right]\right)=\begin{cases}
0 & ~\text{if}~ p = 2,3,\\[.5em]
    \frac{p-5}{6} & ~\text{if}~ p \equiv -1 \bmod 3,\\[1em]
    \frac{p-1}{6} & ~\text{if}~ p \equiv 1 \bmod 3.
    \end{cases} $$
\end{corollary}
\begin{proof}
The case $p=3$ is immediate from \autoref{faceorbit}, which shows that $\SBDA(\F_3^2)$ contains no simplices on which the action of $\Pb_2^\pm(\F_3)$ is orientation-preserving. 
Thus,
\[
\dim\Q\left[\Pb_2^\pm(\F_3)\backslash
\SBDA(\F_3^2)^\mathrm{pr}_2/\Sigma_3\right]=0.\] 
Now let $p\neq 3$. We recall from \autoref{lem:goodmat} that \[\Pb_2^\pm(\F_3)\backslash
\SBDA(\F_3^2)^\mathrm{pr}_2/\Sigma_3\cong \Pb_2^\pm(\F_3)\backslash
\MDA_2(\F_p)_2^\mathrm{pr}/G_2.\]
We also recall from \autoref{face:c} that every orbit of $\Pb_2^\pm(\F_p)\backslash\MDA_2(\F_p)_2/G_2$ may be written in the from $\E(1,c)\cdot T_1$ for some $c\in\F_p$.

It follows from \autoref{cor:face} the bijection  \begin{align*}\Pb_2^\pm(\F_p)\backslash \MDA_2(\F_p)_2/G_2&\overset{\cong}\rightarrow \F_p/\approx\\
\Face(1,c,-1-c)\cdot G_2&\mapsto c\end{align*} where the equivalence class of $\approx$ is 
\begin{equation}\label{rel3}\{c, c^{-1}, -(1+c), -(1+c^{-1}), -(1+c)^{-1}, -1+(1+c)^{-1}\}\end{equation}
and two of these elements coincide if and only if \[c\in\{0,-2,-2^{-1},-1,1,\gamma_1,\gamma_2\}\quad\text{with }\gamma_i^2+\gamma_i+1=0.\]Note that that $\gamma_i^2+\gamma_i+1=0$ has a solution in $\F_p$ if and only if $p\equiv 1 \bmod 3$.

Additionally, by \autoref{faceorbit}, the value $c$ corresponds to an orientation-reversing action if and only if \[c=\pm 1,\quad c=\pm(-1-c)\quad\text{or}\quad -1-c=\pm 1;\] that is, $c\in\{0,-2,-2^{-1},-1,1\}$ where $-2^{-1}$ is defined for $p\neq 2$. 

\begin{itemize}
\item If $p=2$, every value of $c$ corresponds to an orientation-reversing action. So \[\dim\Q\left[\Pb_2^\pm(\F_2)\backslash \SBDA(\F_2^2)^\mathrm{pr}_2/\Sigma_3\right]=0.\]
    \item If $p\equiv -1\bmod3$, an easy calculation shows that the five values $c\in\{0,-2,-2^{-1},-1,1\}$ are exactly those for which the action is orientation-reversing. The remaining $p-5$ values of $c$ split into equivalence classes of size $6$, yielding $\frac{p-5}{6}$ orientation-preserving orbits.
    \item If $p\equiv 1\bmod 3$, there are seven exceptional values of $c$, namely $\{0,-2,-2^{-1},-1,1,\gamma_1,\gamma_2\}$, five of which correspond to orientation-reversing actions. For the remaining two, $\gamma_1$ and $\gamma_2$, we have the following equalities:
    
    \[\Face(1,\gamma_1,-1-\gamma_1)\cdot G_2\overset{\gamma_1=\gamma_2^{-1}}{\underset{\gamma_2=-1-\gamma_1}{=}}\Face(1,\gamma_2^{-1},\gamma_2)\cdot G_2=\Face(\gamma_2,1,\gamma_2^{-1})\cdot G_2\overset{\gamma_2^2=-1-\gamma_2}
    =\Face(\gamma_2,1,-1-\gamma_2)\cdot G_2.\]It follows that \[\Face(1,\gamma_1,-1-\gamma_1)\cdot G_2=\Face(1,\gamma_2,-1-\gamma_2)\cdot G_2.\]Thus, $\gamma_1$ and $\gamma_2$ lie in the same orbit in \[\Pb_2^\pm(\F_p)\backslash \MDA_2(\F_p)_2^\mathrm{pr}/G_2,\cong \Pb_2^\pm(\F_p)\backslash \SBDA(\F_p^2)_2^\mathrm{pr}/\Sigma_3.\] Therefore, the dimension is $1+\frac{p-7}{6}=\frac{p-1}{6}$. \qedhere
\end{itemize}
\end{proof}

\begin{theorem}\label{dim: St2}
The dimension of $\left(\St_2(\Q)\otimes\Q\right)_{\Gamma_{0}^\pm(p)}$ is
    \[\dim\left(\St_2(\Q)\otimes\Q\right)_{\Gamma_{0}^\pm(p)}=\begin{cases}
        1&\text{if $p=2,3$},\\[1em]
        \frac{p-1}{12}&\text{if $p\equiv 1\bmod 12$},\\[1em]
        \frac{p+7}{12}&\text{if $p\equiv 5\bmod 12$},\\[1em]
        \frac{p+5}{12}&\text{if $p\equiv 7\bmod 12$},\\[1em]
        \frac{p+13}{12}&\text{if $p\equiv 11\bmod 12$}.\\[1em]
    \end{cases}\]
\end{theorem}
\begin{proof}
    From \eqref{dim=St}, we have \[\dim \left((\St_2(\Q)\otimes\Q)_{\Gamma_{0}^\pm(p)}\right)=\dim\Q\left[\Pb_2^\pm(\F_p)\backslash \SBDA(\F_p^2)_1^\mathrm{pr}/\Sigma_2\right]-\dim\Q\left[\Pb_2^\pm(\F_p)\backslash \SBDA(\F_p^2)_2^\mathrm{pr}/\Sigma_3\right].\]It follows from \autoref{edgecount} and \autoref{facecount} that  \[\begin{gathered}[b]\dim\left(\St_2(\Q)\otimes\Q\right)_{\Gamma_{0}^\pm(p)}=\begin{cases}
        1-0=1&\text{if $p=2,3$},\\[1em]
        \frac{p-1}{4}-\frac{p-1}{6}= \frac{p-1}{12}&\text{if $p\equiv 1\bmod 12$},\\[1em]
        \frac{p-1}{4}-\frac{p-5}{6}=\frac{p+7}{12}&\text{if $p\equiv 5\bmod 12$}.\\[1em]
        \frac{p+1}{4}-\frac{p-1}{6}=\frac{p+5}{12}&\text{if $p\equiv 7\bmod 12$},\\[1em]
         \frac{p+1}{4}-\frac{p-5}{6}= \frac{p+13}{12}&\text{if $p\equiv 11\bmod 12$}.
    \end{cases}\\ \end{gathered}\qedhere\]
\end{proof}

We also prove that the symmetric $\Delta$-complex $\Gamma^\pm_{0}(p)\backslash \SBA(\Z^2)$ is highly $\Q$-acyclic, which is a main step in showing the acyclicity in high dimensions.
\begin{proposition}\label{n=2,BA}
    The homology group $\redhom_k\left(\Gamma^\pm_{0}(p)\backslash \SBA(\Z^2);\Q\right)\cong 0$ for all $k\leq 1$ if and only if $p\in\{2,3,5,7,13\}$.
\end{proposition}
\begin{proof}
We recall by \autoref{good q} the isomorphism \[\Gamma^\pm_{0}(p)\backslash \SBA(\Z^2)\cong \Pb_2^\pm(\F_p)\backslash \SBDA(\F_p^2).\]
Thus, by \autoref{cor:vert}, the chain complex in degree $0$, $\redchain_0(\Gamma^\pm_{0}(p)\backslash \SBA(\Z^2);\Q)$, is generated by the vertices\[U=\{\lambda e_1\mid \lambda\in \F_p^\times\}\quad\text{and}\quad V = \{v \in \F_p^2 \setminus\{\underline{0}\} \mid v \notin U\}.\] It implies that the augmentation map \[\varepsilon\colon \redchain_0(\Gamma^\pm_{0}(p)\backslash \SBA(\Z^2);\Q)\longrightarrow \redchain_{-1}(\Gamma_0^\pm(\F_p)\backslash\SBA(\Z^2);\Q)=\Q,\] has kernel \[\ker \varepsilon\cong \Q\left[U-V\right].\] 
For later use, we observe that the kernel has dimension $1$.

Moreover, we have from \autoref{good q} that the chain complex in degree $1$ is \[\redchain_1(\Gamma^\pm_{0}(p)\backslash \SBA(\Z^2);\Q)\cong\redchain_1\left(\Pb_2^\pm(\F)\backslash \SBDA(\F_p^2);\Q\right),\]which is again isomorphism to $\Q\left[\Pb_2^\pm(\F)\backslash \SBDA(\F_p^2)_1^\mathrm{pr}/ T_1\right]$ by \autoref{coinv}. Furthermore, this is isomorphic to  $\Q\left[\Pb_2^\pm(\F_p)\backslash \MD_2(\F_p)_1^\mathrm{pr}/T_1\right]$ by \autoref{lem:goodmat}.

We know from \autoref{edge:c} that orbits of $\Pb_2^\pm(\F_p)\backslash \MD_2(\F_p)_1/T_1$ have the form $E(1,c)\cdot T_1$ for some $c\in\F_p$. 

Let $B=(v_1|v_2)\in \E(1,0)$. Then $\lambda e_1=1\cdot v_1+0\cdot v_2$ for some $\lambda\in\F_p^\times$. It implies that $v_1\in U$ and $v_2\in V$. Thus the degree-$1$ differential map $\partial_1$, sends $B\cdot T_1$ to \[\partial_1(B\cdot T_1)= V\cdot T_0-U\cdot T_0=V-U.\]
As $\ker \varepsilon$ is generated by such elements, we obtain $\operatorname{im}(\partial_1)=\ker \varepsilon$. Therefore\[\redhom_0\left(\Gamma^\pm_0(p)\backslash\SBA(\Z^2);\Q\right)\cong 0.\] 
It remains to show that,\[\redhom_1\left(\Gamma^\pm_0(p)\backslash\SBA(\Z^2);\Q\right)\cong 0\quad\text{if and only if}\quad p\in\{2,3,5,7,13\}.\]
Let $\SBDA(\F_p^2)'$ be the discrete $0$-dimensional subcomplex of $\SBDA(\F_p^2)$. Then \[
\redchain_k\left(\Pb_n^\pm(\F_p)\backslash\SBDA(\F_p^2)';\Q\right)=\begin{cases}
\redchain_k\left(\Pb_n^\pm(\F_p)\backslash\SBDA(\F_p^2);\Q\right)&\text{if $k=-1,0$}\\
    0&\text{otherwise}
\end{cases}\]
Let $\rel_2(p)'=\left(\Pb_2^\pm(\F_p)\backslash\SBDA(\F_p^2),\Pb_2^\pm(\F_p)\backslash\SBDA(\F_p^2)' \right)$. 
Since \[\redhom_1\left(\Pb_2^\pm(\F_p)\backslash\SBDA(\F_p^2)';\Q\right)\cong \redhom_0\left(\Pb_2^\pm(\F_p)\backslash\SBDA(\F_p^2);\Q\right)\cong 0,\] it follows by the long exact sequence \[0\rightarrow \redhom_1\left(\Pb_2^\pm(\F_p)\backslash\SBDA(\F_p^2);\Q\right)\rightarrow \homology_1(\rel_2(p)';\Q)\rightarrow \redhom_0\left(\Pb_2^\pm(\F_p)\backslash\SBDA(\F_p^2)';\Q\right)\rightarrow  0\]that \[\dim \redhom_1\left(\Pb_2^\pm(\F_p)\backslash\SBDA(\F_p^2);\Q\right)=\dim \homology_1(\rel_2(p)';\Q)-\dim \redhom_0\left(\Pb_2^\pm(\F_p)\backslash\SBDA(\F_p^2)';\Q\right).\]
We have 
\[\redhom_0\left(\Pb_2^\pm(\F_p)\backslash\SBDA(\F_p^2)';\Q\right)=\ker\left(\redchain_0\left(\Pb_2^\pm(\F_p)\backslash\SBDA(\F_p^2)';\Q\right)\longrightarrow \Q\right)=\ker \varepsilon,\]which as mentioned above, has dimension $1$. As for $\homology_1(\rel_2(p)';\Q)$, we will show that it is isomorphic to the coinvariants $\left(\St_2(\Q)\otimes\Q\right)_{\Gamma_{0}^\pm(p)}$, which we know its dimension from \autoref{dim: St2}. 

The relative chains are \[\C_k\left(\rel_2(p)';\Q\right)=\begin{cases}
\redchain_k\left(\Pb_2^\pm(\F_p)\backslash\SBDA(\F_p^2);\Q\right)&\text{if $k=1,2$}\\
0&\text{otherwise}.
\end{cases}\]
Thus, \[\homology_1(\rel_2(p)';\Q)=\coker\left(\redchain_2(\Pb_2^\pm(\F_p)\backslash\SBDA(\F_p^2);\Q)\longrightarrow \redchain_1(\Pb_2^\pm(\F_p)\backslash\SBDA(\F_p^2);\Q)\right).\] 
Moreover, we have by \autoref{chaincoinv} that \[\redchain_k(\Pb_2^\pm(\F_p)\backslash\SBDA(\F_p^2);\Q)\cong \redchain_k\left(\SBDA(\F_p^2);\Q\right)_{\Pb_2^\pm(\F_p)}.\]
It follows by \autoref{good q}, \[\redchain_k(\Pb_2^\pm(\F_p)\backslash\SBDA(\F_p^2);\Q)\cong \redchain_k\left(\SBA(\Z^2);\Q\right)_{\Gamma_{0}^\pm(p)}.\]
It then implies that \[\homology_1(\rel_2(p)';\Q)\cong \coker\left( \redchain_2\left(\SBA(\Z^2);\Q\right)_{\Gamma_{0}^\pm(p)}\longrightarrow \redchain_1\left(\SBA(\Z^2);\Q\right)_{\Gamma_{0}^\pm(p)}\right).\] Using now the short exact sequence \eqref{St:ses}  \[0\longrightarrow \redchain_2\left(\SBA(\Z^2);\Q\right)\longrightarrow \redchain_1\left(\SBA(\Z^2);\Q\right)\longrightarrow \St_2(\Q)\otimes \Q\longrightarrow 0,\]and the fact that the coinvariants is a right exact functor, we conclude that \[\homology_1(\rel_2(p)';\Q)\cong \left(\St_2(\Q)\otimes\Q\right)_{\Gamma_{0}^\pm(p)}.\] 
We recall from \autoref{dim: St2} that \[\dim\left(\St_2(\Q)\otimes\Q\right)_{\Gamma_{0}^\pm(p)}=\begin{cases}
        1&\text{if $p=2,3$},\\[1em]
         \frac{p-1}{12}&\text{if $p\equiv 1\bmod 12$},\\[1em]
        \frac{p+7}{12}&\text{if $p\equiv 5\bmod 12$},\\[1em]
        \frac{p+5}{12}&\text{if $p\equiv 7\bmod 12$},\\[1em]
         \frac{p+13}{12}&\text{if $p\equiv 11\bmod 12$}.\\[1em]
    \end{cases}\]
Therefore,  
\[\redhom_1(\Gamma_{0}^\pm(p)\backslash \SBA(\Z^2);\Q)=\dim \left(\St_2(\Q)\otimes\Q\right)_{\Gamma_{0}^\pm(p)}-1=\begin{cases}
     0  & ~\text{if}~ p = 2,3,\\[.5em]
    \frac{p-13}{12} & ~\text{if}~ p \equiv 1 \bmod 12,\\[1em]
    \frac{p-5}{12} & ~\text{if}~ p \equiv 5 \bmod 12,\\[1em]
    \frac{p-7}{12} & ~\text{if}~ p \equiv 7 \bmod 12,\\[1em]
    \frac{p+1}{12} & ~\text{if}~ p \equiv 11 \bmod 12.
    \end{cases}\\ \]
    Consequently, $\redhom_1(\Gamma_{0}^\pm(p)\backslash \SBA(\Z^2);\Q)\cong 0$ if and only if $p\in\{2,3,5,7,13\}$.
\end{proof}

\begin{comment}\begin{remark}
   The reader might notice that \[\redhom_1(\Gamma_{0}^\pm(p)\backslash \SBA(\Z^2);\Q)\]
has dimension equal to twice the genus of the classical compactified modular curve $X_0(p)$(i.e., the modular curve with cusps filled in). Miller, Patzt, and Putman proved in \cite[Lemma 2.44]{MPP} that the quotient $\Gamma_0(p)\backslash \BA_2(\Z)$
is homeomorphic to the level $p$-modular curve $X(p)$ of genus \[\frac{(p+2)(p-3)(p-5)}{24}.\] 
Consequently, our quotient
$\Gamma_{0}^\pm(p)\backslash \SBA(\Z^2)$ is homeomorphic to the classical modular curve compactified by filling in the cusps.
\end{remark}\end{comment}

\subsection{The coinvariants \texorpdfstring{$\left(\St_2(\Q)\otimes\Q^{\det}\right)_{\Gamma_{0}^\pm(p)}$}{Lg}}\label{n=2,coinv2}

\FloatBarrier
We recall from the introduction of \autoref{n=2,coinv1}, that the presentation of the Steinberg module induces the long exact sequence 

\begin{figure}[H]
    \begin{tikzpicture}
       
        \tikzstyle{myarr}  = [-stealth] 
        \node (A2) at (0.4,0) {$\left(\redchain_2(\SBA(\Z^2);\Q)\otimes\Q^{\det}\right)_{\Gamma_{0}^\pm(p)}$};
        \node (A1) at (5.7,0)     {$\left(\redchain_1(\SBA(\Z^2);\Q)\otimes\Q^{\det}\right)_{\Gamma_{0}^\pm(p)}$};
        \node (A0) at (10.5,0)       {$\left(\St_2(\Q)\otimes\Q^{\det}\right)_{\Gamma_{0}^\pm(p)}$};
        \node (Zeroo) at (13,0)     {$0$};
        \node (A4) at (4.7,1.5)       {$0$};
        \node (A3) at (8,1.5)       {$\homology_1\left(\Gamma_{0}^\pm(p);\St_2(\Q)\otimes\Q^{\det}\right)$};
        \draw [myarr] (A2) -- (A1);
        \draw [myarr] (A1) -- (A0);
        \draw [myarr] (A0) -- (Zeroo);
        \draw [myarr] (A4) -- (A3);
        
        \draw [myarr] (A3)  .. controls (15,1.2) and (-2,1) .. (A2);
    \end{tikzpicture}\end{figure}
Putman--Studenmund in \cite[Theorem C]{Putman-Studenmund} showed that \[\homology_1(\Gamma_{0}^\pm(p);\St_2(\Q)\otimes\Q^{\det})\cong \homology^0(\Gamma_{0}^\pm(p);\Q)\cong \Q.\]
As a consequence, the dimension of $(\St_2(\Q)\otimes\Q^{\det})_{\Gamma_{0}^\pm(p)}$ is given by\begin{equation}\label{dim:St,tw}\dim \left((\St_2(\Q)\otimes\Q^{\det})_{\Gamma_{0}^\pm(p)}\right)=1-\dim\left(\left(\redchain_2\left(\SBA(\Z^2);\Q\right)\otimes\Q^{\det}\right)_{\Gamma_{0}^\pm(p)}\right)+\dim\left(\left(\redchain_1\left(\SBA(\Z^2);\Q\right)\otimes\Q^{\det}\right)_{\Gamma_{0}^\pm(p)}\right).\end{equation} 
By \eqref{ses}, we have for $k=1,2$, \[\left(\redchain_k\left(\SBA(\Z^2);\Q\right)\otimes\Q^{\det}\right)_{\Gamma_{0}^\pm(p)}\cong \left(\left(\redchain_k\left(\SBA(\Z^2);\Q\right)\otimes\Q^{\det}\right)_{\Gamma_{2}(p)}\right)_{\Pb_2^\pm(\F_p)}\cong \left(\redchain_k\left(\SBA(\Z^2);\Q\right)_{\Gamma_{2}(p)}\otimes\Q^{\det}\right)_{\Pb_2^\pm(\F_p)}.\]
By \autoref{chaincoinv}, there is an isomorphism $\redchain_k\left(\SBA(\Z^2);\Q\right)_{\Gamma_{2}(p)}\cong \redchain_k\left(\Gamma_2(p)\backslash\SBA(\Z^2);\Q\right)$. Moreover, \autoref{chaincoinv} together with \autoref{lem1} and \autoref{lem2} give that
\[\left(\redchain_k\left(\SBA(\Z^2);\Q\right)\otimes\Q^{\det}\right)_{\Gamma_{0}^\pm(p)}\cong
\left(\redchain_k\left(\SBDA(\Z^2);\Q\right)\otimes\Q^{\det}\right)_{\Pb_2^\pm(\F_p)}\quad\text{for $k=1,2$}.\]
From \autoref{coinv,tw}, $\left(\redchain_1\left(\SBA(\Z^2);\Q\right)\otimes\Q^{\det}\right)_{\Gamma_{0}^\pm(p)}$ has a non-canonical basis bijective to  $\Pb_2^\pm(\F_p)\backslash\SBDA(\F_p^2)_1^\mathrm{utw}/\Sigma_2$,
and $\left(\redchain_2\left(\SBA(\Z^2);\Q\right)\otimes\Q^{\det}\right)_{\Gamma_{0}^\pm(p)}$ has a non-canonical basis bijective to $\Pb_2^\pm(\F_p)\backslash\SBDA(\F_p^2)_2^\mathrm{utw}/\Sigma_3$. To determine their dimensions, we will find all the untwisted $1$ and $2$-simplices in $\SBDA(\F_p^2)$. We start with the $1$-simplices.

Recall from \autoref{rmk:ident}, that for every $\sigma\in \SBDA(\F_p^2)_1$, we may write $\sigma \Sigma_2= B\cdot T_1$ with $B\in \MD_2(\F_p)_1$. Moreover, we recall from \autoref{edge:c}, that every orbit of $\Pb_2^\pm(\F_p)\backslash \MD_2(\F_p)_1/T_1$ has the form $E(1,c)\cdot T_1$ for some $c\in\F_p$. 

\begin{proposition}\label{twisted edge}
    Let $\sigma\in \SBD(\F_p^2)_1$ and let $B\in \MD_2(\F_p)_1$ such that $ \sigma \Sigma_2=B\cdot T_1$. Suppose $B\in \E(1,c)$ for some $c\in\F_p$. Then $\sigma$ is twisted under the action of $\Pb_2^\pm(\F_p)$ if and only if one of the following is satisfied.\begin{itemize}
    \item $p=2$,
    \item $c=0$,
    \item $c^2=-1$.
    \end{itemize}  
\end{proposition}
\begin{proof}
     Let $\sigma \Sigma_2=B\cdot T_1$ with $B\in \E(1,c)$ for some $c\in\F_p$. By \autoref{good q}, we have \[\SBDA(\F_p^2)_1^\mathrm{tw}/\Sigma_{2}\cong \MDA_2(\F_p)_1^\mathrm{tw}/T_1.\]
     Thus $\sigma$ is twisted if and only if there exist some $A\in \Pb_2^\pm(\F_p)$ and $X=(\pi,\varepsilon_1,\varepsilon_2)\in T_1$ such that \[\sign(X)\det(A)=-1\quad\text{and}\quad B\cdot X=A\cdot B.\] Taking determinants in the second equality, we get \[\det(A)=\det(X)=\sign(X)\varepsilon_1\varepsilon_2.\] Therefore, \[\sign(X)\det(A)=\sign(X)^2\varepsilon_1\varepsilon_2=\varepsilon_1\varepsilon_2,\] and so the above conditions are equivalent to \[\varepsilon_1\varepsilon_2=-1\quad\text{and}\quad B\cdot X=A\cdot B.\]
     By properties (c) and (d) of \autoref{edgetype}, the condition $B\cdot X = A\cdot B$ is equivalent to \begin{equation*}\label{orbit cond}\E((1,c)\cdot X)=E(1,c).\end{equation*}
      Moreover, by property (b) of \autoref{edgetype}, we have \[E((1,c)\cdot X)=E(1,c)\Longleftrightarrow (1,c)\cdot X=\lambda(1,c)~\text{for some $\lambda\in \F_p^\times$}.\]
Combining this with the condition $\varepsilon_1\varepsilon_2=-1$, we conclude that $\sigma$ is twisted if and only if there exist $\lambda \in \F_p^\times$ and $X=(\pi,\varepsilon_1,\varepsilon_2)\in T_1$ such that
\[
\varepsilon_1 \varepsilon_2 = -1
\quad \text{and} \quad
(1,c)\cdot X = (\lambda, \lambda c).
\]
Consequently, $\sigma$ is twisted if and only if there exist $\lambda \in \F_p^\times$, $\varepsilon_1=-\varepsilon_2 \in \{-1,1\}$ such that one of the following systems holds, depending on whether $\pi=\Id$ or $\pi=(1\; 2)$.
    \[
\begin{cases}
\lambda = \varepsilon_1, \\
\lambda c = -\varepsilon_1 c,
\end{cases}
\quad
\begin{cases}
\lambda = \varepsilon_1 c, \\
\lambda c = -\varepsilon_1.
\end{cases}
\] 
We now prove the proposition using this equivalence.

If $\sigma$ is twisted, the first system implies that $p=2$ or $c=0$, while the second system implies that $c^2=-1$. This proves the ``only if '' direction. 
     
     Conversely, if $p=2$, $\lambda=\varepsilon$ satisfies the first system. If $c=0$, choosing $\lambda=\varepsilon$ satisfies the first system. If $c^2=-1$, the second system is satisfied by taking $\varepsilon=-1$ and $\lambda=-c$. 

     This completes the proof.
\end{proof}
We compute the dimension of $\Q\left[\Pb_2^\pm(\F_p)\backslash\SBDA(\F_p^2)_1^\mathrm{utw}/\Sigma_2\right]$ in the following lemma.

\begin{proposition}\label{edgeutwcount}
   The dimension of $\Q\left[\Pb_2^\pm(\F_p)\backslash\SBDA(\F_p^2)_1^\mathrm{utw}/\Sigma_2\right]$ is \[\dim\left(\Q\left[\Pb_2^\pm(\F_p)\backslash\SBDA(\F_p^2)_1^\mathrm{utw}/\Sigma_2\right]\right)=\begin{cases}
    0 &\text{if $p=2$},\\[1em]
    \frac{p+1}{4}&\text{if $p\equiv 3 \bmod 4$},\\[1em]
    \frac{p-1}{4}&\text{if $p\equiv 1 \bmod 4$}.\\[1em]
\end{cases}\]
\end{proposition}

\begin{proof}
We recall from \autoref{lem:goodmat} that \[\Pb_2^\pm(\F_p)\backslash\SBDA(\F_p^2)^\mathrm{pr}/\Sigma_2\cong \Pb_2^\pm(\F_p)\backslash\MD_2(\F_p)_1^\mathrm{pr}/T_1.\]
We also recall from \autoref{edge:c} that every orbit of $\Pb_2^\pm(\F_p)\backslash\MD_2(\F_p)_1/T_1$ may be written in the from $\E(1,c)\cdot T_1$ for some $c\in\F_p$.

It follows from \autoref{cor:edge} the bijection \begin{align*}\Pb_2^\pm(\F_p)\backslash\MD_2(\F_p)_1/T_1&\overset{\cong}\rightarrow\F_p/\sim\\\E(1,c)\cdot T_1&\mapsto c\end{align*} where the equivalence classes of $\sim$ are\begin{equation*}\{c, c^{-1}, -c,-c^{-1}\}.\end{equation*}  The only values of $c$ for which two of the elements $c,c^{-1},-c,-c^{-1}$ coincide are $c\in\{0,-1,1,-i,i\}$ where $i$ denotes an element in $\F_p$ such that $i^2=-1$. Note that the equation $i^2=-1$ has a solution in $\F_p$ if and only if $p\equiv 1\bmod 4$.
 
Additionally, by \autoref{twisted edge}, the value $c$ corresponds to a twisted action if and only if either $p=2$, $c=0$ or $c=\pm i$. Therefore, \begin{itemize}
    \item For $p=2$, all the values of $c$ correspond to twisted actions, so \[\dim\Q\left[\Pb_2^\pm(\F_2)\backslash\SBDA(\F_2^2)_1^\mathrm{utw}/\Sigma_2\right]=0.\]
    \item If $p\equiv 3\bmod 4$, two of the elements $c,c^{-1},-c-c^{-1}$ coincide if and only if $c\in\{0,-1,1\}$. The remaining $p-3$ values of $c$ form $\frac{p-3}{4}$ untwisted orbits. The value $c=0$ is the only value corresponding to a twisted action. Also, since $1\sim -1$, the values $c=1$ and $c=-1$ lie in the same orbit in $\Pb_2^\pm(\F_p)\backslash\MD_2(\F_p)_1^\mathrm{pr}/T_1$. Thus the dimension is $1+\frac{p-3}{4}=\frac{p+1}{4}$.
    \item If $p\equiv 1\bmod 4$, two of the elements $c,c^{-1},-c-c^{-1}$ coincide if and only if  $c\in\{0,-1,1,-i,i\}$.  The remaining $p-5$ values of $c$ form $\frac{p-5}{4}$ untwisted orbits. The values $c=0,-i,i$ are the only values corresponding to twisted actions. Also, the values $c=-1$ and $c=1$ lie in the same orbit. Thus the dimension is $1+\frac{p-5}{4}=\frac{p-1}{4}$.\qedhere\end{itemize}\end{proof}
We proceed to show that $\SBDA(\F_p^2)$ does not contain any twisted $2$-simplices.

\begin{proposition}\label{facetwisted}
     Let $\sigma$ be a $2$-simplex in $\SBDA(\F_p^2)$. Then $\sigma$ is not twisted under the action of $\Pb_2^\pm(\F_p)$.
\end{proposition}

\begin{proof}
    Let $\sigma \Sigma_3=B\cdot X_2$ with $B\in\MDA_2(\F_p)_2$. Suppose by contradiction that $\sigma$ is twisted. 
    By \autoref{lem:goodmat}, we have \[\SBDA(\F_p^2)_2^\mathrm{tw}/\Sigma_{3}\cong \MDA_2(\F_p)_2^\mathrm{tw}/G_2.\]Then there exist some $A\in \Pb_2^\pm(\F_p)$ and $X=(\pi,\varepsilon)\in G_2$ such that \[\sign(X)\det(A)=-1\quad\text{and}\quad B\cdot X=A\cdot B.\] The relation $B\cdot X=A\cdot B$ forces\[\varepsilon^2\sign(X)=\det(g)=\det(A).\]It then follows that $\varepsilon^2=-1$, which gives a contradiction to $\varepsilon\in\{-1,1\}$. Therefore, $\sigma$ is not twisted.
\end{proof}

\begin{corollary}\label{faceutwcount}
The dimension of $\Q\left[\Pb_2^\pm(\F_p)\backslash\SBDA(\F_p^2)_2^\mathrm{utw}/\Sigma_3\right]$ is 
\[\dim\left(\Q\left[\Pb_2^\pm(\F_p)\backslash\SBDA(\F_p^2)_2^\mathrm{utw}/\Sigma_3\right]\right)=\begin{cases}
    1 &\text{if $p=2$},\\[1em]
    2&\text{if $p=3$},\\[1em]
    \frac{p+7}{6}&\text{if $p\equiv -1 \bmod 3$},\\[1em]
    \frac{p+11}{6}&\text{if $p\equiv 1 \bmod 3$}.\\[1em]
\end{cases}\] 
    
\end{corollary}
\begin{proof}
We have by \autoref{facetwisted} that $\SBDA(\F_p^2)$ does not contain any twisted $2$-simplices. Thus \[\dim\left(\Q\left[\Pb_2^\pm(\F_p)\backslash\SBDA(\F_p^2)_2^\mathrm{utw}/\Sigma_3\right]\right)=\dim\left(\Q\left[\Pb_2^\pm(\F_p)\backslash\SBDA(\F_p^2)_2/\Sigma_3\right]\right).\]
It then follows from \autoref{matrixrep} that \begin{equation}\label{=}\dim\left(\Q\left[\Pb_2^\pm(\F_p)\backslash\SBDA(\F_p^2)_2^\mathrm{utw}/\Sigma_3\right]\right)=\dim\left(\Q\left[\Pb_2^\pm(\F_p)\backslash\MDA_2(\F_p)_2/G_2\right]\right).\end{equation}\begin{itemize}
\item For $p=3$, we have by property (d) of \autoref{F,p=3} that the orbits in $\Pb_2^\pm(\F_3)\backslash\MDA(\F_3^2)_2/G_2$ have the form $\Face'(a_1,a_2)\cdot G_2$ with $(a_1,a_2)\in\F_3^2\setminus\{0,0\}$. Moreover, \[
\Face'(1,1)\cdot G_2 
\overset{\eqref{prop,b}}{=} 
\Face'(-1,-1)\cdot G_2 
\overset{\eqref{prop,b}}{\underset{\eqref{Face',2}}{=}} 
\Face'(\pm 1,0)\cdot G_2 
\overset{\eqref{Face',1}}{=} 
\Face'(0,\pm 1)\cdot G_2.
\]and \[\Face'(1,-1)\cdot G_2\overset{\eqref{prop,b}}=\Face'(-1,1)\cdot G_2.\] 
We show that $\Face'(1,1)\cdot G_2$ and $\Face'(1,-1)\cdot G_2$ are distinct. Suppose, for contradiction, that there exist $B=(v_0|v_1|v_2)\in\Face'(1,1)$ and $X=(\pi,\varepsilon)\in G_2$ such that $B\cdot X\in\Face'(1,-1)$. Then there exists $\lambda\in\F_p^\times$ with 
\[v_{\pi(1)}-v_{\pi(2)}=\lambda (v_1+v_2),\]since both $v_{\pi(1)}-v_{\pi(2)}$ and $v_1+v_2$ are multiples of $e_1$.
We consider the possible permutations $\pi$.\begin{itemize}
    \item If $\pi=\Id$, then $v_1-v_2=\lambda(V_1+v_2)$. By linear independence of $v_1$ and $v_2$, this forces $\lambda = 0$, a contradiction.
    \item If $\pi=(1\; 2)$, then $v_2-v_1=\lambda (v_1+v_2)$, which again implies $\lambda = 0$, contradicting $\lambda \in \F_p^\times$.
    \item If $\pi(i)=0$ for some $i \in \{1,2\}$, assume without loss of generality that $\pi(1)=0$. Using $v_0 = -v_1 - v_2$, we obtain $v_0-v_{\pi(2)}=-\lambda v_0.$ This implies that $v_{\pi(2)} = (1+\lambda)v_0$, which contradicts the assumption that $v_{\pi(2)} \neq 0$ and is independent from $v_0$.
\end{itemize}
Thus $\Pb_2^\pm(\F_3)\backslash\MDA(\F_3^2)_2^\mathrm{rv}/G_2$ contains exactly the two orbits, $\Face'(1,1)\cdot G_2$ and $\Face'(1,-1)\cdot G_2$. So the dimension is \[\dim\left(\Q\left[\Pb_2^\pm(\F_3)\backslash\MDA(\F_3^2)_2/G_2\right]\right)=2.\]
\item For $p\neq 3$, we use the following equality \[\dim\left(\Q\left[\Pb_2^\pm(\F_p)\backslash\MDA_2(\F_p)_2/G_2\right]\right)=\dim\left(\Q\left[\Pb_2^\pm(\F_p)\backslash\MDA_2(\F_p)_2^\mathrm{pr}/G_2\right]\right)+\dim\left(\Q\left[\Pb_2^\pm(\F_p)\backslash\MDA_2(\F_p)_2^\mathrm{rv}/G_2\right]\right).\]
From \autoref{face:c}, every orbit in $\Pb_2^\pm(\F_p)\backslash \MDA_2(\F_p)_2/G_2$ has the form $\Face(1,c,-1-c)\cdot G_2$ for some $c\in\F_p$. Moreover, by \autoref{cor:face}, there is a bijection \begin{align*}\Pb_2^\pm(\F_p)\backslash \MDA_2(\F_p)_2/G_2&\overset{\cong}\rightarrow \F_p/\approx\\
\Face(1,c,-1-c)\cdot G_2&\mapsto c\end{align*} where the equivalence class of $\approx$ is 
\begin{equation}\label{rel4}\{c, c^{-1}, -(1+c), -(1+c^{-1}), -(1+c)^{-1}, -1+(1+c)^{-1}\}.\end{equation} Furthermore, by \autoref{faceorbit}, an orbit corresponds to an orientation-reversing action if and only if $c\in\{0,-2,-2^{-1},-1,1\}$, where $-2^{-1}$ is defined for $p\neq 2$. 
\begin{itemize}
    \item For $p=2$, $c\in\{0,1\}$. So the double quotient $\Pb_2^\pm(\F_2)\backslash\MDA_2(\F_2)_2^\mathrm{rv}/G_2$ contains a single orbit, represented by $\Face(1,1,0)\cdot G_2$. Therefore,  \[
    \dim\left(\Q\left[\Pb_2^\pm(\F_2)\backslash\MDA_2(\F_2)_2^{\mathrm{rv}}/G_2\right]\right)=1.
    \]
    \item For $p\geq 5$, using \eqref{rel4}, we compute \[
    0 \approx -(1+0) = -1, \quad
    1 \approx -(1+1) = -2 \approx (-2)^{-1} = -2^{-1}.
    \] Thus the values of $c$ corresponding to orientation-reversing actions form exactly two equivalence classes, represented by $0$ and $1$. It follows that the double quotient $\Pb_2^\pm(\F_p)\backslash\MDA_2(\F_p)_2^{\mathrm{rv}}/G_2$
    contains exactly two orbits. Hence,
    \[\dim\left(\Q\left[\Pb_2^\pm(\F_p)\backslash\MDA_2(\F_p)_2^{\mathrm{rv}}/G_2\right]\right)=2.\]
\end{itemize}
Furthermore, by \autoref{good q}, we have \[\Pb_2^\pm(\F_p)\backslash\MDA_2(\F_p)_2^\mathrm{pr}/G_2\cong \Pb_2^\pm(\F_p)\backslash\SBDA(\F_p^2)_2^\mathrm{pr}/\Sigma_3.\] Consequently, using \autoref{facecount}, we obtain 
\[
\begin{gathered}[b]\dim\left(\Q\left[\Pb_2^\pm(\F_p)\backslash \SBDA(\F_p^2)^\mathrm{utw}_2/\Sigma_3\right]\right)=\begin{cases}
0+1=1& \text{if $p=2$},\\[1em]

    \frac{p-5}{6}+2=\frac{p+7}{6} & ~\text{if}~ p \equiv -1 \bmod 3,\\[1em]
    \frac{p-1}{6}+2=\frac{p+11}{6} & ~\text{if}~ p \equiv 1 \bmod 3.
    \end{cases}\\ \end{gathered}\qedhere\]
    \end{itemize}
\end{proof}
\begin{theorem}\label{det St}
The dimension of $\left(\St_2(\Q)\otimes\Q^{\det}\right)_{\Gamma_{0}^\pm(p)}$ is
    \[\dim\left(\St_2(\Q)\otimes\Q^{\det}\right)_{\Gamma_{0}^\pm(p)}=\begin{cases}
        0&\text{if $p=2,3$},\\[1em]
        \frac{p-13}{12}&\text{if $p\equiv 1\bmod 12$},\\[1em]
        \frac{p-5}{12}&\text{if $p\equiv 5\bmod 12$},\\[1em]
        \frac{p-7}{12}&\text{if $p\equiv 7\bmod 12$},\\[1em]
        \frac{p+1}{12}&\text{if $p\equiv 11\bmod 12$}.\\[1em]
    \end{cases}\]
\end{theorem}
\begin{proof}
    From \eqref{dim:St,tw}, we have \[\dim \left((\St_2(\Q)\otimes\Q^{\det})_{\Gamma_{0}^\pm(p)}\right)=1+ \dim\Q\left[\Pb_2^\pm(\F_p)\backslash\SBDA(\F_p^2)_1^\mathrm{utw}/\Sigma_2\right] -\dim\Q\left[\Pb_2^\pm(\F_p)\backslash\SBDA(\F_p^2)_2^\mathrm{utw}/\Sigma_3\right].\]It follows by \autoref{edgeutwcount} and \autoref{faceutwcount} that  \[\begin{gathered}[b]\dim\left(\St_2(\Q)\otimes\Q^{\det}\right)_{\Gamma_{0}^\pm(p)}=\begin{cases}
        1+0-1=0&\text{if $p=2$},\\[1em]
        1+1-2=0&\text{if $p=3$},\\[1em]
        1+\frac{p-1}{4}-\frac{p+11}{6}= \frac{p-13}{12}&\text{if $p\equiv 1\bmod 12$},\\[1em]
        1+\frac{p-1}{4}-\frac{p+7}{6}=\frac{p-5}{12}&\text{if $p\equiv 5\bmod 12$},\\[1em]
        1+\frac{p+1}{4}-\frac{p+11}{6}=\frac{p-7}{12}&\text{if $p\equiv 7\bmod 12$},\\[1em]
         1+\frac{p+1}{4}-\frac{p+7}{6}= \frac{p+1}{12}&\text{if $p\equiv 11\bmod 12$}.
    \end{cases}\\ \end{gathered}\qedhere\]
\end{proof}

\section{Quotients}\label{sec:quot}
Let $p$ be a prime. In this section, we study how the congruence subgroup $\Gamma_{0,n}^\pm(p)$ acts on both the poset $\T_n(\Q)$ and the partial (augmented) frames complex over $\Z$ for a general $n$, and we give an explicit description of the resulting quotient spaces.

\subsection{Tits building}
\begin{definition}\label{titsbldg}
    Let $F$ be a field and let $V$ be a finite dimensional vector space over $F$. Define $\T(V)$ to be the poset of proper nonzero subspaces of $V$, ordered by inclusion. We set $\T_n(F)=\T(F^n)$. We call the associated order complex of $\T_n(F)$, the \emph{Tits building} and we denote it by $\mathcal{T}_n(F)$.
\end{definition}

\begin{definition}
    Let $F$ be a field and let $V$ be a $d$-dimensional vector space. We call an element $\omega \in \wedge ^d V \cong F$ an \emph{orientation} on $V$ if it generates it as an $\F$-vector space. A $\pm$-\emph{orientation} on $V$ is a $\pm$-vector $\pm \omega$, where $\omega$ is an orientation on $V$.
\end{definition}
\begin{definition}\label{pmTitsbldg}
    Let $F$ be a field and let $V$ be a $d$-dimensional vector space. Define $\T^{\pm}(V)$ to be the poset whose elements are pairs $(U,\pm \omega)$, where $U\subsetneq V$ is a proper nonzero subspace and $\pm\omega$ is a $\pm$-orientation on $U$. The order relation is
    \[(U,\pm \omega)\subseteq (U',\pm \omega')\quad\text{if and only if}\quad U\subseteq U'.\]We set $\T^{\pm}(F^n) = \T^{\pm}_n(F)$. We call the associated order complex of $\T^\pm_n(F)$, the \emph{$\pm$-oriented Tits building} and we denote it by $\mathcal{T}^\pm_n(F)$.
\end{definition}

In \cite{MPP}, Miller--Patzt--Putman proved the following proposition.

\begin{proposition}[{\cite[Proposition 3.16]{MPP}}]\label{TitsMPP} For all $n \geq 1$, $\Gamma_n(p)\backslash \T_n(\Q)\cong \T^{\pm}_n(\F_p)$.
\end{proposition}

\begin{lemma}\label{Titsbldg:dq}
For all $n\geq 1$,
    \[\Gamma_{0,n}^\pm(p)\backslash \T_n(\Q) \cong \Pb_n^\pm(\F_p)\backslash \T^{\pm}_n(\F_p).\]
\end{lemma}
\begin{proof}
    It follows directly from \eqref{ses} and \autoref{TitsMPP}.
\end{proof}
\begin{lemma}\label{subspaces}
For all $n\geq 2$ and all $1\leq k\leq n-1$, the set of $\pm$-oriented $k$-dimensional subspaces of $\F_p^n$ decomposes into two $\Pb_n^\pm(\F_p)$-orbits. These are represented by
\[
(U_k,\pm \omega_k)
\quad \text{and} \quad
(V_k,\pm \omega_k'),
\]
where
\[
U_k=\operatorname{span}\{ e_1,\dots,e_k\}, 
\quad
V_k=\operatorname{span}\{ e_2,\dots,e_{k+1}\},
\]
and $\omega_k=e_1\wedge\cdots\wedge e_k$, $\omega_k'=e_2\wedge\cdots\wedge e_{k+1}$.
\end{lemma}

\begin{proof}
We recall the definition of $\Pb_n^\pm(\F_p)$, \begin{equation*}\Pb_n^\pm(\F_p)=\left\{A\in \GL_n(\F_p)\mid \det(A)=\pm 1~\text{and}~A\cdot e_1=\lambda e_1~\text{for some $\lambda\in\F_p^\times$} \right\}.\end{equation*}
Let $1\leq k\leq n-1$. Let $(U,\pm\omega)$ be a $\pm$-oriented $k$-dimensional subspace of $\F_p^n$.

First suppose that $e_1\in U$. We want to find some $A\in \Pb_n^\pm(\F_p)$ such that $A\cdot (U\,\pm\omega)=(U_k,\pm\omega_k)$. Choose a basis $\{u_1,\dots,u_k\}$ of $U$ such that \[\pm\omega=\pm u_1\wedge\dots\wedge u_k.\]Additionally, there is for some $(a_1,\dots,a_k)\in \F_p^k\setminus\{\underline{0}\}$ such that
\[e_1=\sum_{i=1}^k a_i u_i.\]
 After reordering the basis if necessary, we may assume that $a_1\neq 0$. Extend $\{u_1,\dots,u_k\}$ to a basis $\{u_1,\dots,u_n\}$ of $\F_p^n$ in such a way that
\[
\det(u_1|\dots|u_n)=1.
\]
Define a linear map $A$ on this basis by
\begin{align*}
A\colon u_1&\mapsto e_1-a_1^{-1}\sum_{i=2}^k a_i e_i,\\
u_i&\mapsto e_i \quad \text{for } i\neq 1.
\end{align*}
Then
\[
A\cdot e_1=\sum_{i=1}^k a_i (A\cdot u_i)
= a_1\left(e_1-a_1^{-1}\sum_{i=2}^k a_i e_i\right)+\sum_{i=2}^k a_i e_i
= a_1 e_1,
\]
so $A\cdot e_1=a_1e_1$. Moreover,
\[
\det(A)=\det(A)\det(u_1|\dots|u_n)
=\det(A\cdot u_1|\dots|A\cdot u_n)=\det\left(e_1-a_1^{-1}\sum_{i=2}^k a_i e_i \,\middle|\, e_2 \,\middle|\, \dots \,\middle|\, e_n\right)=1,
\]
since all terms involving $e_i$ with $i\geq 2$ vanish by repetition of columns. Thus
\[
A\in \Pb_n^\pm(\F_p).
\]
Now
\[
A\cdot U=\operatorname{span}\left\{e_1-a_1^{-1}\sum_{i=2}^ka_ie_i,e_2,\dots,e_k \right\}=\operatorname{span}\{e_1,\dots,e_k\}\]
and
\begin{align*}
A\cdot(\pm\omega)
&=\pm A\cdot u_1\wedge\cdots\wedge A\cdot u_k \\
&=\pm \left(e_1-a_1^{-1}\sum_{i=2}^k a_i e_i\right)\wedge e_2\wedge\cdots\wedge e_k \\
&=\pm e_1\wedge\cdots\wedge e_k.
\end{align*}
 Therefore
\[
A\cdot (U,\pm\omega)
=
(U_k,\pm \omega_k).\]

Next suppose that $e_1\notin U$. Choose a basis $\{u_1,\dots,u_k\}$ of $U$ such that
\[
\pm \omega=\pm u_1\wedge\cdots\wedge u_k
\]is a $\pm$-orientation on $U$.
Extend it to a basis $\{u_1,\dots,u_n\}$ of $\F_p^n$ with $u_n=\lambda e_1$ for some $\lambda\in\F_p^\times$, such that
\[
\det(u_1|\dots|u_n)=1.
\]
Define $A$ by
\begin{align*}
A\colon u_i&\mapsto e_{i+1}\quad \text{for }1\leq i\leq n-1,\\
u_n&\mapsto e_1.
\end{align*}
Then
\[
A\cdot e_1=\lambda^{-1}A\cdot u_n=e_1.
\]
 Also,
\[
\det(A)=\det(A)\det(u_1|\dots|u_n)
=\det(e_2|\dots|e_n|e_1)=(-1)^{n-1}.
\]
Hence
\[
A\in \Pb_n^\pm(\F_p).
\]
Furthermore,
\[
A\cdot U=\operatorname{span}\{ e_2,\dots,e_{k+1}\},\]
and
\[
A\cdot(\pm \omega)=\pm A\cdot u_1\wedge\cdots\wedge A\cdot u_k=\pm e_2\wedge\cdots\wedge e_{k+1}.
\]
Thus
\[
A\cdot (U,\pm\omega)
=(V_k,
\pm \omega_k').\]

Since an element of $\Pb_n^\pm(\F_p)$ sends $e_1$ to a scalar multiple of $e_1$, it cannot send a subspace containing $e_1$ to one not containing $e_1$. Thus the orbits represented by $(U_k,\pm \omega_k)$ and $(V_k,\pm \omega_k')$ are distinct.
\end{proof}

\begin{proposition}\label{tits}
For all $n\ge 2$, the quotient\[\Gamma_{0,n}^\pm(p)\backslash \T_n(\Q)\]
is isomorphic to the poset whose elements are the $\Pb_n^\pm(\F_p)$-orbits of $\pm$-oriented $k$-dimensional subspaces of $\F_p^n$, indexed by the representatives\[(U_k,\pm\omega_k)\quad\text{and}\quad (V_k,\pm\omega_k'), \quad 1\le k\le n-1.\]The order relations are\[\left[(U_k,\pm\omega_k)\right]\subset \left[(U_{k+1},\pm\omega_{k+1})\right],\]\[ \left[(V_k,\pm\omega_k')\right]\subset \left[(V_{k+1},\pm\omega_{k+1}')\right],\]
\[\left[(V_k,\pm\omega_{k})\right]\subset \left[(U_{k+1},\pm\omega_{k+1})\right],\] and \[\left[(U_k,\pm\omega_k)\right]\not\subset \left[(V_{k+1},\pm\omega_{k+1}')\right].\] 
\end{proposition}This is recorded in \ref{fig0}.
\begin{proof}
By \autoref{Titsbldg:dq}, we have \[\Gamma_{0,n}^\pm(p)\backslash \T_n(\Q) \cong \Pb_n^\pm(\F_p)\backslash \T^{\pm}_n(\F_p).\]
Recall that $\T^\pm_n(\F_p)$ is the poset of nonzero proper $\pm$-oriented subspaces of $\F_p^n$, ordered by inclusion. By \autoref{subspaces}, for each $1\leq k\leq n-1$, the set of $\pm$-oriented $k$-dimensional subspaces of $\F_p^n$ decomposes into two $\Pb_n^\pm(\F_p)$-orbits, represented by
\[
(U_k,\pm \omega_k)
\quad \text{and} \quad
(V_k,\pm \omega_k'),
\]
where
\[
U_k=\operatorname{span}\{ e_1,\dots,e_k\}, 
\quad
V_k=\operatorname{span}\{ e_2,\dots,e_{k+1}\},
\]
and $\omega_k=e_1\wedge\cdots\wedge e_k$, $\omega_k'=e_2\wedge\cdots\wedge e_{k+1}$. Thus the elements of $\Pb_n^\pm(\F_p)\backslash \T_n^\pm(\F_p)$ are exactly the orbit classes  \[\left[(U_k,\pm\omega_k)\right]\quad\text{and}\quad\left[(V_k,\pm\omega_k')\right],\quad 1\leq k\leq n-1.\]
Since $U_k\subset U_{k+1}$, we obtain \[\left[(U_k,\pm\omega_k)\right]\subset \left[(U_{k+1},\pm\omega_{k+1})\right].\] Similarly, since $V_{k}\subset V_{k+1}$, we have \[\left[ (V_k,\pm\omega_k')\right]\subset \left[(V_{k+1},\pm\omega_{k+1}')\right].\] Also,
\[V_k\subset U_{k+1},\]
so\[\left[(V_k,\pm\omega_k')\right]\subset\left[(U_{k+1},\pm\omega_{k+1})\right].\] 
Furthermore, $\Pb_n^\pm(\F_p)$ preserves the line spanned by $e_1$. Hence every subspace in the orbit of $(U_k,\pm\omega_k)$ contains $e_1$, whereas no subspace in the orbit of $(V_{k+1},\pm\omega_{k+1}')$ does. This completes the proof.\end{proof}

\begin{figure}[H]
\centering
\begin{tikzpicture}[node distance=0.8cm and 0.3cm]

  % Left column: bullet and text nodes
  \node (A1) {\tiny$\bullet$};
  \node (A1.1)[left=-0.1cm of A1]{$\left[(V_{n-1},\pm \omega_{n-1}')\right]$};
  
  \node (B1) [below=of A1] {\tiny$\bullet$}; 
  \node (B1.1)[left=-0.1cm of B1]{$\left[(V_{n-2},\pm\omega_{n-2}')\right]$};

  \node (C1) [below=of B1] {\tiny$\bullet$};

   \node (D1) [below=-0.2cm of C1]
   {$\vdots$};
   
  \node (E1) [below=of D1] {\tiny$\bullet$};
  \node (E1.1)[left=-0.1cm of E1]{$\left[(V_2,\pm\omega_2')\right]$};
  
  \node (F1) [below=of E1] {\tiny$\bullet$};
  \node (F1.1) [left=-0.1cm of F1] {$\left[(V_1,\pm \omega_1')\right]$};
  
  % Right column: bullet and text nodes
  \node (A2) [right=1cm of A1] {\tiny$\bullet$};
  \node (A2.1)[right=-0.1cm of A2]{$\left[(U_{n-1},\pm \omega_{n-1})\right]$};
  
  \node (B2) [below=of A2] {\tiny$\bullet$};
   \node (B2.1)[right=-0.1cm of B2]{$\left[(U_{n-2},\pm\omega_{n-2})\right]$};

  \node (C2) [below=of B2] {\tiny$\bullet$};
  
   \node (D2) [below=-0.2cm of C2]
   {$\vdots$};
   
  \node (E2) [below=of D2] {\tiny$\bullet$};
  \node (E2.1)[right=-0.1cm of E2]{$\left[(U_2,\pm\omega_2)\right]$};

  \node (F2) [below=of E2] {\tiny$\bullet$};
  \node (F2.1) [right=-0.1cm of F2] {$\left[(U_1,\pm \omega_1)\right]$};
 
  % Connect the bullet nodes with lines
  \draw (A1) -- (B1) -- (C1);
  \draw (A2) -- (B2) -- (C2);
  \draw (D2) -- (E2);
  \draw (D1) -- (E1);
  \draw (B1) -- (A2);
  \draw (C1) -- (B2);
  \draw (E1) -- (F1) -- (E2) -- (F2);

\end{tikzpicture}
\caption{The order complex whose poset is $\Pb_n^\pm(\F_p)\backslash \T^\pm_n(\F_p)$}\label{fig0}
\end{figure}
\subsection{\texorpdfstring{Determinant-$1$ partial frames}{Lg}}
Let $p$ be a prime. In this section, we will describe the actions of $\Pb_n^\pm(\F_p)$ and $T_k$ on $\MD(\F_p^n)_k$ for all $n$. We recall that $T_k=\Sigma_{k+1}\ltimes\{-1,1\}^{k+1}$ and \[\MD(\F_p^n)_k=\left\{B=(v_0|\dots|v_k)~\middle\vert\begin{array}{c}v_0,\dots,v_k~\text{are linearly independent vectors in $\F_p^n$}\\\text{and}~ \det(B)
=\pm 1 ~\text{if $k=n-1$}\end{array}\right\}.\]
\begin{definition}\label{def:a}
    Let $n\geq 1$. For all $0\leq k\leq n-1$, the group $T_k$ acts on $\F_p^{k+1}\setminus\{\underline{0}\}$ as follows: for $\underline{a}=(a_0,\dots,a_k)\in\F_p^{k+1}\setminus\{\underline{0}\}$ and $X=(\pi,
    \varepsilon_0,\dots,\varepsilon_k)\in T_k$, we define \[\underline{a}\cdot X=\left(\varepsilon_0a_{\pi(0)},\dots,\varepsilon_ka_{\pi(k)}\right).\]
\end{definition}

\begin{proposition} \label{BD/Pn}
    Let $n\geq 1$ and $0\leq k \leq n-1$. Define the set \[\D_1^{k,n} =  \left\{ \left(v_0|\dots|v_k\right)\in \MD(\F_p^n)_k \mid e_1\not\in \operatorname{span}\{v_0,\dots,v_k\}\right\},\]
    and for every $\underline{a}=(a_0,\dots,a_k)\in\F_p^{k+1}\setminus\{\underline{0}\}$, define the set \[\D_2^{k,n}(\underline{a})=\left\{B=\left(v_0|\dots|v_k\right)\in \MD(\F_p^n)_k ~\middle\vert \lambda e_1=a_0v_0+\dots+a_kv_k~\text{for some $\lambda\in\F_p^\times$}\right\}.\] These sets satisfy the following properties. For $\underline{a},\underline{b}\in\F_p^{k+!}\setminus\{\underline{0}\}$, \begin{enumerate}[label=(\alph*)]
    \item $\D_2^{k,n}(\underline{a})$ is nonempty.
    \item $\D_2^{k,n}(\underline{a})=\D_2^{k,n}(\underline{b})\quad\Longleftrightarrow\quad \underline{a}= \lambda \underline b$ for some $\lambda\in \F_p^\times$.
    \item The right action of $T_k$ on $\MD(\F_p^n)_k$ satisfies \begin{equation*}\label{D1,rel}\D_1^{k,n}\cdot X=\D_1^{k,n}\end{equation*} \begin{equation}\label{D2,rel2}\D_2^{k,n}(\underline{a})\cdot X=\D_2^{k,n}(\underline{a}\cdot X)\end{equation}for all $X\in T_k$.
    \item The set of the $\Pb_n^\pm(\F_p)$-orbits in $\MD(\F_p^n)_k$ is $\left\{\D_1^{k,n}\right\}\cup\left\{ \D_2^{k,n}(\underline{a})~|~ \underline{a}\in \F_p^{k+1}\setminus\{\underline{0}\}\right\}$.
    \end{enumerate}
Note that $\D_1^{k,n}$ only exists for $k<n-1$.
\end{proposition}
\begin{proof}
\noindent\emph{Proof of (a).} Let $\underline{a}\in\F_p^{k+1}\setminus\{\underline{0}\}$. Then $\D_2^{k,n}(\underline{a})\neq \emptyset $: since $\underline{a}\in\F_p^{k+1}\setminus\{\underline{0}\}$, there exists some $i_0$ such that $a_{i_0}\neq 0$. Choose some vectors $v_0,\dots,v_{i_0-1},v_{i_0+1},\dots,v_{n-1}$ in $\F_p^n$ such that \[\det\left(e_1|v_0|\dots|v_{i_0-1}|v_{i_0+1}|\dots|v_{n-1}\right)=a_{i_0}.\] Set \[v_{i_0}=a_{i_0}^{-1}\left(e_1-\sum_{\substack{i=0// i\neq i_0}}^ka_iv_i\right).\] Thus \[\sum_{i=0}^ka_iv_i=e_1,\quad \det\left(v_0|\dots|v_n\right)=1,\]and \[(v_0|\dots|v_k)\in \D_2^{k,n}(\underline{a}).\] 

\noindent\emph{Proof of (b).} Let $\underline{b}=(b_0,\dots,b_k)\in\F_p^{k+1}\setminus\{\underline{0}\}$, and suppose that $\D_2^{k,n}(\underline{a})=\D_2^{k,n}(\underline{b})$. Since $\D_2^{k,n}(\underline{a})\neq \emptyset$, choose an arbitrary  \[B=(v_0|\dots|v_k)\in \D_2^{k,n}(\underline{a}).\] Then there exist $\lambda_1,\lambda_2\in \F_p^\times$ such that \[\lambda_1 e_1=\sum_{i=0}^ka_iv_i\quad\text{and}\quad \lambda_2 e_1=\sum_{i=0}^kb_iv_i.\] Thus, \[\lambda_1^{-1}\sum_{i=0}^ka_iv_i=\lambda_2^{-1}\sum_{i=0}^kb_iv_i.\]Since the vectors $v_i$ are linearly independent in $\F_p^{n}$, it implies that \[\underline{a}=\lambda_1\lambda_2^{-1}\underline{b}.\] Conversely, suppose $\underline{a}=\mu \underline{b}$ for some $\mu\in \F_p^\times$. Then \begin{align*}
    \D_2^{k,n}(\underline{a})&=\left\{ (v_0|\dots|v_k)\in\MD(\F_p^n)_k~\middle\vert\ \lambda e_1=\sum_{i=0}^ka_iv_i~\text{for some $\lambda\in \F_p^\times$}\right\}\\
    &=\left\{ (v_0|\dots|v_k)\in\MD(\F_p^n)_k~\middle\vert\ \lambda e_1=\sum_{i=0}^k\mu b_iv_i~\text{for some $\lambda\in \F_p^\times$}\right\}\\
    &=\left\{ (v_0|\dots|v_k)\in\MD(\F_p^n)_k~\middle\vert\ \mu^{-1}\lambda e_1=\sum_{i=0}^kb_iv_i~\text{for some $\lambda\in \F_p^\times$}\right\}\\
    &=\D_2^{k,n}(\underline{b}).
\end{align*}
\noindent\emph{Proof of (c).} Let $X=(\pi,\varepsilon_0,\dots,\varepsilon_k)\in T_k$, and $B=(v_0|\dots|v_k)\in\MD(\F_p^n)_k$. We recall that \[B\cdot X =\left(\varepsilon_0 v_{\pi(0)}\mid \dots\mid \varepsilon_{k}v_{\pi(k)}\right),\]and\[\underline{a}\cdot X=(\varepsilon_0a_{\pi(0)},\dots,\varepsilon_ka_{\pi(k)}).\] 
 If $B\in \D_1^{k,n}$, then \[e_1\notin \operatorname{span}\{v_0,\dots,v_k\}.\] 
 It follows that $B\cdot X\in \D_1^{k,n}$, implying \begin{equation}\label{D1}\D_1^{k,n}\cdot X\subseteq \D_1^{k,n}.\end{equation}
 Applying \eqref{D1} to $X^{-1}$, we obtain \[ \D_1^{k,n}\cdot X^{-1}\subseteq \D_1^{k,n}.\] Thus \[ \D_1^{k,n}\subseteq \D_1^{k,n}\cdot X.\]Therefore $\D_1^{k,n}\cdot X=\D_1^{k,n}$.

If $B\in \D_2^{k,n}(\underline{a})$ for some $\underline{a}=(a_0,\dots,a_k)\in \F_p^{k+1}\setminus\{\underline{0}\}$, then $\lambda e_1=\sum\limits_{i=0}^ka_iv_i$ for some $\lambda_1\in \F_p^\times$. This implies that \[\lambda e_1=\varepsilon_0 a_{\pi(0)}\left(\varepsilon_0v_{\pi(0)}\right)+\dots+\varepsilon_k a_{\pi(k)}\left(\varepsilon_kv_{\pi(k)}\right).\]So the matrix \[B\cdot X=\left(\varepsilon_0 v_{\pi(0)}\mid \dots\mid\varepsilon_k v_{\pi(k)}\right)\in \D_2^{k,n}\left(\varepsilon_0 a_{\pi(0)},\dots,\varepsilon_ka_{\pi(k)}\right)= \D_2^{k,n}\left(\underline{a}\cdot X\right).\] Thus, \begin{equation}\label{x}\D_2^{k,n}(\underline{a})\cdot X\subseteq \D_2^{k,n}(\underline{a}\cdot X)\quad\text{for all $X\in T_k$}. \end{equation}
Applying \eqref{x} to $X^{-1}$, we obtain \[ \D_2^{k,n}(\underline{a}\cdot X)\cdot X^{-1}\subseteq \D_2^{k,n}(\underline{a}).\] Thus \[ \D_2^{k,n}(\underline{a}\cdot X)\subseteq \D_2^{k,n}(\underline{a})\cdot X,\]concluding \eqref{D2,rel2}. 

\noindent\emph{Proof of (d).}  We consider two cases.
    \begin{mycases}
        \case We show that $\D_1^{k,n}$ is a $\Pb_n^\pm(\F_p)$-orbit in $\MD(\F_p^n)_k$. Let $B=\left(v_0|\dots|v_k\right)\in \D_1^{k,n}$, then $e_1\not\in\operatorname{span}\{v_0,\dots,v_k\}$. It follows that for any $A\in \Pb_n^\pm(\F_p)$,  \[e_1\notin\operatorname{span}\left\{A\cdot v_0,\dots,A\cdot v_k\right\}.\]Thus, \[A\cdot B\in \D_1^{k,n}.\]  
        
        We will now show that $\Pb_n^\pm(\F_p)$ acts transitively on $\D_1^{k,n}$. Take another matrix $C=\left(u_0|\dots| u_k\right)\in \MD(\F_p^n)_k$ such that $e_1\notin \operatorname{span}\{u_0,\dots,u_k\}$. We complete the partial bases $\{v_0,\dots,v_k\}$ and $\{u_0,\dots,u_k\}$, respectively, to determinant-$1$ bases $\{v_0,\dots,v_{n-1}\}$, and $\{u_0,\dots,u_{n-1}\}$ of $\F_p^n$ with $v_{n-1}=\mu_1 e_1$ and $u_{n-1}=\mu_2$ for some $\mu_1,\mu_2\in\F_p^\times$. We consider the element $A'$ defined by $A'\cdot v_i=u_i$ for all $i$. We have \[A'\cdot  B=C,\]and so $\det(A')=\pm 1$. Additionally, \[A'\cdot  e_1=\mu_1^{-1}A'\cdot  v_{n-1}=\mu_1^{-1}u_{n-1}=\mu_1^{-1}\mu_2e_1.\]Thus $A'\in\Pb_n^\pm(\F_p)$, and so $\Pb_n^\pm(\F_p)$ acts transitively on $\D_1^{k,n}$.
        \case We now show that for any $(\underline{a})\in\F_p^{k+1}\setminus\{\underline{0}\}$, $\D_2^{k,n}(\underline{a})$ is a $\Pb_n^\pm(\F_p)$-orbit in $\MD(\F_p^n)_k$. Let $B=\left(v_0|\dots|v_k\right)\in \D_2^{k,n}(\underline{a})$ for some $\underline{a}=(a_0,\dots,a_k)\in\F_p^{k+1}\setminus\{\underline{0}\}$, then \begin{equation*}\label{eq,lin comb}
            \lambda_1 e_1=a_0v_0+\dots+a_kv_k
        \end{equation*}
        for some $\lambda_1\in\F_p^\times$.
        For any $A\in \Pb_n^\pm(\F_p)$, we have by definition that $A\cdot e_1=\lambda e_1$ for some $\lambda \in \F_p^\times$. Thus\begin{align*}a_0(A\cdot v_0)+\dots+ a_k(A\cdot v_k)&=A\cdot\left(a_0v_0+\dots+a_kv_k\right)\\&=A\cdot (\lambda_1 e_1)\\&=\lambda_1\lambda e_1.\end{align*}It implies that $A\cdot B\in \D_2^{k,n}(\underline{a})$. Thus,  $\D_2^{k,n}(\underline{a})$ is a $\Pb_n^\pm(\F_p)$-invariant subset of $\MD(\F_p^n)_k$.

        Subsequently, we show that it acts transitively. Let $C=\left(u_0|\dots| u_k\right)$ be another matrix in $\D_2^{k,n}(\underline{a})$. Then \[\lambda_2  e_1=a_0 u_0+\dots+a_ku_k,\]for some $\lambda_2 \in\F_p^\times$. We complete the partial bases $\{v_0,\dots,v_k\}$ and $\{u_0,\dots,u_k\}$, respectively, to determinant-$1$ bases $\{v_0,\dots,v_{n-1}\}$, and $\{u_0,\dots,u_{n-1}\}$ of $\F_p^n$. We now consider the element $A'$ defined on the basis as follows. \[
            v_i\longmapsto u_i\quad\text{for all $0\leq i\leq n-1$}.
           \]Observe that \[A'\cdot  B= C,\]and so $\det(A')=\pm 1.$ Moreover, \begin{align*}
            A'\cdot  e_1 &=A'\cdot\left(\lambda_1^{-1}\sum_{i=0}^ka_iv_i\right)\\&=\lambda_1^{-1}\left(a_0(A'\cdot  v_0)+\dots+a_k(A'\cdot  v_k)\right)\\&=\lambda_1^{-1}\left(a_0u_0+\dots+a_ku_k\right)\\&=\lambda_1^{-1}\lambda_2  e_1.
        \end{align*}It follows that $A'\in\Pb_n^\pm(\F_p)$, and therefore $\Pb_n^\pm(\F_p)$ acts transitively on $\D_2^{k,n}(\underline{a})$.
         \end{mycases}
        As the union of \[\D_1^{k,n}\quad\text{and}\quad \bigcup\limits_{\underline{a}\in \F_p^{k+1}\setminus\{\underline{0}\}} \D_2^{k,n}(\underline{a})\] comprises the whole set $\MD(\F_p^n)_k$, we are done.  
\end{proof}

\subsection{\texorpdfstring{Determinant-$1$ partial augmented frames}{Lg}}
 Let $p$ be a prime and $n\geq 2$. We will describe the actions of $\Pb_n^\pm(\F_p)$ and $G_k$ on $\MDA(\F_p^n)_k$. We recall that \[G_k=\Sigma_{\{0,1,2\}}\times \{-1,1\}\times\left(\Sigma_{\{3,\dots,k\}}\ltimes \{-1,1\}^{k-2}\right)\] and \[\MDA(\F_p^n)_k=\left\{B=(v_0|\dots|v_k)~\middle\vert\begin{array}{c}v_0,\dots,v_k~\text{are linearly independent vectors in $\F_p^n$}\\\text{and}~ \det\left(v_1|\dots|v_n \right)
=\pm 1 ~\text{if $k=n$}\end{array}\right\}.\]
\begin{definition}\label{def:a'}
    Let $2\leq k\leq n$. Define the set \[A_k=\left\{\underline{a}=(a_0,\dots,a_k)\in\F_p^{k+1}\setminus\{\underline{0}\}~\middle\vert~
a_0+a_1+a_2=0\right\}.\] The group $G_k$ acts on $A_k$ as follows: for $\underline{a}=(a_0,\dots,a_k)\in A_k$ and $X=(\pi,\varepsilon,\tau,
    \varepsilon_3,\dots,\varepsilon_k)\in G_k$, we define \[\underline{a}\cdot X=\left(\varepsilon a_{\pi(0)},\varepsilon a_{\pi(1)},\varepsilon a_{\pi(2)},\varepsilon_3 a_{\tau(3)},\dots,\varepsilon_ka_{\tau(k)}\right).\]
\end{definition}
As in \autoref{sec:n=2}, the techniques in the cases $p=3$ and $p\neq 3$ are different. So we will address these two cases separately. We begin with the case $p\neq 3$ and we prove the following lemma.

\begin{lemma}\label{lem:lin comb}
    Let $p\neq 3$ be a prime and $2\leq k\leq n$. 
    Let $v_1,\dots,v_k$ be linearly independent vectors in $\F_p^n$, such that \[\lambda e_1= b_1v_1+\dots +b_kv_k,\]for some $\lambda\in\F_p^\times$ and $(b_1,\dots,b_k)\in\F_p^{k}\setminus\{\underline{0}\}$. Let $v_0=-(v_1+v_2)$. Then there exists a nonzero vector $(a_0,\dots,a_k)\in A_{k}$ such that \begin{equation*}
    \label{system}
     \lambda e_1=a_0v_0+\dots+a_kv_k.
    \end{equation*}
    \end{lemma}
    \begin{proof}
   Let \begin{equation*}\lambda e_1=b_1v_1+\dots +b_kv_k\end{equation*}for some $\lambda\in\F_p^\times$ and $(b_1,\dots,b_k)\in\F_p^k\setminus\{\underline{0}\}$. Let $v_0=-(v_1+v_2)$. Then \[
    a_0=-3^{-1}(b_1+b_2),\quad a_1=3^{-1}(2b_1-b_2),\quad a_2=3^{-1}(2b_2-b_1),\quad a_i=b_i~\text{for all $i\geq 3$}\] satisfies \[\lambda e_1=a_0v_0+\dots+a_kv_k,\quad a_0+a_1+a_2=0.\qedhere\]
    \end{proof}
    
\begin{proposition}\label{BDA/Pn}
 Let $p\neq 3$ be a prime and $2\leq k \leq n$. Define the set \[\DA_1^{k,n} =  \left\{ \left(v_0|\dots|v_k\right)\in \MDA(\F_p^n)_k ~\middle\vert e_1\not\in \operatorname{span}\{v_0,\dots,v_k\}\right\},\]
    and for every $\underline{a}=(a_0,\dots,a_k)\in A_k$, define the set \[\DA_2^{k,n}(\underline{a})=\left\{B=\left(v_0|\dots|v_k\right)\in \MDA(\F_p^n)_k ~\middle\vert \lambda e_1=a_0v_0+\dots+a_kv_k~\text{for some $\lambda\in\F_p^\times$}\right\}.\] These sets satisfy the following properties. For $\underline{a}=\underline{b}\in A_k$,\begin{enumerate}[label=(\alph*)]
    \item $\DA_2^{k,n}(\underline{a})$ is nonempty.
    \item $\DA_2^{k,n}(\underline{a})=\DA_2^{k,n}(\underline{b})\quad\Longleftrightarrow\quad \underline{a}= \lambda \underline b$ for some $\lambda\in \F_p^\times$.
    \item The right action of $G_k$ on $\MDA(\F_p^n)_k$ satisfies \begin{equation*}\label{DA1,rel}\DA_1^{k,n}\cdot X=\DA_1^{k,n}\end{equation*} \begin{equation}\label{DA2,rel2}\DA_2^{k,n}(\underline{a})\cdot X=\DA_2^{k,n}(\underline{a}\cdot X)\end{equation}for all $X\in G_k$.
    \item The set of the $\Pb_n^\pm(\F_p)$-orbits in $\MDA(\F_p^n)_k$ is $\left\{\DA_1^{k,n}\right\}\cup \left\{\DA_2^{k,n}(\underline{a})~|~\underline{a}\in A_k\right\}$.
    \end{enumerate}
Note that $\DA_1^{k,n}$ only exists for $k<n$.
\end{proposition}

\begin{proof} 
\noindent\emph{Proof of (a).} Let $\underline{a}\in A_k$. Then $\DA_2^{k,n}(\underline{a}
)\neq \emptyset$: since $\underline{a}=(a_0,\dots,a_k)$ is not zero then we consider two cases. \begin{itemize}
    \item If $a_3=\dots=a_k=0$, then $(a_o,a_1,a_2)$ is not zero. As $a_0+a_1+a_2=0$, the entries $a_0,a_1,a_2$ are not all equal. So there exist $i\neq j$ with $a_i\neq a_j$. Suppose for instance that $a_1\neq a_2$. Define \[v_1=\left((a_1-a_2)^{-1}, -(a_1-a_2),0,\dots,0\right)^t,\]  \[v_2=\left(0,a_1-a_2,0,\dots,a_k\right)^t,\] and set  \[v_0=-v_1-v_2.\] Choose vectors $v_3,\dots,v_k$ in $\F_p^n$ such that $v_1,\dots,v_k$ are linearly independent, and if $k=n$, choose them so that\[\det\left(v_1|\dots|v_n\right)=1.\]Then $v_0+v_1+v_2=0$, so $(v_0|\dots|v_k\in\MDA(\F_p^n)_k$. Moreover, \[\sum_{i=0}^ka_iv_i=(a_1-a_0)v_1+(a_2-a_0)v_2==(a_1-a_2)^{-1}(-a_0+a_1)e_1.\]Thus \[(v_0|\dots|v_k)\in\DA_2^{k,n}(\underline{a}).\] Similar explicit constructions can be done when $a_0\neq a_1$ or $a_0\neq a_2$.
    \item If $a_{i_0}\neq 0$ for some $i_0\geq 3$, define \[\underbar{b}=(b_1,\dots,b_k):=(a_1-a_0,a_2-a_0,a_3,\dots,a_k).\]Then $\underline{b}\neq 0$. By property (a) of \autoref{BD/Pn}, $\D_2^{k,n}(\underline{b})$ is nonempty. So choose linearly independent vectors $v_1,\dots,v_k\in\F_p^n$ such that, if $k=n$, then \[\det(v_1|\dots|v_n)=1,\]and \[(a_1-a_0)v_1+(a_2-a_0)v_2+a_3v_3+\dots+a_kv_k=\lambda e_1\]for some $\lambda\in\F_p^\times$. Now set $v_0=-v_1-v_2$. Then $v_0+v_1+v_2=0$, so $(v_0|\dots|v_k)\in\MDA(\F_p^n)_k$. Moreover, \[\sum_{i=0}^ka_iv_i=a_0(-v_1-v_2)+a_1v_2+a_2v_2+\sum_{i=3}^ka_iv_i=(a_1-a_0)v_1+(a_2-a_0)v_2+\sum_{i=3}^ka_iv_i=\lambda e_1.\]Therefore, \[(v_0|\dots|v_k)\in\DA_2^{k,n}(\underline{a}).\]
\end{itemize}  We conclude that $\DA_2^{k,n}(\underline{a})$ is nonempty. 

\noindent\emph{Proof of (b).} Let $\underline{b}\in A_k$, and suppose that $\DA_2^{k,n}(\underline{a})=\DA_2^{k,n}(\underline{b})$. Since this set is nonempty, fix an arbitrary \[B=(v_0|\dots|v_k)\in \DA_2^{k,n}(\underline{a}).\] Then there exist $\lambda_1,\lambda_2\in \F_p^\times$ such that \[\lambda_1 e_1=\sum_{i=0}^ka_iv_i\quad\text{and}\quad \lambda_2 e_1=\sum_{i=0}^kb_iv_i.\] Thus, \[\lambda_1^{-1}\sum_{i=0}^ka_iv_i=\lambda_2^{-1}\sum_{i=0}^kb_iv_i.\]As $v_0=-v_1-v_2$, it follows that \[\left(\lambda_1^{-1}a_1-\lambda_1^{-1}a_0\right)v_1+\left( \lambda_1^{-1}a_2-\lambda_1^{-1}a_0\right)v_2+\lambda_1^{-1}\sum_{i=3}^ka_iv_i=\left(\lambda_2^{-1}b_1-\lambda_2^{-1}b_0\right)v_1+\left(\lambda_2^{-1}b_2-\lambda_2^{-1}b_0\right)v_2+\lambda_2^{-1}\sum_{i=3}^kb_iv_i.\]
Since $a_0+a_1+a_2=b_0+b_1+b_2=0$, and $v_1,\dots,v_k$ are linearly independent, we obtain \begin{align*}
\lambda_1^{-1}(2a_1+a_2)&=\lambda_2^{-1}(2b_1+b_2)\\
\lambda_1^{-1}(2a_2+a_1)&=\lambda_2^{-1}(2b_2+b_1)\\ \lambda_1^{-1}a_i&=\lambda_2^{-1}b_i\quad\text{for all $i\geq 3$}.
\end{align*}
Thus, \[3\lambda_1^{-1}a_1=3\lambda_2^{-1}b_1\quad\text{and}\quad 3\lambda_1^{-1}a_2=3\lambda_2^{-1}b_2.\]
As $p\neq 3$ , we conclude that $\underline{a}=\lambda_1\lambda_2^{-1}\underline{b}$. 

Conversely, suppose $\underline{a}=\mu \underline{b}$ for some $\mu\in \F_p^\times$. Then \begin{align*}
    \DA_2^{k,n}(\underline{a})&=\left\{ (v_0|\dots|v_k)\in\MDA(\F_p^n)_k~\middle\vert\ \lambda e_1=\sum_{i=0}^ka_iv_i~\text{for some $\lambda\in \F_p^\times$}\right\}\\
    &=\left\{ (v_0|\dots|v_k)\in\MDA(\F_p^n)_k~\middle\vert\ \lambda e_1=\sum_{i=0}^k\mu b_iv_i~\text{for some $\lambda\in \F_p^\times$}\right\}\\
    &=\left\{ (v_0|\dots|v_k)\in\MDA(\F_p^n)_k~\middle\vert\ \mu^{-1}\lambda e_1=\sum_{i=0}^kb_iv_i~\text{for some $\lambda\in \F_p^\times$}\right\}\\
    &=\DA_2^{k,n}(\underline{b}).
\end{align*}
\noindent\emph{Proof of (c).} Let $X=(\pi,\varepsilon,\tau,\varepsilon_3,\dots,\varepsilon_k)\in G_k$ and $B=(v_0|\dots|v_k)\in\MDA(\F_p^n)_k$. We recall that \[B\cdot X=\left(\varepsilon v_{\pi(0)}\mid \varepsilon v_{\pi(1)}\mid \varepsilon v_{\pi(2)}\mid \varepsilon_{\tau(3)}v_{\tau(3)}\mid\dots\mid \varepsilon_{\tau(k)}v_{\tau(k)}\right),\]and\[\underline{a}\cdot X =\left(\varepsilon a_{\pi(0)}\mid \varepsilon a_{\pi(1)}\mid \varepsilon a_{\pi(2)}\mid \varepsilon_{\tau(3)}a_{\tau(3)}\mid\dots\mid \varepsilon_{\tau(k)}a_{\tau(k)}\right).\]
 If $B\in \DA_1^{k,n}$, then \[e_1\notin \operatorname{span}\{v_0,\dots,v_k\}.\] 
 It follows that \begin{equation}\label{DAA1}
     \DA_1^{k,n}\cdot X\subset \DA_1^{k,n}.
 \end{equation}   Applying \eqref{DAA1} to $X^{-1}$, we obtain \[ \DA_1^{k,n}\cdot X^{-1}\subseteq \DA_1^{k,n}.\] Thus \[ \DA_1^{k,n}\subseteq \DA_1^{k,n}\cdot X.\] Therefore $\DA_1^{k,n}\cdot X=\DA_1^{k,n}$.

 If $B\in \DA_2^{k,n}(\underline{a})$ for some $\underline{a}=(a_0,\dots,a_k)\in A_k$, then $\lambda_1 e_1=\sum_{i=0}^ka_iv_i$ for some $\lambda_1\in \F_p^\times$. It implies that \[\lambda_1 e_1=\varepsilon a_{\pi(0)}\left(\varepsilon v_{\pi(0)}\right)+\varepsilon a_{\pi(1)}\left(\varepsilon v_{\pi(1)}\right)+\varepsilon a_{\pi(2)}\left(\varepsilon v_{\pi(2)}\right)+\varepsilon_3 a_{\pi(3)}\left(\varepsilon_3v_{\pi(3)}\right)+\dots+\varepsilon_k a_{\pi(k)}\left(\varepsilon_kv_{\pi(k)}\right).\]So the matrix \[B\cdot X=\left(\varepsilon v_{\pi(0)}| \varepsilon v_{\pi(1)}| \varepsilon v_{\pi(2)}| \varepsilon_{\tau(3)}v_{\tau(3)}|\dots| \varepsilon_{\tau(k)}v_{\tau(k)}\right)\in  \DA_2^{k,n}\left(\underline{a}\cdot X\right),\]implying \begin{equation}\label{g}\DA_2^{k,n}(\underline{a})\cdot X\subseteq \DA_2^{k,n}(\underline{a}\cdot X)\quad\text{for all $X\in G_k$}.\end{equation} 
 Applying \eqref{g} to $X^{-1}$, we obtain \[ \DA_2^{k,n}(\underline{a}\cdot X)\cdot X^{-1}\subseteq \DA_2^{k,n}(\underline{a}).\] Thus \[ \DA_2^{k,n}(\underline{a}\cdot X)\subseteq \DA_2^{k,n}(\underline{a})\cdot X,\]
proving \eqref{DA2,rel2}.  

\noindent\emph{Proof of (d).} We  consider two different cases.
    \begin{mycases}
        \case We show that $\DA_1^{k,n}$ is a $\Pb_n^\pm(\F_p)$-orbit in $\MDA(\F_p^n)_k$. If $e_1\notin \operatorname{span}\{v_0,\dots,v_k\}$, then for any $A\in \Pb_n^\pm(\F_p)$  \[e_1\notin \operatorname{span}\left\{A\cdot v_0,\dots,A\cdot v_k\right\}.\]Thus \[A\cdot B\in \DA_1^{k,n}.\]  This shows that $\DA_1^{k,n}$ is a $\Pb_n^\pm(\F_p)$-invariant subset of $\MDA(\F_p^n)_k$. We will now show that this is a transitive action. Take another matrix $C=\left(u_0|\dots| u_k\right)\in \MDA(\F_p^n)_k$ such that $e_1\notin \operatorname{span}\{u_0,\dots,u_k\}$. We complete the partial bases $\{v_1,\dots,v_k\}$ and $\{u_1,\dots,u_k\}$, respectively, to determinant-$1$ bases $\{v_1,\dots,v_n\}$, and $\{u_1,\dots,u_n\}$ of $\F_p^n$ with $v_{n-1}=\mu_1 e_1$ and $u_{n-1}=\mu_2$ for some $\mu_1,\mu_2\in\F_p^\times$. We define the element $A'$ on the basis $\{v_1,\dots,v_n\}$ as follows. \begin{align*}
            v_n&\longmapsto u_n\\
            v_i&\longmapsto u_i\quad i\neq n,\\
        \end{align*}We have \[A'\cdot  B=C,\] and so $\det(A')=\pm 1$. We also have \[A'\cdot  e_1=\mu_1^{-1}A'\cdot  v_n=\mu_1^{-1}u_n=\mu_1^{-1}\mu_2e_1.\]
        Thus \[A'\in\Pb_n^\pm(\F_p).\]
        For later use in \autoref{BDA,p=3}, we observe that the condition $p\neq 3$ was not required to prove that $\Pb_n^\pm(\F_p)$ acts transitively on $\DA_1^{k,n}$, so transitivity holds for all primes $p$.
        \case We now show that for any $(\underline{a})\in A_k$, $\DA_2^{k,n}(\underline{a})$ is a $\Pb_n^\pm(\F_p)$-orbit in $\MDA(\F_p^n)_k$. If $e_1\in \operatorname{span}\{v_0,\dots, v_k\}$, then by \autoref{lem:lin comb}, we can write \begin{equation}\label{eq:lin comb}
            \lambda_1 e_1=a_0v_0+\dots+a_kv_k
        \end{equation}
        for some $\lambda_1\in\F_p^\times$, and $(a_0,\dots,a_k)\in A_k$. Let $B\in \DA_2^{k,n}(\underline{a})$.
        Then for any $A\in \Pb_n^\pm(\F_p)$, we have by definition that $A\cdot e_1=\lambda e_1$ for some $\lambda \in \F_p^\times$, and \begin{align*}a_0(A\cdot v_0)+\dots+ a_k(A\cdot v_k)&=A\cdot\left(a_0v_0+\dots+a_kv_k\right)\\&=A\cdot (\lambda_1 e_1)\\&=\lambda_1\lambda e_1.\end{align*}It implies that  $A\cdot B\in \DA_2^{k,n}(\underline{a})$. Thus, we showed that $\DA_2^{k,n}(\underline{a})$ is a $\Pb_n^\pm(\F_p)$-invariant subset of $\MDA(\F_p^n)_k$.

       Let $C=\left(u_0|\dots| u_k\right)$ be another matrix in $\DA_2^{k,n}(\underline{a})$. Then \[\lambda_2 e_1=a_0 u_0+\dots+a_ku_k.\]We complete the partial bases $\{v_1,\dots,v_k\}$ and $\{u_1,\dots,u_k\}$, respectively, to determinant-$1$ bases $\{v_1,\dots,v_n\}$, and $\{u_1,\dots,u_n\}$ of $\F_p^n$, and we consider the element $A'$ defined on the basis as follows. \[
            v_i\longmapsto u_i\quad\text{for all $1\leq i\leq n$}.
           \]Since $v_0+v_1+v_2=u_0+u_1+u_2=0$, then \[A'\cdot  v_0=-A'\cdot  v_1-A'\cdot  v_2=-u_1-u_2=u_0,\]and so
           \[A'\cdot  B= C.\]This implies that $\det(A')=\pm 1$. Moreover, \begin{align*}
            A'\cdot  e_1 &=\lambda_1^{-1}\left(a_0(A'\cdot  v_0)+\dots+a_k(A'\cdot  v_k)\right)\\&=\lambda_1^{-1}\left(a_0u_0+\dots+a_ku_k\right)\\&=\lambda_1^{-1}\lambda_2 e_1.
        \end{align*}It follows that $A'\in\Pb_n^\pm(\F_p)$, showing that $\Pb_n^\pm(\F_p)$ acts transitively on $\DA_2^{k,n}(\underline{a})$.
    \end{mycases}
    
     Since it follows by \autoref{lem:lin comb} that the union of \[\DA_1^{k,n}\quad\text{and}\quad \bigcup\limits_{\substack{\underline{a}=(a_0,\dots,a_k)\in\F_p^{k+1}\setminus\{\underline{0}\}\\ a_0+a_1+a_2=0}}\DA_2^{k,n}(\underline{a})\]comprises the whole set $\MDA(\F_p^n)_k$, we are done.
    \end{proof}
    
The case $p=3$ is as follows.
\begin{proposition}\label{BDA,p=3}
 Let $n\geq 2$, and $2\leq k \leq n$. Define the set $\DA_1^{k,n}$ as in \autoref{BD/Pn}, thats is; \[\DA_1^{k,n} =  \left\{ \left(v_0|\dots|v_k\right)\in \MDA(\F_3^n)_k ~\middle\vert e_1\not\in \operatorname{span}\{v_0,\dots,v_k\}\right\},\]and for every $\underline{a}=(a_1,\dots,a_k)\in \F_3^k\setminus\{\underline{0}\}$, define the set
   \[\TA_2^{k,n}(\underline{a})=\left\{\left(v_0|\dots|v_k\right)\in\MDA(\F_3^n)_k~\middle\vert~\lambda e_1= a_1v_1+\dots+a_kv_k ~\text{for some $\lambda\in\F_3^\times$}\right\}.\] These sets satisfy the following properties. For $\underline{a},\underline{b}\in\F_3^k\setminus\{\underline{0}\}$,\begin{enumerate}[label=(\alph*)]
   \item $\TA_2^{k,n}(\underline{a})$ is nonempty.
       \item $
\TA_2^{k,n}(\underline{a})=\TA_2^{k,n}(\underline{b})\quad\Longleftrightarrow\quad \underline{a}= \pm \underline b$.
  \item  The right action of $G_k$ on $\MDA(\F_3^n)_k$ satisfies  \begin{equation*}\label{DA1,rel1}  \DA_1^{k,n}\cdot X=\DA_1^{k,n}\quad\text{for all $X\in G_k$},
     \end{equation*}
     \begin{equation}\label{TA2,rel}
         \TA_2^{k,n}(\underline{a})\cdot X=\TA_2^{k,n}(\varepsilon a_{\pi(1)},\varepsilon a_{\pi(2)},\varepsilon_3a_{\tau(3)},\dots,\varepsilon_k a_{\tau(k)})
     \end{equation}for all $X=(\pi,\varepsilon, \tau,\varepsilon_3,\dots,\varepsilon_k)\in\Sigma_{\{1,2\}}\times \{-1,1\}\times\left(\Sigma_{\{3,\dots,k\}}\ltimes \{-1,1\}^{k-2}\right)\subset G_k$.
\begin{equation}\label{TA2,rel2}\TA_2^{k,n}(\underline{a})\cdot X=\TA_2^{k,n}(1,0,a_3,\dots,a_k)\quad\text{if $a_1\cdot a_2=1$ for some $X\in G_k$}.\end{equation} 
   \item  The set of the $\Pb_n^\pm(\F_3)$-orbits in $\MDA(\F_3^n)_k$ is $\left\{\DA_1^{k,n}\right\}\cup\left\{\TA_2^{k,n}(\underline{a})~|~\underline{a}\in \F_3^k\backslash \{\underline{0}\}\right\}$.
   \end{enumerate} 

\end{proposition}
\begin{proof}

\noindent\emph{Proof of (a).} Similar to property (a) of \autoref{BDA/Pn}.

\noindent\emph{Proof of (b).} Similar to property (b) of \autoref{BDA/Pn}.

\noindent\emph{Proof of (c).} We will only show \eqref{TA2,rel2} as \eqref{TA2,rel} is similar to property (c) of \autoref{BDA/Pn}.
Let $B=\left(v_0|\dots|v_k\right)\in\MDA(\F_3^n)_k$. If $B\in \TA_2^{k,n}(\underline{a})$, such that \begin{equation}\label{e}\lambda e_1=a_1 v_1+\dots+a_k v_k\end{equation}for some $\lambda\in\F_3^\times$. Let $a_1\cdot a_2=1$; that is $a_1=a_2=\pm 1$. Take $\pi=(1\; 0)\in\Sigma_{\{0,1,2\}}$ and consider \[X=\left(\pi,-a_1,\Id,1,\dots,1\right)\in G_k.\] Then \[B\cdot X=\left(-a_1v_1|-a_1v_0|- a_1v_2|v_3|\dots|v_k\right).\]Now define new vectors:\[w_0=-a_1v_1,\quad w_1=-a_1v_0,\quad w_2=-a_1v_2,\quad w_i=v_i~\text{for $i\geq 3$}.\]So \[B\cdot X=(w_0|w_1|w_2|\dots |w_k).\] Using $v_1=-v_0-v_2$, we write \eqref{e} in the following way: \[\lambda e_1=-a_1 v_0+(a_2-a_1)v_2+a_3v_3+\dots a_k v_k.\] Since $a_1=a_2$, we obtain \[\lambda e_1=-a_1v_0+a_3v_3+\dots+a_kv_k.\]Rewriting in terms of the $w_i$, we obtain 
\[\lambda e_1=1\cdot w_1+0\cdot w_2+a_3w_3+\dots+a_kw_k.\] Therefore, \[B\cdot X\in \TA_2^{k,n}(1,0,a_3,\dots,a_k).\] 

For the same element $X=(\pi,-a_1,\tau,1,\dots,1)$, we also have
\[
\TA_2^{k,n}(1,0,a_3,\dots,a_k)\cdot X\in  \TA_2^{k,n}(\underline{a}).\]Since $X^2=\Id$, we conclude that $\TA_2^{k,n}(\underline{a})\cdot X\subseteq \TA_2^{k,n}(1,0,a_3,\dots,a_k)$.

\noindent\emph{Proof of (d).}  Note that every element of $\MDA(\F_3^n)_k$ lies either in $\DA_1^{k,n}$ or in $\TA_2^{k,n}(\underline{a})$ for some $\underline{a}\in\F_3^k\setminus\{\underline{0}\}$. As shown in \autoref{BDA/Pn}, $\Pb_n^\pm(\F_p)$ acts transitively on $\DA_1^{k,n}$ for all $p\geq 3$; in particular for $p=3$. It therefore remains to show that $\Pb_n^\pm(\F_3)$ acts transitively on each $\TA_2^{k,n}(\underline{a})$.

Let $\underline{a}=(a_1,\dots,a_k)\in\F_3^k\setminus\{\underline{0}\}$. We will adapt the same reasoning used in the proof of \autoref{BDA/Pn} and we will first check that $\Pb_n^\pm(\F_3)$ acts on $\TA_2^{k,n}(\underline{a})$. Let $B=\left(v_0|\dots|v_k\right)\in \TA_2^{k,n}(\underline{a})$, such that \[\lambda_1 e_1=a_1 v_1+\dots+ a_k v_k\]for some $\lambda_1\in\F_3^\times$. Then for any $A\in \Pb_n^\pm(\F_3)$, we have by definition that $A\cdot e_1=\lambda e_1$ for some $\lambda\in \F_3^\times$, and \begin{align*}
        a_1\left(A\cdot v_1\right)+\dots+ a_k\left(A\cdot v_k\right)&=A\left(a_1 v_1+\dots +a_k v_k\right)\\&=A\cdot\left(\lambda_1 e_1\right)\\&=\lambda_1 \lambda e_1.
    \end{align*}This implies that $A\cdot B\in \TA_2^{k,n}(\underline{a})$, and thus, $\Pb_n^\pm(\F_3)$ acts on $\TA_2^{k,n}(\underline{a})$.
    Now let $C=\left(u_0|\dots| u_k\right)$ be another matrix in $\TA_2^{k,n}(\underline{a})$. Then \[\lambda_2  e_1=a_1 u_1+\dots+a_ku_k,\]for some $\lambda_2 \in\F_3^\times$. We complete the partial bases $\{v_1,\dots,v_k\}$ and $\{u_1,\dots,u_k\}$, respectively, to determinant-$1$ bases $\{v_1,\dots,v_n\}$ and $\{u_1,\dots,u_n\}$ of $\F_3^n$, and we consider the element $A'$ defined on the basis as follows. \[
            v_i\longmapsto u_i\quad\text{for all $1\leq i\leq n$}.
           \]
     Observe that \[A'\cdot v_0=A'\cdot (-v_1-v_2)=-u_1-u_2=u_0.\]Thus \[A'\cdot B= C\quad\text{and}\quad
        A'\cdot e_1=\lambda_1^{-1}\lambda_2  e_1.
        \]
        Since $\det(A')=\pm 1$, it follows that $A'\in \Pb_n^\pm(\F_3)$, proving that $\Pb_n^\pm(\F_3)$ acts transitively on $\TA_2^{k,n}(\underline{a})$. 
\end{proof}

\begin{comment}\begin{corollary}\label{cor: TA}
    For all $n\geq 2$ and $a_3,\dots,a_k\in\F_3$, the following equalities hold in $\Pb_n^\pm(\F_3)\backslash\MDA_k(\F_3^n)/G_k$.\begin{align*}\TA_2^{k,n}(1,1,a_3,\dots,a_k)&=\TA_2^{k,n}(-1,-1,a_3,\dots,a_k)\\&=\TA_2^{k,n}(\pm 1,0,a_3,\dots,a_k)\\&=\TA_2^{k,n}(0,\pm 1,a_3,\dots,a_k),\end{align*} and \[\TA_2^{k,n}(\pm 1,a_3,\dots,a_k)=\TA_2^{k,n}(1,-1,a_3,\dots,a_k).\]
\end{corollary}
\begin{proof}
    The proof follows immediately from the relations in \autoref{BDA,p=3}.
\end{proof}\end{comment}
\section{Orientation-reversing and preserving actions}\label{sec6}
Let $p$ be a prime, and let $Y=\SBD(\F_p^n)$ or $\SBDA(\F_p^n)$.  In this section, we study the chains \[\redchain_*\left(Y;\Q^{\det}\right)_{\Pb_n^\pm(\F_p)}\]as well as their homology groups.
\subsection{Simplices with orientation-reversing and preserving actions}\label{subs:pr}

Let $p$ be a prime and $n\geq 2$. Let $Y=\SBD(\F_p^n)$ or $\SBDA(\F_p^n)$. 
We classify the simplices in $Y$ based on whether the $\Pb_n^\pm(\F_p)$-action is orientation-preserving or orientation-reversing. We refer to an element $\sigma\in \SBD(\F_p^n)_k$ as a \emph{standard $k$-simplex}, and to $\sigma\in \SBDA(\F_p^n)_k\setminus \SBD(\F_p^n)_k$ as an \emph{additive $k$-simplex}.

\paragraph{Standard simplices}
Recall from \autoref{matrixrep} that for all primes $p$, $n\geq 2$ and $0\leq k\leq n-1$,\[ \SBD(\F_p^n)_k/\Sigma_{k+1}\cong \MD(\F_p^n)_k/T_k.\] 
Thus, we may identify the coset $\sigma \Sigma_{k+1}$ with the coset $B\cdot T_k $ for some $B\in \MD(\F_p^n)_k$. We also recall from \autoref{BD/Pn}, that every orbit of $\Pb_n^\pm(\F_p)\backslash \MD(\F_p^n)_k/T_k$ has the form $\D_1^{k,n}\cdot T_k$ or $\D_2^{k,n}(\underline{a})\cdot T_k$ for some $\underline{a}\in\F_p^{k+1}\setminus\{\underline{0}\}$. 

\begin{lemma} \label{S1}
    Let $p$ be a prime, $n\geq 3$ and $1\leq k < n-1$. Let $\sigma\in \SBD(\F_p^n)_k$ and let $B\in \D_1^{k,n}$ such that $\sigma \Sigma_{k+1}=B\cdot T_k$. Then the action of $\Pb_n^\pm(\F_p)$ on $\sigma$ is orientation-reversing.\end{lemma}

\begin{proof} We recall that \[\D_1^{k,n} =  \left\{ \left(v_0|\dots|v_k\right)\in \MD(\F_p^n)_k ~\middle\vert e_1\not\in \operatorname{span}\{v_0,\dots,v_k\}\right\}.\]Let $ \sigma \Sigma_{k+1}=B\cdot T_k$ with $B\in \D_1^{k,n}$. We want to show that the action of $\Pb_n^\pm(\F_p)$ on $\sigma$ is orientation-reversing. 
We have by \autoref{lem:goodmat} that \[\SBD(\F_p^n)_k^\mathrm{rv}/\Sigma_{k+1}\cong \MD(\F_p^n)_k^\mathrm{rv}/T_k.\]
 Thus, it is enough to show that the action of $\Pb_n^\pm(\F_p)$ on $B$ is orientation-reversing.
 
We have by property (c) of \autoref{BD/Pn} that \[B\cdot X\in \D_1^{k,n}\quad\text{for all $X\in T_k$}.\]In particular, \[B\cdot X\in \D_1^{k,n}\quad\text{for $X\in T_k$ with $\sign(X)=-1$}.\]Thus, there exist $A\in \Pb_n^\pm(\F_p)$ and $X\in T_k$ such that $\sign(X)=-1$ and \[A\cdot B=B\cdot X.\]
Therefore, the action of $\Pb_n^\pm(\F_p)$ on $\sigma$ is orientation-reversing.\end{proof}

\begin{lemma}\label{rev,a}
    Let $p$ be a prime, $n\geq 2$ and $1\leq k\leq n-1$. Let $\sigma\in \SBD(\F_p^n)_k$ and let $B\in \D_2^{k,n}(\underline{a})$ for some $\underline{a}=(a_0,\dots,a_k)\in \F_p^{k+1}\setminus\{\underline{0}\}$ such that $\sigma \Sigma_{k+1}=B\cdot T_k$. Then the action of $\Pb_n^\pm(\F_p)$ on $\sigma$ is orientation-reversing if and only if there exist some $\lambda\in \F_p^\times$ and $X\in T_k$ such that $\sign(X)=-1$, and \[\lambda\underline{a}=\underline{a}\cdot X.\]
\end{lemma}
\begin{proof}Recall that \[\D_2^{k,n}(\underline{a})=\left\{B=\left(v_0|\dots|v_k\right)\in \MD(\F_p^n)_k ~\middle\vert \lambda e_1=a_0v_0+\dots+a_kv_k~\text{for some $\lambda\in\F_p^\times$}\right\}.\] 
Let $\sigma \Sigma_{k+1}=B\cdot T_k$ with $B\in \D_2^{k,n}(\underline{a})$ for some $\underline{a}=(a_0,\dots,a_k)\in \F_p^{k+1}\setminus\{\underline{0}\}$. Since  \[\SBD(\F_p^n)_k^\mathrm{rv}/\Sigma_{k+1}\cong \MD(\F_p^n)_k^\mathrm{rv}/T_k\] by \autoref{lem:goodmat}, we have the following equivalences:
   \begin{align*}\text{The action of $\Pb_n^\pm(\F_p)$ on $\sigma$} &~\text{is orientation-reversing}\\&\Longleftrightarrow \text{there exists $X\in T_k$ with $\sign(X)=-1$ and $B\cdot X=A\cdot B$ for some $A\in \Pb_n^\pm(\F_p)$}\\&\Longleftrightarrow \text{there exists $X\in T_k$ with $\sign(X)=-1$ and $B\cdot X\in \D_2^{k,n}(\underline{a})$}\\ &\Longleftrightarrow \text{there exists $X\in T_k$ with $\sign(X)=-1$ and $\D_2^{k,n}(\underline{a}\cdot X)=\D_2^{k,n}(\underline{a})$},\end{align*}where the last equivalence follows from properties (c) and (d) of \autoref{BD/Pn}. Moreover, we have by property (b) of \autoref{BD/Pn} that \[\D_2^{k,n}(\underline{a})=\D_2^{k,n}(\underline{b})\quad\Longleftrightarrow\quad \underline{a}=\lambda\underline{b}~\text{for some $\lambda\in \F_p^\times$}.\]
   This implies that the action of $\Pb_n^\pm(\F_p)$ on $\sigma$ is orientation-reversing if and only if there exist $\lambda\in \F_p^
    \times$ and $X\in T_k$ such that $\sign(X)=-1$ and \[\lambda \underline{a}=\underline{a}\cdot X.\qedhere\]
\end{proof}
\begin{definition}[\textbf{$(\lambda,m)$-condition}]\label{def1}
    Let $p$ be a prime, $n\geq 2$ and $1\leq k\leq n-1$. Let $\underline{a}=(a_0,\dots,a_k)\in \F_p^{k+1}\setminus\{\underline{0}\}$. We say that \emph{$\underline{a}$ satisfies the $(\lambda,m)$-condition} if it satisfies one of the following cases.\begin{mycases}
    \case If $k$ is odd, and there exist $\lambda\in \F_p^\times$ and an even integer $m$ satisfying \[\lambda~\text{has order $2m$},\]\[m\mid (k+1),\]\[q=\frac{k+1}{m}~\text{is an odd integer},\] and \[\underline{a}\cdot Y=
    (b_1, \lambda b_1, \dots, \lambda^{m-1} b_1,\dots,b_q, \lambda b_q, \dots, \lambda^{m-1} b_q),\]for some $Y\in T_k$ and $b_1,\dots,b_q\in \F_p^\times$.
    \case  If $k$ is even, and there exist $\lambda\in \F_p^\times$ and an even integer $m$ satisfying \[\lambda~\text{has order $2m$},\]\[m\mid k,\]\[q=\frac{k}{m}~\text{is an odd integer},\] and \[\underline{a}\cdot Y=(b_1, \lambda b_1, \dots, \lambda^{m-1} b_1,\dots,b_q, \lambda b_q, \dots, \lambda^{m-1} b_q,0),\]for some $Y\in T_k$ and $b_1,\dots,b_q\in \F_p^\times$.
\end{mycases}
\end{definition}
\begin{lemma}\label{D2,dif,rv} Let $p$ be a prime, $n\geq 2$ and $1\leq k\leq n-1$. Let $\sigma\in \SBD(\F_p^n)_k$ and let $B\in \D_2^{k,n}(\underline{a})$ for some $\underline{a}=(a_0,\dots,a_k)\in \F_p^{k+1}\setminus\{\underline{0}\}$ such that $\sigma \Sigma_{k+1}=B\cdot T_k$. If $\underline{a}$ satisfies the \hyperref[def1]{$(\lambda,m)$-condition}, then the action of $\Pb_n^\pm(\F_p)$ on $\sigma$ is orientation-reversing.
\end{lemma}
\begin{proof}
Let $\underline{a}, \lambda$ be as in Case $1$ and Case $1$. Let $\sigma \Sigma_{k+1}=B\cdot T_k$ with $B\in \D_2^{k,n}(\underline{a})$. In each case, we will show that there exists $X\in T_k$ such that $\sign(X)=-1$ and \[\lambda \underline{a}=\underline{a}\cdot X.\]It will then follow by \autoref{rev,a} that the action of $\Pb_n^\pm(\F_p)$ on $\sigma$ is orientation-reversing.
    \begin{mycases} 
        \case  Let $\underline{b}=(b_1,\lambda b_1, \dots, \lambda^{m-1} b_1,\dots,b_q, \lambda b_q, \dots, \lambda^{m-1} b_q)$. Let $\pi$ be the following permutation \[(0\; 1
        \;\dots\; m-1)(m\;\dots\;2m-1)\dots(qm-m\;\dots\;qm-1)\quad\text{in $\Sigma_{k+1}$}.\]
        Observe that $\pi$ is odd since $m$ is even and $q$ is odd. Let $(\varepsilon_0,\dots,\varepsilon_k)$ be as follows \[\varepsilon_i=\begin{cases}
            -1 &\text{if $i\in\{m-1,2m-1,\dots,qm-1\}$}\\
            1&\text{otherwise}.
        \end{cases}\] Then the element $X=(\pi,\varepsilon_0,\dots,\varepsilon_k)\in T_k$ and $\sign(X)=-1$. Moreover, given that $\lambda$ has order $2m$, it follows that $\lambda^{m}=-1$. Therefore \begin{align*}\lambda \underline{b}&=\lambda(b_1, \lambda b_1, \dots, \lambda^{m-1} b_1,\dots,b_q, \lambda b_q, \dots, \lambda^{m-1} b_q)\\&=\left(\lambda b_1,\lambda^2b_1,\dots,\lambda^m b_1,\dots,\lambda b_q,\lambda^2 b_q,\dots,\lambda^m b_q\right)\\&=\left(\lambda b_1,\lambda^2b_1,\dots,- b_1,\dots,\lambda b_q,\lambda^2 b_q,\dots,- b_q\right)\\&=(b_1, \lambda b_1, \dots, \lambda^{m-1} b_1,\dots,b_q, \lambda b_q, \dots, \lambda^{m-1} b_q)\cdot X\\&=\underline{b}\cdot X.\end{align*}
         Since \[\underline{a}\cdot Y=(b_1, \lambda b_1, \dots, \lambda^{m-1} b_1,\dots,b_q, \lambda b_q, \dots, \lambda^{m-1} b_q),\] it implies that \[\lambda \underline{a}\cdot Y=\lambda\underline{b}=\underline{b}\cdot X=\underline{a}\cdot YX.\]
         Therefore, \[\lambda\underline{a}=\underline{a}\cdot YXY^{-1}.\]
         As $\sign(X)=-1$, we conclude by \autoref{rev,a} that the action of $\Pb_n^\pm(\F_p)$ on $\sigma$ is orientation-reversing.
        \case Let $\underline{b}=(b_1, \lambda b_1, \dots, \lambda^{m-1} b_1,\dots,b_q, \lambda b_q, \dots, \lambda^{m-1} b_q,0)$. Let $\pi$ be the following odd permutation in $\Sigma_{k+1}$ that fixes $k$  \[(0\;1\; \dots\; m-1)(m\;\dots\;2m-1)\dots(qm-m\;\dots\;qm-1)\quad\text{in $\Sigma_{k+1}$}.\] We consider the element $X=(\pi,\varepsilon_0,\dots,\varepsilon_k)$ with $\varepsilon_i$ as in the previous case. Then $X\in T_k$, and $\sign(X)=-1$. By the same reasoning as before, we have that \[\lambda \underline{b}=\underline{b}\cdot X,\]and thus $\lambda\underline{a}=\underline{a}\cdot YXY^{-1}$ with $\sign(YXY^{-1})=-1$.
      \end{mycases}\end{proof}

 \begin{lemma}\label{D2,dif,pv}
         Let $p$ be a prime, $n\geq 2$ and $1\leq k\leq n-1$. Let $\sigma\in \SBD(\F_p^n)_k$ and let $B\in \D_2^{k,n}(\underline{a})$ for some $\underline{a}=(a_0,\dots,a_k)\in \F_p^{k+1}\setminus\{\underline{0}\}$ such that $\sigma \Sigma_{k+1}=B\cdot T_k$. Assume $a_i\neq\pm a_j$ for all $i\neq j$. If $\underline{a}$ does not satisfy the \hyperref[def1]{$(\lambda,m)$-condition}, then the action of $\Pb_n^\pm(\F_p)$ on $\sigma$ is orientation-preserving.
     \end{lemma}
   \begin{proof} Let $\sigma \Sigma_{k+1}=B\cdot T_k$ with $B\in \D_2^{k,n}(\underline{a})$, such that $a_i\neq \pm a_j$ for all $i\neq j$. We argue by contradiction, assuming that $\underline{a}$ does not satisfy the \hyperref[def1]{$(\lambda,m)$-condition}, while the action of $\Pb_n^\pm(\F_p)$ on $\sigma$ is orientation-reversing. Then by \autoref{rev,a}, there exist $\lambda\in\F_p^\times$ and $X=(\pi,\varepsilon_0,\dots,\varepsilon_k)\in T_k$ such that $\sign(X)=-1$ and \[\lambda \underline{a}=\underline{a}\cdot X.\]This implies that \begin{equation}\label{lambdaa} \lambda a_i=\varepsilon_i a_{\pi(i)}\quad\text{for all $0\leq i\leq k$}.
    \end{equation}
    \begin{claim}\label{claim1}
        If $\pi$ has a cycle of length $\ell>1$, then $\lambda^{\ell}=\pm 1$. Further, if $\ell$ is even, then $\lambda^{\ell}=-1.$
    \end{claim}\begin{proof}[Proof of Claim] Let $i_0$ be an entry in a cycle of length $\ell>1$. Then $\pi^{\ell}(i_0)=i_0$. 
    By \eqref{lambdaa}, we have \begin{align*}
     \lambda a_{i_0}&=\varepsilon_{i_0}a_{\pi(i_0)},\\
     \lambda a_{\pi(i_0)}&=\varepsilon_{\pi(i_0)}a_{\pi^2(i_0)},\\ &\ \vdots\\\lambda a_{\pi^{\ell-2}(i_0)}&=\varepsilon_{\pi^{\ell-2}(i_0)}a_{\pi^{\ell-1}(i_0)},\\\lambda a_{\pi^{\ell-1}(i_0)}&=\varepsilon_{\pi^{\ell-1}(i_0)}a_{\pi^{\ell}(i_0)}=\varepsilon_{\pi^{\ell-1}(i_0)}a_{i_0}.
     \end{align*}
        As $\ell>1$ and $a_i\neq \pm a_j$ for all $i\neq j$, it follows that $a_{i_0},\dots,a_{\pi^{\ell-1}(i_0)}$ are all nonzero. It then implies that \begin{equation}\label{eq:aii}\lambda^sa_{i_0}=\varepsilon_{i_0}\varepsilon_{\pi(i_0)}\varepsilon_{\pi^2(i_0)}\dots\varepsilon_{\pi^{s-1}(i_0)}a_{\pi^{s}(i_0)}~\text{for all $0\leq s\leq \ell$}.\end{equation}
        Since $\pi^{\ell}(i_0)=i_0$, it  follows by \eqref{eq:aii} that\[\lambda^{\ell}=\varepsilon_{i_0}\varepsilon_{\pi(i_0)}\dots\varepsilon_{\pi^{\ell-1}(i_0)}=\pm 1.\]
        
        If $\ell$ is even, \[\lambda^{\frac{\ell}{2}}a_{i_0}=\varepsilon_{i_0}\varepsilon_{\pi(i_0)}\dots\varepsilon_{\pi^{\frac{\ell}{2}-1}(i_0)}a_{\pi^{\frac{\ell}{2}}(i_0)}.\]
        Thus,\[\lambda^{\ell}a^2_{i_0}=\left(\lambda^{\frac{\ell}{2}}a_{i_0}\right)^2=\left(\varepsilon_{\pi(i_0)}\varepsilon_{\pi^2(i_0)}\dots\varepsilon_{\pi^{\frac{\ell}{2}-1}(i_0)} a_{\pi^{\frac{\ell}{2}}(i_0)}\right)^2= \left(a_{\pi^{\frac{\ell}{2}}(i_0)}\right)^2.\]
        As $a_i\neq \pm a_j$ for all $i\neq j$, we conclude that \[\lambda^{\ell}=-1.\qedhere\] 
\end{proof}
        Since $\sign(X)=-1$, then the permutation $\pi$ must contain a cycle of even length $m_0$. Let $i_0$ be an entry in that cycle. 
        Assume $\pi$ has a cycle of length $\ell>1$ such that $\ell\neq m_0$ and let $i_1$ be an entry in that cycle. It follows by the claim that $\ell$ is odd, and \[\lambda^{\ell-m_0}=\lambda^\ell(\lambda^{m_0})^{-1}=\pm 1.\] As $\ell>1$ and $a_i\neq \pm a_j$ for all $i\neq j$, \eqref{lambdaa} gives that $a_{i_1}\neq 0$. Moreover, \eqref{eq:aii} gives \[\lambda^{m_0-\ell} a_{i_0}=\varepsilon_{i_0}\dots\varepsilon_{\pi^{m_0-\ell-1}(i_0)}a_{\pi^{m_0-\ell}(i_0)}\quad\text{if $m_0>\ell$},\]and \[\lambda^{\ell-m_0} a_{i_1}=\varepsilon_{i_1}\dots\varepsilon_{\pi^{\ell-m_0-1}(i_1)}a_{\pi^{\ell-m_0}(i_1)}\quad\text{if $\ell>m_0$}.\]Since$\lambda^{m_0-\ell}=\lambda^{\ell-m_0}=\pm 1,$ we obtain \[a_{i_0}=\pm a_{\pi^{m_0-\ell}(i_0)}\quad\text{if $m_0>\ell$},\]and\[a_{i_1}=\pm a_{\pi^{\ell-m_0}(i_1)}\quad\text{if $\ell>m_0$}.\] As $a_i\neq \pm a_j$ for all $i\neq j$, this gives a contradiction.

        Thus we conclude that $\pi$ has only cycles of length $m_0$ or $1$. We consider the following two cases.
        \begin{itemize}
            \item If all cycles of $\pi$ have length $m_0$, then $m_0$ divides $k+1$. Additionally, \eqref{eq:aii} implies that  the tuple $(a_{i_0},a_{\pi(i_0)},\dots,a_{\pi^{\ell-1}(i_0)})$ is equal to\[(a_{i_0},\varepsilon_{i_0}\lambda a_{i_0},\dots,\varepsilon_{i_0}\varepsilon_{\pi(i_0)}\varepsilon_{\pi^2(i_0)}\dots\varepsilon_{\pi^{m_0-1}(i_0)}\lambda^{m_0-1}a_{i_0}).\] Write $\pi$ as a product of disjoint cycles \[(r_1\; r_2 \; \dots \;r_{m_1})(r_{m_1+1}\; \dots\; r_{m_1+m_2})\dots .\] Let $\tau\in\Sigma_{k+1}$ be the permutation defined by $\tau(i)=r_i$ for all $i$. Consequently, for $Z=(\tau,\varepsilon'_0,\dots,\varepsilon'_k)\in T_k$ for some $\varepsilon'_i=\{-1,1\}$, the tuple $\underline{a}\cdot Z$, for $Z=(\tau,\varepsilon'_0,\dots,\varepsilon'_k)\in T_k$, is composed entirely of tuples of the form \[(b_j,\lambda b_j,\dots,\lambda^{m_0-1}b_j)\]for some $b_j\in \F_p^\times$. There are \[q=\frac{k+1}{m_0}\] such tuples. Moreover, $q$ is odd since $\sign(Z)=\sign(\tau)=\sign(\pi)=\sign(X)=-1$. Thus $\underline{a}$ satisfies the \hyperref[def1]{$(\lambda,m)$-condition}, which gives a contradiction to our assumption.  
            \item If $\pi$ has a cycle of length $1$, so $\pi(i)=i$ for some $i$. Then from \eqref{lambdaa}, \[\lambda a_i=\varepsilon_ia_i.\]
If $a_i\neq 0$, then $\lambda=\varepsilon_i=\pm 1$. Since $m_0$ is even, it implies that \[\lambda^{m_0}=(\pm 1)^{m_0}=1,\]which gives a contradiction to $\lambda^{m_0}=-1$. Thus, for all fixed points $i$ of $\pi$ we must have $a_i=0$. Since $a_i\neq \pm a_j$ for all $i\neq j$, there is exactly one fixed point of $\pi$, and so exactly one cycle of length $1$. Furthermore, by the previous case, we have that the tuple $\underline{a}\cdot Z$ is composed of $q$ tuples of the form\[(b_j,\lambda b_j,\dots,\lambda^{m_0-1}b_j),\]for some $b_j\in\F_p^\times$, and one entry equal to $0$, where $q$ is odd. Thus $\underline{a}$ satisfies the \hyperref[def1]{$(\lambda,m)$-condition}, which gives a contradiction to our assumption.\qedhere\end{itemize} \end{proof}  

 \begin{lemma}\label{D2,eq,rv} 
   Let $p$ be a prime, $n\geq 2$ and $1\leq k\leq n-1$. Let $\sigma\in \SBD(\F_p^n)_k$ and let $B\in \D_2^{k,n}(\underline{a})$ for some $\underline{a}=(a_0,\dots,a_k)\in \F_p^{k+1}\setminus\{\underline{0}\}$ such that $\sigma \Sigma_{k+1}=B\cdot T_k$. Assume $a_i=\pm a_j$ for some $i\neq j$. Then the action of $\Pb_n^\pm(\F_p)$ on $\sigma$ is orientation-reversing.
\end{lemma}   
      \begin{proof}
      Let $\sigma \Sigma_{k+1}=B\cdot T_k$ with $B\in \D_2^{k,n}(\underline{a})$ such that $a_i=\pm a_j$ for some $i\neq j$.
          We recall from \autoref{rev,a} that the action of $\Pb_n^\pm(\F_p)$ on $\sigma$ is orientation-reversing if and only if there exist $\lambda\in \F_p^\times$ and $X=(\pi,\varepsilon_0,\dots,\varepsilon_k)\in T_k$ such that $\sign(X)=-1$ and \[\lambda \underline{a}=\underline{a}\cdot X.\] 
          Suppose $a_{i_0}=\varepsilon a_{i_1}$ for some $ i_0\neq i_1$ and $\varepsilon\in\{-1,1\}$. Let $\pi$ be the transposition $(i_0\; i_1)$ in $\Sigma_{k+1}$ and \[\varepsilon_i=\begin{cases}
              \varepsilon&\text{if $i=i_0,i_1$}\\
              1&\text{otherwise}.
          \end{cases}\]Then, the element $X=(\pi,\varepsilon_0,\dots,\varepsilon_k)\in T_k$ and $\sign(X)=-1$. Thus, for $\lambda=1$, we obtain\[\lambda \underline{a}=\underline{a}\cdot X.\]This completes the proof.
\end{proof}
We summarize the above lemmas in the following corollary.
\begin{corollary}\label{cor:st,pr}
Let $p$ be a prime, $n\geq 2$ and $2\leq k\leq n-1$. Let $\sigma\in\SBD(\F_p^n)_k$ and let $B\in\MD(\F_p^n)_k$ such that $ \sigma\cdot\Sigma_{k+1}=B\cdot T_k$. Then $\sigma$ belongs to $\SBD(\F_p^n)_k^\mathrm{pr}$ if and only if $B$ belongs to some set \[\D_2^{k,n}(\underline{a})\quad\text{with $\underline{a}=(a_0,\dots,a_k)\in\F_p^{k+1}\setminus\{\underline{0}\}$},\] where the entries satisfy $a_i\neq \pm a_j$ for all $i\neq j$ and do not satisfy the \hyperref[def1]{$(\lambda,m)$-condition}. 

\end{corollary}
\begin{definition}\label{pr,0}
    Let $p$ be a prime, $n\geq 2$ and $2\leq k\leq n-1$. Define 
    \[ \MD(\F_p^n)_k^{\mathrm{pr},0}=\left\{B\in\MD(\F_p^n)_k^\mathrm{pr}\middle\vert\begin{array}{c} B\in \D_2^{k,n}(\underline{a})\text{ for some $\underline{a}=(a_0,\dots,a_k)\in\F_p^n\setminus\{\underline{0}\}$},\\
        a_{i_0}=0\text{ for a unique $i_0$}, a_i\neq \pm a_j\text{ for all $i\neq j$},\\
        \underline{a}\text{ does not satisfy the \hyperref[def1]{$(\lambda,m)$-condition}}\end{array}\right\}\]
         and \[ \MD(\F_p^n)_k^{\mathrm{pr},1}=\left\{B\in\MD(\F_p^n)_k^\mathrm{pr}\middle\vert\begin{array}{c} B\in\D_2^{k,n}(\underline{a})\text{for some $\underline{a}=(a_0,\dots,a_k)\in\F_p^n\setminus\{\underline{0}\}$},\\a_{i}\neq 0\text{ for all $i$}, a_i\neq \pm a_j\text{ for all $i\neq j$},\\
        \underline{a}\text{ does not satisfy the \hyperref[def1]{$(\lambda,m)$-condition}}\end{array}\right\}.\]

\end{definition}
\begin{remark}\label{rmk,pr}
    It follows by \autoref{cor:st,pr}  \[\Pb_n^\pm(\F_p)\backslash\MD(\F_p^n)_k^\mathrm{pr}/T_k=\left(\Pb_n^\pm(\F_p)\backslash\MD(\F_p^n)_k^{\mathrm{pr},0}/T_k\right)\quad\sqcup\quad \left(\Pb_n^\pm(\F_p)\backslash\MD(\F_p^n)_k^{\mathrm{pr},1}/T_k\right).\]
\end{remark}
Additionally, we prove the following result.
\begin{lemma}\label{bij}
    Let $p$ be a prime, $n\geq 2$ and $2\leq k\leq n-1$. There is a bijection, \[\Pb_n^\pm(\F_p)\backslash \MD(\F_p^n)_{k+1}^{\mathrm{pr},0}/T_{k+1}\rightarrow \Pb_n^\pm(\F_p)\backslash \MD(\F_p^n)_k^{\mathrm{pr},1}/T_k
    \]given by by forgetting the unique zero entry: \[\D_2^{k+1,n}(a_0,\dots,a_{k+1})\cdot T_{k+1}\mapsto\D_2^{k,n}(a_0,\dots,\widehat{0},\dots,a_{k+1})\cdot T_k.\]
\end{lemma}
\begin{proof}
    Let $\D_2^{k+1,n}(a_0,\dots,a_{k+1})\cdot T_{k+1}\in \Pb_n^\pm(\F_p)\backslash \MD(\F_p^n)_{k+1}^{\mathrm{pr},0}/T_{k+1}$ with $a_{i_0}=0$. Define \[\Phi\colon \Pb_n^\pm(\F_p)\backslash \MD(\F_p^n)_{k+1}^{\mathrm{pr},0}/T_{k+1}\rightarrow \Pb_n^\pm(\F_p)\backslash \MD(\F_p^n)_{k}^{\mathrm{pr},1}/T_{k}\] by \[\Phi\left(\D_2^{k+1,n}(a_0,\dots,a_{k+1})\cdot T_{k+1}\right)=\D_2^{k,n}(a_0,\dots,\widehat{a_{i_0}},\dots,a_{k+1})\cdot T_{k}.\]Since $a_i\neq\pm a_j$ for all $i\neq j$, then none of the entries of $(a_0,\dots,\widehat{a_{i_0}},\dots,a_{k+1})$ is zero. Moreover, the tuple $(a_0,\dots,\widehat{a_{i_0}},\dots,a_{k+1})$ satisfies $a_i\neq \pm a_j$ for all $i\neq j$, and does not satisfy the \hyperref[def1]{$(\lambda,m)$-condition} since $\D_2^{k+1,n}(a_0,\dots,a_{k+1})\cdot T_{k+1}\in \Pb_n^\pm(\F_p)\MD(\F_p^n)_{k+1}^{\mathrm{pr},0}/T_{k+1}$.
    
Now let $\D_2^{k,n}(a_0,\dots,a_{k})\cdot T_k\in \Pb_n^\pm(\F_p)\MD(\F_p^n)_{k}^{\mathrm{pr},}/T_{k}$. Define the inverse map $\Psi$ by \[\Psi\left(\D_2^{k,n}(a_0,\dots,a_{k})\cdot T_k\right)=\D_2^{k,n}(a_0,\dots,a_{k},0)\cdot T_k\cdot T_{k+1}.\qedhere\]
\end{proof}

\paragraph{Additive simplices} 
Recall from \autoref{matrixrep} that for all primes $p$, $n\geq 2$ and $2\leq k\leq n$,\[ \left(\SBDA(\F_p^n)_k - \SBD(\F_p^n)_k\right)/\Sigma_{k+1}\cong \MDA(\F_p^n)_k/G_k.\] Thus, we may identify the coset $\sigma \Sigma_{k+1}$ with the coset $B\cdot G_k $ for some $B\in \MDA(\F_p^n)_k$.
We start with the following lemma.

\begin{lemma} \label{DA1}
     Let $p$ be a prime, $n\geq 3$ and $2\leq k < n-1$. Let $\sigma$ be an additive $k$-simplex in $\SBDA(\F_p^n)_k$ and let $B\in \DA_1^{k,n}$ such that $\sigma \Sigma_{k+1}=B\cdot G_k$. Then the action of $\Pb_n^\pm(\F_p)$ on $\sigma$ is orientation-reversing.\end{lemma}

\begin{proof} We recall that \[\DA_1^{k,n} =  \left\{ \left(v_0|\dots|v_k\right)\in \MDA(\F_p^n)_k ~\middle\vert e_1\not\in \operatorname{span}\{v_0,\dots,v_k\}\right\}.\]Let $ \sigma \Sigma_{k+1}=B\cdot G_k$ with $B\in \DA_1^{k,n}$. We want to show that the action of $\Pb_n^\pm(\F_p)$ on $\sigma$ is orientation-reversing.
We have by \autoref{lem:goodmat} that \[\left(\SBDA(\F_p^n)_k ^\mathrm{rv} -\SBD(\F_p^n)_k^\mathrm{rv}\right)/\Sigma_{k+1}\cong \MDA(\F_p^n)_k^\mathrm{rv}/G_k.\] 
It is enough to show that the action of $\Pb_n^\pm(\F_p)$ on $B$ is orientation-reversing. We have by property (b) of \autoref{BDA/Pn}, that \[B\cdot X\in \DA_1^{k,n}\quad\text{for all $X\in G_k$}.\]In particular, \[B\cdot X\in \DA_1^{k,n}\quad\text{for $X\in G_k$ with $\sign(X)=-1$}.\]Thus, there exist some $A\in \Pb_n^\pm(\F_p)$ and $X\in G_k$ such that $\sign(X)=-1$, and \[A\cdot B=B\cdot X.\]
Therefore, the action of $\Pb_n^\pm(\F_p)$ on $\sigma$ is orientation-reversing.\end{proof}
Let $\sigma$ be a additive $k$-simplex such that
$\sigma\Sigma_{k+1}=B\cdot G_k$ for some $B\in \DA_2^{k,n}(\underline{a})$.
We give a general description of such simplices $\sigma$ on which
the action of $\Pb_n^\pm(\F_p)$ is orientation-reversing.
\begin{lemma}\label{rev',a}
    Let $p\neq 3$ be a prime, $n\geq 2$ and $2\leq k\leq n-1$. Let $\sigma$ be an additive $k$-simplex in $\SBDA(\F_p^n)_k$ and let $B\in \DA_2^{k,n}(\underline{a})$ for some $\underline{a}=(a_0,\dots,a_k)\in A_k$ such that $\sigma \Sigma_{k+1}=B\cdot G_k$. Then the action of $\Pb_n^\pm(\F_p)$ on $\sigma$ is orientation-reversing if and only if there exist some $\lambda\in \F_p^\times$ and $X\in G_k$ with $\sign(X)=-1$, such that \[\lambda\underline{a}=\underline{a}\cdot X.\]
\end{lemma}
\begin{proof}
 Recall that \[A_k=\left\{\underline{a}=(a_0,\dots,a_k)\in\F_p^{k+1}\setminus\{\underline{0}\} ~\middle\vert a_0+a_1+a_2=0\right\}\] and \[\DA_2^{k,n}(\underline{a})=\left\{B=\left(v_0|\dots|v_k\right)\in \MDA(\F_p^n)_k ~\middle\vert \lambda e_1=a_0v_0+\dots+a_kv_k~\text{for some $\lambda\in\F_p^\times$}\right\}.\] 
Let $\sigma \Sigma_{k+1}=B\cdot G_k$ with $B\in \DA_2^{k,n}(\underline{a})$ for some $\underline{a}=(a_0,\dots,a_k)\in A_k$. Since  \[\left(\SBDA(\F_p^n)_k ^\mathrm{rv} -\SBD(\F_p^n)_k^\mathrm{rv}\right)/\Sigma_{k+1}\cong \MDA(\F_p^n)_k^\mathrm{rv}/G_k\] by \autoref{lem:goodmat}, we have the following equivalences:
   \begin{align*}\text{The action of $\Pb_n^\pm(\F_p)$ on $\sigma$} &~\text{is orientation-reversing}\\&\Longleftrightarrow \text{there exists $X\in G_k$ with $\sign(X)=-1$ and $B\cdot X=A\cdot B$ for some $A\in \Pb_n^\pm(\F_p)$}\\&\Longleftrightarrow \text{there exists $X\in G_k$ with $\sign(X)=-1$ and $B\cdot X\in \DA_2^{k,n}(\underline{a})$}\\ &\Longleftrightarrow \text{there exists $X\in G_k$ with $\sign(X)=-1$ and $\DA_2^{k,n}(\underline{a}\cdot X)=\DA_2^{k,n}(\underline{a})$},\end{align*}where the last equivalence follows from \eqref{DA2,rel2}. We have by \autoref{BDA/Pn} that \[\DA_2^{k,n}(\underline{a})=\DA_2^{k,n}(\underline{b})\quad\Longleftrightarrow\quad \underline{a}=\lambda\underline{b}~\text{for some $\lambda\in \F_p^\times$}.\]
   It implies that the action of $\Pb_n^\pm(\F_p)$ on $\sigma$ is orientation-reversing if and only if there exist $\lambda\in \F_p^
    \times$ and $X\in G_k$ such that $\sign(X)=-1$ and \[\lambda \underline{a}=\underline{a}\cdot X.\qedhere\]
\end{proof}

\begin{lemma}\label{DA2,eq,rv}
 Let $p\neq 3$ be a prime, $n\geq 2$ and $2\leq k\leq n-1$. Let $\sigma$ be an additive $k$-simplex in $\SBDA(\F_p^n)_k$ and let $B\in \DA_2^{k,n}(\underline{a})$ for some $\underline{a}=(a_0,\dots,a_k)\in A_k$ such that $\sigma \Sigma_{k+1}=B\cdot G_k$. Assume $a_i=\pm a_j$ for some $i, j\leq 2$ or $i,j\geq 3$. Then the action of $\Pb_n^\pm(\F_p)$ on $\sigma$ is orientation-reversing. 
\end{lemma}

\begin{proof}
By \autoref{rev',a}, it suffices to find 
$\lambda\in\F_p^\times$ and 
$X\in G_k$ with $\sign(X)=-1$ such that
\[
  \lambda\underline{a}
  =\underline a\cdot X.
\]
Set $\lambda=1$. Choose indices $i_0<i_1$ 
either in $\{0,1,2\}$ or in $\{3,\dots,k\}$ so that
$a_{i_0}=\varepsilon'a_{i_1}$ with $\varepsilon'\in\{-1,1\}$.  
Define
\[
  X=(\pi,\varepsilon,\tau,\varepsilon_3,\dots,\varepsilon_k)
  \in G_k
\]
by
\[
  \pi =
  \begin{cases}
    (i_0\; i_1)\in \Sigma_{\{0,1,2\}} &\text{if $i_0,i_1\leq2$},\\
    \Id                  & \text{otherwise},
  \end{cases}\]
 \[
  \tau =
  \begin{cases}
    \Id                  &\text{$i_0,i_1\leq2$},\\
    (i_0\; i_1)\in\Sigma_{\{3,\dots,k\}} & \text{otherwise},
  \end{cases}
\]
 
\[\varepsilon =
 \begin{cases}
         -1&\text{if $i_0,i_1\leq 2$ and $a_{i_0}=-a_{i_1}$}\\
         1&\text{otherwise}.
     \end{cases}\] 
     and 
\[
  \varepsilon_i =
  \begin{cases}
    \varepsilon' &\text{if $i=i_0,i_1$ and $i_0,i_1\geq 3$},\\
    1            & \text{otherwise}.
  \end{cases}
\]
 
We have $\sign(X)=\sign(\pi)\sign(\tau)=-1$.

\medskip
\begin{mycases}
    \case  $i_0,i_1\in\{0,1,2\}$. Then $\pi=(i_0\; i_1)$ swaps two of the first three entries, and all
$\varepsilon_i=1$. Without loss of generality, we assume $i_0=1$ and $i_1=2$.
 \subcase $a_1=a_2$. Then $\varepsilon=1$. We obtain \begin{align*}\underline{a}\cdot X&=(a_0,a_2,a_1,a_3,\dots,a_k)\\&=(a_0,a_1,a_2,a_3,\dots,a_k)\\&=\lambda\underline{a}.\end{align*}
\subcase $a_1=-a_2$. Then $a_0=0$ and $\varepsilon=-1$. We compute\begin{align*}\underline{a}\cdot X&=(0,-a_2,-a_1,a_2,\dots,a_k)\\&=(0,a_1,a_2,a_3,\dots,a_k)\\&=\lambda\underline{a}.\end{align*}
\case $i_0,i_1\in\{3,\dots,k\}$. Then $\tau=(i_0\; i_1)$ and $\pi=\Id$, with
$\varepsilon_{i_0}=\varepsilon_{i_1}=\varepsilon'$ and $\varepsilon_i=1$ otherwise. Thus
\[
  \underline a\cdot g
  =(\dots,\varepsilon' a_{i_1},\dots,\varepsilon' a_{i_0},\dots)
  =(\dots,a_{i_0},\dots,a_{i_1},\dots)
  =\lambda \underline a.
\]
\end{mycases}
We conclude that the action on $\sigma$ is
orientation-reversing.
\end{proof}

\begin{lemma}\label{rev,eq}Let $p\neq 3$ be a prime, $n\geq 2$ and $2\leq k\leq n-1$. Let $\sigma$ be an additive $k$-simplex in $\SBDA(\F_p^n)_k$ and let $B\in \DA_2^{k,n}(\underline{a})$ for some $\underline{a}=(a_0,\dots,a_k)\in A_k$ such that $\sigma \Sigma_{k+1}=B\cdot G_k$. If the action of $\Pb_n^\pm(\F_p)$ on $\sigma$ is orientation-reversing, then $a_i=\pm a_j$ for some $ i, j\leq 2$ or $i, j\geq 3$. 
\end{lemma}
\begin{proof}
Let $\sigma \Sigma_{k+1}=B\cdot G_k$ with $B\in \DA_2^{k,n}(\underline{a})$.
   Assume that the action of $\Pb_n^\pm(\F_p)$ on $\sigma$ is orientation-reversing. Then by \autoref{rev',a}, there exist $\lambda\in\F_p^\times$ and $X=(\pi,\varepsilon,
    \tau,
    \varepsilon_3,\dots,\varepsilon_k)\in G_k$ such that $\sign(X)=-1$ and\[\lambda \underline{a}=\underline{a}\cdot X.\]This implies that \begin{equation}\label{0eq} \lambda a_i=\varepsilon a_{\pi(i)}\quad\text{for all $0\leq i\leq 2$},
    \end{equation} \begin{equation}\label{3eq}\lambda a_i=\varepsilon_i a_{\tau(i)}\quad\text{for all $3\leq i\leq k$}.
    \end{equation}Since $\sign(X)=\sign(\pi)\sign(\tau)=-1$, we split the proof into three cases:\begin{mycases}
    \case $\lambda=\pm 1$.
    \subcase If $\sign(\pi)=-1$, then $\pi=(i\,j)$ is a transposition in $\{0,1,2\}$. \eqref{0eq} then implies that \[\lambda a_i=\varepsilon a_j,\quad\text{and}\quad \lambda a_j=\varepsilon a_i.\] 
    Since $\lambda=\pm 1$, we conclude that $a_i=\pm a_j$.
    \subcase If $\sign(\tau)=-1$, then $\tau(i)\neq i$ for some $3\leq i\leq n$. It then implies by that \eqref{3eq} \[a_i=\pm a_j\quad\text{for some $i\neq j$}.\]
    \case $\lambda\neq \pm 1$ and $\lambda^3\neq \pm 1$. 
    \subcase If $\sign(\pi)=-1$, then $\pi$ has a fixed point in $\{0,1,2\}$. So $\pi(i_0)=i_0$ for some $0\leq i_0\leq 2$. Since $\lambda\neq \pm 1$, it implies by \eqref{0eq} that $a_{i_0}=0$. Using the identity \[a_0+a_1+a_2=0,\]we conclude that $a_i=-a_j$ for some $0\leq i,j\leq 2$.
    \subcase If $\sign(\pi)=1$ and $\pi=\Id$, then since $\lambda\neq \pm 1$, \eqref{0eq} implies that $a_0=a_1=a_2=0$. 
    \subcase If $\sign(\pi)=1$ and $\pi=(i\; j\; \ell)$ is a rotation in $\{0,1,2\}$, then \begin{align*} \lambda a_i&=\varepsilon a_j\\ \lambda a_j&=\varepsilon a_{\ell}\\ \lambda a_{\ell}&=\varepsilon a_i. 
    \end{align*}Observe that since $\lambda^3\neq \pm 1$, it follows that $a_i=a_j=a_{\ell}=0$.
    \case $\lambda\neq\pm 1$ and $\lambda^3=\pm 1$. 
    \subcase If $\sign(\pi)=-1$, then by the same argument from Case $2$, we obtain $a_i=-a_j$ for some $0\leq i,j \leq 2$.
    \subcase If $\sign(\tau) = -1$, pick any cycle of even length
    $m$ under $\tau$ and let $i_0$ be an index in that cycle. We use the following claim on $m$.
     \begin{claim}
        If $\tau$ has a cycle of length $\ell>1$, then $\lambda^{\ell}=\pm 1$. 
    \end{claim}
    \begin{proof}
         Let $i_0$ be an entry in a cycle of length $\ell>1$. Then $\pi^{\ell}(i_0)=i_0$. 
    By \eqref{3eq}, we have \begin{align*}
     \lambda a_{i_0}&=\varepsilon_{i}a_{\tau(i_0)},\\
     \lambda a_{\pi(i_0)}&=\varepsilon_{\pi(i_0)}a_{\tau^2(i_0)},\\ &\ \vdots\\\lambda a_{\tau^{\ell-2}(i_0)}&=\varepsilon_{\tau^{\ell-2}(i_0)}a_{\tau^{\ell-1}(i_0)},\\\lambda a_{\tau^{\ell-1}(i_0)}&=\varepsilon_{\tau^{\ell-1}(i_0)}a_{\tau^{\ell}(i_0)}=\varepsilon_{\tau^{\ell-1}(i_0)}a_{_0}.
     \end{align*}
        As $\ell>1$ and $a_i\neq \pm a_j$ for all $i\neq j$, it follows that $a_{i_0},\dots,a_{\pi^{\ell-1}(i_0)}$ are all nonzero. It then implies that \begin{equation}\label{eq:aiii}\lambda^sa_{i_0}=\varepsilon_{i_0}\varepsilon_{\pi(i_0)}\varepsilon_{\pi^2(i_0)}\dots\varepsilon_{\pi^{s-1}(i_0)}a_{\pi^{s}(i_0)}~\text{for all $0\leq s\leq \ell$}.\end{equation}
        Since $\pi^{\ell}(i_0)=i_0$, it  follows by \eqref{eq:aiii} that\[\lambda^{\ell}=\varepsilon_{i_0}\varepsilon_{\pi(i_0)}\dots\varepsilon_{\pi^{\ell-1}(i_0)}=\pm 1.\qedhere\]
    \end{proof}
    It follows by the claim that $\lambda^{m}=\pm 1$. Since $\lambda^3=\pm 1$ and $\lambda\neq \pm 1$, then $m>3$. Now using \eqref{eq:aiii} on $\ell=m$ and $s=3$, we obtain 
    \[a_{i_0}=\pm a_{\tau^3(i_0)}.\] As $m$ is even, we have that $\tau^3(i_0) \neq i_0$. Thus, $a_i=\pm a_j$ for some $i\neq j$.
\end{mycases}
\end{proof}
\begin{comment}\begin{corollary}\label{cor:p=2,5}
    Let $p=2,5$, $n\geq 2$ and $2\leq k\leq n$. Let $\sigma$ be an additive $k$-simplex in $ \SBDA(\F_p^n)$. Then the action of $\Pb_n^\pm(\F_p)$ on $\sigma$ is orientation-reversing.
\end{corollary}\end{comment}
The preceding lemmas can be summarized in the following corollary.
\begin{corollary}\label{cor:add,pr}
Let $p\neq 3$ be a prime, $n\geq 2$ and $2\leq k\leq n$. Let $\sigma$ be an additive $k$-simplex in $\SBDA(\F_p^n)_k$ and let $B\in\MDA(\F_p^n)_k$ such that $\sigma\Sigma_{k+1}=B\cdot G_k$. Then $\sigma$ belongs to $\left(\SBDA(\F_p^n)_k ^\mathrm{pr}-  \SBD(\F_p^n)_k^\mathrm{pr}\right)$ if and only if $B$ belongs to some set \[\DA_2^{k,n}(\underline{a})\quad\text{with $\underline{a}=(a_0,\dots,a_k)\in A_k$},\] where the entries satisfy $a_i\neq \pm a_j$ for all $i,j\leq 2$ and $i,j\geq 3$.
    
\end{corollary}
\begin{definition}\label{prr,0}
    Let $p$ be a prime, $n\geq 2$ and $2\leq k\leq n-1$. Define the following.
    \begin{itemize}
        \item Define \[ \MDA(\F_p^n)_k^{\mathrm{pr},0}=\left\{B\in\MDA(\F_p^n)_k^{\mathrm{pr}}\middle\vert\begin{array}{c} 
       B\in\DA_2^{k,n}(\underline{a})\text{ for some $\underline{a}\in A_k$} \\ a_{i_0}=0\text{ for a unique $i_0\geq 3$}, a_i\neq \pm a_j\text{ for all $i\neq j$},\\
        \underline{a}\text{ does not satisfy the \hyperref[def1]{$(\lambda,m)$-condition}}\end{array}\right\}\]
         \item Define \[ \MDA(\F_p^n)_k^{\mathrm{pr},1}=\left\{B\in\MDA(\F_p^n)_k^{\mathrm{pr}}\middle\vert\begin{array}{c} 
       B\in\DA_2^{k,n}(\underline{a})\text{ for some $\underline{a}\in A_k$}\\ a_{i}\neq 0\text{ for all $i\geq 3$}, a_i\neq \pm a_j\text{ for all $i\neq j$},\\
        \underline{a}\text{ does not satisfy the \hyperref[def1]{$(\lambda,m)$-condition}}\end{array}\right\}.\]
    \end{itemize}
    
\end{definition}
\begin{remark}\label{rmk,prr}
    It follows by \autoref{cor:add,pr} that \[\Pb_n^\pm(\F_p)\backslash\MDA(\F_p^n)_k^\mathrm{pr}/\Sigma_{k+1}\cong \Pb_n^\pm(\F_p)\backslash\MDA(\F_p^n)_k^{\mathrm{pr},0}/G_k\quad\bigsqcup\quad \Pb_n^\pm(\F_p)\backslash\MDA(\F_p^n)_k^{\mathrm{pr},1}/G_k.\]
\end{remark}
\begin{lemma}\label{bij'}
    Let $p$ be a prime, $n\geq 2$ and $2\leq k\leq n-1$. There is a bijection, \[\Pb_n^\pm(\F_p)\backslash \MDA_{k+1}^{\mathrm{pr},0}(\F_p^n)/G_{k+1}\rightarrow \Pb_n^\pm(\F_p)\backslash \MDA(\F_p^n)_k^{\mathrm{pr},1}/G_k
    \]given by forgetting the unique zero entry:\[\DA_2^{k+1,n}(a_0,\dots,a_{k+1})\cdot G_{k+1}\mapsto\DA_2^{k,n}(a_0,\dots,\widehat{0},\dots,a_{k+1})\cdot G_k.\]
\end{lemma}
\begin{proof}
     Let $\DA_2^{k+1,n}(a_0,\dots,a_{k+1})\cdot G_{k+1}\in \Pb_n^\pm(\F_p)\backslash \MDA_{k+1}^{\mathrm{pr},0}(\F_p^n)/G_{k+1}$ with $a_{i_0}=0$. Define \[\Phi\colon \Pb_n^\pm(\F_p)\backslash\MDA(\F_p^n)_{k+1}^{\mathrm{pr},0}/G_{k+1}\rightarrow \Pb_n^\pm(\F_p)\backslash\MDA(\F_p^n)_k^{\mathrm{pr},1}/G_k\] by \[\Phi\left(\DA_2^{k+1,n}(a_0,\dots,a_{k+1})\cdot G_{k+1}\right)=\DA_2^{k,n}(a_0,\dots,\widehat{a_{i_0}},\dots,a_{k+1})\cdot G_{k}.\]Since $a_i\neq\pm a_j$ for all $i\neq j$, then none of the entries of $(a_0,\dots,\widehat{a_{i_0}},\dots,a_{k+1})$ is zero. Moreover, the tuple $(a_0,\dots,\widehat{a_{i_0}},\dots,a_{k+1})$ satisfies $a_i\neq \pm a_j$ for all $i\neq j$, and does not satisfy the \hyperref[def1]{$(\lambda,m)$-condition} since $\DA_2^{k+1,n}(a_0,\dots,a_{k+1})\cdot G_{k+1}\in \MDA(\F_p^n)_{k+1}^{\mathrm{pr},0}/G_{k+1}$.
    
Now let $\DA_2^{k,n}(a_0,\dots,a_{k})\cdot G_k\in\Pb_n^\pm(\F_p)\backslash\MDA(\F_p^n)_k^{\mathrm{pr},1}/G_k$. Define the inverse map $\Psi$ by \[\Psi\left(\DA_2^{k,n}(a_0,\dots,a_{k})\cdot G_k\right)=\DA_2^{k+1,n}(a_0,\dots,a_{k},0)\cdot G_k\cdot G_{k+1}.\qedhere\]  
\end{proof}
We now treat the case $p=3$.
\begin{lemma} \label{p=3}
     Let $n\geq 2$ and $2\leq k\leq n$. Let $\sigma$ be an additive $k$-simplex in $\SBDA(\F_3^n)_k$ and let $B\in \TA_2^{k,n}(\underline{a})$ for some $\underline{a}=(a_1,\dots,a_k)\in \F_3^{k}\setminus\{\underline{0}\}$ such that $\sigma \Sigma_{k+1}=B\cdot G_k$. Then the action of $\Pb_n^\pm(\F_3)$ on $\sigma$ is orientation-reversing.
\end{lemma}
\begin{proof}
     Let $\sigma \Sigma_{k+1}=B\cdot G_k$ with $B\in \TA_2^{k,n}(\underline{a})$. Since  \[\left(\SBDA(\F_3^n)_k ^\mathrm{rv} -\SBD(\F_3^n)_k^\mathrm{rv}\right)/\Sigma_{k+1}\cong \MDA(\F_3^n)_k^\mathrm{rv}/G_k\] by \autoref{lem:goodmat}, we have the following equivalences:
   \begin{align*}\Pb_n^\pm(\F_3)~\text{reverses the orientation} &~\text{of $\sigma$}\\&\Longleftrightarrow \text{there exists $X\in G_k$ with $\sign(X)=-1$ and $B\cdot X=A\cdot B$ for some $A\in \Pb_n^\pm(\F_3)$}\\&\Longleftrightarrow \text{there exists $X\in G_k$ with $\sign(X)=-1$ and $B\cdot X\in \TA_2^{k,n}(\underline{a})$}\\ &\Longleftrightarrow \text{there exists $X\in G_k$ with $\sign(X)=-1$ and $\TA_2^{k,n}(\underline{a})\cdot X=\TA_2^{k,n}(\underline{a})$},\end{align*} We have $a_1,a_2\in\{0,1,-1\}$. Thus, either \[a_1=\pm a_2\quad\text{or}\quad (a_1,a_2)\in\{(\pm 1,0), (0,\pm 1)\}.\]
        \begin{itemize}
             \item If $a_1=-a_2$, let $X=((1\,2),-1,\Id,1,\dots,1)\in G_k$. Then $\sign(X)=-1$ and \begin{align*}\TA_2^{k,n}(\underline{a})\cdot X&= \left\{\left(-v_0| -v_2| -v_1|v_3|\dots| v_k\right)\in\MDA(\F_3^n)_k~\middle\vert~\begin{array}{c}\lambda e_1= a_1(-v_2)+a_2(-v_1)+a_3v_3+\dots+a_kv_k \\ \text{for some $\lambda\in\F_3^\times$}\end{array}\right\}\\&=\TA_2^{k,n}(\underline{a}).\end{align*}
              \item If $a_1=a_2$, let $X=((1\,2),1,\Id,1,\dots,1)\in G_k$. Then $\sign(X)=-1$ and \begin{align*}\TA_2^{k,n}(\underline{a})\cdot X&=\left\{\left(v_0| v_2| v_1| v_3|\dots| v_k\right)\in\MDA(\F_3^n)_k~\middle\vert~\begin{array}{c}\lambda e_1= a_1v_2+a_2v_1+a_3v_3+\dots+a_kv_k \\ \text{for some $\lambda\in\F_3^\times$}\end{array}\right\}\\&= \TA_2^{k,n}(\underline{a}).
             \end{align*}
              \item If $(a_1,a_2)\in\{(\pm 1, 0),(0,\pm 1)\}$, it follows by \autoref{BDA,p=3} that \[\TA_2^{k,n}(\underline{a})\cdot Y=\TA_2^{k,n}(1,1,a_3,\dots,a_k)\]for some $Y\in G_k$. Let $X=((1\,2),1,\Id,1,\dots,1)\in G_k$ be as in the latter case, then \[\TA_2^{k,n}(1,1,a_3,\dots,a_k)\cdot X=\TA_2^{k,n
              }(1,1,a_3,\dots,a_k).\] Then, \[\TA_2^{k,n}(\underline{a})\cdot YXY^{-1}=\TA_2^{k,n}(1,1,a_3,\dots,a_k)\cdot XY^{-1}=\TA_2^{k,n}(1,1,a_3,\dots,a_k)\cdot Y^{-1}=\TA_2^{k,n}(\underline{a}),\]
             where $\sign(YXY^{-1})=\sign(X)=-1$. 
            \end{itemize} This completes the proof.\end{proof}
\subsection{Homology of the complex of partial frames}
We will prove that the symmetric $\Delta$-complex $\Gamma_{0,n}^\pm(p)\backslash\SB(\Z^n)$ is highly $\Q$-acyclic. Recall that 
$\Gamma_{0,n}^\pm(p)\backslash\SB(\Z^n)$ has dimension $n-1$. 
\begin{theorem}\label{hom:B}
    Let $p$ be a prime and $n\geq 2$.  \[\redhom_k(\Gamma_{0,n}^\pm(p)\backslash\SB(\Z^n);\Q)\cong 0 \quad\text{for}\quad k\leq n-2,\]and for all $n\geq 3$,\[\redhom_{n-1}(\Gamma_{0,n}^\pm(p)\backslash\SB(\Z^n);\Q)\cong\Q\left[\Pb_n^\pm(\F_p)\backslash\MD(\F_p^n)_{n-1}^{\mathrm{pr},1}\right]\otimes_{\Q[T_k]}\Q^\mathrm{sgn}.\]
\end{theorem}
 Recall from \autoref{good q} that there is an isomorphism\[\Gamma_{0,n}^\pm(p)\backslash\SB(\Z^n)\cong \Pb_n^\pm(\F_p)\backslash\SBD(\F_p^n).\]Thus, it suffices to prove all our claims for $\Pb_n^\pm(\F_p)\backslash\SBD(\F_p^n)$. Additionally, we have from \autoref{coinv} that \[\redchain_k(\Pb_n^\pm(\F_p)\backslash\SBD(\F_p^n);\Q)=\Q\left[\Pb_n^\pm(\F_p)\backslash\SBD(\F_p^n)_k\right]\otimes_{\Q[\Sigma_{k+1}]}\Q^\mathrm{sgn}\]and has a non-canonical basis bijective to $\Pb_n^\pm(\F_p)\backslash\SBD(\F_p^n)^\mathrm{pr}_k/T_k$.
We start with the following lemma.

\begin{lemma}\label{lemm:vert}
Let $p$ be a prime. The chain group $\redchain_0\left(\Pb_n^\pm(\F_p)\backslash \SBD(\F_p^n);\Q\right)$ is two-dimensional with basis elements $\D_1^{0,n}\otimes 1$ and $\D_2^{0,n}(1)\otimes 1$.
\end{lemma}
\begin{proof}
We recall that \[\D_1^{k,n} =  \left\{ \left(v_0|\dots|v_k\right)\in \MD(\F_p^n)_k ~\middle\vert e_1\not\in \operatorname{span}\{v_0,\dots,v_k\}\right\},\]
    and \[\D_2^{k,n}(\underline{a})=\left\{B=\left(v_0|\dots|v_k\right)\in \MD(\F_p^n)_k ~\middle\vert \lambda e_1=a_0v_0+\dots+a_kv_k~\text{for some $\lambda\in\F_p^\times$}\right\}\] for $\underline{a}=(a_0,\dots,a_k)\in\F_p^{k+1}\setminus\{\underline{0}\}$.

     By \autoref{qut}, the quotient $\Pb_n^\pm(\F_p)\backslash \SBD(\F_p^n)$ is a symmetric $\Delta$-complex. It therefore follows from \autoref{chaincpx} that \[\redchain_0(\Pb_n^\pm(\F_p)\backslash \SBD(\F_p^n);\Q)\cong\Q\left[\Pb_n^\pm(\F_p)\backslash \SBD(\F_p^n)_0\right]\otimes_{\Q[\Sigma_1]}\Q^\mathrm{sgn},\]which by \autoref{matrixrep}, is isomorphic to $\Q\left[\Pb_n^\pm(\F_p)\backslash \MD(\F_p^n)_0\right]\otimes_{\Q[T_0]}\Q^\mathrm{sgn}$.
    We recall by \autoref{BD/Pn} that $\Pb_2^\pm(\F_p)\backslash \MD(\F_p^n)_0=\{\D_1^{0,n}\}\cup \{\D_2^{0,n}(\lambda)\mid \text{ for some }\lambda\in\F_p^\times\}$, and \[\D_2^{k,n}(\underline{a})=\D_2^{k,n}(\underline{b})\quad\Longleftrightarrow\quad \underline a= \mu \underline b~\text{for some $\mu\in \F_p^\times$}.\] 
    In particular, $\D_2^{0,n}(\pm \lambda)=\D_2^{0,n}(1)$ for all $\lambda\in\F_p^\times$. We conclude \[\redchain_0(\Pb_n^\pm(\F_p)\backslash \SBDA(\F_p^n);\Q)\cong \Q[\D_1^{0,n}\otimes 1,\D_2^{0,n}(1)\otimes 1].\qedhere\]
\end{proof}

In order to prove \autoref{hom:B}, we consider the cases $p=2$ and $p\neq 2$ separately, starting with $p=2$.
\begin{proposition}\label{BD,p=2}
     Let $n\geq 2$.\[\redhom_k\left(\Pb_n^\pm(\F_2)\backslash\SBD(\F_2^n);\Q\right)\cong 0.\]
 \end{proposition}
 \begin{proof}

Recall from \autoref{coinv} \[\redchain_k\left(\Pb_n^\pm(\F_2)\backslash \SBD(\F_2^n);\Q\right)\cong \Q\left[\Pb_n^\pm(\F_2)\backslash \SBD(\F_2^n)^\mathrm{pr}_k\right]\otimes_{\Q[\Sigma_{k+1}]}\Q^\mathrm{sgn}.\] 
we consider the following chain complex \[
\oset{$n-1$}{\Q\left[\Pb_n^\pm(\F_2)\backslash \SBD(\F_2^n)^\mathrm{pr}_n\right]\otimes_{\Q[\Sigma_{n+1}]}\Q^\mathrm{sgn}} \overset{\partial_{n-1}}\rightarrow \dots\overset{\partial_{2}}\rightarrow \oset{$1$}{\Q\left[\Pb_n^\pm(\F_2)\backslash \SBD(\F_2^n)^\mathrm{pr}_1\right]\otimes_{\Q[\Sigma_{2}]}\Q^\mathrm{sgn}} \overset{\partial_{1}}\rightarrow \oset{$0$}{\Q^2}\overset{\varepsilon} \rightarrow \Q
\] where $\partial_{k}$ is the differential and $\varepsilon$ is the augmentation map. We recall that by \autoref{lemm:vert}, we have \[ \redchain_0\left(\Pb_n^\pm(\F_2)\backslash\SBD(\F_2^n);\Q\right)\cong \Q^2.\]
From \autoref{lem:goodmat}, $\redchain_k\left(\Pb_n^\pm(\F_2)\backslash \SBD(\F_2^n);\Q\right)$ has a non-canonical basis bijective to \[\Pb_n^\pm(\F_2)\backslash\SBD(\F_2^n)_k^{\mathrm{pr}}/\Sigma_{k},\]which by \autoref{lem:goodmat} is again bijective to \[\Pb_n^\pm(\F_2)\backslash\MD(\F_p^n)_k^{\mathrm{pr}}/T_k.\]Recall from \autoref{cor:st,pr} that $B\in\MD(\F_2^n)^\mathrm{pr}_k$ if and only if $B\in\D_2^{k,n}(\underline{a})$ for some $\underline{a}=(a_0,\dots,a_k)\in\F_2^{k+1}\backslash\{\underline{0}\}$  where the entries satisfy $a_i\neq \pm a_j$ for all $i\neq j$ and do not satisfy the \hyperref[def1]{$(\lambda,m)$-condition}. 

For $k\geq 2$, $a_0,\dots,a_k\in\F_2$ cannot all be distinct, so there exist some $i\neq j$ with $a_i=a_j$. This implies that $\Pb_n^\pm(\F_2)\backslash\MD(\F_2^n)_k^{\mathrm{pr}}/T_k$ is empty, and so 
\[\redchain_k\left(\Pb_n^\pm(\F_2)\backslash \SBD(\F_2^n);\Q\right)\cong 0\quad\text{for all $k\geq 2$}.\]
As for degree $k=1$, $B\in\MD(\F_2^n)^\mathrm{pr}_k$ if and only if $B\in \D_2^{1,n}(1,0)$ or in $\D_2^{1,n}(0,1)$. We know from \autoref{coinv} that\[\D_2^{1,n}(\underline{a})\otimes 1=\sign(X)\D_2^{1,n}(\underline{a})\otimes\sign(X)~\text{for all $X\in T_1$}.\]
It is then enough to compute\begin{equation}\label{eq:B}\partial_1\left(\D_2^{1,n}(1,0)\otimes 1\right)=\D_1^{0,n}\otimes 1-\D_2^{0,n}(1)\otimes 1.\end{equation}Note that this a similar argument to that in \autoref{n=2,BA} for $n=2$, where $\D_1^{0,2}=V$ and $\D_2^{0,2}=U$.

As $\ker \varepsilon$ is generated by such elements, by \autoref{lemm:vert}, we obtain $\operatorname{im}(\partial_1)=\ker \varepsilon$. 
Additionally, it follows from \eqref{eq:B} that \[\ker\partial_1\cong  0.\] 
Therefore \[\redhom_k\left(\Pb_n^\pm(\F_2)\backslash\SBD(\F_2^n);\Q\right)\cong 0\quad\text{for all $k$}.\]This also shows that \[\redhom_{n-1}(\Pb_n^\pm(\F_2)\backslash\SBD(\F_2^n);\Q)\cong\Q\left[\Pb_n^\pm(\F_2)\backslash\MD(\F_2^n)_{n-1}^{\mathrm{pr},1}\right]\otimes_{\Q[T_k]}\Q^\mathrm{sgn}\cong 0.\qedhere\]
 \end{proof}
We now treat the case $p\geq 3$.
\begin{proof}[Proof of \autoref{hom:B}]

We previously showed in \autoref{n=2,BA} that the result holds for $n=2$ and all primes $p$, and in \autoref{BD,p=2} that it holds for $p=2$ and all $n\geq 2$. In both cases,\[\redhom_k\left(\Pb_n^\pm(\F_p)\backslash\SBD(\F_p^n);\Q\right)\cong 0\] for all $k\leq n-2$. It remains to verify the statement for $n\geq 3$ and $p\geq 3$. So we will assume that $n,p\geq 3$ for the rest of this proof. Recall from \autoref{coinv} \[\redchain_k\left(\Pb_n^\pm(\F_p)\backslash \SBD(\F_p^n);\Q\right)\cong \Q\left[\Pb_n^\pm(\F_p)\backslash \SBD(\F_p^n)_k\right]\otimes_{\Q[\Sigma_{k+1}]}\Q^\mathrm{sgn},\]with \begin{equation}\label{rvv}\Pb_n^\pm(\F_p)\sigma\otimes1=0 \quad\text{if $\sigma\in\SBD(\F_p^n)_k^\mathrm{rv}$.}\end{equation} In particular, \[\redchain_k\left(\Pb_n^\pm(\F_p)\backslash \SBD(\F_p^n);\Q\right)\cong \Q\left[\Pb_n^\pm(\F_p)\backslash \SBD(\F_p^n)_k^\mathrm{pr}\right]\otimes_{\Q[\Sigma_{k+1}]}\Q^\mathrm{sgn}.\]
We consider the following chain complex \[
\oset{$n-1$}{\Q\left[\Pb_n^\pm(\F_p)\backslash \SBD(\F_p^n)^\mathrm{pr}_n\right]\otimes_{\Q[\Sigma_{n+1}]}\Q^\mathrm{sgn}} \overset{\partial_{n-1}}\rightarrow \dots\overset{\partial_{2}}\rightarrow \oset{$1$}{\Q\left[\Pb_n^\pm(\F_p)\backslash \SBD(\F_p^n)^\mathrm{pr}_1\right]\otimes_{\Q[\Sigma_{2}]}\Q^\mathrm{sgn}} \overset{\partial_{1}}\rightarrow \oset{$0$}{\Q^2}\overset{\varepsilon} \rightarrow \Q
\] where $\partial_{k}$ is the differential and $\varepsilon$ is the augmentation map. We recall that by \autoref{lemm:vert}, we have \[ \redchain_0\left(\Pb_n^\pm(\F_p)\backslash\SBD(\F_p^n);\Q\right)\cong \Q^2.\] To prove the statement, we will show that at each degree of the chain complex the image coincides with the kernel. 
From \autoref{coinv}, $\redchain_k\left(\Pb_n^\pm(\F_p)\backslash \SBD(\F_p^n);\Q\right)$ has a non-canonical basis bijective to \[\Pb_n^\pm(\F_p)\backslash\SBD(\F_p^n)_k^{\mathrm{pr}}/\Sigma_{k+1},\]which by \autoref{lem:goodmat} and \autoref{rmk,pr}, is again bijective to \[\Pb_n^\pm(\F_p)\backslash\MD(\F_p^n)_k^{\mathrm{pr},0}/T_k\quad\bigsqcup\quad \Pb_n^\pm(\F_p)\backslash\MD(\F_p^n)_k^{\mathrm{pr},1}/T_k.\]
   Recall from \autoref{cor:st,pr} that $B\in\MD(\F_p^n)^\mathrm{pr}_k$ if and only if $B\in\D_2^{k,n}(\underline{a})$ for some $\underline{a}=(a_0,\dots,a_k)\in\F_p^{k+1}\backslash\{\underline{0}\}$  where the entries satisfy $a_i\neq \pm a_j$ for all $i\neq j$ and do not satisfy the \hyperref[def1]{$(\lambda,m)$-condition}. 

Under the above identifications, we compute the differentials in terms of the elements $\D_2^{k,n}(\underline{a})\otimes1\subset \Q\left[\Pb_n^\pm(\F_p)\backslash \MD(\F_p^n)^{\mathrm{pr},i}_k\right]\otimes_{\Q[T_k]}\Q^\mathrm{sgn}$, for $i=0,1$, satisfying the above conditions.

Consider the case $a_i\neq0$ for all $i$. Then $\partial_k$ sends $\D_2^{k,n}(\underline{a})\otimes1$ to $\D_1^{k,n}\otimes1$ which is zero by \eqref{rvv} and \autoref{cor:st,pr}. 

For the other case, suppose that $a_{i_0}=0$ for some $i_0$. $\partial_k$ sends $\D_2^{k,n}(\underline{a})\otimes1$ to \[\pm \D_2^{k-1,n}(a_0,\dots,\widehat{a_{i_0}},\dots,a_k)\otimes 1\] if $k\geq 2$, and to \[\pm (\D_2^{0,n}(1)\otimes 1-\D_1^{0,n}\otimes 1)\] if $k=1$. We recall from \autoref{lemm:vert}, that \[\redchain_0(\Pb_n^\pm(\F_p)\backslash\SBD(\F_p^n);\Q)\cong \Q[\D_1^{0,n}\otimes 1,\D_2^{0,n}(1)\otimes 1].\] Thus \begin{equation}\label{k}\ker\varepsilon\cong\Q\left[\D_1^{0,n}\otimes 1-\D_2^{0,n}(1)\otimes 1 \right]\cong \Image\partial_1,\end{equation}and 
 $\partial_k$ is zero on $\Q\left[ \Pb_n^\pm(\F_p)\backslash\MD_k^{\mathrm{pr},1}(\F_p^n)\right]\otimes_{\Q[T_k]}\Q^\mathrm{sgn}$. We conclude from \autoref{bij} \begin{equation*}\label{ker1}
    \ker\partial_k=\Q\left[ \Pb_n^\pm(\F_p)\backslash\MD_k^{\mathrm{pr},1}(\F_p^n)\right]\otimes_{\Q[T_k]}\Q^\mathrm{sgn}=\Image\partial_{k+1}\quad\text{for all $k\geq 1$.}
\end{equation*}
Together with \eqref{k}, we obtain
\[\redhom_k(\Pb_n^\pm(\F_p)\backslash \SBD(\F_p^n);\Q)\cong 0\quad\text{for all $k\leq n-2$},\]and \[\redhom_{n-1}(\Pb_n^\pm(\F_p)\backslash \SBD(\F_p^n);\Q)\cong\Q\left[ \Pb_n^\pm(\F_p)\backslash\MD_k^{\mathrm{pr},1}(\F_p^n)\right]\otimes_{\Q[T_k]}\Q^\mathrm{sgn}.\qedhere\]
\end{proof}

\begin{corollary}\label{BD,p=3,5}
    Let $p$ be a prime and $n\geq 2$ such that $p\leq 2n-1$. Then, \[\redhom_{n-1}\left( \Pb_n^\pm(\F_p)\backslash \SBD(\F_p^n);\Q\right)\cong 0.\] 
\end{corollary}
\begin{proof}
    If $p=2$, then $\F_2^\times=\{1\}$, so for any choice of $(a_0,\dots,a_{n-1})$ we necessarily have $a_i=\pm a_j$ for $i\neq j$. Hence every generator vanishes in this case.

Now assume that $p$ is odd. The quotient $\F_p^\times/\{\pm1\}$ has cardinality $(p-1)/2$. Since $p\leq 2n-1$, we have
\[
\frac{p-1}{2}\leq n-1.
\]
Thus, among any $n$ elements $a_0,\dots,a_{n-1}\in\F_p^\times$, there exist distinct indices $i\neq j$ such that $a_i=\pm a_j$. Consequently, no element $\D_2^{n-1,n}(a_0,\dots,a_{n-1})$ can satisfy $a_i\neq \pm a_j$ for all $i\neq j$. Therefore the result follows by \autoref{hom:B}.
\end{proof}

\subsection{Homology of the complex of partial augmented frames}
We now prove a high acyclicity result for  $\Gamma_{0,n}^\pm(p)\backslash \SBA(\Z^n)$.

\begin{theorem} \label{hom:BA} Let $p$ be a prime and $n\geq 2$ such that $p\in\{2,3,5,7,13\}$ or $p\leq 6n-8$.
    \[\redhom_k(\Gamma_{0,n}^\pm(p)\backslash\SBA(\Z^n);\Q)\cong 0\quad\text{for $k\leq n-1$}.\]
\end{theorem}
Recall from \autoref{good q} that there is an isomorphism\[\Gamma_{0,n}^\pm(p)\backslash\SB(\Z^n)\cong \Pb_n^\pm(\F_p)\backslash\SBD(\F_p^n).\]Thus, it suffices to prove all our claims for $\Pb_n^\pm(\F_p)\backslash\SBD(\F_p^n)$.

We will treat the cases $p=2$ and $p\neq 2$ separately, starting with $p=2$. 
 \begin{proposition}\label{BDA,p=2}
     Let $n\geq 2$. \[\redhom_k\left(\Pb_n^\pm(\F_2)\backslash\SBD(\F_2^n);\Q\right)\cong 0\quad\text{for all $k$}\]
 \end{proposition}
 \begin{proof}
We consider the following chain complex of the pair $\rel_n(2)=\left(\Pb_n^\pm(\F_2)\backslash\SBDA(\F_2^n), \Pb_n^\pm(\F_2)\backslash\SBD(\F_2^n)\right)$,\[
\oset{$n$}{\C_{n}\left(\rel_n(2);\Q\right)} \overset{\partial_{n}}\longrightarrow \oset{$n-1$}{\C_{n-1}\left(\rel_n(2);\Q\right)} \overset{\partial_{n-1}}\longrightarrow\dots\overset{\partial} \longrightarrow \oset{$2$}{\C_{2}\left(\rel_n(2);\Q\right)} \longrightarrow \oset{$1$}{0} \longrightarrow \oset{$0$}{0}
\] where $\partial_{k}$ is the differential. From \autoref{BD,p=2}, we already know that \[\redhom_k\left(\Pb_n^\pm(\F_2)\backslash\SBD(\F_2^n);\Q\right)\cong 0\quad\text{for all $k$}.\]
Thus, it suffices to prove that  \begin{equation}\label{C}\C_k\left(\rel_n(2);\Q\right)\cong 0\quad\text{for all $k\geq 2$}.\end{equation}
Then, applying the long exact sequence of the pair $\rel_n(2)$, together with \eqref{C}, we conclude that\[ \redhom_k\left( \Pb_n^\pm(\F_2)\backslash\SBDA(\F_2^n);\Q\right)\cong 0\quad\text{for all $k$}.\] 
As \[\C_k\left( \rel_n(2);\Q\right)=\redchain_k\left( \Pb_n^\pm(\F_2)\backslash\SBDA(\F_2^n);\Q\right)~/~\redchain_k\left( \Pb_n^\pm(\F_2)\backslash\SBD(\F_2^n);\Q\right), \]it implies by \autoref{coinv} that $\C_k(\rel_n(2);\Q)$ has a non-canonical basis bijective to  \[\Pb_n^\pm(\F_2)\backslash\left(\SBDA(\F_2^n)_k^\mathrm{pr}-\SBD(\F_p^2)_k^\mathrm{pr}\right)/\Sigma_{k+1},\] which is again bijective by \autoref{lem:goodmat} to $\Pb_n^\pm(\F_2)\backslash\MDA(\F_2^n)_k^\mathrm{pr}/G_k.$ We have from \autoref{cor:add,pr} that $B\in\MDA(\F_p)_k^\mathrm{pr}$ if and only if $B\in\DA_2^{k,n}(\underline{a})$ for $\underline{a}=(a_0,\dots,a_k)\in A_k$ such that $a_i \neq \pm a_j$ for all $i \neq j\leq 2$ and $i,j\geq 3$.  

As $p=2$ and $k\geq 2$, there exist some $i\neq j\geq 2$ with $a_i=a_j$. This implies that \[\Pb_n^\pm(\F_2)\backslash\MDA(\F_2^n)_k^\mathrm{pr}/G_k\] is empty and so
\[\C_k\left(\rel_n(2);\Q\right)\cong 0\quad\text{for all $k\geq 2$}.\qedhere\]
\end{proof}
We now treat the case $p\neq 2$. The proof consists of three steps. In the first step, we show that\[\redhom_k\left(\Pb_n^\pm(\F_p)\backslash\SBDA(\F_p^n);\Q\right)\cong 0~\text{for all $k\leq n-2$ and all primes $p$}.\] In the second step, we show that \[\redhom_{n-1}\left(\Pb_n^\pm(\F_p)\backslash\SBDA(\F_p^n);\Q\right)\cong 0\quad\text{for $n\geq 3$ and $p\leq 6n-8$}.\] Finally, in the third step, we show that \[\redhom_{n-1}\left(\Pb_n^\pm(\F_p)\backslash\SBDA(\F_p^n);\Q\right)\cong 0\quad\text{for $n=3$ and $p=13$}.\] 

\begin{remark} The range from the second step already covers all cases with $n \geq 3$ and $p \leq 7$, as well as $n \geq 4$ and $p = 13$. Together with the third step (the case $n=3, p=13$) and \autoref{n=2,BA} (which treats $n=2$ for $p\in\{2,3,5,7,13\}$), this establishes the result for every $n$ when $p\in\{2,3,5,7,13\}$.
\end{remark}
In addition to proving vanishing results, our approach for the first step (\autoref{BA connected}) gives an explicit description of a generating set for the top relative homology group \[\homology_{n}\left( \Pb_n^\pm(\F_p)\backslash \SBDA(\F_p^n),\Pb_n^\pm(\F_p)\backslash \SBD(\F_p^n);\Q\right),\]which we use in the proofs of \autoref{BA connected'} and \autoref{BA ctd}.
\begin{lemma}\label{BDA:p=3}
    Let $n\geq 2$ and $p=3$. \[\redhom_k\left(\Pb_n^\pm(\F_3)\backslash\SBDA(\F_3^n);\Q\right)\cong 0\quad\text{for all $k$}.\]
\end{lemma}
\begin{proof}
    Let $\sigma$ be an additive $k$-simplex in $\SBDA(\F_p)_k$. We have from \autoref{DA1} and \autoref{p=3}, that the action of $\Pb_n^\pm(\F_3)$ on $\sigma$ is orientation-reversing. Let $\rel_n(3)=\left(\Pb_n^\pm(\F_3)\backslash\SBDA(\F_3^n), \Pb_n^\pm(\F_3)\backslash\SBD(\F_3^n)\right)$. It then implies by \autoref{coinv}, that $\C_k\left( \rel_n(3);\Q\right)\cong 0$ for all $k$. Thus $\homology_k(\rel_n(3);\Q)\cong 0$ for all $k$. Moreover, by \autoref{BD,p=3,5} and \autoref{hom:B}, \[\redhom_k(\Pb_n^\pm(\F_3)\backslash\SBD(\F_3^n);\Q)\cong 0\quad\text{for all $k$}.\] 
It then follows by the long exact sequence of $\rel_n(3)$ that \begin{equation*}\redhom_k\left(\Pb_n^\pm(\F_3)\backslash\SBDA(\F_3^n);\Q\right)\cong 0\quad\text{for all $k$}.\qedhere\end{equation*} 
\end{proof}
\begin{proposition}[Step One]\label{BA connected}
    Let $p$ be a prime and $n\geq 3$. \[\redhom_k\left(\Pb_n^\pm(\F_p)\backslash\SBDA(\F_p^n);\Q\right)\cong 0\quad\text{for all $k\leq n-2$},\]and \[\redhom_{n}(\rel_n(p);\Q)\cong \Q\left[ \Pb_n^\pm(\F_p)\backslash\MDA_k^{\mathrm{pr},1}(\F_p^n)\right]\otimes_{\Q[G_k]}\Q^\mathrm{sgn}.\]
\end{proposition}
\begin{proof}
The case $p\leq 3$ has already been established in \autoref{BDA,p=2} and \autoref{BDA:p=3}. So we let $p\geq 5$. Consider the long exact sequence of the pair $\rel_n(p)=\left(\Pb_n^\pm(\F_p)\backslash\SBDA(\F_p^n), \Pb_n^\pm(\F_p)\backslash\SBD(\F_p^n)\right)$,  \[
\oset{$n$}{\C_n(\rel_n(p);\Q)} \overset{\partial_{n}}\longrightarrow\oset{$n-1$}{\C_{n-1}(\rel_n(p);\Q)} \overset{\partial_{n-1}}\longrightarrow \dots \longrightarrow \oset{$2$}{\C_{2}(\rel_n(p);\Q)} \overset{\partial_{2}}\longrightarrow \oset{$1$}0 \longrightarrow \oset{$0$} 0
\] where $\partial_{k}$ is the differential.
To prove our claim, we will prove that the image coincides with the kernel.
\autoref{coinv} says that \[\C_k(\rel_n(p);\Q)\cong\Q\left[\Pb_n^\pm(\F_p)\backslash\left(\SBDA(\F_p^n) _k-\SBD(\F_p^n)_k\right)\right]\otimes_{\Q[\Sigma_{k+1}]}\Q^\mathrm{sgn},\]with \begin{equation}\label{rvv'}\Pb_n^\pm(\F_p)\sigma\otimes 1=0\quad\text{if $\sigma\in \left(\SBDA(\F_p^n) _k-\SBD(\F_p^n)_k\right)^\mathrm{pr}$.}\end{equation}

From \autoref{lem:goodmat}, $\redchain_k\left(\Pb_n^\pm(\F_p)\backslash \SBDA(\F_p^n);\Q\right)$ has a non-canonical basis bijective to \[\Pb_n^\pm(\F_p)\backslash\left(\SBD(\F_p^n)_k-\SBD(\F_p^n)_k\right)^{\mathrm{pr}}/\Sigma_{k},\]which by \autoref{lem:goodmat} and \autoref{rmk,pr}, is again bijective to \[\Pb_n^\pm(\F_p)\backslash\MDA(\F_p^n)_k^{\mathrm{pr},0}/G_k\quad\bigsqcup\quad \Pb_n^\pm(\F_p)\backslash\MDA(\F_p^n)_k^{\mathrm{pr},1}/G_k.\]
Recall from \autoref{cor:add,pr} that $B\in\MDA(\F_p^n)_k^\mathrm{pr}$ if and only if $B\in\DA_2^{k,n}(\underline{a})$ for $\underline{a}=(a_0,\dots,a_k)\in A_k$ such that $a_i \neq \pm a_j$ for all $i \neq j\leq 2$ and $i,j\geq 3$.

Under the above identifications, we compute the differentials in terms of the elements $\DA_2^{k,n}(\underline{a})\otimes1\subset \Q\left[\Pb_n^\pm(\F_p)\backslash \MDA(\F_p^n)^{\mathrm{pr},i}_k\right]\otimes_{\Q[G_k]}\Q^\mathrm{sgn}$, for $i=0,1$, satisfying the above conditions.

We observe the following: For any \[B=(v_0|\dots|v_k) \in \DA_2^{k,n}(a_0,\dots,a_k),\] there exists $\lambda\in\F_p^\times$, such that \begin{equation}\label{e'}\lambda e_1=a_0v_0+\dots+a_kv_k.\end{equation}
Since $v_0+v_1+v_2=0$, we can write $\lambda e_1$ in several equivalent forms: \begin{equation}\begin{aligned}\label{e11}
    \lambda e_1&=(2a_1+a_2)v_1+(2a_2+a_1)v_2+a_3v_3+\dots+a_kv_k\\&=(2a_0+a_2)v_0+(2a_2+a_0)v_2+a_3v_3+\dots+a_kv_k\\&=(2a_0+a_1)v_0+(2a_1+a_0)v_1+a_3v_3+\dots+a_kv_k.
\end{aligned}\end{equation}

Moreover, if $i\in\{0,1,2\}$, \eqref{e'} implies that the $i^\mathrm{th}$ face map of $\partial_k$ sends $\DA_2^{k,n}(\underline{a})\otimes 1$ to $\D_2^{k-1,n}(\underline{c})\otimes 1$ for some $\underline{c}\in\F_p^{k}$. Since $\partial_k$ is the differential on the relative chain complex, then $\D_2^{k-1,n}(\underline{c})\otimes 1$ is zero.

We now consider the case where $a_{i} \neq 0$ for all $i\geq 3$. Then \eqref{e'} implies that the $i^\mathrm{th}$ face map of $\partial_k$ sends $\DA_2^{k,n}(\underline{a})\otimes 1$ to $\DA_1^{k-1,n}\otimes 1$, which is zero by \eqref{rvv'} and \autoref{cor:add,pr}. 

Now suppose that $a_{i_0}=0$ for some $i_0\geq 3$. Then $\partial_k$ sends $\DA_2^{k,n}(\underline{a})\otimes 1$ to \[\pm \DA_2^{k-1,n}(a_0,\dots,\widehat{a_{i_0}},\dots,a_k)\otimes 1.\]

Thus $\partial_k$ is zero on $\Q\left[ \Pb_n^\pm(\F_p)\backslash\MDA_k^{\mathrm{pr},1}(\F_p^n)\right]\otimes_{\Q[G_k]}\Q^\mathrm{sgn}$. We conclude from \autoref{bij'}, that for all $k\geq 2$,
\begin{equation*}
    \ker\partial_k=\Q\left[ \Pb_n^\pm(\F_p)\backslash\MDA_k^{\mathrm{pr},1}(\F_p^n)\right]\otimes_{\Q[G_k]}\Q^\mathrm{sgn}=\Image\partial_k.
\end{equation*}
  Therefore, \begin{equation}\label{rel=0}\homology_k(\rel_n(p);\Q)\cong 0\quad\text{for all $2\leq k\leq n-1$}.\end{equation}
  Since $\homology_k(\Pb_n^\pm(\F_p)\backslash \SBD(\F_p^n);\Q)\cong 0$ for all $k\leq n-2$ and all $p$ (by \autoref{hom:B}), it follows by the long exact sequence on $\rel_n$ that \[\redhom_k\left(\Pb_n^\pm(\F_p)\backslash \SBDA(\F_p^n);\Q\right)\cong 0\quad\text{for all $k\leq n-2$}.\]Additionally, \[\redhom_{n}(\rel_n(p);\Q)=\ker\partial_n= \Q\left[ \Pb_n^\pm(\F_p)\backslash\MDA_k^{\mathrm{pr},1}(\F_p^n)\right]\otimes_{\Q[G_k]}\Q^\mathrm{sgn}.\qedhere\]\end{proof} 
  Before proving \autoref{BA connected'}, we state a theorem by Dias da Silva and Hamidoune \cite{[DdSH}.
  \begin{theorem}
      [\cite{[DdSH}]\label{DH}Let $A\subseteq \F_p$. Define the set \[S_A=\{a_1+a_2+a_3\mid a_i\in A, a_i\neq a_j~\text{for all $i\neq j$}\}.\]Then $|S_A|\geq \min\{p,3|A|-8\}$.
\end{theorem}
 \begin{corollary}\label{bound}
   Let $p>5$ be a prime and let $n\geq \frac{p+8}{6}$. Then for any $a_1,\dots,a_n\in \F_p^\times/\{-1,1\}$ distinct classes, there exist distinct indices $i,j,\ell$ such that $a_i+a_j+a_{\ell}=0$.
\end{corollary}
\begin{proof}
     Let $A = \{ a_1,-a_1,\dots, a_n,-a_n\} \subseteq \F_p^\times$. So $|A| = 2n$. Since $p>5$, we have $\frac{p+8}{6}\leq \frac{p-1}{2}$. 
Thus the assumption $n\geq \frac{p+8}{6}$ is compatible with the existence of $n$ distinct classes in $\F_p^\times/\{\pm1\}$.

Moreover, since $
n \geq \frac{p+8}{6},$ it follows that
   \[|A| = 2n \geq \frac{2(p+8)}{6} = \frac{p+8}{3},\]
   and hence $3|A|-8 \geq p.$ 
   Now apply the Dias da Silva--Hamidoune \autoref{DH} on $A$, we obtain
   \[
   |S_A| = |\{a + b + c \mid a,b,c \in A \text{ are distinct}\}|\geq p. \]
   Thus $S_A = \F_p $. In particular, $0 \in S_A$, so there exist distinct elements $ b_i, b_j, b_\ell \in A$ such that  \[
   b_i + b_j + b_\ell = 0.\]
   By definition of $A$, each element $ b_m$ is of the form $ \pm a_k$ for some index $ k $. Since $b_i,b_j,b_{\ell}$ are distinct elements in $A$, two of them, say $b_i$ and $b_j$ can correspond to the same class in $\F_p^\times/\{-1,1\}$ only if they differ by sign, i.e. if $ b_i = -b_j $. However, if $b_i=-b_j$, then \[b_i+b_j+b_{\ell}=b_{\ell},\]which cannot be zero because $0\not\in A\subseteq\F_p^\times$. Thus no such cancellation can occur, and the elements $b_i,b_j,b_{\ell}$ must correspond to three distinct classes $ a_i, a_j, a_\ell \in\F_p^\times\setminus\{-1,1\}$ satisfying
   \[a_i+a_j+a_\ell =0,\]as required.
\end{proof}
  \begin{proposition}[Step Two]\label{BA connected'}
    Let $n\geq 3$ and $p\leq 6n-8$. Then \[\redhom_{n-1}\left(\Pb_n^\pm(\F_p)\backslash\SBDA(\F_p^n);\Q\right)\cong 0.\]
\end{proposition}
  \begin{proof} Let $n\geq 3$. We recall that we have already proven the statement for $p=2$ in \autoref{BDA,p=2}, and $p=3$ in \autoref{BDA:p=3}. We now treat the cases $p=5$ and $p>5$ separately, starting with $p=5$.
  \begin{itemize}
      \item If $p=5$, we can apply \autoref{BD,p=3,5} since $n\geq 3$ implies $5\leq 2n-1$. Thus $\redhom_{n-1}(\Pb_n^\pm(\F_5)\backslash\SBD(\F_5^n);\Q)\cong 0$. Additionally, we have just shown in \eqref{rel=0} that $\homology_{n-1}(\rel_n(5);\Q)\cong 0.$ Therefore, \[\redhom_{n-1}(\Pb_n^\pm(\F_5)\backslash\SBDA(\F_5^n);\Q)\cong 0.\]
      \item If $p>5$, we will show that the map \[\partial\colon \homology_n(\rel_n(p);\Q)\longrightarrow \redhom_{n-1}(\Pb_n^\pm(\F_p)\backslash \SBD(\F_p^n);\Q)\]is surjective for $5< p\leq 6n-8$. Since $\homology_{n-1}(\rel_n(p);\Q)\cong 0$ for all primes $p$ (by \eqref{rel=0}), it then follows that $ \redhom_{n-1}\left(\Pb_n^\pm(\F_p)\backslash\SBDA(\F_p^n):\Q\right)\cong 0$ for $5< p\leq 6n-8$. 
      
     By \autoref{hom:B}, $\redhom_{n-1}(\Pb_n^\pm(\F_p)\backslash \SBD(\F_p^n);\Q)$ is generated by elements \[\D_2^{n-1,n}(\underline{a})\otimes 1\subset \Q\left[ \Pb_n^\pm(\F_p)\backslash\MDA_k^{\mathrm{pr},1}(\F_p^n)\right]\otimes_{\Q[G_k]}\Q^\mathrm{sgn}\] for $\underline{a}=(a_0,\dots,a_{n-1})\in\F_p^n\setminus\{\underline{0}\}$ such that \begin{equation}\begin{aligned}\label{nonzero}&a_i\neq 0~\text{for all } i,\quad a_i\neq \pm a_j~\text{for all $i\neq j$},\\ &\underline{a}\text{ does not satisfy the \hyperref[def1]{$(\lambda,m)$-condition}}.\end{aligned}\end{equation} Let $\underline{a}\in\F_p^n\setminus\{\underline{0}\}$ satisfying \eqref{nonzero}, such that $a_i-a_j=\pm a_\ell$ for some distinct indices $i,j,\ell$. Without loss of generality, we assume $i=1$ and $j=2$. If necessary, we first replace $a_0$ or $a_1$ by its negative. Namely, if
$a_0=2a_1$, then we may replace $a_0$ by $-a_0=-2a_1$ since
\[\D_2^{n-1,n}(a_0,a_1,\dots,a_{n-1})\otimes 1
=\D_2^{n-1,n}(-a_0,a_1,\dots,a_{n-1})\otimes 1.\]
Similarly, if $a_1=2a_0$, then we replace $a_1$ by
$-a_1=-2a_0$. After making this replacement if necessary, we define $\underline{b}=(b_0,\dots,b_n)\in \F_p^{n+1}\setminus\{\underline{0}\}$ such that \[b_1=3^{-1}(2a_0-a_1),\]\[b_2=3^{-1}(2a_1-a_0),\]\[b_0=-b_1-b_2=3^{-1}(-a_0-a_1),\]\[b_i=a_{i-1}~\text{for all $i\geq 3$}.\]
  It implies by \eqref{nonzero} that $b_i\neq 0$ for all $i\geq 3$, and $b_i\neq \pm b_j$ for all $i,j\leq 2$ and $i,j\geq 3$. Thus, by \autoref{BA connected}, $\DA_2^{n,n}(\underline{b})\otimes 1\in \homology_n(\rel_n(p);\Q).$
  
  Now since $a_0-a_1=\pm a_\ell$ for some $\ell\geq 2$, \[\D_2^{n-1,n}(-a_0,a_1-a_0,\dots,a_{n-1})\otimes 1=\D_2^{n-1,n}(-a_1,a_0-a_1,\dots,a_{n-1})\otimes 1=0\]by \eqref{rvv} and \autoref{cor:st,pr}. We thus obtain by \eqref{e11}\begin{align*}\partial\left(\DA_2^{n,n}(\underline{b})\otimes 1\right)=&\D_2^{n-1,n}(a_0,\dots,a_{n-1})\otimes 1\\&-\D_2^{n-1,n}(-a_0,a_1-a_0,\dots,a_{n-1})\otimes 1\\&+\D_2^{n-1,n}(-a_1,a_0-a_1,\dots,a_{n-1})\otimes 1\\=&\D_2^{n-1,n}(a_0,\dots,a_{n-1})\otimes 1\end{align*}
  Consequently \[\D_2^{n-1,n}(a_0,\dots,a_{n-1})\otimes 1\in \Image \partial.\] 
  We deduce that if for every tuple $(a_0,\dots,a_{n-1})\in (\F_p/\{-1,1\})^n$ satisfying \eqref{nonzero}, there exist distinct indices $i,j,\ell$ such that \[a_i+a_j+a_{\ell}=0,\] then $\partial$ is surjective. 
Since $n\geq \frac{p+8}{6}$, it follows by \autoref{bound} that $\partial$ is surjective. Therefore, \[\homology_n(\Pb_n^\pm(\F_p)\backslash \SBDA(\F_p^n);\Q)\cong 0\quad\text{for $p\leq 6n-8$}.\qedhere\] \end{itemize} \end{proof}

\begin{proposition}[Step Three]\label{BA ctd}
    Let $n= 3$ and $p=13$. Then \[\redhom_{n-1}\left(\Pb_n^\pm(\F_p)\backslash\SBDA(\F_p^n);\Q\right)\cong 0.\]
\end{proposition}
\begin{proof} Set \[\rel_3(13)=\left(\Pb_3^\pm(\F_{13})\backslash \SBDA(\F_{13}^3),\Pb_3^\pm(\F_{13})\backslash \SBD(\F_{13}^3)\right).\]
Similarly to the proof of \autoref{BA connected'}, we will show that the map \[\partial:\homology_3(\rel_3(13);\Q)\longrightarrow \redhom_{2}(\Pb_3^\pm(\F_{13})\backslash \SBD(\F_{13}^3);\Q)\] is surjective, though we will use a different approach. By \autoref{hom:B}, \[\redhom_{2}( \Pb_3^\pm(\F_p)\backslash \SBD(\F_{13}^3);\Q) \cong\Q\left[\Pb_3^\pm(\F_{13})\backslash\MD(\F_{13}^3)_{2}^{\mathrm{pr},1}\right]\otimes_{\Q[T_2]}\Q^\mathrm{sgn}.\]
By property (b) of \autoref{BD/Pn}, for every generator of $\redhom_2(\Pb_3^\pm(\F_{13})\backslash \SBD(\F_{13}^3);\Q)$, we have \[\D_2^{2,3}(a_0,a_1,a_2)\otimes 1=\D_2^{2,3}(1,a_1a_0^{-1},a_2a_0^{-1})\otimes 1.\]Moreover, by \autoref{coinv} \begin{equation}\label{tensor}\D_2^{2,3}(a_0,a_1,a_2)\otimes 1=\sign(X)\D_2^{2,3}(\varepsilon_0 a_{\pi(0)},\varepsilon_1a_{\pi(1)},\varepsilon_2a_{\pi(2)})\otimes1,\end{equation}for all $X=(\pi,\varepsilon_0,\varepsilon_1,\varepsilon_2)\in T_2$. Thus we may assume that $a_0=1$ and $a_1,a_2$ lie in $\{1,\dots,6\}\subset\F_{13}$. We then obtain the following spanning set for $\redhom_2(\Pb_3^\pm(\F_{13})\backslash\SBD(\F_{13}^3);\Q)$: \[
\begin{array}{llll}
\D_2^{2,3}(1,2,3)\otimes 1 & \D_2^{2,3}(1,2,4)\otimes 1 & \D_2^{2,3}(1,2,5)\otimes 1 & \D_2^{2,3}(1,2,6)\otimes 1\vspace{0.5em}\\
& \D_2^{2,3}(1,3,4)\otimes 1 & \D_2^{2,3}(1,3,5)\otimes 1 & \D_2^{2,3}(1,3,6)\otimes 1\vspace{0.5em}\\
& & \D_2^{2,3}(1,4,5)\otimes 1 & \D_2^{2,3}(1,4,6)\otimes 1\vspace{0.5em}\\
& & & \D_2^{2,3}(1,5,6)\otimes 1
\end{array}
\]
We now apply \eqref{tensor} together with property (b) of \autoref{BD/Pn} to identify further generators in the homology group.
\begin{itemize} 
\item $\D_2^{2,3}(1,2,6)\otimes 1=\D_2^{2,3}(2,4,12)\otimes 1=\D_2^{2,3}(2,4,-1)\otimes 1=\D_2^{2,3}(1,2,4)\otimes 1$
\item $\D_2^{2,3}(1,3,5)\otimes 1=\D_2^{2,3}(5,15,25)\otimes 1=\D_2^{2,3}(5,2,-1)\otimes 1=-\D_2^{2,3}(1,2,5)\otimes 1$
\item $\D_2^{2,3}(1,3,6)\otimes 1=\D_2^{2,3}(4,12,24)\otimes 1=\D_2^{2,3}(4,-1,-2)\otimes 1=\D_2^{2,3}(1,2,4)\otimes 1$
\item $\D_2^{2,3}(1,4,5)\otimes 1=\D_2^{2,3}(3,12,15)\otimes 1=\D_2^{2,3}(3,-1,2)\otimes 1=\D_2^{2,3}(1,2,3)\otimes 1$
\item $\D_2^{2,3}(1,4,6)\otimes 1=\D_2^{2,3}(2,8,12)\otimes 1=\D_2^{2,3}(2,-5,-1)\otimes 1=\D_2^{2,3}(1,2,5)\otimes 1$
\item $\D_2^{2,3}(1,5,6)\otimes 1=\D_2^{2,3}(2,10,12)\otimes 1=\D_2^{2,3}(2,-3,-1)\otimes 1=\D_2^{2,3}(1,2,3)\otimes 1$
\end{itemize}
It follows that $\redhom_{2}( \Pb_3^\pm(\F_p)\backslash \SBD(\F_{13}^3);\Q)$ is generated by \[\{\D_2^{2,3}(1,2,3)\otimes 1, \D_2^{2,3}(1,2,4)\otimes 1 ,\D_2^{2,3}(1,2,5)\otimes 1, \D_2^{2,3}(1,3,4)\otimes 1\}.\] 
 Recall by \autoref{BA connected} that\[
\homology_3\left(\Pb_3^\pm(\F_{13})\backslash \SBDA(\F_{13}^3),\Pb_3^\pm(\F_{13})\backslash \SBD(\F_{13}^3);\Q\right)
\cong \Q\left[ \Pb_3^\pm(\F_{13})\backslash\MDA_3^{\mathrm{pr},1}(\F_{13}^3)\right]\otimes_{\Q[G_3]}\Q^\mathrm{sgn}.\]
Moreover, following \eqref{e11}, the boundary map $\partial$ is given on generators by 
\begin{align*}\partial\left(\DA_2^{3,3}(\underline{b})\otimes 1\right)=&\D_2^{2,3}(2b_1+b_2,2b_2+b_1,b_3)\otimes 1\\-&\D_2^{2,3}(2b_0+b_2,2b_2+b_0,b_3)\otimes 1\\+&\D_2^{2,3}(2b_0+b_1,2b_1+b_0,b_3)\otimes 1.\end{align*}
Additionally, we have from \eqref{rvv} and \autoref{cor:add,pr} that $\D_2^{2,3}(a_0,a_1,a_2)\otimes 1=0$ if $a_i=\pm a_j$ for some $i\neq j$. Using again property (b) of \autoref{BD/Pn} and \eqref{tensor} on each term on the right hand side, we compute the following.
\begin{itemize}[leftmargin=*]
\item $\begin{aligned}[t]\partial\left(\DA_2^{3,3}(4,3,6,3)\otimes 1\right)
    &= \D_2^{2,3}(-1,2,3)\otimes 1
       - \D_2^{2,3}(1,3,3)\otimes 1
       + \D_2^{2,3}(-2,-3,3) \otimes 1\\
    &= \D_2^{2,3}(1,2,3)\otimes 1-0+0\\&=\D_2^{2,3}(1,2,3)\otimes 1\end{aligned}$
  \item $\begin{aligned}[t]
  \partial\left(\DA_2^{3,3}(4,3,6,5)\otimes 1\right)
    &= \D_2^{2,3}(-1,2,5)\otimes 1
       \quad- \underbrace{\D_2^{2,3}(1,3,5)\otimes 1}_{-\D_2^{2,3}(1,2,5)\otimes 1}
       + \underbrace{\D_2^{2,3}(-2,-3,5)\otimes 1}_{\D_2^{2,3}(10,15,-25)\otimes 1=\D_2^{2,3}(-3,2,1)\otimes 1} \\
       &=\D_2^{2,3}(1,2,5)\otimes 1+\D_2^{2,3}(1,2,5)\otimes 1-\D_2^{2,3}(1,2,3)\otimes 1\\
    &= 2\D_2^{2,3}(1,2,5)\otimes 1
       - \D_2^{2,3}(1,2,3)\otimes 1.
  \end{aligned}
  $
  \item $\begin{aligned}[t]
  \partial\left(\DA_2^{3,3}(-5,7,-2,4)\otimes 1\right)
    &= \D_2^{2,3}(-1,3,4)\otimes 1
       - \D_2^{2,3}(1,4,4)\otimes 1
       + \D_2^{2,3}(-3,-4,4)\otimes 1\\&= \D_2^{2,3}(1,3,4)\otimes 1-0+0\\
    &= \D_2^{2,3}(1,3,4)\otimes 1\end{aligned}$
   \item $\begin{aligned}[t]
  \partial\left(\DA_2^{3,3}(4,3,6,4)\otimes 1\right)
    &= \D_2^{2,3}(-1,2,4)\otimes 1
       - \D_2^{2,3}(1,3,4)\otimes 1
       \quad+ \underbrace{\D_2^{2,3}(-2,-3,4)\textbf{}}_{\D_2^{2,3}(12,18,-24)\otimes 1=\D_2^{2,3}(-1,5,2)\otimes 1} \\
    &= \D_2^{2,3}(1,2,4)\otimes 1
       - \D_2^{2,3}(1,3,4)\otimes 1
       - \D_2^{2,3}(1,2,5)\otimes 1.
  \end{aligned}
  $
\end{itemize}
These computations show that each generator of $\redhom_{2}( \Pb_3^\pm(\F_p)\backslash \SBD(\F_{13}^3);\Q)$ lies in the image of $\partial$. Therefore, $\partial$ is surjective for $p=13$.\end{proof}
\section{Twisted and untwisted actions}\label{sec7}
Let $p$ be a prime, and let $Y=\SBD(\F_p^n)$ or $\SBDA(\F_p^n)$.  In this section, we study the chains \[\redchain_*\left(\Pb_n^\pm(\F_p)\backslash Y;\Q\right)\]as well as their homology groups.

\subsection{Simplices with twisted and untwisted actions}\label{subsec:tw}
Let $p$ be a prime and $n\geq 2$. Let $Y=\SBD(\F_p^n)$ or $\SBDA(\F_p^n)$.
We classify the simplices in $Y$ based on whether they are twisted or untwisted under the action of $\Pb_n^\pm(\F_p)$, as defined in \autoref{def:tw} with respect to $\chi=\det\colon \Pb_n^\pm(\F_p)\rightarrow\Q^\times$.

We recall that an element $\sigma \in \SBD(\F_p^n)_k$ is called a standard $k$-simplex, while an element $\sigma \in \SBDA(\F_p^n)_k -   \SBD(\F_p^n)_k$ is called an additive $k$-simplex.

\paragraph{Standard simplices}

Recall that \autoref{matrixrep} says that \[\redchain_k\left(\Pb_n^\pm(\F_p)\backslash\SBD(\F_p^n);\Q\right)\cong \Q\left[\Pb_n^\pm(\F_p)\backslash\SBD(\F_p^n)\right]\otimes_{\Q[\Sigma_{k+1}]}\Q^\mathrm{sgn}\cong \MD(\F_p^n)_k/T_k.\] Thus, we may identify the coset $\sigma \Sigma_{k+1}$ with the coset $B\cdot T_k $ for some $B\in \MD(\F_p^n)_k$.
\begin{lemma}\label{ksmall}
    Let $p$ be a prime, $n\geq 2$ and $1\leq k < n-1$. If $\sigma\in \SBD(\F_p^n)_k$, then $\sigma$ is twisted.
\end{lemma}
\begin{proof}
    Write $\sigma\Sigma_{k+1}=B\cdot T_k$ for some $B=(v_0|\dots|v_k)\in\MD(\F_p^n)_k$. We want to show that the action of $\Pb_n^\pm(\F_p)$ on $\sigma$ is twisted. We have by \autoref{lem:goodmat} that \[\SBD(\F_p^n)_k^\mathrm{tw}/\Sigma_{k+1}\cong \MD(\F_p^n)_k^\mathrm{tw}/T_k.\]
Thus, it is enough to show that the action of $\Pb_n^\pm(\F_p)$ on $B$ is twisted.
    Extend $\{v_0,\dots,v_k\}$ to a determinant-$1$ basis of $\F_p^n$
\[
  \{v_0,\dots,v_k,\dots,v_{n-1}\},
  \quad\text{with}~
  v_{n-1}=\lambda e_1,
\]for some $\lambda\in\F_p^\times$.
Now set
\[
  X = (\Id,1,\dots,1)\in T_k.
\]
Next define $A\in \Pb_n^\pm(\F_p)$ by its action on the basis:
\begin{align*}
    v_{n-1}&\longmapsto -v_{n-1}\\
    v_i&\longmapsto v_i~\text{otherwise}.
\end{align*}Note that $A\cdot e_1=-e_1$.
Since $k<n-1$, then \[A\cdot B=B\cdot X.\]Moreover, \[\det(A)\sign(X)=-1.\qedhere\]
\end{proof}
\begin{corollary}\label{D1,tw}
Let $p$ be a prime, $n\geq 2$ and $1\leq k<n-1$. Let $\sigma\in\SBD(\F_p^n)_k$ and let $B\in\D_1^{k,n}$ such that $\sigma\Sigma_{k+!}B\cdot T_k$. Then $\sigma$ is twisted.
\end{corollary}
\begin{proof}
Since $\D_1^{k,n}$ only exists for $k<n-1$, the result follows from \autoref{ksmall}.
\end{proof}
    We now find all twisted simplices in $\SBD(\F_p^n)_k$ for $k=n-1.$
\begin{lemma}\label{twisted:BD}
Let $p$ be a prime and $n\geq 2$. Let $\sigma\in \SBD(\F_p^n)_{n-1}$ and let $B\in \D_2^{n-1,n}(\underline{a})$ for some $\underline{a}=(a_0,\dots,a_k)\in \F_p^{k+1}\setminus\{\underline{0}\}$ such that $\sigma \Sigma_{n}=B\cdot T_{n-1}$. Then $\sigma$ is twisted if and only if there exist some $\lambda\in \F_p^\times$ and $X=(\pi,\varepsilon_0,\dots,\varepsilon_k)\in T_k$ such that\[\lambda \underline{a}=\underline{a}\cdot X,\]and \[\prod_{i=0}^k\varepsilon_i=-1. \] 
\end{lemma}
\begin{proof}
 Recall that \[\D_2^{n-1,n}(\underline{a})=\left\{B=\left(v_0|\dots|v_{n-1}\right)\in \MD(\F_p^n)_{n-1} ~\middle\vert \lambda e_1=a_0v_0+\dots+a_{n-1}v_{n-1}~\text{for some $\lambda\in\F_p^\times$}\right\}.\] 
Let $\sigma \Sigma_{n}=B\cdot T_{n-1}$ with $B\in \D_2^{n-1,n}(\underline{a})$ for some $\underline{a}=(a_0,\dots,a_{n-1})\in \F_p^{n}\setminus\{\underline{0}\}$. Since  \[\SBD(\F_p^n)_{n-1}^\mathrm{tw}/\Sigma_{n}\cong \MD(\F_p^n)_{n-1}^\mathrm{tw}/T_{n-1}\] by \autoref{lem:goodmat}, we have the following equivalences:
\begin{align*}\Pb_n^\pm(\F_p)~\text{twists the}&~\text{simplex $\sigma$}\\&\Longleftrightarrow \text{there exist $A\in \Pb_n^\pm(\F_p)$ and $X\in T_{n-1}$ such that $\det(A)\sign(X)=-1$ and $B\cdot X=A\cdot B$}\\&\Longleftrightarrow \text{there exist $X=(\pi,\varepsilon_0,\dots,\varepsilon_{n-1})\in T_{n-1}$ such that $B\cdot X\in \D_2^{n-1,n}(\underline{a})$, and $\prod\limits_{i=0}^{n-1}\varepsilon_i=-1$.}\\ &\Longleftrightarrow \text{there exist $X=(\pi,\varepsilon_0,\dots,\varepsilon_{n-1})\in T_{n-1}$ such that $\D_2^{n-1,n}(\underline{a}\cdot X)=\D_2^{n-1,n}(\underline{a})$, and $\prod\limits_{i=0}^{n-1}\varepsilon_i=-1$},\end{align*} where the second equivalence follows from the fact that $B\cdot X=A\cdot B$ implies $\prod\limits_{i=0}^{n-1}\varepsilon_i=\det (A)\sign(X)$, and the last equivalence follows from properties (c) and (d) of \autoref{BD/Pn}. 

Moreover, we have by property (b) of \autoref{BD/Pn} that \[\D_2^{n-1,n}(\underline{a})=\D_2^{n-1,n}(\underline{b})\quad\Longleftrightarrow\quad \underline{a}=\lambda\underline{b}~\text{for some $\lambda\in \F_p^\times$}.\]
This implies that $\sigma$ is twisted if and only if there exist $\lambda\in \F_p^
    \times$ and $X=(\pi,\varepsilon_0,\dots,\varepsilon_{n-1})\in T_k$ such that \[\lambda \underline{a}=\underline{a}\cdot X,\]and \[\prod\limits_{i=0}^{n-1}\varepsilon_i=-1.\qedhere\]
\end{proof}
We show in the following lemma that if $n$ is odd, then all simplices $\sigma\in\SBD(\F_p^n)_{n-1}$ are twisted.
\begin{lemma}\label{D2:tw,odd}
     Let $p$ be a prime and $n\geq 3$ is odd. Let $\sigma\in \SBD(\F_p^n)_{n-1}$ and let $B\in \D_2^{n-1,n}(\underline{a})$ for some $\underline{a}=(a_0,\dots,a_{n-1})\in \F_p^{n}\setminus\{\underline{0}\}$ such that $\sigma \Sigma_{n}=B\cdot T_{n-1}$. Then $\sigma$ is twisted. \end{lemma}
\begin{proof}
    Let $\sigma \Sigma_{n}=B\cdot T_{n-1}$ with $B\in \D_2^{n-1,n}(\underline{a})$ for some $\underline{a}=(a_0,\dots,a_{n-1})\in \F_p^{n}\setminus\{\underline{0}\}$. We have by \autoref{twisted:BD} that $\sigma$ is twisted if there exists $\lambda\in \F_p^\times$ and $X=(\pi,\varepsilon_0,\dots,\varepsilon_{n-1})\in T_{n-1}$ such that \[\lambda \underline{a}=\underline{a}\cdot X,\]and \[\prod\limits_{i=0}^{n-1}\varepsilon_i=-1.\]Take $X=(\Id,-1,\dots,-1)\in T_{n-1}$ and $\lambda=-1$. Then \[\lambda \underline{a}=\underline{a}\cdot X,\]and \[\prod\limits_{i=0}^{n-1}\varepsilon_i=(-1)^n=-1.\]
        This completes the proof.\end{proof}
We now turn to the case where $n$ is even, where additional phenomena can occur.
\begin{lemma}\label{D2:even,tw}
    Let $p$ be a prime, $n\geq 2$ even. Let $\sigma\in \SBD(\F_p^n)_{n-1}$ and let $B\in \D_2^{n-1,n}(\underline{a})$ for some $\underline{a}=(a_0,\dots,a_{n-1})\in \F_p^{n}\setminus\{\underline{0}\}$ such that $\sigma \Sigma_{n}=B\cdot T_{n-1}$. Then $\sigma$ is twisted if and only if either $a_{i_0}=0$ for some $0\leq i_0\leq n-1$ or $\underline{a}$ satisfies the \hyperref[def1]{$(\lambda,m)$-condition}\footnote{Only  Case 1 of the \hyperref[def1]{$(\lambda,m)$-condition} applies with $k=n-1$ odd.}
\end{lemma}

\begin{proof} Let $\sigma \Sigma_{n}=B\cdot T_{n-1}$ with $B\in\D_2^{n-1,n}(\underline{a})$ for $\underline{a}=(a_0,\dots,a_{n-1})\in\F_p^n\setminus\{\underline{0}\}$.
    We will first prove the ``if” direction by showing that, in each case, there exist $\lambda\in \F_p^\times$ and $X=(\pi,\varepsilon_0,\dots,\varepsilon_{n-1})\in T_{n-1}$ such that \[\lambda\underline{a}=\underline{a}\cdot X,\] and \[\prod_{i=0}^{n-1}\varepsilon_i=-1.\] 
    It will then follow by \autoref{twisted:BD} that $\sigma$ is twisted. \begin{mycases}
        \case If $a_{i_0}=0$ for some $0\leq i_0\leq n-1$, then take $X=(\Id,\varepsilon_0,\dots,\varepsilon_{n-1})$ with \[\varepsilon_i=\begin{cases}-1&\text{if $i=i_0$}\\             1&\text{otherwise}  
        \end{cases},\]and set $\lambda=1$. Thus,  \[\lambda \underline{a}=\underline{a}\cdot X,\]and \[\prod_{i=0}^{n-1}\varepsilon_i=\varepsilon_{i_0}=-1.\]
        \case Let $\underline{b}=(b_1, \lambda b_1, \dots, \lambda^{m-1} b_1,\dots,b_q, \lambda b_q, \dots, \lambda^{m-1} b_q)$. Let $\pi$ be the following odd permutation \[\pi=(0\;1\; \dots\; m-1)(m\;\dots\; 2m-1)\dots(qm-m\;\dots\;qm-1)\quad\text{in $\Sigma_{n}$},\]and \[\varepsilon_i=\begin{cases}
            -1 &\text{if $i\in\{m-1,2m-1,\dots,qm-1\}$}\\
            1&\text{otherwise}.
        \end{cases}.\]
        As $q$ is odd, taking the product over all $i=0,\dots,k$ shows\[\prod_{i=0}^{n-1}\varepsilon_i=(-1)^q=-1.\]
        Let $X=(\pi,\varepsilon_0,\dots,\varepsilon_{n-1})$. Then $\lambda \underline{b}=\underline{b}\cdot X$ as in the proof of \autoref{D2,dif,rv}.
        
        Since \[\underline{a}\cdot Y=(b_1, \lambda b_1, \dots, \lambda^{m-1} b_1,\dots,b_q, \lambda b_q, \dots, \lambda^{m-1} b_q),\] then \[\lambda \underline{a}\cdot Y=\lambda\underline{b}=\underline{b}\cdot X.\] It follows that\[\lambda\underline{a}=\underline{a}\cdot YXY^{-1}.\]\end{mycases}
        Therefore, $\sigma$ is twisted by \autoref{twisted:BD}.
        
We now proceed to prove the converse. Assume the simplex $\sigma$ is twisted. Then by \autoref{twisted:BD}, there exists $\lambda\in \F_p^\times$ and $X=(\pi,\varepsilon_0,\dots,\varepsilon_{n-1})\in T_{n-1}$ such that \[\lambda\underline{a}=\underline{a}\cdot X,\]and \[\prod_{i=0}^{n-1}\varepsilon_i=-1.\] This implies \begin{equation}\label{eq:l}\lambda a_i=\varepsilon_ia_{\pi(i)}\quad\text{for all $0\leq i\leq n-1$}.\end{equation}
We consider two different cases. \begin{itemize}
        \item If $a_i=0$ for some $i$, then we are in Case 1.
        \item If $a_i\neq 0$ for all $i$, then for each $0\leq i\leq n-1$, we let $m_i$ be the length of the cycle of $\pi$ that contains $i$. Then $\pi^{m_i}(i)=i$ for every $i$. It then implies by \eqref{eq:l} that \begin{equation}\label{eq:L}\lambda^\ell a_{i}=\varepsilon_i\varepsilon_{\pi(i)}\dots\varepsilon_{\pi^{\ell-1}(i)}a_{\pi^\ell(i)}\quad\text{for every $1\leq \ell\leq m_i$}.\end{equation}
        Thus, for every $0\leq i\leq n-1$, the tuple $(a_{i},a_{\pi(i)},\dots,a_{\pi^{m_i-1}(i)})$ is equal to\[(a_{i},\varepsilon_{i}\lambda a_{i},\dots,\varepsilon_{i}\varepsilon_{\pi(i)}\varepsilon_{\pi^2(i)}\dots\varepsilon_{\pi^{m_i-1}(i)}\lambda^{m_i-1}a_{i}).\]  
    Let $m=\gcd\limits_i(m_i)$. Thus each cycle of length $m_i$ can be partitioned into $\frac{m_i}{m}$ consecutive blocks of length $m$. Write $\pi$ as a product of disjoint cycles \[(r_1\; r_2 \; \dots \;r_{m_1})(r_{m_1+1}\; \dots\; r_{m_1+m_2})\dots .\] Let $\tau\in\Sigma_{n+1}$ be the permutation defined by $\tau(j)=r_j$ for all $j$. Consequently, for $Z=(\tau,\varepsilon'_3,\dots,\varepsilon'_n)\in T_{n-1}$ for some $\varepsilon'_i=\{-1,1\}$, the tuple $\underline{a}\cdot Z$, is equal to a tuple of the form \[(b_1,\lambda b_1,\dots,\lambda^{m-1}b_1,\dots,b_q,\lambda b_q,\dots,\lambda^{m-1}b_q)\]for some $b_j\in \F_p^\times$ with $q=\frac{n}{m}$. 
    Moreover, applying \eqref{eq:L} to $\ell=m_i$, we obtain \[\lambda^{m_i}a_{i}=\varepsilon_i\varepsilon_{\pi(i)}\dots\varepsilon_{\pi^{m_i-1}(i)}a_{\pi^{m_i}(i)}=\varepsilon_i\varepsilon_{\pi(i)}\dots\varepsilon_{\pi^{m_i-1}(i)}a_i.\]Since $a_i\neq 0$ for all $i$, it implies that\[\lambda^{m_i}=\varepsilon_i\varepsilon_{\pi(i)}\dots\varepsilon_{\pi^{m_i-1}(i)}=\pm 1.\]
    As $m=\gcd\limits_i(m_i)$, so $m=\sum_i\mu_im_i$ for some $\mu_i\in\Z$. Thus, \[\lambda^m=\prod_i\left(\lambda^{m_i}\right)^{\mu_i}=\prod_i(\pm 1)^{\mu_i}=\pm 1.\]
    On the other hand, multiplying $\lambda a_i=\varepsilon_ia_{\pi(i)}$ over all $i$ gives \[\lambda^n=\prod_{i=0}^{n-1}\varepsilon_i.\] But $n=qm$, so \[\prod_{i=0}^{n-1}\varepsilon_i=\lambda^{n}=(\lambda^m)^q=(\pm 1)^q.\] Since  $\prod\limits_{i=0}^{n}\varepsilon_i=-1$, we conclude that $(\pm 1)^q=-1$. Therefore, \[\lambda^m=-1\quad\text{and}\quad q~\text{is odd},\] 
   Finally, because $n=qm$ is even and $q$ is odd, $m$ must be even, which is exactly the condition in Case $2$.\qedhere\end{itemize}\end{proof}

We summarize the above lemmas in the following corollary.
\begin{corollary}\label{cor:st,utw}
Let $p$ be a prime, $n\geq 2$ and $2\leq k\leq n-1$. Let $\sigma\in\SBD(\F_p^n)_k$ and let $B\in\MD(\F_p^n)_k$ such that $\sigma\Sigma_{k+1}=B\cdot T_k$. Then $\sigma$ belongs to $\SBD(\F_p^n)^\mathrm{utw}$ if and only if $k=n-1$ is odd and $B$ belongs to some set \[\D_2^{n-1,n}(\underline{a})\quad\text{with $\underline{a}=(a_0,\dots,a_k)\in\left(\F_p^\times\right)^{n-1}$},\] such that $\underline{a}$ does not satisfy the \hyperref[def1]{$(\lambda,m)$-condition}.
\end{corollary}

\paragraph{Additive simplices}
Recall that \autoref{matrixrep} says that for all $k\leq n$,\[ \left(\SBDA(\F_p^n)_k - \SBD(\F_p^n)_k\right)/\Sigma_{k+1}\cong \MDA(\F_p^n)_k/G_k.\] Thus, we may identify the coset $\sigma \Sigma_{k+1}$ with the coset $B\cdot G_k $ for some $B\in \MDA(\F_p^n)_k$. We also recall that a simplex $\sigma\in \SBDA(\F_p^n)_k - \SBD(\F_p^n)_k$, is called an additive $k$-simplex.
\begin{lemma}\label{ksmalll}
    Let $p$ be a prime, $n\geq 3$ and $2\leq k < n$. If $\sigma\in \SBDA(\F_p^n)_k -  \SBD(\F_p^n)_k$, then $\sigma$ is twisted.
\end{lemma}
\begin{proof}
    Write $\sigma\Sigma_{k+1}=B\cdot G_k$ for some $B=(v_0|\dots|v_k)\in\MDA(\F_p^n)_k$.
 We want to show that $\sigma$ is twisted.
We have by \autoref{lem:goodmat} that \[\left(\SBDA(\F_p^n)_k ^\mathrm{tw}- \SBD(\F_p^n)_k^\mathrm{tw}\right)/\Sigma_{k+1}\cong \MDA(\F_p^n)_k^\mathrm{tw}/G_k.\] 
It is enough to show that $B$ is twisted.
  Extend $\{v_1,\dots,v_k\}$ to a determinant-$1$ basis of $\F_p^n$
\[
  \{v_1,\dots,v_n\},
  \quad
  v_{n}=\lambda e_1,
\]for some $\lambda\in\F_p^\times$.
Now set
\[
  X=(\Id,1,\Id,1\dots,1)\in G_k.
\]
Next define $A\in \Pb_n^\pm(\F_p)$ by its action on the basis:
\begin{align*}
    v_{n}&\longmapsto -v_{n}\\
    v_i&\longmapsto v_i~\text{otherwise}.
\end{align*}Note that $A\cdot e_1=-e_1$.
Since $k<n$, then \[A\cdot B=B\cdot X.\]Furthermore,\[\det(A)\sign(X)=-1.\] \
This completes the proof.
\end{proof}
\begin{corollary}\label{DA1,tw}
      Since $\DA_1^{k,n}$ only exists for $k<n$, then every $\sigma\in \SBDA(\F_p^n)_k$ with $\sigma \Sigma_{k+1}=B\cdot G_k$ such that $B\in \DA_1^{k,n}$ is twisted.
\end{corollary}
\begin{lemma}\label{product}
    Let $p$ be a prime and $n\geq 2$. Let $B\in\MDA(\F_p^n)_n$, $A\in\Pb_n^\pm(\F_p)$ and $X=(\pi,\varepsilon,\tau,\varepsilon_3,\dots,\varepsilon_n)\in G_n$ such that $B\cdot X=A\cdot B$. Then \[\det(A)=\sign(X)\prod_{i=3}^n\varepsilon_i.\]
\end{lemma}
\begin{proof}
    Suppose $B=(v_0|\dots|v_n)\in \MDA(\F_p^n)_n$, then the vectors $v_1,\dots,v_n$ form a basis of $\F_p^n$ and $v_0+v_1+v_2=0$. Let
\[C=(v_1|\dots|v_n)\in \GL_n(\F_p).\]
From the relation $v_0+v_1+v_2=0$, the subspace
\[W=\operatorname{span} \{v_1,v_2\}\]
is preserved by the action of $\Sigma_{\{0,1,2\}}$ on $\{v_0,v_1,v_2\}$. For $\pi\in \Sigma_{\{0,1,2\}}$, let
\[T_\pi\colon W\to W\]
be the linear map determined by
\[T_\pi(v_i)=v_{\pi(i)} \quad (i=0,1,2).\]
\begin{claim}
    $\det(T_\pi)=\sign(\pi)$.
\end{claim}
\begin{proof}
    It suffices to check the claim on generators of $\Sigma_3$, for example the transpositions $(0\;1)$ and $(1\;2)$.

For $\pi=(1\;2)$, we have
\[T_{(1\;2)}(v_1)=v_2,\quad T_{(1\;2)}(v_2)=v_1.
\]So with respect to the basis $(v_1,v_2)$,
\[T_{(1\;2)}=
\begin{pmatrix}
0&1\\
1&0
\end{pmatrix},\]
and hence
\[\det(T_{(1\;2)})=-1=\sign((1\;2)).\]

For $\pi=(0\;1)$, we have
\[T_{(0\;1)}(v_1)=v_0=-v_1-v_2,\quad T_{(0\;1)}(v_2)=v_2.\]
Therefore\[T_{(0\;1)}=
\begin{pmatrix}
-1&0\\
-1&1
\end{pmatrix},\]
so
\[\det(T_{(0\;1)})=-1=\sign((0\;1)).\]

Since $(0\;1)$ and $(1\;2)$ generate $S_3$, every $\pi\in \Sigma_3$ can be written as a product of these transpositions. Thus
\[
\det(T_\pi)=\sign(\pi)\]
for all $\pi\in \Sigma_3$.
\end{proof}
$X$ acts on $B$ in the following way: on $W=\operatorname{span} \{v_1,v_2\}$, this action is $\varepsilon T_\pi$, while on $\operatorname{span} \{v_3,\dots,v_n\}$ it is given by sending $v_i$ to $\varepsilon_iv_{\tau(i)}$. Since $B\cdot X=A\cdot B$, then 
\[\det(A)=\det(X)
=\det(\varepsilon T_\pi)\cdot \det(\varepsilon_3v_{\tau(3)}|\dots|\varepsilon_n v_{\tau(n)}).\]Since $W$ has dimension $2$,
\[\det(\varepsilon T_\pi)=\varepsilon^2\det(T_\pi)=\det(T_\pi)=\sign(\pi),
\]and the second factor is
\[
\sign(\tau)\prod_{i=3}^n \varepsilon_i.
\]Hence
\[\det(A)=\sign(\pi)\sign(\tau)\prod_{i=3}^n \varepsilon_i.\]
By definition,
\[\sign(X)=\sign(\pi)\sign(\tau),
\]so
\[\det(A)=\sign(X)\prod_{i=3}^n \varepsilon_i.\qedhere\]
\end{proof}
\begin{lemma}\label{twisted:BDA}
    Let $p\neq 3$ be a prime and $n\geq 2$. Let $\sigma$ be an additive $n$-simplex in $\SBDA(\F_p^n)_{n}$ and let $B\in \DA_2^{n,n}(\underline{a})$ for some $\underline{a}=(a_0,\dots,a_n)\in A_n$ such that $\sigma \Sigma_{n+1}=B\cdot G_n$. Then $\sigma$ is twisted if and only if there exist some $\lambda\in\F_p^\times$ and $X=(\pi,\varepsilon,\tau,\varepsilon_3,\dots,\varepsilon_n)\in G_n$ such that\[\lambda \underline{a}=\underline{a}\cdot X,\]and \[\prod_{i=3}^n\varepsilon_i=-1. \] 
\end{lemma}
\begin{proof}
  Recall that \[A_n=\left\{\underline{a}=(a_0,\dots,a_n)\in\F_p^{n+1}\setminus\{\underline{0}\} ~\middle\vert a_0+a_1+a_2=0\right\}\] and \[\DA_2^{n,n}(\underline{a})=\left\{B=\left(v_0|\dots|v_n\right)\in \MDA(\F_p^n)_n ~\middle\vert \lambda e_1=a_0v_0+\dots+a_nv_n~\text{for some $\lambda\in\F_p^\times$}\right\}.\] 
Let $\sigma \Sigma_{n+1}=B\cdot G_n$ with $B\in \DA_2^{n,n}(\underline{a})$ for some $\underline{a}=(a_0,\dots,a_n)\in A_n$. Since  \[\left(\SBDA(\F_p^n)_n - \SBD(\F_p^n)_n\right)^\mathrm{tw}/\Sigma_{n+1}\cong \MDA(\F_p^n)_n^\mathrm{tw}/G_n\] by \autoref{lem:goodmat}, we have the following equivalences:
\begin{align*}\text{$\sigma$ is twisted}&\Longleftrightarrow \text{there exist $A\in \Pb_n^\pm(\F_p)$ and $X\in G_n$ such that $\det(A)\sign(X)=-1$ and $B\cdot X=A\cdot B$}\\&\Longleftrightarrow \text{there exist $X=(\pi,\varepsilon,\tau,\varepsilon_3,\dots,\varepsilon_n)\in G_n$ such that $B\cdot X\in \DA_2^{n,n}(\underline{a})$, and $\prod\limits_{i=3}^n\varepsilon_i=-1$.}\\ &\Longleftrightarrow \text{there exist $X=(\pi,\varepsilon,\tau,\varepsilon_3,\dots,\varepsilon_n)\in G_n$ such that $\DA_2^{n,n}(\underline{a}\cdot X)=\DA_2^{n,n}(\underline{a})$, and $\prod\limits_{i=3}^n\varepsilon_i=-1$},\end{align*} where the second equivalence follows from \autoref{product}, and the last equivalence from \autoref{BDA/Pn}.

    Moreover, we have by \autoref{BDA/Pn} that \[\DA_2^{n,n}(\underline{a})=\DA_2^{n,n}(\underline{b})\quad\Longleftrightarrow\quad \underline{a}=\lambda\underline{b}~\text{for some $\lambda\in \F_p^\times$}.\]
   It follows that $\sigma$ is twisted if and only if there exist $\lambda\in \F_p^
    \times$ and $X=(\pi,\varepsilon,\tau,\varepsilon_3,\dots,\varepsilon_n)\in G_n$ such that \[\lambda \underline{a}=\underline{a}\cdot X,\]and \[\prod\limits_{i=3}^n\varepsilon_i=-1.\qedhere\]\end{proof}

\begin{lemma}
     Let $p\neq 3$ be a prime and $n\geq 3$ odd. Let $\sigma$ be an additive $n$-simplex in $\SBDA(\F_p^n)_n$ and let $B\in \DA_2^{n,n}(\underline{a})$ for some $\underline{a}=(a_0,\dots,a_n)\in A_n$ such that $\sigma \Sigma_{n+1}=B\cdot G_n$. Then $\sigma$ is twisted.\end{lemma}
\begin{proof}
  Let $\sigma \Sigma_{n+1}=B\cdot G_n$ with $B\in \DA_2^{n,n}(\underline{a})$ for some $\underline{a}=(a_0,\dots,a_n)\in A_n$. We have by \autoref{twisted:BDA} that $\sigma$ is twisted if there exists $\lambda\in \F_p^\times$ and $X=(\pi,\varepsilon,\tau,\varepsilon_3,\dots,\varepsilon_n)\in G_n$ such that \[\lambda \underline{a}=\underline{a}\cdot X,\]and \[\prod\limits_{i=3}^n\varepsilon_i=-1.\]
       Take $X=(\Id,-1,\Id,-1,\dots,-1)\in G_k$ and $\lambda=-1$. Then \[\lambda\underline{a}=\underline{a}\cdot X,\]and \[\prod\limits_{i=3}^{n}\varepsilon_i=(-1)^{n-2}=-1.\qedhere\]
\end{proof}

\begin{lemma}\label{DA,tw,even}
     Let $p\neq 3$ be a prime, $n\geq 2$ is even. Let $\sigma$ be an additive $n$-simplex in $\SBDA(\F_p^n)_n$ and let $B\in \DA_2^{n,n}(\underline{a})$ for some $\underline{a}=(a_0,\dots,a_n)\in A_n$ such that $\sigma \Sigma_{n+1}=B\cdot G_n$. Then $\sigma$ is twisted if and only if one the following cases is satisfied. \begin{description}
        \item[Case 1.] If $a_{i_0}=0$ for some $3\leq i_0\leq n$
        \item[Case 2. (Additive $(\lambda,m)$-condition)]\label{twistedcond} If there exist $\lambda\in\F_p^\times$ and an even integer $m$ satisfying \[\lambda~\text{has order $2m$},\] \[m\mid n-2,\]\[q=\frac{n-2}{m}~\text{is an odd integer},\]and\[\underline{a}\cdot Y=(0,0,0,b_1,\lambda b_1,\dots,\lambda^{m-1}b_1,\dots,b_q,\lambda b_q,\dots,\lambda^{m-1}b_q),\]for some $Y\in G_n$ and $b_1,\dots,b_q\in \F_p^\times$. 
    \end{description}
\end{lemma}

\begin{proof}
    This proof is very similar to that of \autoref{D2:even,tw}. We recall that \autoref{twisted:BDA} says that $\sigma$ is twisted if and only if $\lambda\in \F_p^\times$ and $X=(\pi,\varepsilon,\tau,\varepsilon_3,\dots,\varepsilon_n)\in G_{n}$ such that \[\lambda\underline{a}=\underline{a}\cdot X,\]and \[\prod_{i=3}^{n}\varepsilon_i=-1.\]We start with proving the ``if" direction.\begin{mycases}
        \case If $a_{i_0}=0$ for some $i_0\geq 3$, we take $X=(\Id,1,\Id,\varepsilon_3,\dots,\varepsilon_{n})\in G_n$ where \[\varepsilon_i=\begin{cases}-1&\text{if $i=i_0$}\\             1&\text{otherwise}  
        \end{cases}.\] Set $\lambda=1$. Then,  \[\lambda\underline{a}=\underline{a}\cdot X,\]and \[\prod_{i=3}^{n}\varepsilon_i=\varepsilon_{i_0}=-1.\]
        \case Let $\underline{b}=(0,0,0,b_1, \lambda b_1, \dots, \lambda^{m-1} b_1,\dots,b_q, \lambda b_q, \dots, \lambda^{m-1} b_q)$.  Let $X=(\Id,1,\tau,\varepsilon_3,\dots,\varepsilon_{n})\in G_n$ where $\tau$ is the permutation  \[\tau=\left(3\; \dots\; m+2\right)\left(m+3\;\dots\; 2m+2\right)\dots((q-1)m+3\;\dots\; qm+2),\]and \[\varepsilon_i=\begin{cases}
            -1 &\text{if $i\in\{m+2,2m+2,\dots,qm+2$\}}\\
            1&\text{otherwise}.
        \end{cases}.\]
        As $q$ is odd, it implies that \[\prod_{i=3}^{n-1}\varepsilon_i=(-1)^q=-1.\]
        Moreover, given that $\lambda$ has order $2m$, it follows that $\lambda^{m}=-1$. Therefore \begin{align*}\lambda \underline{b}&=\lambda(0,0,0,b_1, \lambda b_1, \dots, \lambda^{m-1} b_1,\dots,b_q, \lambda b_q, \dots, \lambda^{m-1} b_q)\\&=\left(0,0,0,\lambda b_1,\lambda^2b_1,\dots,\lambda^m b_1,\dots,\lambda b_q,\lambda^2 b_q,\dots,\lambda^m b_q\right)\\&=\left(0,0,0,\lambda b_1,\lambda^2b_1,\dots,- b_1,\dots,\lambda b_q,\lambda^2 b_q,\dots,- b_q\right)\\&=(0,0,0,b_1, \lambda b_1, \dots, \lambda^{m-1} b_1,\dots,b_q, \lambda b_q, \dots, \lambda^{m-1} b_q)\cdot X\\&=\underline{b}\cdot X.\end{align*}
        Since \[\underline{a}\cdot Y=(0,0,0,b_1, \lambda b_1, \dots, \lambda^{m-1} b_1,\dots,b_q, \lambda b_q, \dots, \lambda^{m-1} b_q),\] it implies that $\lambda \underline{a}\cdot Y=\lambda \underline{b}=\underline{b}\cdot X$. Thus, \[\lambda\underline{a}=\underline{a}\cdot YXY^{-1}.\]
    \end{mycases}We conclude by \autoref{twisted:BDA} that $\sigma$ is twisted.
    
   We now prove the converse and assume that $\sigma$ is twisted. Then by \autoref{twisted:BDA}, there exist $\lambda\in\F_p^\times$ and $X=(\pi,\varepsilon,\tau,\varepsilon_3,\dots,\varepsilon_n)\in G_n$ such that \[\lambda \underline{a}=\underline{a}\cdot X\quad\text{and}\quad \prod_{i=3}^n\varepsilon_i=-1.\]
   Thus,\begin{equation}\label{eq0} \lambda a_i=\varepsilon a_{\pi(i)}\quad\text{for all $0\leq i\leq 2$},
    \end{equation} \begin{equation}\label{eq3} \lambda a_i=\varepsilon_i a_{\tau(i)}\quad\text{for all $3\leq i\leq n$}.
    \end{equation}\begin{itemize}
        \item If $a_i=0$ for some $i\geq 3$, then we are in Case 1.
        \item If $a_i\neq 0$ for all $i\geq 3$, we show that $\lambda^3\neq \pm 1$.
     \begin{itemize}
    \item If $\lambda=\pm 1$, we set \[\underline{a'}=\underline{a}\cdot Y\]for some $Y=(\Id,\varepsilon',\Id,\varepsilon_3',\dots,\varepsilon_n')\in G_n$ with $\varepsilon,\varepsilon_i$ are chosen such that, for each $i$, $a_i'$ is equal to the unique value of $a_i$ in $\{1,\dots,\frac{p-1}{2}\}$. We have \[\lambda \underline{a'}=\underline{a'}\cdot (Y^{-1}XY)=\underline{a'
    }\cdot X',\]where $X'=Y^{-1}XY.$
    It then follows by \eqref{eq3} that \[\lambda a'_i=\varepsilon_ia'_{\tau(i)}\quad\text{for all $i\geq 3$}.\] 
    Since $a_i'\neq 0$ for all $i\geq 3$, we deduce that \[\lambda =\varepsilon_i\quad\text{for all $i\geq 3$}.\]Moreover, multiplying $\lambda a_i'=\varepsilon_ia'_{\tau(i)}$ over all $i$ gives \[\lambda^{n-2}=\prod_{i=3}^{n}\varepsilon_i.\]As $n$ is even, it follows that \[\prod\limits_{i=3}^n\varepsilon_i=\lambda^{n-2}=1.\]Since $\prod\limits_{i=3}^n\varepsilon_i=-1$, then $\lambda\neq \pm 1$.
    \item If $\lambda\neq\pm 1$ and $\lambda^3=\pm 1$, then for each $i\geq 3$, we let $m_i$ be the length of the cycle of $\tau$ that contains $i$. In particular, \[\tau^{m_i}(i)=i.\]  
    Since $a_i\neq 0$ for all $i\geq 3$, it then implies by \eqref{eq3} that for every $3\leq i\leq n$ and $ \ell\leq m_i$,\begin{equation}\label{eq:ll}\lambda^\ell  a_{i}=\varepsilon_{i}\varepsilon_{\tau(i)}\varepsilon_{\tau^2(i)}\dots\varepsilon_{\tau^{l-1}(i)}a_{\tau^{l}(i)}.\end{equation} Furthermore,  since $a_i\neq 0$ for all $i$, it implies that 
    \[\lambda^{m_i}=\varepsilon_{i}\varepsilon_{\tau(i)}\varepsilon_{\tau^2(i)}\dots\varepsilon_{\tau^{m_i-1}(i)}=\pm 1.\]
Since $\lambda\neq \pm 1$ and $\lambda^3=\pm 1$, then every cycle-length $m_i$ must be a multiple of $3$.
Hence, $n-2=3m$ for some positive integer $m$. Moreover, \[\lambda^{n-2}=\prod\limits_{i=3}^n\varepsilon_i=-1.\] As $\lambda^3=\pm1$, it follows that \[(\pm 1)^m=(\lambda^3)^m=\lambda^{n-2}=-1.\]
Consequently, $m$ is odd. This gives a contradiction to $n$ being even. Then this case cannot occur.\end{itemize} 
  Thus \[\lambda^3\neq \pm 1.\] Then by the same argument used in \autoref{rev,eq}, we have the following three cases: \begin{itemize} 
    \item If $\sign(\pi)=-1$, then $\pi$ has a fixed point $i_0$ in $\{0,1,2\}$, and is a transposition $(ij)$ for $i,j\in\{0,1,2\}\backslash i_0$. This implies by \eqref{eq0} that $a_{i_0}=0$. Using the identity \[a_0+a_1+a_2=0,\]we conclude that $a_{i}=-a_{j}$. Moreover, we have \[\lambda a_i=\varepsilon a_j\quad\text{and}\quad \lambda a_j=\varepsilon a_i.\] Since $\lambda\neq \pm 1$, it follows that $a_1=a_2=a_2=0$. 
    \item If $\pi=\Id$, then $a_0=a_1=a_2=0$. 
    \item If $\pi=(i\; j\; \ell)$ is a rotation in $\{0,1,2\}$, then \begin{align*} \lambda a_i&=\varepsilon a_j\\ \lambda a_j&=\varepsilon a_{\ell}\\ \lambda a_{\ell}&=\varepsilon a_i. 
    \end{align*} It implies that $\lambda^3a_ia_ja_\ell=\varepsilon^3a_ia_ja_\ell$. If $a_{i_0}\neq 0$ for some $i_0\in\{i,j,\ell\}$, then all $a_i,a_j,a_\ell$ are non zero, which contradicts $\lambda^3\neq \pm 1$. Thus $a_i=a_j=a_{\ell}=0$.\end{itemize}
    Thus we have showed that $a_0=a_1=a_2=0$. Since $a_i\neq 0$ for all $i\geq 3$, we have by \eqref{eq3}, that for every $3\leq i\leq n$, \[(a_{i},a_{\pi(i)},\dots,a_{\pi^{m_i-1}(i)})=(a_{i},\varepsilon_{i}\lambda a_{i},\dots,\varepsilon_{i}\varepsilon_{\pi(i)}\varepsilon_{\pi^2(i)}\dots\varepsilon_{\pi^{m_i-1}(i)}\lambda^{m_i-1}a_{i}).\] Let $m_i$ be the length of the cycle of $\tau$ that contains $i$, and let $m=\gcd\limits_i(m_i)$. Thus each cycle of length $m_i$ decomposes into $\frac{m_i}{m}$ consecutive blocks of length $m$. Write $\tau$ as a product of disjoint cycles \[(r_1\; r_2 \; \dots \;r_{m_1})(r_{m_1+1}\; \dots\; r_{m_1+m_2})\dots .\] Let $\pi'\in\Sigma_{n-2}$ be the permutation defined by $\pi'(\ell)=r_\ell$ for all $\ell$. Consequently, for $Z=(\tau,\varepsilon'_3,\dots,\varepsilon'_n)\in \Sigma_{\{3,\dots,n\}}\ltimes \{-1,1\}^n-2$ for some $\varepsilon'_i=\{-1,1\}$, the tuple $(a_3,\dots,a_n)\cdot Z$ is equal to a tuple of the form \[(b_1,\lambda b_1,\dots,\lambda^{m_0-1}b_1,\dots, b_q\lambda b_q,\dots,\lambda^{m-1}b_q)\]for some $b_j\in \F_p^\times$ with $q=\frac{n-2}{m}$. Let $Y=(\Id,1,\pi',\varepsilon_3',\dots,\varepsilon_n')$.
    Then \[\underline{a}\cdot Y=(b_1,\lambda b_1,\dots,\lambda^{m_0-1}b_1,\dots, b_q\lambda b_q,\dots,\lambda^{m-1}b_q).\] Moreover, we have by \eqref{eq:ll}  $\lambda^{m_i}=\pm 1$ for all $i$, and so $\lambda^m=\pm 1$. Additionally, \[(\pm 1)^q=(\lambda^m)^q=\lambda^{n-2}=\prod_{i=3}^n\varepsilon_i=-1.\] We conclude that $\lambda^m=-1$ and that $q$ is odd. Consequently, $m$ is even, and we are in Case $2$.\qedhere\end{itemize}
\end{proof}
The preceding lemmas can be summarized in the following corollary.
\begin{corollary}\label{cor:add,utw}
Let $p\neq 3$ be a prime, $n\geq 2$ and $2\leq k\leq n$. Let $\sigma$ be an additive $k$-simplex in $\SBDA(\F_p^n)_k$ and let $B\in\MDA(\F_p^n)_k$ such that $\sigma\Sigma_{k+1}=B\cdot G_k$. Then $\sigma$ belongs to $\left(\SBDA(\F_p^n)_k\setminus \SBD(\F_p^n)_k\right)^\mathrm{utw}$ if and only if $k=n$ is even and $B$ belongs to some set \[\DA_2^{n,n}(\underline{a})\quad\text{with $\underline{a}=(a_0,\dots,a_n)\in A_n$},\] where the entries do not satisfy $a_i= 0$ for some $i\geq 3$ and the \hyperref[twistedcond]{Additive $(\lambda,m)$-condition}.
\end{corollary}
We now treat the case $p=3$.
\begin{lemma} \label{tw:p=3}
     Let $n\geq 2$, and let $\sigma$ be an additive $n$-simplex in $\SBDA(\F_3^n)_n$ and let $B\in \TA_2^{n,n}(\underline{a})$ for some $\underline{a}=(a_1,\dots,a_n)\in \F_3^{n}\setminus\{\underline{0}\}$ such that $\sigma \Sigma_{n+1}=B\cdot G_n$. Then $\sigma$ is twisted if and only if one of the following cases is satisfied.\begin{mycases}
         \case If $n$ is odd,
         \case If $a_{i_0}=0$ for some $3\leq i_0\leq n$.
     \end{mycases}
\end{lemma}
\begin{proof}
     Let $\sigma \Sigma_{n+1}=B\cdot G_n$ with $B\in \TA_2^{n,n}(\underline{a})$. Since  \[\left(\SBDA(\F_p^n)_n - \SBD(\F_p^n)_n\right)^\mathrm{tw}/\Sigma_{n+1}\cong \MDA(\F_p^n)_n^\mathrm{tw}/G_n\] by \autoref{lem:goodmat}, we have the following equivalences:
     \begin{align*}\text{$\sigma$ is twisted}&\Longleftrightarrow \text{there exist $A\in \Pb_n^\pm(\F_3)$ and $X\in G_n$ such that $\det(A)\sign(X)=-1$ and $B\cdot X=A\cdot B$}\\&\Longleftrightarrow \text{there exist $X=(\pi,\varepsilon,\tau,\varepsilon_3,\dots,\varepsilon_n)\in G_n$ such that $B\cdot X\in \TA_2^{n,n}(\underline{a})$, and $\prod\limits_{i=3}^n\varepsilon_i=-1$.}\\ &\Longleftrightarrow \text{there exist $X=(\pi,\varepsilon,\tau,\varepsilon_3,\dots,\varepsilon_n)\in G_n$ such that $\TA_2^{n,n}(\underline{a})\cdot X=\TA_2^{n,n}(\underline{a})$, and $\prod\limits_{i=3}^k\varepsilon_i=-1$},\end{align*} where the second equivalence follows from \autoref{product}.\begin{mycases}
        \case If $n$ is odd, we take $X=(\Id,-1,\Id,-1,\dots,-1)\in G_n$. Then \begin{align*}\TA_2^{n,n}(\underline{a})\cdot X&=\left\{\left(-v_0| -v_1|\dots| -v_n\right)\in\MDA(\F_3^n)_n~\middle\vert~\begin{array}{c}\lambda e_1= a_1v_1+a_2v_2+\dots+a_nv_n \\ \text{for some $\lambda\in\F_3^\times$}\end{array}\right\}\\&=\left\{\left(-v_0| -v_1|\dots| -v_n\right)\in\MDA(\F_3^n)_n~\middle\vert~\begin{array}{c}-\lambda e_1= -a_1v_1-a_2v_2-\dots-a_nv_n \\ \text{for some $\lambda\in\F_3^\times$}\end{array}\right\}\\&= \left\{\left(-v_0| -v_1| \dots| -v_n\right)\in\MDA(\F_3^n)_n~\middle\vert~\begin{array}{c}-\lambda e_1= a_1(-v_1)+a_2(-v_2)+\dots+a_n(-v_n) \\ \text{for some $\lambda\in\F_3^\times$}\end{array}\right\}\\&=\TA_2^{k,n}(\underline{a}),
             \end{align*} and \[\prod\limits_{i=3}^n\varepsilon_i=(-1)^{n-2}=-1.\]
        \case If $k=n$ and $a_{i_0}=0$ for some $3\leq i_0\leq n$,  we take $X=(\Id,1,\Id,\varepsilon_3,\dots,\varepsilon_{n-1})$ where \[\varepsilon_i=\begin{cases}-1&\text{if $i=i_0$}\\    1&\text{otherwise.}  
        \end{cases}\] 
        Then, \begin{align*}\TA_2^{n,n}(\underline{a})\cdot X&=\left\{\left(v_0|\dots| -v_{i_0}|\dots| v_n\right)\in\MDA(\F_3^n)_n~\middle\vert~\begin{array}{c}\lambda e_1= a_1v_1+\dots+0\cdot(-v_{i_0})+\dots+a_nv_n \\ \text{for some $\lambda\in\F_3^\times$}\end{array}\right\}\\&= \TA_2^{k,n}(\underline{a}),
             \end{align*}and \[\prod_{i=3}^{n}\varepsilon_i=\varepsilon_{i_0}=-1.\]
        \end{mycases}
        Now conversely, assume that $\sigma$ is twisted. We further assume that $n$ is even, and $a_i \neq 0$ for all $i \geq 3$, and we argue by contradiction. Let $B=(v_0|\dots| v_n)\in \TA_2^{n,n}(\underline{a})$. Then, there exists $C=(u_0|\dots| u_n)\in \TA_2^{n,n}(\underline{a})$ and $X=(\pi,\varepsilon,\tau,\varepsilon_3,\dots,\varepsilon_n)\in G_n$, such that \[B=C\cdot X\quad\text{and}\quad\prod\limits_{i=3}^n\varepsilon_i=-1.\]
        It implies that there exists some $\lambda\in \F_3^\times$ such that\[\lambda e_1=\sum_{i=1}^na_iv_i=\varepsilon a_1u_{\pi(1)}+\varepsilon a_2u_{\pi(2)}+\varepsilon_3 a_3u_{\tau(3)}+\dots+\varepsilon_n a_nu_{\tau(n)}.\]
        Furthermore, $\mu e_1=\sum\limits_{i=1}^na_iu_i$ for some $\mu\in\F_3^\times$. Thus we obtain \[\lambda\mu^{-1}a_1u_1+\dots+\lambda\mu^{-1}a_nu_n=\varepsilon a_1u_{\pi(1)}+\varepsilon a_2u_{\pi(2)}+\varepsilon_3 a_3u_{\tau(3)}+\dots+\varepsilon_n a_nu_{\tau(n)}.\]It follows by the linear independence of $u_3,\dots,u_n$ that \[\lambda a_{i}=\mu \varepsilon_ia_{\tau(i)}\quad\text{for all $i=3,\dots,n$}.\]Multiplying those $n-2$ equations, we get\[\lambda^{n-2}\prod_{i=3}^na_i=\mu^{n-2}\prod_{i=3}^na_i\varepsilon_i.\]As $a_i\neq 0$ for all $i\geq 3$, we can cancel $\prod\limits_{i=3}^na_i$ to get  \[\left(\lambda\mu^{-1}\right)^{n-2}=\prod_{i=3}^n\varepsilon_i=-1.\] Since $\lambda,\mu\in\F_3^\times=\{-1,1\}$, this gives a contradiction to $n$ being even.
\end{proof}
\subsection{Relation to partial frames complexes}
Following the flat $\SL_n(\Z)$-resolution of $\St_n(\Q)\otimes\Q$ in \cite[Lemma 3.2]{CP},
\[\redchain_n(\SBA(\Z^n);\Q) \rightarrow \redchain_{n-1}(\SB(\Z^n);\Q)\rightarrow  \St_n(\Q)\otimes\Q \rightarrow 0,\] we will explain in this section how the Steinberg module is related to the homology of our symmetric $\Delta$-complexes $\Gamma^\pm_{0,n}(p)\backslash \SB(\Z^n)$ and $\Gamma^\pm_{0,n}(p)\backslash \SBA(\Z^n)$.

In \cite{MPP}, Miller--Patzt--Putman use the above resolution to establish a relation between $\St_n(\Q)_{\Gamma_n(p)}$ and the complex of determinant-$1$ partial (augmented) frames, ultimately proving \autoref{isoSt1}.
\begin{definition}
    Let $R$ be a Euclidean domain. Define $\BDA(R^n)'$ to be the subcomplex of $\BDA(R^n)$ consisting of simplices $\{v_0^\pm,\dots,v_k^\pm\}$ such that the $R$-$\operatorname{span}$ of the $\pm$-vectors is a proper submodule for $R^n$. 
\end{definition}
Recall that, via \autoref{notation:sdelta}, $\mathrm{S}(\BDA(R^n)')$ denotes the symmetric $\Delta$-complex associated to $\BDA(R^n)'$. By abuse of notation, we will denote it by $\SBDA(R^n)'$.

\begin{lemma} [{\cite[Lemma 3.23]{MPP}}]\label{isoSt1}
    Let $p$ be a prime and $n \geq 2$. Then,
    \[\St_n(\Q)_{\Gamma_n(p)} \cong \homology_{n-1}(\BDA(\F_p^n),\BDA(\F_p^n)').\]
\end{lemma}

\begin{lemma} \label{isoSt2}
    Let $p$ be a prime and $n \geq 2$. Let $\rel'_n(p)=\left( \Pb_n^\pm(\F_p)\backslash \SBDA(\F_p^n),\Pb_n^\pm(\F_p)\backslash \SBDA(\F_p^n)'\right)$. Then,
    \[(\St_n(\Q)\otimes\Q)_{\Gamma_{0,n}^\pm(p)} \cong \homology_{n-1}(\rel'_n(p);\Q).\]
\end{lemma}
Our proof follows the same argument used by Miller--Patzt--Putman in \autoref{isoSt1}.
\begin{proof}
Since $\Pb_n^\pm(\F_p)\backslash \SBDA(\F_p^n)$ is $n$-dimensional and $\Pb_n^\pm(\F_p)\backslash \SBDA(\F_p^n)'$ is $(n-1)$-dimensional complex containing the $(n-2)$-skeleton of $\Pb_n^\pm(\F_p)\backslash \SBDA(\F_p^n)$, we have 
\begin{equation}\label{rel}\C_{k}(\rel'_n(p))\cong 
\begin{cases}
    \redchain_n(\Pb_n^\pm(\F_p)\backslash \SBDA(\F_p^n);\Q) &~\text{if}~k=n\\
    \redchain_{n-1}(\Pb_n^\pm(\F_p)\backslash \SBD(\F_p^n);\Q) &~\text{if}~k=n-1\\
    0 &~\text{otherwise}\\
\end{cases}.\end{equation} Consequently,
    \[\homology_{n-1}(\rel'_n(p);\Q) \cong \coker\left(\redchain_n(\Pb_n^\pm(\F_p)\backslash \SBDA(\F_p^n);\Q)\rightarrow \redchain_{n-1}(\Pb_n^\pm(\F_p)\backslash \SBD(\F_p^n);\Q)\right).\]
    Recall by \eqref{ses} the short exact sequence \[1 \longrightarrow \Gamma_n(p) \longrightarrow \Gamma_{0,n}^\pm(p) \longrightarrow  \Pb_n^\pm(\F_p) \longrightarrow 1 .\] Thus,
    \[\redchain_{n-1}(\SB(\Z^n);\Q)_{\Gamma_{0,n}^\pm(p)} \cong \left(\redchain_{n-1}(\SB(\Z^n);\Q)_{\Gamma_n(p)}\right)_{\Pb_n^\pm(\F_p)} \]and \[\redchain_{n}(\SBA(\Z^n);\Q)_{\Gamma_{0,n}^\pm(p)} \cong \left(\redchain_{n}(\SBA(\Z^n);\Q)_{\Gamma_n(p)}\right)_{\Pb_n^\pm(\F_p)}.\]
    It then implies by \autoref{chaincoinv} and \autoref{good q} that\[\redchain_{n-1}(\SB(\Z^n);\Q)_{\Gamma_{0,n}^\pm(p)}\cong \redchain_{n-1}(\Pb_n^\pm(\F_p)\backslash \SBD(\F_p^n);\Q) \]and \[\redchain_{n}(\SBA(\Z^n);\Q)_{\Gamma_{0,n}^\pm(p)} \cong \redchain_n(\Pb_n^\pm(\F_p)\backslash \SBDA(\F_p^n);\Q). \]
    Therefore, the lemma follows from the fact that taking coinvariants is right-exact.
\end{proof}
We prove a version of the previous lemma with the determinant representation.
\begin{lemma} \label{isoSt3}
    Let $p$ be a prime and $n \geq 2$.
    \[(\St_n(\Q)\otimes\Q^{\det})_{\Gamma_{0,n}^\pm(p)} \cong \homology_{n-1}\left( \left(\C_*(\SBDA(\F_p^n),\SBDA(\F_p^n)';\Q)\otimes\Q^{\det} \right)_{\Pb_n^\pm(\F_p)}\right).\]
\end{lemma}
\begin{proof}Since $\Gamma_n(p)$ acts trivially on $\Q^{\det}$, it follows from \eqref{ses}, as well as \autoref{chaincoinv} and \autoref{good q} that\[\left(\redchain_{n-1}(\SB(\Z^n);\Q)\otimes\Q^{\det}\right)_{\Gamma_{0,n}^\pm(p)}\cong \left(\redchain_{n-1}(\SBD(\F_p^n);\Q)\otimes\Q^{\det}\right)_{\Pb_n^\pm(\F_p)}. \]
Similarly, we have\[\left(\redchain_{n}(\SBA(\Z^n);\Q)\otimes\Q^{\det}\right)_{\Gamma_{0,n}^\pm(p)}\cong \left(\redchain_{n}(\SBDA(\F_p^n);\Q)\otimes\Q^{\det}\right)_{\Pb_n^\pm(\F_p)}.\] Therefore, the lemma follows from the fact that tensoring and taking coinvariants is right-exact.\end{proof}

\subsection{Proof of \autoref{thD}}
We are now ready to prove \autoref{thD}. 
    Let $p$ be a prime and $n\geq 2$ such that one of the following conditions holds \begin{itemize}
        \item $n=2$ and $p\in\{2,3,5,7,13\}$,
        \item $n$ is odd,
        \item $p\leq 6n-8$.
    \end{itemize}Then \[\left(\St_n(\Q)\otimes\Q^{\det}\right)_{\Gamma_{0,n}^\pm(p)}\cong 0.\]

\begin{proof}[Proof of \autoref{thD}]
Recall that the case $n=2$ has already been proven in \autoref{det St}. So it remains to prove the statement for $n\geq 3$.

We showed in \autoref{isoSt3} that for all $n\geq 2$ and all primes, \begin{align*}(\St_n(\Q)\otimes\Q^{\det})_{\Gamma_{0,n}^\pm(p)} &\cong \homology_{n-1}\left( \left(\C_*(\SBDA(\F_p^n),\SBDA(\F_p^n)';\Q)\otimes\Q^{\det} \right)_{\Pb_n^\pm(\F_p)}\right)\\&\overset{\eqref{rel}}{\cong} \coker \left(\left(\redchain_n(\SBDA(\F_p^n);\Q)\otimes\Q^{\det} \right)_{\Pb_n^\pm(\F_p)} \overset{\partial}\longrightarrow \left(\redchain_{n-1}(\SBD(\F_p^n);\Q)\otimes\Q^{\det} \right)_{\Pb_n^\pm(\F_p)}\right).
    \end{align*}
    By \autoref{coinv,tw}, \[\left(\redchain_n(\SBD(\F_p^n);\Q)\otimes\Q^{\det} \right)_{\Pb_n^\pm(\F_p)}\cong \left(\Q\left[\SBD(\F_p)_{n-1}^\mathrm{utw}\right]\otimes_{\Q[\Sigma_{n}]}\Q^\mathrm{sgn}\otimes\Q^{\det}\right)_{\Pb_n^\pm(\F_p)}\] has a non-canonical basis bijective to \[\Pb_n^\pm(\F_p)\backslash\SBD(\F_p^n)_{n-1}^\mathrm{utw}/\Sigma_{n},\] which is empty by \autoref{cor:st,utw} when $n$ is odd. Thus \[\left(\St_n(\Q)\otimes\Q^{\det}\right)_{\Gamma_{0,n}^\pm(p)}\cong 0\quad\text{if $n$ is odd}.\]
   To prove the final part of the theorem, we will show that $\partial$ is surjective for even $n$ and $p \leq 6n-8$. We first treat the case $p\neq 3$. Our argument is very similar to that of \autoref{BA connected}.

   From \autoref{lem:goodmat}, we have \[\Pb_n^\pm(\F_p)\backslash\SBD(\F_p^n)_{n-1}^\mathrm{utw}/\Sigma_{n}\cong \Pb_n^\pm(\F_p)\backslash\MD(\F_p^n)_{n-1}^\mathrm{utw}/T_{n-1}.\] And by \autoref{cor:st,utw}, $B\in \MD(\F_p^n)_{n-1}^\mathrm{utw}$ if and only if $n$ is even and $B$ in $\D_2^{n-1,n}(\underline{a})$ such that \begin{equation}\begin{aligned}\label{nonzeros}
       \underline{a}=(a_0,\dots,a_{n-1})\in\left(\F_p^\times\right)^{n},\quad \underline{a}\text{ does not satisfy the \hyperref[def1]{$(\lambda,m)$-condition}.}
   \end{aligned}\end{equation}
 Fix $\underline{a}\in\F_p^n$ satisfy \eqref{nonzeros}. We first consider the case where $a_i=\pm a_j$ for some $i\neq j$. From \autoref{coinv,tw}, we have \begin{equation}\label{tww}
     \D_2^{n-1,n}(\underline{a})\otimes1\otimes1=\pm\D_2^{n-1,n}(\underline{a}\cdot X)\otimes1\otimes1
 \end{equation} for all $X\in T_k$. Thus we may assume without loss of generality that $a_i=a_j$ and $i=0$, $j=1$. Then we define $\underline{b}=(b_0,\dots,b_n)\in \F_p^{n+1}\setminus\{\underline{0}\}$ such that \[b_1=3^{-1}(2a_0-a_1),\]\[b_2=3^{-1}(2a_1-a_0),\]\[b_0=-b_1-b_2=3^{-1}(-a_0-a_1),\]\[b_i=a_{i-1}~\text{for all $i\geq 3$}.\]
  It implies by \eqref{nonzeros} that $b_i\neq 0$ for all $i\geq 3$, and $\underline{b}$ does not satisfy the \hyperref[twistedcond]{Additive $(\lambda,m)$-condition}.
  
  Under the above identifications, we compute $\partial$ in terms of the elements \[\DA_2^{n,n}(\underline{b})\otimes 1\otimes 1,\quad\text{and}\quad \D_2^{n-1,n}(\underline{a})\otimes 1\otimes 1 .\] By \eqref{e11}, we have \begin{equation}\label{DA}\begin{aligned}\partial\left(\DA_2^{n,n}(\underline{b})\otimes 1\otimes 1)\right)=&\D_2^{n-1,n}(a_0,\dots,a_{n-1})\otimes 1\otimes 1\\&-\D_2^{n-1,n}(-a_0,a_1-a_0,\dots,a_{n-1})\otimes 1\otimes1\\&+\D_2^{n-1,n}(-a_1,a_0-a_1,\dots,a_{n-1})\otimes 1\otimes 1\end{aligned}\end{equation}Since $a_1-a_0=0$, then by \autoref{coinv,tw} and \autoref{cor:st,utw},\[\D_2^{n-1,n}(-a_0,a_1-a_0,\dots,a_{n-1})\otimes 1\otimes 1=\D_2^{n-1,n}(-a_1,a_0-a_1,\dots,a_{n-1})\otimes1\otimes 1=0.\] 
  We thus conclude that for any $\underline{a}\in\F_p^n\setminus\{\underline{0}\}$ satisfying \eqref{nonzeros} such that $a_i=\pm a_j$, \[\D_2^{n-1,n}(a_0,\dots,a_{n-1})\otimes 1\otimes 1\in\Image\partial.\]
Now assume that $a_i\neq \pm a_j$ for all $i\neq j$. As $p\leq 6n-8$, we have by \autoref{bound} that there exist distinct indices $i,j,\ell$ such that $a_i-a_j=\pm a_{\ell}$. Using \eqref{tww}, we assume without loss of generality that $i=0$ and $j=1$.
Using the same construction as before, \eqref{DA} shows that \[\D_2^{n-1,n}(a_0,\dots,a_{n-1})\otimes 1\otimes 1-\D_2^{n-1,n}(-a_0,a_1-a_0,\dots,a_{n-1})\otimes 1+\D_2^{n-1,n}(-a_1,a_0-a_1,\dots,a_{n-1})\otimes 1\otimes 1\in\Image\partial.\]Since $a_1-a_0=\pm a_\ell$, then the previous case gives that \[\D_2^{n-1,n}(-a_0,a_1-a_0,\dots,a_{n-1})\otimes 1 \otimes 1\quad\text{and}\quad\D_2^{n-1,n}(-a_1,a_0-a_1,\dots,a_{n-1})\otimes 1\otimes 1\]lie in the image of $\partial$. It follows that \[\D_2^{n-1,n}(a_0,\dots,a_{n-1})\otimes 1\otimes 1\in\Image\partial.\] Therefore, $\partial$ is surjective for $p>3$.

 We now treat the case $p=3$.
 By \autoref{tw:p=3}, $B\in\MDA(\F_3^n)_n^\mathrm{utw}$ if and only if $n$ is even and $B\in \TA_2^{n,n}(b_1,\dots,b_n)$ such that $b_i\neq 0$ for all $i\geq 3$. Recall that \[\TA_2^{n,n}(b_1,\dots,b_n)=\left\{\left(v_0|\dots|v_n\right)\in\MDA(\F_3^n)_n~\middle\vert~\lambda e_1= b_1v_1+\dots+b_nv_n ~\text{for some $\lambda\in\F_3^\times$}\right\}.\]

 We first make the following observation: For $B=(v_0|\dots|v_n)\in \TA_n^{n,n}(b_1,\dots,b_n)$ with $b_i\neq 0$ for all $i\geq 3$, using $v_0=-v_1-v_2$, we can write \begin{align*}
     \lambda e_1&=b_1v_1+b_2v_2+b_3v_3+\dots+b_nv_n\\ &=-b_1v_0+(b_2-b_1)v_2+b_3v_3+\dots+b_nv_n\\&=-b_2v_0+(b_1-b_2)v_2+b_3v_3+\dots+b_nv_n,
 \end{align*}for some $\lambda\in\{-1,1\}$. Consequently, $\partial$ is defined in the following way\begin{align*}\partial\left(\TA_2^{n,n}(b_1,b_2,\dots,b_n)\otimes 1\otimes 1\right)=&\D_2^{n-1,n}(b_1,b_2,b_3,\dots,b_{n})\otimes 1\otimes 1\\-&\D_2^{n-1,n}(-b_1,b_2-b_1,b_3,\dots,b_n)\otimes 1\otimes 1\\+&\D_2^{n-1,n}(-b_2,b_1-b_2,b_3,\dots,b_{n})\otimes 1\otimes 1.\end{align*}

 Now let $n$ be even and let $\underline{a}\in\F_p^{n}$ that satisfies \eqref{nonzeros}. Since $\TA_2^{n,n}(\underline{a})\otimes 1\otimes 1=\pm\TA_2^{n,n}(\underline{a})\cdot X\otimes 1\otimes 1$ by \autoref{coinv,tw}, and $a_i\neq 0$ for all $i$, we assume without loss of generality that $a_0=a_1=1$. Choose \[(b_1,b_2,b_3,\dots,b_n)=(1,1,a_2,\dots,a_{n-1}).\] So, for every $i\geq 3$, $b_i$ is nonzero. Moreover, 
 \begin{align*}\partial\left(\TA_2^{n,n}(1,1,a_2,\dots,a_{n-1})\otimes 1\otimes 1\right)&=\D_2^{n-1,n}(1,1,a_2,\dots,a_{n-1})\otimes 1\otimes 1\\&-\D_2^{n-1,n}(-1,0,a_2,\dots,a_{n-1})\otimes 1\otimes 1\\&+\D_2^{n-1,n}(-1,0,a_2,\dots,a_{n-1})\otimes 1\otimes 1.\end{align*}Since \[\D_2^{n-1,n}(-1,0,a_2,\dots,a_{n-1})\otimes 1\otimes 1=\D_2^{n-1,n}(-1,0,a_2,\dots,a_{n-1})\otimes 1\otimes 1=0,\]by \autoref{coinv,tw} and \autoref{cor:st,utw}, thus \[\partial\left(\TA_2^{n,n}(b_1,b_2,\dots,b_{n-1})\otimes 1\otimes 1\right)=\D_2^{n-1,n}(a_0,a_1,a_2,\dots,a_{n-1})\otimes 1\otimes 1.\]
 We conclude that $\partial$ is surjective for $p=3$.
  This completes the proof.
 \end{proof}

\section{Homology of the Tits building}\label{sec8}
For a poset $X$, we denote its order complex by $\oc(X)$, whose $k$-simplices are subsets $\{x_0,\dots,x_k\}\subseteq X$ with $x_0<\dots<x_k$. 
In the special case of the posets $\T_n(R)$ and $\T_n^\pm(R)$, we write $\mathcal{T}_n(R)$ and $\mathcal{T}_n^\pm(R)$ for their order complexes, respectively. (See \autoref{titsbldg} and \autoref{pmTitsbldg}).

We will prove that, for all primes $p$ and all $n\geq 3$, the order complex $\oc\left(\Gamma_{0,n}^\pm(p)\backslash \T_n(\Q)\right)$ is contractible. Our proof uses combinatorial Morse theory.

\begin{proposition}
[\cite{Bestvina}]\label{morse theory} Let $X$ be a simplicial complex and let $Y$ be a full subcomplex of $X$. Let $S$ be the set of all vertices of $X$ that are not in $Y$. Suppose there is no edge between any pair of vertices $s,t\in S$. Then \[X/Y\cong \bigvee_{s\in S}\Sigma\left(\Link_X(s)\right)\]
    In particular, if $Y$ is contractible and $\Link_X(s)$ is contractible for every $s \in S$, then $X$ is contractible. 
\end{proposition}
\begin{proposition} \label{tits:acyclic}
    For all primes $p$ and  $n \geq 3$, the order complex $\oc\left(\Gamma_{0,n}^\pm(p)\backslash \T_n(\Q)\right)$ is contractible.
\end{proposition} 
\begin{proof}
We recall from \autoref{Titsbldg:dq} the isomorphism \[\Gamma_{0,n}^\pm(p)\backslash \T_n(\Q)\cong \Pb_n^\pm(\F_p)\backslash \T_n^\pm(\F_p).\]
In order to prove our claim, we will prove that the order complex $\oc\left(\Pb_n^\pm(\F_p)\backslash \T_n^\pm(\F_p)\right)$ is contractible.

We have by \autoref{tits} that $\Pb_n^\pm(\F_p)\backslash \T_n^\pm(\F_p)$ is the poset \[\{\left[(U_1,\pm\omega_1)\right],\dots,\left[(U_{n-1},\pm \omega_{n-1})\right],\left[(V_1,\pm \omega_1')\right],\dots,\left[(V_{n-1},\pm \omega_{n-1}')\right]\},\] where
\[
U_k=\operatorname{span}\{ e_1,\dots,e_k\}, 
\quad
V_k=\operatorname{span}\{ e_2,\dots,e_{k+1}\},
\]
and $\omega_k=e_1\wedge\cdots\wedge e_k$, $\omega_k'=e_2\wedge\cdots\wedge e_{k+1}$.
Let $n\geq 3$. Let $Y$ be the full poset \[\{\left[(U_1,\pm \omega_1)\right],\dots,\left[(U_{n-1},\pm \omega_{n-1})\right], \left[(V_1,\pm \omega_1')\right],\dots,\left[(V_{n-2},\pm\omega_{n-2}')\right]\}.\]

\begin{figure}[H]
\centering
\begin{tikzpicture}[node distance=0.8cm and 0.3cm]

  % Left column: bullet and text nodes
  \node (A1) {\tiny$\bullet$};
  \node (A1.1)[left=-0.1cm of A1]{$\left[(V_{n-1},\pm \omega_{n-1}')\right]$};
  
  \node (B1) [below=of A1] {\tiny$\bullet$}; 
  \node (B1.1)[left=-0.1cm of B1]{$\left[(V_{n-2},\pm\omega_{n-2}')\right]$};

  \node (C1) [below=of B1] {\tiny$\bullet$};

   \node (D1) [below=-0.2cm of C1]
   {$\vdots$};
   
  \node (E1) [below=of D1] {\tiny$\bullet$};
  \node (E1.1)[left=-0.1cm of E1]{$\left[(V_2,\pm\omega_2')\right]$};
  
  \node (F1) [below=of E1] {\tiny$\bullet$};
  \node (F1.1) [left=-0.1cm of F1] {$\left[(V_1,\pm \omega_1')\right]$};
  
  % Right column: bullet and text nodes
  \node (A2) [right=1cm of A1] {\tiny$\bullet$};
  \node (A2.1)[right=-0.1cm of A2]{$\left[(U_{n-1},\pm \omega_{n-1})\right]$};
  
  \node (B2) [below=of A2] {\tiny$\bullet$};
   \node (B2.1)[right=-0.1cm of B2]{$\left[(U_{n-2},\pm\omega_{n-2})\right]$};

  \node (C2) [below=of B2] {\tiny$\bullet$};
  
   \node (D2) [below=-0.2cm of C2]
   {$\vdots$};
   
  \node (E2) [below=of D2] {\tiny$\bullet$};
  \node (E2.1)[right=-0.1cm of E2]{$\left[(U_2,\pm\omega_2)\right]$};

  \node (F2) [below=of E2] {\tiny$\bullet$};
  \node (F2.1) [right=-0.1cm of F2] {$\left[(U_1,\pm \omega_1)\right]$};
 
  % Connect the bullet nodes with lines
  \draw (A1) -- (B1) -- (C1);
  \draw (A2) -- (B2) -- (C2);
  \draw (D2) -- (E2);
  \draw (D1) -- (E1);
  \draw (B1) -- (A2);
  \draw (C1) -- (B2);
  \draw (E1) -- (F1) -- (E2) -- (F2);
 \begin{scope}
  \draw[red, thick, fill=none, rounded corners=4pt]
   ($(B1.north west)+(0pt,0pt)$)
  -- ($(A2.north east)+(0pt,7pt)$)
  -- (B2.east)
  -- (C2.east)
  -- (D2.south east)
  -- (E2.south east)
  -- (F2.south east)
  -- (F1.south west)
  -- (E1.west)
  -- (D1.north west)
  -- ($(C1.west)+(0.5pt,2pt)$)   
  -- cycle;
\end{scope}\end{tikzpicture}
\caption{The order complex whose poset is $\Pb_n^\pm(\F_p)\backslash \T^\pm_n(\F_p)$, with the subposet $Y$ enclosed in red.}\label{fig1}
\end{figure}
We observe that $\oc(Y)$ has a cone point at $\left[(U_{n-1},\pm \omega_{n-1})\right]$. Thus, $\oc(Y)$ is contractible. Moreover, \[\Link\left[(V_{n-1},\pm \omega_{n-1}')\right]=\oc\left(\{\left[(V_1,\pm \omega_1')\right]<\dots<\left[(V_{n-2},\pm\omega_{n-2}')\right]\}\right)\]is also contractible. Therefore, by \autoref{morse theory}, $\oc\left(\Pb_n^\pm(\F_p)\backslash \T_n^\pm(\F_p)\right)$ is contractible.\end{proof}

\begin{remark}
    For $n=2$, $\oc\left(\Pb_2^\pm(\F_p)\backslash \T_2^\pm(\F_p)\right)$ is a discrete set of $2$ elements, and so not contractible.
\end{remark}

\section{Spectral sequence argument}\label{sec:spsq}

In this section, we review results from \cite{Quillen} and \cite{Charney} about map-of-poset spectral sequence, and we prove \autoref{thB} and \autoref{thC}. 

\subsection{Map-of-poset spectral sequence}
\begin{definition}
    Let $X$ be a poset. We consider $X$ as a category whose elements as objects and a single morphism from $x_0\in X$ and $x_1 \in X$ exactly when $x_0\leq x_1$. \end{definition}
For every functor $F \colon X \rightarrow \mathrm{Ab}$, we define the chain complex \[\C_p(X;F)=\bigoplus\limits_{x_0<\dots <x_p}F(x_0)\] where the differential is the alternating sum of $d_i\colon C_k(X;F) \rightarrow C_{k-1}(X;F)$ given by \begin{itemize}
    \item $d_0\colon F(x_0)\rightarrow F(x_1)$ is induced by the relation $x_0<x_1$,
    \item $d_i\colon F(x_0)\rightarrow F(x_0)$ for $i>0$ is the identity map, since deleting $x_i$ does not change $x_0$. 
\end{itemize} The poset homology of $X$ is defined as \[\homology_k(X;F):=\homology_k(\C_*(X;F)).\] We can also define the augmented chain complex \[\redchain_k(X;\Q)=\begin{cases} 
        \Q &\text{if $k=-1$},\\
        \bigoplus\limits_{x_0<\dots <x_k}\Q &\text{if $k\geq 0$},\\
        0 &\text{otherwise}.
        \end{cases}\] The reduced poset homology is then defined as the homology of this augmented chain complex
\[\redhom_k(X; \Q) := \homology_k(\redchain_*(X; \Q)).\]Note that $\redhom_k(X;\Q)\cong \redhom_k(\mathcal{O}(X);\Q)$.
 Let \[X_{>x}=\{x'\in X\mid x'>x\}.\] For a map $f \colon Y \rightarrow X$ of posets, we write \[f_{\leq x}=\{y\in Y\mid  f(y)\leq x \}.\] 
    We recall the following lemma of Charney \cite{Charney}, which is very useful for computing poset homology.
\begin{lemma} \label{lem: Charney} Let $X$ be a poset and let $F\colon X \rightarrow \mathrm{Ab}$ be a functor that is supported on elements
of an antichain $A\subset X$. Then 
\[\homology_k(X;F) = \bigoplus\limits_{x_0\in A}\redhom_{k-1}(X_{>x_0}; F(x_0)).\]
\end{lemma} 
    \begin{proof}
        See e.g. {\cite[Lemma 1.3]{Charney}} for a proof.
    \end{proof}
    
Our main interest in poset homology is a reduced version of the map-of-poset spectral sequence from Patzt and Wilson's unpublished notes \cite{PW}.

\begin{theorem} \label{red spsq}Let $f \colon Y \rightarrow X$ be a map of posets. Then there exists a spectral sequence 
    \[E^2_{kh}=\homology_k(Y;y\mapsto \redhom_h(f_{\leq y};\Q))\Longrightarrow\homology_{k+h+1}(f),\]where $\homology_{*}(f)$ is the homology of the mapping cone of $f_*\colon\C_*(X;\Q)\rightarrow \C_*(Y;\Q)$.
    \end{theorem}
    \begin{proof}
        See the Appendix in {\cite{PW}} for a proof.
    \end{proof}
\subsection{The proof of \autoref{thC}}
We recall the statement of \autoref{thC}. Let $p$ be a prime and $n\geq 3$. If $p\in\{2,3,5,7,13\}$ or $p\leq 6n-14$, then \[\left(\St_n(\Q)\otimes\Q\right)_{\Gamma_{0,n}^\pm(p)}\cong 0.\] Since \autoref{tits:acyclic} says that the order complex $\oc\left(\Pb_n^\pm(\F_p)\backslash \T_n^\pm(\F_p)\right)$ is contractible for all $n\geq 3$, the result will follow after proving that the map \begin{equation}\label{map} \left(\St_n(\Q)\otimes \Q\right)_{\Gamma_{0,n}^\pm(p)}\longrightarrow \redhom_{n-2}(\Gamma_{0,n}^\pm(p)\backslash \T_n(\Q);\Q)\end{equation} defined in the introduction is an isomorphism when $p\in\{2,3,5,7,13\}$ or $p\leq 6n-14$.
\begin{definition}
    Let $Y$ be a $\Delta$-complex. Define the associated poset $\mathcal{P}(Y)=\{\sigma\in Y\mid \sigma\text{ is a simplex}\}$ with order given by\[\sigma\leq \tau\Longleftrightarrow \text{$\sigma$ is a face of $\tau$}.\]
\end{definition}
We first prove the following two lemmas.
\begin{lemma} \label{lem:compmap}
      For all primes $p$ and $n\geq 2$, let \[\rel'_n(p):=\left(\Pb_n^\pm(\F_p)\backslash \SBDA(\F_p^n),\Pb_n^\pm(\F_p)\backslash \SBDA(\F_p^n)'\right).\]The map in \eqref{map} equals to the composition  \begin{equation}\label{compmap}\homology_{n-1}(\rel'_n(p);\Q) \overset{\partial}\longrightarrow \redhom_{n-2}(\Pb_n^\pm(\F_p)\backslash \SBDA(\F_p^n)';\Q)\overset{\Psi_*}\longrightarrow \redhom_{n-2}(\Pb_n^\pm(\F_p)\backslash \T_n^{\pm}(\F_p);\Q)\end{equation}
    where the maps are:
    \begin{itemize}
        \item $\partial$ is the boundary map in the long exact sequence of a pair in the homology.
        \item $\Psi\colon \mathcal{P}((\Pb_n^\pm(\F_p)\backslash \SBDA(\F_p^n)')\longrightarrow \Pb_n^\pm(\F_p)\backslash \T_n^\pm(\F_p)$ is the poset map taking the orbit of $\sigma= \left(v_0^\pm,\dots,v_k^\pm\right)$ in $\SBDA(\F_p^n)'$ to $(U_k,\pm \omega_k)$ if $e_1\in \operatorname{span}\{v_0,\dots,v_k\}$ and to $(V_k,\pm \omega_k')$ otherwise.
    \end{itemize} 
\end{lemma}
\begin{proof}
      We recall that \[
U_k=\operatorname{span}\{ e_1,\dots,e_k\}, 
\quad
V_k=\operatorname{span}\{ e_2,\dots,e_{k+1}\},
\]
and $\omega_k=e_1\wedge\cdots\wedge e_k$, $\omega_k'=e_2\wedge\cdots\wedge e_{k+1}$. We also recall from \autoref{Titsbldg:dq} that \[\Gamma_{0,n}^\pm(p)\backslash \T_n(\Q) \cong \T_n^\pm(\F_p).\] 
    We showed in \autoref{isoSt2} that \[(\St_n(\Q)\otimes\Q)_{\Gamma_{0,n}^\pm(p)} \cong \homology_{n-1}(\rel'_n(p);\Q).\] 
    Under this identification, the map in \eqref{map} agrees with the map in \eqref{compmap}.
\end{proof}
\begin{lemma}\label{mapcomp}
    For all primes $p$ and $n\geq 2$, let \[\rel'_n(p):=\left(\Pb_n^\pm(\F_p)\backslash \SBDA(\F_p^n),\Pb_n^\pm(\F_p)\backslash \SBDA(\F_p^n)'\right).\]The map \[\homology_{n-1}(\rel'_n(p);\Q) \overset{\partial}\longrightarrow \redhom_{n-2}(\Pb_n^\pm(\F_p)\backslash \SBDA(\F_p^n)';\Q)\] always surjective, and is injective if $p\in\{2,3,5,7,13\}$ or $p\leq 6n-8$. 
\end{lemma}

\begin{proof}
 Consider the long exact sequence of the pair $\rel'_n(p)$:
    \begin{align*}\redhom_{n-1}(\Pb_n^\pm(\F_p)\backslash \SBDA(\F_p^n);\Q) \longrightarrow &\homology_{n-1}(\rel'_n(p);\Q) \\&\overset{\partial}\longrightarrow\redhom_{n-2}((\Pb_n^\pm(\F_p)\backslash \SBDA(\F_p^n)';\Q) \longrightarrow\redhom_{n-2}(\Pb_n^\pm(\F_p)\backslash \SBDA(\F_p^n);\Q).\end{align*} By \autoref{BDA,p=2} and \autoref{BA connected}, we have that \[\redhom_{n-2}(\Pb_n^\pm(\F_p)\backslash \SBDA(\F_p^n);\Q)\cong0 \]for all primes $p$ and all $n\geq 2$, which implies that $\partial$ is surjective. Furthermore, \autoref{hom:BA} states that \[\redhom_{n-1}(\Pb_n^\pm(\F_p)\backslash \SBDA(\F_p^n);\Q)\cong 0\] for $p\in\{2,3,5,7,13\}$ or $p\leq 6n-8$.  Therefore, the result follows.
\end{proof}
\begin{proposition}\label{hom(Psi)}
      For all primes $p$ and $n\geq 2$, let \[\Psi\colon \mathcal{P}((\Pb_n^\pm(\F_p)\backslash \SBDA(\F_p^n)')\longrightarrow \Pb_n^\pm(\F_p)\backslash \T_n^\pm(\F_p)\]be the poset map defined in \eqref{compmap}. Then for all $k$,      \[\homology_k(\Psi;\Q)\cong \redhom_{k-1}\left(\Pb_n^\pm(\F_p)\backslash \SBDA(\F_p^{n-1});\Q\right).\]
\end{proposition}
\begin{proof}
      We consider the reduced map-of-posets spectral sequence of the map \[\Psi\colon \mathcal{P}(\Pb_n^\pm(\F_p)\backslash \SBDA(\F_p^n)')\longrightarrow \Pb_n^\pm(\F_p)\backslash \T^\pm_n(\F_p).\] 
    \autoref{red spsq} says that there exists a spectral sequence \[E_{kh}^2=\homology_k(\Pb_n^\pm(\F_p)\backslash \T^\pm_n(\F_p);[V\mapsto\redhom_h(\Psi_{\leq V};\Q)])\Rightarrow \homology_{k+h+1}(\Psi;\Q).\]
    We first consider the case $V \in \Pb_n^\pm(\F_p)\backslash \T^\pm_n(\F_p)$ with $e_1\not\in V$. For each $0\leq k\leq \dim V$, set \[Z_k=\left\{\Pb_n^\pm(\F_p)\sigma \in\Pb_n^\pm(\F_p)\backslash \SBDA(\F_p^n)'_k~\middle\vert\begin{array}{c}\sigma\Sigma_{k+1}=B\cdot T_k\text{ for some $B\in \D_1^{k,n}$}\\ \text{or $\sigma\Sigma_{k+1}=B\cdot G_k$ for some $B\in\DA_1^{k,n}$}\end{array}\right\},\]and $Z:=\{Z_k\}_{0\leq k\leq \dim V}$ considered as a subcomplex of  $\Pb_n^\pm(\F_p)\backslash \SBDA(\F_p^n)'$. Then\[\Psi_{\leq V}\cong \mathcal{P}(Z).\] It follows by \autoref{S1} and \autoref{DA1}, that for every $1\leq k\leq \dim V$, the action of $\Pb_n^\pm(\F_p)$ on every simplex representing an orbit in $Z_k$ is orientation-reversing. Therefore, \[\redhom_h(\Psi_{\leq V};\Q)\cong 0\quad\text{for all $h$}.\] 
    Now we consider the other case that $V \in \Pb_n^\pm(\F_p)\backslash \T^\pm_n(\F_p)$ with $e_1\in V$, we have \[\Psi_{\leq V}\cong \Pb^\pm_{\dim V}(\F_p)\backslash \SBDA(\F_p^{\dim V}),\]and so
    \[\redhom_h(\Psi_{\leq V};\Q)\cong \redhom_h(\Pb_{\dim V}^\pm(\F_p)\backslash\SBDA(\F_p^{\dim V});\Q).\] Thus, we fix $V \in \Pb_n^\pm(\F_p)\backslash \T^\pm_n(\F_p)$ with $e_1\in V$. We have from \autoref{BA connected} that $\Pb_{\dim V}^\pm(\F_p)\backslash\SBDA(\F_p^{\dim V})$ is $(\dim V-2)$-acyclic. It implies that \[\redhom_h(\Psi_{\leq V};\Q)) \cong 0\quad\text{for all}~ \dim V \neq h, h +1.\] 
    Consider the short exact sequence of functors
    \begin{equation}\label{functorseq}0 \longrightarrow B_h \longrightarrow \left(V\mapsto\redhom_h(\Psi_{\leq V};\Q) \right)\longrightarrow A_h \longrightarrow 0,\end{equation} where
    \[A_h(V) = \begin{cases}
        \redhom_h(\Psi_{\leq V};\Q) &~\text{if}~\dim V =h,\\
        0&~\text{otherwise}.
    \end{cases}\]and  \[B_h(V) = \begin{cases}
        \redhom_h(\Psi_{\leq V};\Q) &~\text{if}~\dim V =h+1,\\
        0&~\text{otherwise}.
    \end{cases}\]
     Now using \autoref{lem: Charney}, we have
    \[\homology_k\left(\Pb_n^\pm(\F_p)\backslash \T_n^\pm(\F_p);A_h\right) = \bigoplus\limits_{\substack{V\in \Pb_n^\pm(\F_p)\backslash \T_n^\pm(\F_p)\\ \dim V=h\\e_1\in V}} \redhom_{k-1}\left(\left(\Pb_n^\pm(\F_p)\backslash \T_n^\pm(\F_p)\right)_{>V};\redhom_h(\Psi_{\leq V};\Q)\right),\]and \[\homology_k\left(\Pb_n^\pm(\F_p)\backslash \T_n^\pm(\F_p);B_h\right) = \bigoplus\limits_{\substack{V\in \Pb_n^\pm(\F_p)\backslash \T_n^\pm(\F_p)\\\dim V=h+1\\e_1\in V}} \redhom_{k-1}\left(\left(\Pb_n^\pm(\F_p)\backslash \T_n^\pm(\F_p)\right)_{>V};\redhom_h(\Psi_{\leq V};\Q)\right).\]
    Since $e_1\in V$, then as indicated in Figure \ref{fig2}, the order complex $\oc\left(\Pb_n^\pm(\F_p)\backslash \T_n^\pm(\F_p))_{>V}\right)$ is isomorphic to $\partial_{n-\dim V-2}$. Thus, it is contractible unless $\dim V=n-1$. In fact, $\oc\left(\left(\Pb_n^\pm(\F_p)\backslash \T_n^\pm(\F_p)\right)_{>V}\right)$ is empty if $\dim V=n-1$.
\begin{figure}[H]
\centering
\begin{tikzpicture}[node distance=0.8cm and 0.3cm]

  % Left column: bullet and text nodes
  \node (A1) {\tiny$\bullet$};
  
  \node (B1) [below=of A1] {\tiny$\bullet$}; 
  
  \node (C1) [below=of B1] {\tiny$\bullet$};

   \node (D1) [below=-0.2cm of C1]
   {$\vdots$};
   
  \node (E1) [below=of D1] {\tiny$\bullet$};
  
  \node (F1) [below=of E1] {\tiny$\bullet$};
  
  \node (G1) [below=-0.1cm of F1] {};
  \node[left=of G1] (G1t) {Orbits not containing $e_1$};
  
  % Right column: bullet and text nodes
  \node (A2) [right=1cm of A1] {\tiny$\bullet$};
  
  \node (B2) [below=of A2] {\tiny$\bullet$}; 

  \node (C2) [below=of B2] {\tiny$\bullet$};
  \node[right=-0.1cm of C2] (C2t) {$V$};
  
   \node (D2) [below=-0.2cm of C2]
   {$\vdots$};
   
  \node (E2) [below=of D2] {\tiny$\bullet$};

  \node (F2) [below=of E2] {\tiny$\bullet$};
  
  \node (G2) [below=-0.1cm of F2] {};
  \node[right=of G2] (G2t) {Orbits containing $e_1$};
  
  % Connect the bullet nodes with lines
  \draw (A1) -- (B1) -- (C1);
  \draw (A2) -- (B2) -- (C2);
  \draw (D2) -- (E2);
  \draw (D1) -- (E1);
  \draw (B1) -- (A2);
  \draw (C1) -- (B2);
  \draw (E1) -- (F1) -- (E2) -- (F2);

 \draw[red]  ellipse[x radius=0.3cm,y radius=1cm, color=red, fit=(A2) (B2) , inner sep=2mm];

\end{tikzpicture}
\caption{$\Pb_n^\pm(\F_p)\backslash \T^\pm_n(\F_p)$, with $(\Pb_n^\pm(\F_p)\backslash \T^\pm_n(\F_p))_{>V}$ enclosed in red.}\label{fig2}\end{figure}
   This implies that 
    \[\homology_k\left(\Pb_n^\pm(\F_p)\backslash \T_n^\pm(\F_p);A_h\right) \cong 0\quad\text{unless $(k,h)=(0,n-1)$},\]and
    \[\homology_k\left(\Pb_n^\pm(\F_p)\backslash \T_n^\pm(\F_p);B_h\right)\cong0\quad\text{unless $(k,h)=(0,n-2)$}.\]Thus, the only potential nonzero terms in our spectral sequence are \[E^2_{0,n-1}\quad\text{and}\quad E^2_{0,n-2},\]where 
\[E^2_{0,n-1}\cong \homology_0\left(\Pb_n^\pm(\F_p)\backslash \T_n^\pm(\F_p);A_{n-1}\right)\cong \bigoplus_{\substack{V\in \Pb_n^\pm(\F_p)\backslash \T_n^\pm(\F_p)\\\dim V=n-1\\e_1\in V}} \redhom_{n-1}(\Pb_{n-1}^\pm(\F_p)\backslash\SBDA(\F_p^{n-1});\Q)\]
    and
    \[E^2_{0,n-2}\cong \homology_0\left(\Pb_n^\pm(\F_p)\backslash \T_n^\pm(\F_p);B_{n-2}\right)\cong \bigoplus_{\substack{V\in \Pb_n^\pm(\F_p)\backslash \T_n^\pm(\F_p)\\\dim V=n-1\\e_1\in V}} \redhom_{n-2}(\Pb_{n-1}^\pm(\F_p)\backslash\SBDA(\F_p^{n-1});\Q).\]
Since \autoref{tits} says that there is exactly one orbit of $\T_n^\pm(\F_p)$ of dimension $n-1$ containing the vector $e_1$, it implies that
\[E^2_{0,n-1}\cong \redhom_{n-1}(\Pb_{n-1}^\pm(\F_p)\backslash\SBDA(\F_p^{n-1});\Q),\]
and
\[E^2_{0,n-2}\cong\redhom_{n-2}(\Pb_{n-1}^\pm(\F_p)\backslash\SBDA(\F_p^{n-1});\Q).\]
  Thus, \begin{equation*}
  \homology_k(\Psi;\Q)\cong\begin{cases}
      \redhom_{n-1}(\Pb_{n-1}^\pm(\F_p)\backslash\SBDA(\F_p^{n-1});\Q)&\text{if $k=n$}\\
      \redhom_{n-2}(\Pb_{n-1}^\pm(\F_p)\backslash\SBDA(\F_p^{n-1});\Q)&\text{if $k=n-1$}\\
      0&\text{otherwise}.
  \end{cases}\end{equation*}Furthermore, we know from \autoref{BA connected} that \[\redhom_{k}(\Pb_{n-1}^\pm(\F_p)\backslash\SBDA(\F_p^{n-1});\Q)\cong0\quad\text{for all $k\leq n-3$}.\] This completes the proof.
\end{proof}
\begin{corollary}\label{iso, BA,BA'}
Let $p$ be a prime and $n\geq 3$. Then
    \[\redhom_{k}\left(\Pb_n^\pm(\F_p)\backslash \SBDA(\F_p^n)';\Q\right)\cong  \redhom_{k}(\Pb_{n-1}^\pm(\F_p)\backslash\SBDA(\F_p^{n-1});\Q).\]
\end{corollary}
\begin{proof}
We have by \autoref{tits:acyclic} that for all $n\geq 3$, $\oc\left(\Pb_n^\pm(\F_p)\backslash \T_n^\pm(\F_p)\right)$ is contractible. It implies by \eqref{les}, that for all $n\geq 3$\[\redhom_{k}\left(\Pb_n^\pm(\F_p)\backslash \SBDA(\F_p^n)';\Q\right)\cong \homology_{k+1}(\Psi;\Q).\]
The result thus follows by \autoref{hom(Psi)}.
\end{proof}
We now proceed to prove \autoref{thC}. 

\begin{proof}[Proof of \autoref{thC}]
    By \autoref{lem:compmap} we have that the map in \eqref{map} equals to the composition\begin{equation}\label{spsq}\homology_{n-1}(\rel'_n(p);\Q) \overset{\partial}\longrightarrow \redhom_{n-2}(\Pb_n^\pm(\F_p)\backslash \SBDA(\F_p^n)';\Q)\overset{\Psi_*}\longrightarrow \redhom_{n-2}(\Pb_n^\pm(\F_p)\backslash T_n^{\pm}(\F_p);\Q),\end{equation} where $\partial$ is the boundary map from the long exact sequence of the pair \[\rel'_n(p)=\left(\Pb_n^\pm(\F_p)\backslash \SBDA(\F_p^n),\Pb_n^\pm(\F_p)\backslash \SBDA(\F_p^n)'\right),\] and $\Psi$ is the poset map defined in that lemma. 
    Moreover, \autoref{mapcomp} states that $\partial$ is always surjective, and injective if $p\in\{2,3,5,7,13\}$ or $p\leq 6n-8$. 
    
    Since the order complex $\oc\left(\Pb_n^\pm(\F_p) \backslash \T_n^\pm(\F_p)\right)$ is contractible for $n \geq 3$ and all primes $p$ (by \autoref{tits:acyclic}), it follows that $\Psi_*$ is surjective.
    Thus, it remains to show that $\Psi_*$ is injective when $p\in\{2,3,5,7,13\}$ or $p\leq 6n-14$. 
  
  Consider the following the long exact sequence
 \begin{equation}\label{les}\dots \longrightarrow \homology_{n-1}(\Psi;\Q)\longrightarrow\redhom_{n-2}(\mathcal{P}(\Pb_n^\pm(\F_p)\backslash \SBDA(\F_p^n)');\Q)\overset{\Psi_*}{\longrightarrow} \redhom_{n-2}(\Pb_n^\pm(\F_p)\backslash \T_n^\pm(\F_p);\Q)\longrightarrow\dots\end{equation} It implies by \autoref{hom(Psi)} that $\Psi_*$ is injective if \[ \redhom_{n-2}(\Pb_{n-1}^\pm(\F_p)\backslash\SBDA(\F_p^{n-1});\Q)\cong 0.\]
By \autoref{hom:BA}, this term vanishes if \[p\in\{2,3,5,7,13\}\] or \[p\leq 6(n-1)-8=6n-14.\] Thus, $\Psi_*$ is injective if $p\in\{2,3,5,7,13\}$ or $p\leq 6n-14$. 
This completes the proof.\end{proof}

\subsection{The proof of \autoref{thB}}
We recall that \autoref{thB} states that $\left(\St_n(\Q)\otimes\Q \right)_{\Gamma^+_{0,n}(p)}$ does not vanish for $n = 2$ for every prime $p$, and for $n=3$ for all primes $p \notin \{2,3,5,7,13\}$.

\begin{proof}
We begin by recalling the isomorphism from \eqref{rel:SL,GL}:
 \[\left(\St_n(\Q)\otimes\Q\right)_{\Gamma_{0,n}^+(p)}\cong \left(\St_n(\Q)\otimes\Q\right)_{\Gamma_{0,n}^\pm(p)}\oplus\left(\St_n(\Q)\otimes\Q^{\det}\right)_{\Gamma_{0,n}^\pm(p)}.\] 
 To prove non-vanishing of the left-hand side, it suffices to show that at least one of the summands on the right does not vanish.
We already showed in \autoref{dim: St2} that
\[\left(\St_2(\Q)\otimes\Q\right)_{\Gamma_{0}^\pm(p)}\ncong 0\]for all primes $p$. So the case $n=2$ is settled.

Now let $n = 3$. By \autoref{n=2,BA}, we have
\[\redhom_{1}(\Pb_{2}^\pm(\F_p)\backslash\SBDA(\F_p^{2});\Q)\ncong 0\quad\text{for $p\not\in\{2,3,5,7,13\}$}.\]
By \autoref{iso, BA,BA'}, we have an isomorphism \[\redhom_{1}\left(\Pb_3^\pm(\F_p)\backslash \SBDA(\F_p^3)';\Q\right)\cong  \redhom_{1}(\Pb_{2}^\pm(\F_p)\backslash\SBDA(\F_p^{2});\Q).\]It follows that $\redhom_{1}\left(\Pb_3^\pm(\F_p)\backslash \SBDA(\F_p^3)';\Q\right)$ also does not vanish for $p\not\in\{2,3,5,7,13\}$.

Moreover, by \autoref{mapcomp}, the map \[\homology_{2}(\rel'_3(p);\Q) \overset{\partial}\longrightarrow \redhom_{1}(\Pb_3^\pm(\F_p)\backslash \SBDA(\F_p^3)';\Q)\] is always surjective. Since $\left(\St_3(\Q)\otimes\Q\right)_{\Gamma^\pm_{0,3}(p)}\cong\homology_{2}(\rel'_3(p);\Q)$ by \autoref{isoSt2}, we conclude that
\[\left(\St_3(\Q)\otimes\Q \right)_{\Gamma^\pm_{0,3}(p)}\ncong 0\quad\text{for $p\not\in\{2,3,5,7,13\}$}.\qedhere\]\end{proof}

\printbibliography
\end{document}